\documentclass[11pt,twoside]{article} 

\usepackage{eqnarray,amsmath}
\usepackage[utf8]{inputenc} 
\usepackage[T1]{fontenc}    
\usepackage{booktabs}       
\usepackage{amsfonts}       
\usepackage{nicefrac}       
\usepackage{microtype}      
\usepackage{xcolor}         
\usepackage{subcaption}
\expandafter\def\csname 
ver@subfig.sty\endcsname{}
\usepackage{eqnarray,amsmath}
\usepackage{epsf}
\usepackage{epsfig}
\usepackage{fancyhdr}
\usepackage{graphics}
\usepackage{graphicx}
\usepackage{psfrag}
\usepackage{fullpage}
\usepackage{pdfpages}

\usepackage{url}

\usepackage{color}

\usepackage{amsthm}
\usepackage{amsmath}
\usepackage{amssymb,bbm}
\usepackage[skip=0pt]{caption}  

\usepackage{textcomp}
\usepackage{siunitx}
\usepackage{wrapfig}
\usepackage{algorithm}

\usepackage{multirow}
\usepackage{multicol}
\usepackage{makecell}
\usepackage{colortbl} 
\usepackage{changepage}
\usepackage{tocloft}
\addtocontents{toc}{\protect\setcounter{tocdepth}{-1}}

\definecolor{customyellow}{HTML}{fedf8a}
\usepackage{eqnarray,amsmath}
\usepackage{epsf}
\usepackage{epsfig}
\usepackage{fancyhdr}
\usepackage{graphics}
\usepackage{graphicx}
\usepackage{psfrag}
\usepackage{fullpage}
\usepackage{pdfpages}
\usepackage{ragged2e}
\usepackage{comment}

\newtheorem{remark}{Remark}

\newtheorem{assumption}{Assumption}

\newtheorem{lemma}{Lemma}

\newtheorem{theorem}{Theorem}
\newtheorem{proposition}{Proposition}
\newtheorem{definition}{Definition}
\newtheorem{corollary}{Corollary}

\usepackage{eqnarray,amsmath}
\usepackage{booktabs}       
\usepackage{nicefrac}       
\renewcommand{\arraystretch}{1.35}

\usepackage{subfig}
\usepackage{epsf}
\usepackage{epsfig}
\usepackage{fancyhdr}
\usepackage{graphics}
\usepackage{graphicx}
\usepackage{epstopdf}
\usepackage{psfrag}
\usepackage{fullpage}
\usepackage{pdfpages}
\usepackage{ragged2e}

\usepackage{enumerate}   
\usepackage{multirow}
\usepackage{bm}

\usepackage{mathtools}

\usepackage{url}
\usepackage[colorlinks,linkcolor=black,citecolor=blue, pagebackref=true]{hyperref}
\renewcommand*{\backrefalt}[4]{%
    \ifcase #1 \footnotesize{(Not cited.)}%
    \or        \footnotesize{(Cited on page~#2.)}%
    \else      \footnotesize{(Cited on pages~#2.)}%
    \fi}

\usepackage{color}

\usepackage{amsthm}
\usepackage{amsmath}
\usepackage{amssymb,bbm}
\usepackage{caption}
\usepackage{algorithmic}
\usepackage{algorithm}
\usepackage{textcomp}
\usepackage{siunitx}
\usepackage{wrapfig}
\usepackage{algorithmic}
\usepackage{algorithm}
\usepackage{multirow}
\usepackage{multicol}
\usepackage{mathtools}

\def\1{\bm{1}}

\DeclareMathAlphabet{\mathsfit}{\encodingdefault}{\sfdefault}{m}{sl}
\SetMathAlphabet{\mathsfit}{bold}{\encodingdefault}{\sfdefault}{bx}{n}

\DeclareMathOperator*{\argmax}{arg\,max}

\usepackage{eqnarray,amsmath}

\usepackage{amsthm}
\usepackage{amsmath}
\usepackage{amssymb,bbm}

\allowdisplaybreaks

\begin{document}

\begin{center}

{\bf{\Large{On the Geometry of Separation in Finite Gaussian Mixtures}}}
  
\vspace*{.2in}
{\large{
\begin{tabular}{cccccc}
Huy Nguyen$^{\star}$ & Dung Le$^{\star}$ & Alessandro Rinaldo & Nhat Ho 
\end{tabular}
}}

\vspace*{.2in}

\begin{tabular}{c}
Department of Statistics and Data Sciences\\
The University of Texas at Austin
\end{tabular}

\vspace*{.2in}
\today

\begin{abstract}
    We study an open problem of understanding the effects of the minimum component separation on the convergence rates of parameter estimation in finite Gaussian mixtures. We address this by developing a unified geometric framework based on novel Hellinger lower bounds that directly relate discrepancies between mixture densities directly to Wasserstein distances between their underlying mixing measures, with explicit dependence on both the minimum separation and the minimum weight. Our approach combines carefully designed interpolation polynomials with confluent divided difference techniques to construct specialized moment-extraction test functions. When the number of components is known, these bounds uncover a localization phenomenon: the separation complexity is driven strictly by the spatial configuration of mixture components, namely, whether they are concentrated in a single cluster, partitioned into multiple clusters separated by a macroscopic gap, or arranged without any structural constraints. On the other hand, when the number of components becomes unknown and is over-specified, the separation complexity is slightly reduced, while the minimum mixture weight disappears entirely from the convergence rates due to a transition from first-order to second-order Wasserstein geometry. As a consequence, we obtain separation-dependent convergence rates that continuously interpolate between point-wise and uniform estimation regimes, thereby settling the fundamental limits of parameter recovery in finite Gaussian mixtures.
\end{abstract}
\end{center}
\let\thefootnote\relax\footnotetext{$\star$ Equal contribution.}

\section{Introduction}
\label{sec:introduction}
Finite mixture models (FMMs) \cite{McLachlan2004,Lindsay-1995} provide a principled and flexible approach for modeling heterogeneous data generated from multiple latent subpopulations. Rather than assuming that all observations arise from a single homogeneous distribution, mixture models posit that each data point is drawn from one of several underlying components, with the component identity treated as an unobserved (latent) variable. In particular, the data $X_1,X_2,\ldots, X_n \in\mathcal{X}\subseteq \mathbb{R}^d$ are assumed to be generated from a probability density function of the form
\begin{align}
    \label{eq:pdf_fmm}
    p_{G_*}(x) := \int_{\Theta} f(x \mid \theta)dG_*(\theta)=\sum_{i=1}^{k_*}\pi^*_i f(x\mid \theta^*_i),
\end{align}
where $G_*:=\sum_{i=1}^{k_*}\pi^*_i\delta_{\theta^*_i}$ is a probability mixing measure with $k_*$ atoms $\theta^*_{i}$ belonging to a parameter space $\Theta\subseteq\mathbb{R}^d$ and the mixing proportions $0\leq\pi_i^*\leq 1$ satisfying $\sum_{i=1}^{k_*}\pi_i^*=1$. Meanwhile, $f(\cdot \mid \theta)$ denotes a probability density kernel parameterized by $\theta\in\Theta$.
This formulation allows FMMs to naturally accommodate complex distributional features such as multi-modality, skewness, and varying dispersion, making mixture models a powerful tool for density estimation \cite{ghosal2001entropy} and clustering \cite{maitra2010clustering}. As a result, they have been adopted in a wide range of disciplines, including machine learning \cite{bishop2006pattern,murphy2012machine}, biology \cite{Patra2016testing}, economics \cite{compiani2016econometric}, and physical sciences \cite{Kuusela2012physics}.

\textbf{Density estimation and parameter estimation.} A fundamental problem in mixture models is the analysis of density estimation, namely, understanding how well a density estimator $p_{\widehat{G}_n}$ approximates the true data-generating density $p_{G_*}$. This perspective is natural, as the primary goal of statistical modeling is often to accurately capture the distribution of the observed data. Thanks to the flexibility of mixture models, a rich body of literature has established consistency and convergence rates for density estimators under various metrics, including Hellinger distance \cite{Vandegeer,Birge1986hellinger}, Total Variation distance \cite{Ho-Nguyen-Ann-16,li2017robust}, and Kullback-Leibler divergence \cite{nguyen2016latentmixing,LiBarron1999}. This line of work provides a comprehensive study on how quickly the estimated density approaches the truth as the sample size increases.

Beyond density estimation, another important yet more challenging problem in mixture models is parameter estimation \cite{kalai2010efficient,heinrich2018minimax,liu2023robust,moitra2010polynomial}, namely, recovering the underlying mixing measure $G_*$, which encodes the latent structure of the model. 
A key development for the parameter estimation problem in FMMs is the work of Nguyen \cite{nguyen2016latentmixing}, which establishes a fundamental connection between optimal transport \cite{Villani-03,Villani-09} and the convergence of latent mixing measures. In this framework, the discrepancies between mixing measures are quantified using Wasserstein distances, providing a natural and geometrically meaningful way to compare discrete or non-overlapping supports. More specifically, for any two probability mixing measures $G=\sum_{i=1}^{k}\pi_i\delta_{\theta_i}$ and $G'=\sum_{i=1}^{k'}\pi'_i\delta_{\theta'_i}$, the $r$-Wasserstein distance with the Euclidean norm between $G$ and $G'$ is defined as
\begin{align*}
    W_r(G,G'):=\left(\inf\sum_{i,j}q_{ij}\|\theta_i-\theta'_j\|_2^r\right)^{1/r},
\end{align*}
where the infimum is taken over all couplings $(q_{ij})_{ij}\in[0,1]^{k\times k'}$ such that $\sum_{i=1}^{k}q_{ij}=\pi'_j$, for any $1\leq j\leq k'$, and $\sum_{j=1}^{k'}q_{ij}=\pi_i$, for any $1\leq i\leq k$. Lemma 1 in \cite{nguyen2016latentmixing} indicates that, under mild continuity assumptions on the density kernel $f$,
the convergence of mixing measures implies that of densities through the inequality $d_{H}(p_G,p_{G_*}) \lesssim W_1(G, G_*)$, where $d_{H}(p_{G},p_{G_*}):=({\frac{1}{2}} \int_{\mathcal{X}}(\sqrt{p_{G}(x)}-\sqrt{p_{G_*}(x)})^2dx)^{1/2}$ stands for the Hellinger distance between $p_{G}$ and $p_{G_*}$. 
On the other hand, convergence at the density level does not necessarily transfer to convergence at the parameter level. In fact, significantly different mixing measures may only reflect in small discrepancies between the mixture densities. Indeed, it is well known that  distinct mixing measures may generate nearly indistinguishable densities, especially when components $\theta^*_i$ overlap or are weakly separated \cite{nguyen2016latentmixing}. In response to this problem, Ho and Nguyen \cite{Ho-Nguyen-EJS-16} demonstrated that, under appropriate regularity and strong identifiability conditions on the family of probability density kernels $\{f(\cdot\mid\theta):\theta\in\Theta)\}$, the Hellinger distance between densities is bounded below by the Wasserstein distance between mixing measures; that is,
\begin{align}
    \label{eq:old_Hellinger_bound}
    d_{H}(p_G,p_{G_*}) \geq C(G_*)W_1(G, G_*),
\end{align} 
where $C(G_*)$ is a positive quantity that depends on $G_*$. Notably, Ho and Nguyen \cite{Ho-Nguyen-EJS-16} combine the inequality \eqref{eq:old_Hellinger_bound} with empirical process results for density estimation (see, e.g., Chapter 7 of van de Geer \cite{Vandegeer}) to demonstrate a parametric estimation rate for the latent mixing measure in the distance $W_1$. In detail, the authors show that 
\begin{equation}\label{eq:MLE.rate}
W_1(\widehat{G}_n,G_*)=\mathcal{O}_P(\sqrt{\log(n)/n}),
\end{equation}
where $\widehat{G}_n$ is the maximum likelihood estimator (MLE) defined as
\begin{align}
    \label{eq:MLE}
    \widehat{G}_n:=\sum_{i=1}^{\widehat{k}_n}\widehat{\pi}_{n,i}\delta_{\widehat{\theta}_{n,i}}=\argmax_{G}\frac{1}{n}\sum_{i=1}^{n}\log(p_{G}(X_i)).
\end{align}
Above, when the true number of components $k_*$ is known, the $\argmax$ operator is subjected to $G\in\mathcal{E}_{k_*}(\Theta)$, where $\mathcal{E}_{k_*}(\Theta):=\{G=\sum_{i=1}^{k_*}\pi_i\delta_{\theta_i}: \theta\in\Theta\}$ stands for the set of mixing measures with $k_*$ atoms. On the other hand, when the true number of experts $k_*$ is unknown and over-specified by $k$, then the $\argmax$ operator is rather subject to $G\in\mathcal{G}_{k}(\Theta)$ with
$\mathcal{G}_{k}(\Theta):=\{G=\sum_{i=1}^{k'}\pi_i\delta_{\theta_i}:1\leq k'\leq k, \ \theta\in\Theta\}$ denoting the set of mixing measures with at most $k$ atoms, where $k>k_*$. Returning to the MLE rate, despite bridging a major gap between density estimation and parameter estimation, the constant $C(G_*)$ in the Hellinger lower bound in equation~\eqref{eq:old_Hellinger_bound} is left implicitly in the parameter estimation rate in equation~\eqref{eq:MLE.rate}. As result, the above MLE rate does not appropriately account for two critical factors impacting the difficulty of the estimation task, namely, (i) \emph{the separation among mixture components} $\Delta_{\mathrm{sep}} = \min_{j \neq l} \|\theta_j^* - \theta_l^*\|$ and (ii) \emph{the minimum mixture weight} $\pi_{\text{min}}^{*} = \min_{1 \leq j \leq k_{*}} \pi_{j}^{*}$. 
Therefore, we will fill in these gaps in this paper.


\textbf{The roles of separation and minimum weight.} Firstly, the separation $\Delta_{\mathrm{sep}}$ quantifies how distinguishable the latent subpopulations are from each other \cite{vempala2002spectral,arora2001gaussian}. When the separation is sufficiently large, FMMs exhibit several favorable statistical and computational properties. In particular, it improves the identifiability of the model, as each component contributes a clearly distinguishable signature to the observed density \cite{Teicher-63, Yakowitzspragins-1968}. This reduces ambiguity in assigning observations to latent subpopulations and leads to more stable parameter estimation and faster convergence rates \cite{Chen1992, Ho-Nguyen-Ann-16,wu2020moment}. 
From a computational perspective, optimization procedures, such as the EM algorithm \cite{dempster1977maximum}, rely on separation conditions in the analysis \cite{Kwon_minimax_EM,redner2016em} and often perform more reliably in well-separated regimes due to the reduced overlap among mixture components \cite{Xu_Jordan-1995,Siva_2017,hopkins2018robust,kwon2020wellseparated}. 
In contrast, when mixture components become weakly separated, their corresponding densities overlap substantially, making it increasingly difficult to distinguish individual components from the observed data. This phenomenon often results in (nearly) singular statistical structures, slower convergence rates, and increased sensitivity of estimators to perturbations in the data \cite{ho2022weakseparation}. At the same time, the minimum mixture weight $\pi^*_{\min}$ determines how much statistical information is available for recovering each latent subpopulation \cite{dasgupta1999gaussian}. If it is large enough, every component contributes a meaningful number of observations to the sample, making the corresponding parameters easier to identify and estimate accurately. Otherwise, the associated components become difficult to detect from finite samples \cite{rousseau2011mixture}, often leading to pathological behavior as mentioned in the case of weakly-separated mixtures.
To summarize, the statistical and computational difficulty of estimating the mixing measure is heavily impacted by the separation and minimum weight parameters, both of which should be explicitly accounted for.  

\textbf{Goal of the article and contributions.} 
The primary goal of this work is to provide a precise characterization of the convergence rate of the maximum likelihood estimator $\widehat{G}_{n}$ to the true mixing measure $G_{*}$ in terms of the separation among the components and the smallest weights of $G_{*}$ for finite  high-dimensional Gaussian location mixtures. Specifically, we will study mixture densities of the form 
\begin{align*}
    p_{G_*}(x) := \int_{\Theta} f(x \mid \theta, \Sigma)G_*(d\theta)=\sum_{i=1}^{k_*}\pi^*_i f(x\mid \theta^*_i,\Sigma),
\end{align*}
where the $\pi_i^*$'s are positive probability weights and $\{f(\cdot\mid\theta,\Sigma):\theta\in\Theta\}$ denotes the family of multivariate Gaussian density functions with {\it known, positive-definite} covariance matrix $\Sigma$ and mean vector $\theta$ in a compact subset $\Theta$ of $\mathbb{R}^d$. We seek to provide estimation rates for the mixing measure $G_*$ in the Wasserstein distance using the MLE $\widehat{G}_n$ from equation \eqref{eq:MLE}.
Gaussian mixtures are among the most widely studied and practically important classes of FMMs due to their flexibility and strong approximation capabilities \cite{Mclachlan-1988},  
thereby offering a natural blueprint for a refined analysis of estimation rates in FMMs. Though the study of finite Gaussian mixtures is a classic and heavily studied topic in the statistical literature, minimax estimation rates for the mixing measure in high dimensions have only been obtained recently by Doss et al. \cite{doss_optimal_2023}, without assuming any separation conditions on the mixing parameters. In this paper, we develop a set of analytic tools for studying finite high-dimensional Gaussian location mixtures and derive complementary results to those of \cite{doss_optimal_2023}, yielding {\it local} estimation rates of the MLE $\widehat{G}_n$, in the sense of exhibiting an explicit dependence on the separation and the minimum weight parameters of the true mixing measure $G_*$.    
 
As is customary in the analysis of mixture models, we will consider two scenarios: the \emph{exact-specified setting,} in which the number of components $k_{*}$ is known, and the
 \emph{over-specified setting,} in which the number of components $k_{*}$ is unknown and we fit a mis-specified mixture model with  $k$ components, where $k$ is known to be larger than  $k_{*}$.
The exact-specified setting provides a baseline for understanding the role of component separation and minimum weight in parameter estimation rates, while the over-specified setting captures the practical scenario in which the true number of mixture components is unknown. Although these two settings exhibit fundamentally different geometric behaviors, we show that they can be analyzed within a unified framework based on separation-dependent Hellinger lower bounds.


It is worth noting that existing results by Doss et al. \cite{doss_optimal_2023} characterize the relationship between density estimation and parameter estimation in ways that do not account for the local properties of the mixing measure, delivering  minimax, worst-case uniform estimation rates over the space of mixing measures. In contrast, we exhibit an array of local convergence rates of the MLE that depends on the spatial arrangements and the minimal probability weights of the true mixing components. Towards that goal, we analyze three distinct geometric regimes, depending on how clustered the points over which mixing measure is supported are. 
\begin{itemize}
     \item \emph{Single-cluster regime:} We first study the regime where all true components are concentrated within a cluster whose diameter is proportional to the minimum separation $\Delta_{\mathrm{sep}}$. We establish local and global Hellinger lower bounds with explicit dependence on both $\Delta_{\mathrm{sep}}$ and the minimum mixture weight $\pi_{\min}^*$, yielding a convergence rate of order $\Delta_{\mathrm{sep}}^{-(2k_*-1)}(\pi_{\min}^*)^{-1}d^{1/2}n^{-1/2}$  
     (up to some logarithmic factor). Our analysis reveals that the estimation difficulty deteriorates polynomially as the components become increasingly concentrated, with the exponent $2k_*-1$ quantifying the severity of the resulting geometric singularity. Moreover, our analysis also provides a transition between the point-wise convergence rate for parameter estimation established in this paper and the existing uniform rate derived by Wu et al.~\cite{wu2020moment} and Doss et al. \cite{doss_optimal_2023} based on the method of moments \cite{pearson1894,Hardt2015mixture}. 
    \item \emph{Multi-cluster regime:} We next consider the regime where the true components are partitioned into several tightly concentrated clusters separated by a fixed macroscopic gap. In this setting, we uncover a localization phenomenon showing that the sample complexity of parameter estimation is no longer governed by the total number of components $k_*$, but rather by the size $s_{\max}$ of the densest cluster. More specifically, we establish Hellinger lower bounds whose dependence on the minimum separation improves from $\Delta_{\mathrm{sep}}^{2k_*-1}$ to $\Delta_{\mathrm{sep}}^{2s_{\max}-1}$, leading to substantially sharper convergence rates whenever the cluster structure is relatively balanced. This result demonstrates that favorable cluster structure can substantially reduce the impact of component overlap and lead to significantly faster convergence rates than those predicted by the worst-case single-cluster analysis.
    \item \emph{Unstructured regime:} Finally, we investigate the most general setting in which no assumptions are imposed on the cluster structure of the true components. We therefore derive uniform Hellinger lower bounds and corresponding MLE convergence rates that remain valid across all possible component arrangements. The resulting factor $\Delta_{\mathrm{sep}}^{4k_*-3}$ characterizes the worst-case impact of component overlap on parameter estimation and serves as a universal benchmark for finite Gaussian mixtures. This establishes a complete picture of the estimation problem by complementing the sharper, structure-dependent guarantees obtained in the single-cluster and multi-cluster regimes.
\end{itemize}
We similarly analyze the over-specified setting and show that over-specification fundamentally changes the geometry of the parameter estimation problem. As expected, the presence of additional fitted components causes a slower convergence rate for the MLE,  from $n^{-1/2}$ up to $n^{-1/4}$. At the same time, it also reduces the separation complexity in both the single-cluster and multi-cluster regimes, with the separation exponents decreasing from $2k_*-1$ to $2k_*-2$ and from $2s_{\max}-1$ to $2s_{\max}-2$, respectively. Interestingly, this phenomenon does not occur in the unstructured regime, where the worst-case exponent $4k_*-3$ remains unchanged. Moreover, the minimum mixture weight $\pi_{\min}^*$ disappears entirely from the convergence rates, indicating that the dominant source of sample complexity under over-specification arises from geometric ambiguities created by component separation rather than from low-contributing components.

In terms of technical contributions, we develop a new suite of techniques based on carefully designed interpolation polynomials and moment-extraction test functions. The key idea is to decompose the Wasserstein distance between mixing measures into a collection of geometric discrepancies and then connect each discrepancy to the Hellinger distance through polynomial identities. This approach allows us to derive explicit Hellinger lower bounds whose constants are characterized in terms of the minimum separation and the minimum mixture weight. We remark that, while interpolating polynomials and moment matching arguments are commonly used in the analysis of Gaussian mixture models, in our analysis, we have significantly expanded the scope and uses of these techniques.  


\begin{table}[ht]
\centering
\renewcommand{\arraystretch}{1.5}
\caption{Impacts of the separation and the minimum weight on the MLE convergence rates under three different geometric regimes. For ease of presentation, we omit the terms $[d\log(n)/n]^{1/2}$ in the exact-specified setting and $[d\log(n)/n]^{1/4}$ in the over-specified setting.}
\begin{tabular}{lccc}
\toprule
 & Single-cluster & Multi-cluster & Unstructured \\
\midrule
Exact-specification -- Local ($W_1$) 
&
$\Delta_{\mathrm{sep}}^{-(2k_*-1)}$
&
$\Delta_{\mathrm{sep}}^{-(2s_{\max}-1)}$
&
$\Delta_{\mathrm{sep}}^{-(4k_{*}-3)}$
\\
Exact-specification -- Global ($W_1$) 
&
$(\pi^*_{\min})^{-1}\Delta_{\mathrm{sep}}^{-(2k_*-1)}$
&
$(\pi^*_{\min})^{-1}\Delta_{\mathrm{sep}}^{-(2s_{\max}-1)}$
&
$(\pi^*_{\min})^{-1}\Delta_{\mathrm{sep}}^{-(4k_{*}-3)}$
\\
Over-specification ($W_2$) 
&
$\Delta_{\mathrm{sep}}^{-(k_*-1)}$
&
$\Delta_{\mathrm{sep}}^{-(s_{\max}-1)}$
&
$\Delta_{\mathrm{sep}}^{-\left(2k_{*}-\frac{3}{2}\right)}$
\\
\bottomrule
\end{tabular}
\end{table}

\section{Preliminaries}
\label{sec:preliminaries}
In this section, we first introduce the notation and assumptions used throughout the paper, and then present some preliminary results that will be leveraged in our subsequent analysis. 

\textbf{Notation.}  For any $n\in\mathbb{N}$, we let $[n] = \{1,2,\ldots,n\}$. For any vector $v \in \mathbb{R}^{d}$, we let $\|v\|$ represent the $\ell_2$-norm of $v$. The cardinality of a set $S$ is denoted with $|S|$. For a sequence $(A_n)_{n\geq 1}$ of positive random variables, the notation $A_{n} = \mathcal{O}_{P}(b_{n})$ signifies $A_{n}/b_{n}$ is stochastically bounded, that is, for any $\epsilon>0$, there exists an $M>0$ such that $\mathbb{P}( A_{n}/b_{n} > M) < \epsilon $, for all $n$ large enough. 
For any test polynomial $P$ on a space $\mathcal{S}\subset\Theta$, we define
$\|P\|_{\infty}:=\sup_{\substack{x\in\mathcal{S},\|x\|\leq R}}|P(x)|$,
where $R$ is the radius defined in Assumption (A.1). 
Lastly, for any two real numbers $a$ and $b$, we denote $a\vee b:=\max\{a,b\}$ and $a\wedge b:=\min\{a,b\}$.


\textbf{Assumptions.} Throughout the paper, we will make the following universal assumptions unless stating otherwise:

\emph{(A.1) The parameter space $\Theta$ is compact. Additionally, $\|\theta_{i}^{*}\| \leq R$, for all $1 \leq i \leq k_{*}$.}

\emph{(A.2) $\Sigma\in\mathbb{S}_d^+(\lambda_{\min},\lambda_{\max})$, where $\mathbb{S}_d^+(\lambda_{\min},\lambda_{\max})$ denotes the set of symmetric positive-definite
matrices with eigenvalues bounded below by $\lambda_{\min}$ and above by $\lambda_{\max}$.}

\emph{(A.3) The separation and minimum weight are positive, that is, $\Delta_{\mathrm{sep}},\pi_{\min}^*>0$.}

Next, let us present some auxiliary results that serve as the foundation of our convergence analysis. We recall that the main results of these papers are a series of inequalities to upper bound the Wesserstein distance between mixing measures by multiplicative factors of the Hellinger distance between the corresponding mixture densities, so that parameter estimation rates can be directly derived from density estimation rates. The following proposition, whose proof is given in Appendix~\ref{sec:density_estimation_rate}, provides a high-probability rate for estimating the mixture density under the Hellinger distance.
\begin{proposition}
\label{prop:density_estimation_rate} Under Assumption (A.1) and given the MLE $\widehat{G}_n$ defined in equation~\eqref{eq:MLE}, the following holds:
    \begin{align*}
        \mathbb{P}\left(d_H(p_{\widehat{G}_n},p_{G_*})>C_{\mathrm{density}}[d\log(n)/n]^{1/2}\right)\lesssim \exp(-c_{\mathrm{density}}\log(n))= n^{-c_{\mathrm{density}}},
    \end{align*}
    for universal positive constants $C_{\mathrm{density}}$ and $c_{\mathrm{density}}$ that depend only on $\Theta$ and $k_*$.
\end{proposition}

\begin{remark}
Variants of Proposition~\ref{prop:density_estimation_rate} can be found in the literature: see, e.g., \cite{doss_optimal_2023, Ho-Nguyen-Ann-16, gassiat2013local}. We provide a proof for completeness, based on standard empirical process arguments involving the bracketing entropy for the densities of location Gaussian mixtures in the Hellinger distance (see, e.g., \cite{Vandegeer}).
We remark that Doss et al. \cite[Theorem~1.2]{doss_optimal_2023} provide a refined approach based on  moment comparison techniques that removes an additional logarithmic 
factor $\log(n)$ appearing in the expected convergence rate. Since we consider high-probability rates, for our purposes a calculation based on global bracketing entropy suffices. 
\end{remark}

Proposition~\ref{prop:density_estimation_rate} shows the MLE density estimator $p_{\widehat{G}_n}$ converges under the Hellinger distance to the true density $p_{G_*}$ at the parametric rate $[d\log(n)/n]^{1/2}$, up to some constant depending on the parameter space $\Theta$ and $k_*$. Given this density estimation rate, the problem of capturing the separation quantity and the minimum weight in parameter estimation rates boils down to deriving the Hellinger bound in equation~\eqref{eq:old_Hellinger_bound}, with the two terms $\Delta_{\mathrm{sep}}$ and $\pi^*_{\min}$ made explicit from the constant $C(G_*)$. 

Before proceeding to the main results, we present a simple reduction lemma showing that it is enough to consider Gaussian mixtures with identity covariance matrix: the Hellinger distances will be unaffected while the Wasserstein distances of any order $p \geq 1$ will only change by a multiplicative factor depending only on $\lambda_{\min}$ and $\lambda_{\max}$.  Therefore, without loss of generality, we assume that $\Sigma=I_{d}$ for the rest of the article. The proof of this result is presented in Appendix~\ref{sec:proof_of_preliminary_results}. 
\begin{lemma}
\label{lemma:general_to_identity}
    For any two probability mixing measures given by $G = \sum_{i = 1}^{k} \pi_{i} \delta_{\theta_{i}}$ and $G_{*} = \sum_{i = 1}^{k_{*}} \pi_{i}^{*} \delta_{\theta_{i}^{*}}$, we denote their transformations  as $\tilde{G} = \sum_{i = 1}^{k} \pi_{i} \delta_{\Sigma^{-1/2} \theta_{i}}$ and $\tilde{G}_{*}= \sum_{i = 1}^{k_{*}} \pi_{i}^{*} \delta_{\Sigma^{-1/2} \theta_{i}^{*}}$, respectively. Then, we obtain
    \begin{align}
        \label{eq:invariant_Hellinger}
        d_H(p_G, p_{G_*}) = d_H(p_{\tilde{G}}, p_{\tilde{G}_*}),
    \end{align}
    where $p_{\tilde{G}}(x) = \sum_{i = 1}^{k} \pi_{i}f(x|\Sigma^{-1/2} \theta_{i}, I_{d})$ and $p_{\tilde{G}_{*}}(x) = \sum_{i = 1}^{k} \pi_{i}^{*}f(x|\Sigma^{-1/2} \theta_{i}^{*}, I_{d})$. Additionally, under Assumption (A.2), we also have for any $p\geq 1$ that
    \begin{align}
    \lambda_{\min}W_p(G,G_*) \leq W_p(\tilde{G},\tilde{G}_*)\leq \lambda_{\max}W_p(G,G_*).  
\end{align}
\end{lemma}

We conclude this section with a novel technical result, which may be of independent interest, showing that the Hellinger distance between two mixture densities $p_G$ and $p_{G_*}$ provides an upper bound for a broad class of moments of the signed measure $\nu = G - G_*$, up to some problem-dependent multiplicative quantities. As a result, once suitable interpolating polynomials are constructed to isolate specific geometric quantities of the mixing measures $G$ and $G_*$, such as mean and mass discrepancies, the Hellinger distance between the mixture densities will provide a way to bound them. 
\begin{lemma}
   \label{lemma:Hellinger_to_Polynomial}
    For any two  mixing probability measures $G = \sum_{i = 1}^{k} \pi_{i} \delta_{\theta_{i}}$ and $G_{*} = \sum_{i = 1}^{k_{*}} \pi_{i}^{*} \delta_{\theta_{i}^{*}}$, consider the signed measure  $\nu=G-G_*$ and let $$\mathcal{S} = \mathrm{span} \{\theta_1, \dots, \theta_k\} \cup \{\theta_1^*, \dots, \theta_{k_*}^*\}.$$
    Assume that $\Sigma = I_{d}$. Then, for any polynomial function $U: \mathcal{S}  \to \mathbb{R}$ with degree $D$, we obtain
    \begin{align*}
        \left|\int U(\theta)d\nu(\theta)\right|=\left| \sum_{i=1}^k \pi_i U(\theta_i) - \sum_{j=1}^{k_*} \pi_j^* U(\theta_j^*) \right|\leq C_{\mathrm{poly}} \|U\|_{\infty}d_H(p_G, p_{G_*}),
    \end{align*}
    where $\|U\|_\infty := \sup_{\theta \in \mathcal{S} \cap B(0, R)} |U(\theta)|$, and $C_{\mathrm{poly}}>0$ is some universal constant depending only on $R, \min\{k_{*}+k,d\}$, and $D$.
\end{lemma}
The proof of this result relies on the heat equation identity 
$U(\theta)=\mathbb{E}_{Z\sim N(0,I_d)}[g(Z+\theta)]$, where 
$g=e^{-\frac{\Delta_L}{2}}U$ and $\Delta_L$ denotes the Laplace operator \footnote{The Laplace operator in $\mathbb{R}^d$ is defined as $\Delta_L = \sum_{i=1}^d\frac{\partial^2}{\partial x_i^2}$. }. 
Using this identity, the integral of $U$ with respect to the signed measure 
$\nu$ can be transformed into the integral of $g$ against the difference 
between two Gaussian mixture densities corresponding to $G$ and $G'$
\begin{equation*}
    \int U(\theta)d\nu(\theta) = \int_{\mathbb{R}^d}g(x)(p_{G}(x) - p_{G_*}(x))dx.
\end{equation*}
Then, by writing $$g(x)(p_{G}(x) - p_{G_*}(x)) = (\sqrt{p_{G}(x)} - \sqrt{p_{G_*}(x)})\left[g(x)(\sqrt{p_{G}(x)} + \sqrt{p_{G_*}(x)})\right],$$ we apply the Cauchy-Schwartz inequality as 
\begin{align*}
    &\left|\int_{\mathbb{R}^d}g(x)(p_{G}(x) - p_{G_*}(x))dx\right| \\
    &\hspace{1cm} \leq \left(\int_{\mathbb{R}^d}(g(x))^2(\sqrt{p_G(x)} + \sqrt{p_{G_*}(x)})^2dx\right)^{1/2}\left(\int_{\mathbb{R}^d}(\sqrt{p_{G}(x)} - \sqrt{p_{G_*}(x)})dx\right)^{1/2},
\end{align*}
where the integral of $(\sqrt{p_{G}(x)} - \sqrt{p_{G_*}(x)})^2$ become Hellinger distance. For the integral of $\left[g(x)(\sqrt{p_{G}(x)} + \sqrt{p_{G_*}(x)})\right]^2$, we find a bound of $|g|$ in $B(0,R)$ based on $\|U\|_{\infty}$ and achieve the conclusion. 
The proof of the main results below about the connection between the Hellinger distance and Wasserstein distance, given in Appendix~\ref{sec:proof_of_preliminary_results},  relies on constructing suitable interpolating polynomials to extract the reallocation errors encoded by the moment discrepancies of the signed measure $\nu=G-G_*$. 
 With Lemma~\ref{lemma:Hellinger_to_Polynomial} in hand, we can now proceed to establish the aforementioned  lower bound on the Hellinger distance $d_H(p_{G},p_{G_*})$ between the  $G_*$, true mixing measure, and any other approximating mixing measure $G$ in $\mathcal{E}_{k_*}(\Theta)$, the primary example being the MLE $\widehat{G}_n$. 


\section{Exact-specified Setting}
\label{sec:exact-specified}
We begin by studying the convergence behavior of the maximum likelihood estimator under the exact-specified setting, where the number of true components $k_*$ is known. In this case, the likelihood optimization problem in equation~\eqref{eq:MLE} is carried out over the class $\mathcal{E}_{k^*} (\Theta)$ of Gaussian mixtures with exactly $k^*$ components. Our primary objective is to characterize how the minimum separation $\Delta_{\mathrm{sep}}$ and the minimum mixture weight $\pi_{\min}^*$ influence the convergence rate of parameter estimation. To achieve this goal, we derive lower bounds for the Hellinger distance $d_H(p_G,p_{G_*})$ in terms of the 1-Wasserstein distance $W_1(G,G_*)$, with explicit dependence on both $\Delta_{\mathrm{sep}}$ and $\pi_{\min}^*$.

A central theme of our analysis is that the statistical complexity of parameter estimation depends crucially on the geometric arrangement of the {\it  mixture components}  $\theta_1^*,\theta_2^*,\ldots,\theta_{k_*}^*$, support points of the underlying mixing measure $G_*$. Specifically, the degree by which the points $\theta^*_i$'s cluster has a significant impact on the sample complexity of parameter estimation. To capture this phenomenon accurately, we consider the three clustering scenarios, formulated in terms of the separation parameter $\Delta_{\mathrm{sep}}$.
\begin{itemize}
    \item[(i)] \emph{Single-cluster regime} - All true components make up a single cluster;
    \item[(ii)] \emph{Multi-cluster regime} -  All true components are  partitioned into several clusters;
    \item[(iii)] \emph{Unstructured regime} - No assumptions on the clustering structure of the true components is made. 
\end{itemize}
We formally define these geometric regimes and establish the corresponding convergence results in the following subsections.

\subsection{Single-cluster Regime}
We begin by performing our analysis in the first regime where all true components merge into a single cluster, which is formally defined as follows.

\begin{definition}
    All true components $(\theta^*_i)_{i=1}^{k_*}$ are said to be partitioned into one cluster of diameter $C_0 \Delta_{\mathrm{sep}}$, where $C_0 \ge 1$, if $\max_{1 \leq j \neq l \leq k_{*}} \|\theta_j^* - \theta_l^*\| \le C_0 \Delta_{\mathrm{sep}}$.
\end{definition}
This regime corresponds to the case in which all true components lies in a ball of small diameter,  proportional to the minimum separation quantity $\Delta_{\mathrm{sep}}$. Although the components remain distinct (due to the assumption that $\Delta_{\mathrm{sep}} > 0$), they are all confined to the same {\it microscopic} spatial scale. As a consequence, the corresponding mixture densities exhibit substantial overlap, making the parameter estimation problem highly sensitive to the local geometric configuration of the true components.
To capture this phenomenon,  we relate the Wasserstein distance between the MLE and the true mixing measure, $W_1(\widehat{G}_n,G_*)$, to the Hellinger distance between the corresponding mixture densities, $d_H(p_{\widehat{G}_n},p_{G_*})$. This relationship is crucial because it allows us to transfer the parametric density estimation rate in Proposition~\ref{prop:density_estimation_rate} to the convergence rate for the latent mixing measure. 

Since finite mixtures are identifiable only up to permutations of their components \cite{Teicher-63}, a direct comparison between the atoms of a fitted mixing measure $G$ and those of the true mixing measure $G_*$ is generally not well defined. To resolve this ambiguity, we employ a Voronoi-cell construction in \cite{manole22refined} that assigns each fitted atom to its nearest true component. This construction provides a canonical partition of the fitted atoms and will be used throughout the paper to analyze the convergence of mixing measures.
\begin{definition}
    Given a mixing measure $G$ with $k$ atoms $(\theta_i)_{i=1}^{k}$, a set of corresponding Voronoi cells $\mathcal{V}_{j}\equiv\mathcal{V}_j(G)$ generated by the true atoms $\theta_{j}^{*}$ of $G_*$ is defined as 
    \begin{align}
        \mathcal{V}_{j}\equiv\mathcal{V}_j(G) : = 
  \{i \in [k]: \|\theta_{i} - \theta_{j}^{*}\| \leq \|\theta_{i} - \theta_{\ell}^{*}\|, \ \forall \ell \neq j
  \}, \label{eq_Voronoi_definition}
    \end{align}
    for all $1 \leq j \leq k_{*}$.
\end{definition}

To provide refined results,  we distinguish two cases, depending on whether the support of $G$ is very close to that of $G_*$ relatively to  $\Delta_{\mathrm{sep}}$ or not. We refer to the first type of bounds, in which each point in the support of $G$ is $\Delta_{\mathrm{sep}}/4$-close to some point in the support of $G_*$ as {\it local bounds} and to the second one, in which no proximity assumption between the two supports is imposed, as {\it global bounds}. As the next result demonstrates, global bounds are worse by a multiplicative amount proportional to $\pi_{\mathrm{min}}^{*}$.  

\begin{theorem}
\label{theorem:exact_one_group_univariate}
Suppose that all true components are partitioned into one cluster of diameter $C_0\Delta_{\mathrm{sep}}$, then there exist positive universal constants $C_{\mathrm{local},1}$ and $C_{\mathrm{global},1}$ depending only on $R$, $k_{*}$, $\Sigma$, and $C_{0}$ such that

\noindent
(a) (Local bound) For $G \in \mathcal{E}_{k_{*}}(\Theta)$ such that $\|\theta_{i} - \theta_{j}^{*}\| \leq \frac{\Delta_{\mathrm{sep}}}{4}$, for all $i \in \mathcal{V}_{j}$ and $j\in[k_*]$,
\begin{align} 
d_H(p_{G},p_{G_*}) \ge C_{\mathrm{local},1} \cdot \Delta_{\mathrm{sep}}^{2k_* - 1} \cdot W_1(G, G_*).
\end{align}
(b) (Global bound) For general $G \in \mathcal{E}_{k_{*}}(\Theta)$, we have
\begin{align} 
d_H(p_{G},p_{G_*}) \ge C_{\mathrm{global},1} \cdot \pi_{\mathrm{min}}^{*} \cdot \Delta_{\mathrm{sep}}^{2k_* - 1} \cdot W_1(G, G_*).
\end{align}
\end{theorem}
The proof of Theorem \ref{theorem:exact_one_group_univariate} is deferred to Appendix~\ref{sec:proof_theorem:exact_one_group_univariate}. Crucially, in both the local and global bounds in Theorem~\ref{theorem:exact_one_group_univariate}, the relationship between density estimation and parameter estimation deteriorates polynomially as the the minimum separation $\Delta_{\mathrm{sep}}$ decreases (given the compactness assumption on $\Theta$, the interesting scenario is that of a decreasing). In particular, the factor $\Delta_{\mathrm{sep}}^{2k_*-1}$ shows that the statistical complexity of recovering the latent mixing measure $G_*$ grows rapidly when the true components become increasingly concentrated. Technically, the exponent $2k_*-1$ originates from the Hermite interpolation structure required to isolate the mean and mass discrepancies in the Wasserstein decomposition. It intuitively reflects the geometric singularity induced by severe component overlap. More specifically, since all true components lie within a single cluster, they will merge together when the separation vanishes. In that case, a Wasserstein perturbation of the mixing measure may correspond to only minor changes in the observed density, thereby making parameter estimation substantially more difficult than density estimation.
Consequently, the Gaussian mixture begins to lose the ability to uniquely resolve the underlying latent structure from the observed data. 

On the other hand, the minimum mixture weight $\pi_{\min}^*$ appears only in the global lower bound but not in the local bound. This distinction stems from the fact that the local regime already preserves the correspondence between fitted and true components by requiring each fitted atom to remain sufficiently close to its associated true component. Consequently, low-weight components cannot disappear or migrate to distant regions of the parameter space, and the estimation error is governed primarily by the local overlap geometry and cancellation effects among nearby components.
In contrast, the global regime allows fitted atoms to move arbitrarily far from the true configuration. In this setting, components with very small mixture weights contribute only weakly to the observed density and therefore become substantially harder to detect statistically. As a result, the additional dependence on the minimum mixture weight in the global Hellinger lower bound indicates that weak components further increase the statistical complexity of the parameter estimation problem. 

Combining the Hellinger lower bounds in Theorem~\ref{theorem:exact_one_group_univariate} with the high-probability density estimation rate from Proposition~\ref{prop:density_estimation_rate}, we immediately obtain the MLE convergence rate in the single-cluster regime of the exact-specified setting. 

\begin{corollary}
    \label{corolarry:MLE_exact_I}
Under the assumptions of Proposition~\ref{prop:density_estimation_rate} and Theorem~\ref{theorem:exact_one_group_univariate}, the following holds:
    
\noindent
(a) (Global rate) For any $n \geq 1$, we have
    \begin{align*}
\mathbb{P}\left(W_1(\widehat{G}_n,G_*)>\frac{C_{\mathrm{density}}}{C_{\mathrm{global,1}}}\left(\frac{d\log(n)}{n}\right)^{\frac{1}{2}}(\pi_{\min}^{*})^{-1} \Delta_{\mathrm{sep}}^{-(2k_* - 1)}\right)\lesssim \exp(-c_{\mathrm{density}}\log(n)).
    \end{align*}
(b) (Local rate) There exists $N_1\in\mathbb{N}$ depending on $G_{*}$ such that for all $n \geq N_1$, we have 
\begin{align*}
&\mathbb{P}\left(W_1(\widehat{G}_n,G_*)>\frac{C_{\mathrm{density}}}{C_{\mathrm{local,1}}}\left(\frac{d\log(n)}{n}\right)^{\frac{1}{2}}\Delta_{\mathrm{sep}}^{-(2k_* - 1)}\right) \lesssim   
\exp(-c_{\mathrm{density}}\log(n)).
\end{align*}
\end{corollary}

\begin{remark} 
\label{remark:uni_cluster_exact_fit}
A few remarks concerning the MLE rates in Corollary~\ref{corolarry:MLE_exact_I} are in order.

(1) While the global rate is a direct consequence of Proposition \ref{prop:density_estimation_rate} and the global statement of Theorem \ref{theorem:exact_one_group_univariate}, the local rate requires the evaluation of $N_1$. 
Consider the events $\mathcal{A}_n = \{ W_1(\widehat{G}_n,G_*)>\frac{C_{\mathrm{density}}}{C_{\mathrm{local,1}}}\left(\frac{d\log(n)}{n}\right)^{\frac{1}{2}}\Delta_{\mathrm{sep}}^{-(2k_* - 1)} \}$ and $\mathcal{B}_n = \{ W_1(\widehat{G}_n,G_*) \leq \pi^*_{\min}\Delta_{\mathrm{sep}}/4\}$. Setting $$N_1 = \min \left\{ n\in\mathbb{N} \colon n/\log(n) \geq d \left(4\cdot \frac{C_{\mathrm{density}}}{C_{\mathrm{global,1}}}(\pi_{\min}^{*})^{-2}\Delta_{\mathrm{sep}}^{-2k_*}\right)^2 \right\},$$ 
by part (a) we see that, for all $n\geq N_1$, 
\begin{align*}
    \mathbb{P}(\mathcal{B}^c_n) & = \mathbb{P} \left( W_1(\widehat{G}_n,G_*) > \pi^*_{\min}\Delta_{\mathrm{sep}}/4 \right)\\
    & \leq \mathbb{P}\left(W_1(\widehat{G}_n,G_*)>\frac{C_{\mathrm{density}}}{C_{\mathrm{local,1}}}\left(\frac{d\log(n)}{n}\right)^{\frac{1}{2}}(\pi^*_{\mathrm{min}})^{-1} \Delta_{\mathrm{sep}}^{-(2k_* - 1)}\right) \\
    & \lesssim \exp(-c_{\mathrm{density}}\log(n)). 
\end{align*}
In addition, according to Lemma \ref{lemma:small_wasserstein_implies_center_in_voronoi_cell}, the event $\mathcal{B}_n$ implies that, for each $i \in [k_*]$, the fitted component $\widehat{\theta}_{n,i}$ belongs to a Voronoi cell generated by some true component, say $\theta_j^*$, and, furthermore, $\| \widehat{\theta}_{n,i} - \theta^*_j \| \leq \Delta_{\mathrm{sep}}/4$. Thus, using Proposition \ref{prop:density_estimation_rate} and the local statement of Theorem \ref{theorem:exact_one_group_univariate}, it follows that $\mathbb{P}(\mathcal{A}_n \cap \mathcal{B}_n) \lesssim \exp(-c_{\mathrm{density}}\log(n))$. As a result, we obtain
\begin{equation*}
    \mathbb{P}(\mathcal{A}_n) \leq \mathbb{P}(\mathcal{A}_n \cap \mathcal{B}_n) + \mathbb{P}(\mathcal{B}_n^c) \lesssim \exp(-c_{\mathrm{density}}\log(n)),
\end{equation*}
proving part (b) of Corollary~\ref{corolarry:MLE_exact_I}.


(2) Corollary~\ref{corolarry:MLE_exact_I} shows how weak separation and small mixture weights jointly degrade the statistical convergence rate of the MLE. In particular, when either the separation or the minimum weight becomes small, the MLE convergence rate decreases rapidly, reflecting the increasing complexity of estimating overlapping or low-weight components. As a consequence, significantly more samples are required to accurately recover the latent mixing measure, even when the density itself can already be estimated efficiently. To the best of our knowledge, this is the first work in the finite mixture literature that explicitly captures both the minimum separation and the minimum weight in the convergence rate of MLE.

(3) It is instructive to compare the estimation rate from Corollary~\ref{corolarry:MLE_exact_I} with the minimax convergence rate obtained by Doss et al. \cite[Theorem 1.1]{doss_optimal_2023}, which is given by 
\begin{equation}
    \label{eq:minimax.rate}
\inf_{\widetilde{G}_n} \sup_{G_* \in \mathcal{E}_{k_*}(\Theta)} \mathbb{E}_{G_*} \left[ W_1(\widetilde{G}_n,G_*) \right] \asymp_{k^*}
\left( \frac{d}{n} \right)^{1/4} \wedge 1 + \left( \frac{1}{n} \right)^{1/(4 k_* - 2)},
\end{equation}
where the notation $\asymp_{k^*}$ indicates that both sides match up a multiplicative factor depending (polynomially, in this case) on $k_*$. Interestingly, the minimax rate consists of two terms, only one of which depends on the dimension $d$. Therefore, and as noted in \cite{doss_optimal_2023}, there exists a threshold value $d_{\mathrm{thr}} = n^{\frac{2 k_* - 3}{2k_*-1}}$ for the dimension $d$ at which the two terms on the right hand side of the equation~\eqref{eq:minimax.rate} making up the minimax rate share the same order of $n^{-\frac{1}{4k_*-2}}$. Interestingly, this rate also coincides with the minimax convergence rate for minimum distance estimators established by Heinrich and Kahn \cite{heinrich2018minimax} when components are partitioned into one cluster. In addition, when $d > d_{\mathrm{thr}}$, the minimax rate is driven by the first, dimension-dependent term.

Of course, the minimax rate in equation~\eqref{eq:minimax.rate} is, by its very definition, a rate that holds {\it uniformly} over all possible mixing measures $G_*$ in $\mathcal{E}_{k_*}(\Theta)$. In contrast,  the (high-probability) convergence rate for the MLE $\widehat{G}_n$ from Corollary~\ref{corolarry:MLE_exact_I}, namely $\left(\frac{d\log(n)}{n}\right)^{\frac{1}{2}}\Delta_{\mathrm{sep}}^{-(2k_* - 1)}$, is {\it point-wise}, in the sense that it depends directly on $G_*$. The two rates are, of course, inherently different and thus not comparable, as they capture different features of the statistical task at hand: the minimax rate~\eqref{eq:minimax.rate} provides the worst possible rate achieved by the best possible estimator, while the local or pointwise rates in Corollary~\ref{corolarry:MLE_exact_I} for $\widehat{G}_n$ quantify the performance of one specific estimator, namely the MLE, at a particular mixing measure $G_*$.

In fact, the minimax estimator constructed by Doss et al. \cite{doss_optimal_2023} is {\it not} the MLE. However, we note that, in the present setting, the MLE enjoys a near parametric estimation rate of order $\left( \frac{d}{n} \right)^{1/2}$ (ignoring logarithmic terms) for all sufficiently regular mixing measures $G_*$, i.e., mixing measures for which the separation 
 $\Delta_{\mathrm{sep}}$, viewed as a quantity  changing with $n$, remains bounded away from zero. This is a much faster estimation rate than the minimax rate in equation \eqref{eq:minimax.rate}.  On the other hand, if the minimal weight and separation parameters of $G_*$ vanish (as a function of $n$) too fast, then the estimation error of the MLE can be made to match the minimax rate and, in fact, to be arbitrarily slow. To illustrate, a simple calculation shows that, under the settings of Theorem~\ref{theorem:exact_one_group_univariate} and assuming $\pi^*_{\mathrm{\min}}$ bounded away from $0$, the MLE $\widehat{G}_n$ will achieve the minimax rate whenever $\Delta_{\mathrm{sep}}\asymp n^{-\frac{k_*-1}{(2k_*-1)^2}} d^{\frac{1}{2(2k_* - 1)}}$ if $d \leq d_{\mathrm{thr}}$ and $\Delta_{\mathrm{sep}}\asymp  \left( \frac{d}{n}\right)^{\frac{1}{4(2k_* - 1)}}$ otherwise, and will converge faster whenever $\Delta_{\mathrm{sep}}$ vanishes at slower rates.
\end{remark}

\subsection{Multi-cluster Regime}
We now move to a more general geometric regime where the $k_*$ true components are partitioned into $k_0$ disjoint 
clusters, defined as follows.
\begin{definition}
\label{def:multi.cluster}
    All true components are said to be partitioned into $k_0$ disjoint clusters denoted by $\mathcal{C}_1, \dots, \mathcal{C}_{k_0}$ if three following assumptions are satisfied:
    \begin{itemize}
    \item \textit{Intra-cluster assumption}: For any two components $\theta_j^* \neq \theta_q^* \in \mathcal{C}_m$, their spatial separation is bounded by the cluster diameter $C_0 \Delta_{\mathrm{sep}}$, where $C_0\geq 1$, that is, $\|\theta_j^* - \theta_q^*\| \le C_0 \Delta_{\mathrm{sep}}$.
    \item \textit{Inter-cluster assumption}: For any $\theta_j^* \in \mathcal{C}_m$ and $\theta_p^* \notin \mathcal{C}_m$, their separation is strictly bounded below by the {\it macroscopic gap} parameter $D_0>0$, that is, $\|\theta_j^* - \theta_p^*\| \ge D_0$.
    \item \textit{Rigid topological ordering}: To ensure that the clusters are strictly unambiguously isolated from one another, we further assume that $\Delta_{\mathrm{sep}} \le \frac{D_0}{4 C_0}$. 
\end{itemize}

\end{definition}
Compared to the single-cluster regime studied in the previous section, this regime exhibits a more structured geometry: although strong overlap may still occur among components within the same cluster, different clusters remain macroscopically separated from one another. In this case we will show that the statistical complexity of estimating the mixing measure $G_*$ is no longer governed by the total number of components $k_*$, but rather by the local geometric complexity within each cluster.

Intuitively, components belonging to different clusters become substantially easier to distinguish due to the macroscopic separation between clusters. Thus, the cancellation effects described in the previous subsection occur primarily within individual clusters rather than globally across all components. As a result, the convergence behavior of the MLE is expected to depend mainly on the size of the densest cluster, which represents the most statistically singular local region of the mixture model. To make this precise, for $m\in[k_0]$, we denote by $s_m = |\mathcal{C}_m|$ the number of true components in the cluster $\mathcal{C}_m$,  and let
$$s_{\max} = \max_{1 \le m \le k_0} s_m$$
be the size of the largest cluster. As we will see next, this quantity captures the geometric complexity of estimating the mixing measure. 

\begin{theorem}
\label{theorem:exact_multi_group}
Assume that $G_{*}$ satisfies the multi-cluster condition of Definition~\ref{def:multi.cluster}. Then, there exist positive universal constants $C_{\mathrm{local,2}},C_{\mathrm{global,2}}$ depending only on $R$, $k_{*}$, $\Sigma$, $C_{0}$, and $D_{0}$ such that
\noindent

(a) (Local bound) If $G \in \mathcal{E}_{k_{*}}(\Theta)$ is such that $\|\theta_{i} - \theta_{j}^{*}\| \leq \frac{\Delta_{\mathrm{sep}}}{4}$, for all $i \in \mathcal{V}_{j}$ and $j\in[k_*]$, then
\begin{align} 
d_H(p_{G},p_{G_*}) \ge C_{\mathrm{local,2}} \cdot \Delta_{\mathrm{sep}}^{2s_{\max} - 1} \cdot W_1(G, G_*).
\end{align}
(b) (Global bound) For any $G \in \mathcal{E}_{k_{*}}(\Theta)$, it holds that
\begin{align} 
d_H(p_{G},p_{G_*}) \ge C_{\mathrm{global,2}} \cdot \pi_{\min}^{*} \cdot \Delta_{\mathrm{sep}}^{2s_{\max} - 1} \cdot W_1(G, G_*).
\end{align}
\end{theorem}
The proof of Theorem~\ref{theorem:exact_multi_group} is given in Appendix~\ref{sec:proof_theorem:exact_multi_group_univariate}. Compared to single-cluster regime, the multi-cluster regime exhibits a a fundamentally different geometric behavior, as captured by Theorem~\ref{theorem:exact_multi_group}. Although the minimum separation $\Delta_{\mathrm{sep}}$ still controls the deterioration of the Hellinger lower bound, the exponent now depends the size  $s_{\max}$ of the densest cluster, instead of the total number of components $k_*$. This indicates that the statistical sample complexity of parameter estimation is driven by the most concentrated local region of the mixture rather than by the entire global configuration of components.

The rate improvement over Theorem~\ref{theorem:exact_one_group_univariate} can be understood through the role of the inter-cluster separation condition. Since different clusters are separated by a macroscopic distance $D_0$, components within one cluster have only limited interaction with components from other clusters. Therefore, the strong overlap phenomenon becomes localized inside individual clusters, preventing the large-scale geometric degeneracy observed when all components collapse into a single group. In particular, each cluster behaves approximately as an independent local singular structure, while the overall mixture retains a globally distinguishable structure.

In addition, just as with the single-cluster setting, the global bound also depends on the minimum mixture weight $\pi_{\min}^*$. However, compared to the separation quantity, the exponent of the minimum mixture weight is unaffected by the change in the cluster structure.
This reflects the fact that low-weight components remain statistically difficult to recover even when the cluster geometry becomes more favorable. 


Next, we combined Proposition~\ref{prop:density_estimation_rate} and Theorem~\ref{theorem:exact_multi_group}  to provide convergence rates for the MLE under this multi-cluster regime.

\begin{corollary}
    \label{corolarry:MLE_exact_II}
    Under the assumptions of Theorem~\ref{theorem:exact_multi_group} and Proposition~\ref{prop:density_estimation_rate}, we have
    \noindent
(a) (Global rate) For any $n\geq 1$, we have 
    \begin{align*}
\mathbb{P}\left(W_1(\widehat{G}_n,G_*)>\frac{C_{\mathrm{density}}}{C_{\mathrm{global,2}}}\left(\frac{d\log(n)}{n}\right)^{\frac{1}{2}}(\pi_{\min}^{*})^{-1} \Delta_{\mathrm{sep}}^{-(2s_{\max} - 1)}\right)\lesssim \exp(-c_{\mathrm{density}}\log(n)).
    \end{align*}
(b) (Local rate) There exists $N_2\in\mathbb{N}$ depending on $G_*$ such that for $n \geq N_2$, we have 
    \begin{align*}
\mathbb{P}\left(W_1(\widehat{G}_n,G_*)>\frac{C_{\mathrm{density}}}{C_{\mathrm{local,2}}}\left(\frac{d\log(n)}{n}\right)^{\frac{1}{2}}\Delta_{\mathrm{sep}}^{-(2s_{\max} - 1)}\right)
\lesssim \exp(-c_{\mathrm{density}}\log(n)).
    \end{align*}
\end{corollary}
\begin{remark}
\label{remark:multi_cluster_exact_fit} 
Below we illustrate the significance and implications of of Corollary~\ref{corolarry:MLE_exact_II}. 

(1) For the local rate, using the same argument as in Part (1) of Remark \ref{remark:uni_cluster_exact_fit}, we can choose 
$$N_2 = \min \left\{ n\in\mathbb{N} \colon n/\log(n) \geq d \left(4\cdot \frac{C_{\mathrm{density}}}{C_{\mathrm{global,2}}}(\pi_{\min}^{*})^{-2}\Delta_{\mathrm{sep}}^{-2s_{\max}}\right)^2 \right\}$$
such that for $n\geq N_2$, the local rate is guaranteed. 
(2) From Corollary~\ref{corolarry:MLE_exact_II}, we see that the convergence behavior of the MLE is governed by the size of the largest local cluster rather than the total number of components, as in the single-cluster regime. This leads to a significant improvement in the convergence rate. In particular, the MLE achieves a substantially faster rate when the number of clusters $k_0$ is large, corresponding to the regime where $s_{\max}\ll k_*$. This phenomenon can be illustrated through the following two contrasting scenarios.
\begin{itemize}
    \item \emph{When $k_0$ is small:} 
    In this unfavorable multi-cluster regime, $s_{\max}$ is necessarily large, and the convergence rate of the MLE is heavily impacted by the separation parameter through the term $\Delta_{\mathrm{sep}}^{- (2s_{\max} - 1)}$. In the most extreme case where $k_0 = 1$ (and therefore $s_{\max} = k_*$), the above rates reduce to those established in the single-cluster regime. 
    \item \emph{When $k_0$ is large and the clusters are relatively balanced:} By contrast, the most balanced configuration occurs when the components are distributed approximately uniformly across the $k_0$ clusters. Here, the singularity within each cluster is substantially reduced because no local region contains too many overlapping components. When $k_0$ is large (relative to its upper bound $k_*$), $s_{\max}$ is necessarily small, resulting in a faster convergence rate for the MLE. In the most separated case of one cluster per mixture component (i.e. $k_0 = k_*$, or equivalently $s_{\max} = 1$), the convergence rate of the MLE depends linearly --as opposed to polynomially --  on the inverse of the minimal separation. This is the most favorable case, where the MLE $\widehat{G}_n$ exhibits the fastest rates. Indeed, compared to the minimax rate \eqref{eq:minimax.rate} (and for simplicity ignoring logarithmic terms and assuming $\pi^*_{\mathrm{\min}}$ bounded way from zero), we see that in this case the MLE will enjoy a faster rate than the minimax rate as long as $\Delta_{\mathrm{min}}$ vanishes at a slower rate than $\left( \frac{d}{n} \right)^{1/4}$ (resp. $ \sqrt{d} n^{- k_*/(2k_* - 1)}$) when $d > d_{\mathrm{thr}}$ (resp.   $d \leq d_{\mathrm{thr}}$).
    
\end{itemize}
The phenomenon illustrated above demonstrates that the geometry of finite Gaussian mixtures possesses an intrinsic multiscale structure: local overlap determines the degree of near singularity of mixture, thus negatively impacting the estimation task, whereas global separation stabilizes the sample complexity of parameter estimation.

(3) For comparison, let us recall the minimax adaptive rate established by Wu et al.~\cite[Theorem~2 and Remark~4]{wu2020moment} for the multi-cluster regime in the univariate case: 
\begin{equation}
\label{eqn:adaptive_minimax_rate}
\inf_{\widetilde{G}_n}
\sup_{G_*\in \mathcal{E}_{k_*,k_0,\gamma,\omega}}
W_1(\widetilde{G}_n,G_*)
\asymp_{k^*}
\gamma^{-\frac{2k_0-2}{2(k_*-k_0)+1}}
\left(\frac{1}{n}\right)^{\frac{1}{4(k_*-k_0)+2}}.
\end{equation}
Here, $\mathcal{E}_{k_*,k_0,\gamma,\omega}$ denotes the class of mixing measures with respect to $k_*$-component Gaussian mixtures with $k_0$ $(\gamma,\omega)$-separated clusters\footnote{The Gaussian mixture has $k_0$ $(\gamma,\omega)$-separated clusters if there exists a partition $S_1,\ldots,S_{k_0}$ of $[k_*]$ such that: (i) $\|\mu_i-\mu_{i'}\|\geq \gamma$ for any $i\in S_{\ell}$ and $i'\in S_{\ell'}$ such that $\ell\neq\ell'$; and (ii) $\sum_{i\in S_{\ell}}\omega_i\geq\omega$ for each $\ell$ \cite{wu2020moment}.} defined in Definition~1 in 
\cite{wu2020moment}. 
Note that the above minimax rate also matches the rate of the minimum distance estimator derived by Heinrich and Kahn~\cite[Theorem~3.3]{heinrich2018minimax} (in the univariate case) when ignoring the inter-cluster 
separation $\gamma$.

As discussed previously, the minimax rate and our pointwise estimation rate are 
not directly comparable, as they characterize different aspects of the 
estimation problem. The former describes the worst-case performance over a 
prescribed class of mixing measures, while the latter captures the local 
geometry of a particular mixing measure $G_*$. Nevertheless, comparing these 
two rates provides useful insights into the role of different separation 
quantities.

First, the analysis in \cite{wu2020moment} focuses on the inter-cluster 
separation $\gamma$ and identifies its influence on the minimax convergence 
rate. In contrast, our framework treats the inter-cluster separation $D_0$ as a 
macroscopic quantity and instead focuses on the effect of the intra-cluster 
separation $\Delta_{\mathrm{sep}}$. In the univariate case, ignoring 
logarithmic factors and assuming that the minimum mass $\pi_{\min}^*$ remains 
bounded away from zero, the MLE rate matches the minimax rate in 
\cite[Theorem~2]{wu2020moment} when
\begin{equation*}
    \Delta_{\mathrm{sep}}
    \asymp
    \left(
    n^{-\frac{k_*-k_0}{2(k_*-k_0)+1}}
    \gamma^{\frac{2k_0-2}{2(k_*-k_0)+1}}
    \right)^{\frac{1}{2s_{\max}-1}} 
\end{equation*}
and is faster when $ \Delta_{\mathrm{sep}}$ is of larger order.

Second, our result reveals the role of the largest local cluster size 
$s_{\max}$ in determining the convergence behavior. On the other hand, the 
minimax analysis in \cite[Remark~4]{wu2020moment} corresponds to the worst 
configuration, where the largest possible cluster contains $k_*-k_0+1$ components. 
Although $s_{\max}$ is generally unknown and depends on the local configuration 
of $G_*$, it precisely captures the benefit gained from favorable separation 
structures: mixing measures with smaller local clusters can achieve faster 
pointwise convergence rates.

(4) Finally, as pointed out in \cite{doss_optimal_2023}, extending separation-based 
arguments to the multidimensional setting is non-trivial due to the fact that projection may collapse originally well-separated components into nearby points 
after this procedure. Our polynomial-testing framework avoids this difficulty by 
working directly in the original parameter space. Consequently, the proposed 
approach naturally extends the separation-dependent analysis to arbitrary 
dimensions.
\end{remark}


\subsection{Unstructured Regime} We now turn to the most general regime where no structural assumptions are imposed on the geometric configuration of the true components. In contrast to the previous two regimes, the components are no longer assumed to form either a single localized cluster or multiple well-separated clusters. Consequently, the Gaussian mixture may contain several overlapping regions at different spatial scales, leading to substantially more complicated geometric interactions among components.

The absence of an explicit cluster structure creates significant challenges to relate the Hellinger distance to the Wasserstein distance. In particular, since nearby components may interact across multiple local scales, the cancellation effects induced by component overlap can no longer be localized within a single cluster. As a result, the mixture geometry becomes considerably more singular than in the previous regimes, and the induced density may become substantially less sensitive to perturbations of the latent mixing measure.

To overcome this difficulty, our analysis must simultaneously control both local overlap phenomena and long-range interactions among components without relying on any predefined cluster decomposition. This leads to more conservative yet general lower bounds as provided in the following theorem, whose dependence on the separation parameter expresses the increased geometric complexity of the unrestricted regime.
\begin{theorem}
\label{theorem:exact_no_group}
Under the unstructured regime, there exist positive universal constants $C_{\mathrm{local,3}}$ and $C_{\mathrm{global,3}}$ depending only on $R$, $k_{*}$, and $\Sigma$ such that

\noindent
(a) (Local bound) If $G \in \mathcal{E}_{k_{*}}(\Theta)$ is such that $\|\theta_{i} - \theta_{j}^{*}\| \leq \frac{\Delta_{\mathrm{sep}}}{4}$ for all $i \in \mathcal{V}_{j}$ and $j\in[k_*]$, then
\begin{align} 
d_H(p_{G},p_{G_*}) \ge C_{\mathrm{local,3}} \cdot \Delta_{\mathrm{sep}}^{4k_* - 3} \cdot W_1(G, G_*).
\end{align}
(b) (Global bound) For any $G \in \mathcal{E}_{k_{*}}(\Theta)$, it holds that
\begin{align} 
d_H(p_{G},p_{G_*}) \ge C_{\mathrm{global,3}} \cdot \pi_{\mathrm{min}}^{*} \cdot \Delta_{\mathrm{sep}}^{4k_{*} - 3} \cdot W_1(G, G_*).
\end{align}
\end{theorem}
The proof of Theorem~\ref{theorem:exact_no_group} is provided in Appendix~\ref{sec:proof_of_exact_no_group}.  Theorem~\ref{theorem:exact_no_group} addresses the fully unstructured regime where the true components may organize themselves in an arbitrary geometric manner. In this regime, the convergence behavior becomes markedly more fragile, as evidence by the increased separation factor $\Delta_{\mathrm{sep}}^{4k_*-3}$. Compared with the previous subsections, the exponent nearly doubles, indicating a substantial increase in the sensitivity of the estimation problem to small component separations.
This phenomenon arises because, without an explicit cluster structure, the latent structure is allowed to approach highly irregular configurations in which several local overlap mechanisms coexist.
Such configurations generate substantially stronger degeneracies than those observed in the structured cluster settings, thereby forcing a more conservative dependence on the minimum separation parameter. In particular, the exponent $4k_*-3$ of the minimum separation reflects the additional complexity required to control these overlap interactions without relying on any macroscopic separation assumptions. As a result, the induced density becomes significantly less informative about perturbations of the underlying mixing measure.

Finally, the appearance of the minimum mixture weight $\pi_{\min}^*$ in the global bound again highlights the intrinsic difficulty of recovering components that contribute only weakly to the observed density. 
This effect becomes even more pronounced in the unrestricted regime, where low-weight components may participate in complicated overlapping configurations.  

Before ending this subsection, let us highlight the MLE convergence rates implied by the local and global lower bounds in Theorem~\ref{theorem:exact_no_group}.

\begin{corollary}
    \label{corollary:MLE_exact_III}
    Under the assumptions of Proposition~\ref{prop:density_estimation_rate}, we have
\noindent

(a) (Global rate) For any $n\geq 1$, we have 
\begin{align*}
\mathbb{P}\left(W_1(\widehat{G}_n,G_*)>\frac{C_{\mathrm{density}}}{C_{\mathrm{global,3}}}\left(\frac{d\log(n)}{n}\right)^{\frac{1}{2}}(\pi_{\min}^{*})^{-1} \Delta_{\mathrm{sep}}^{-(4k_* - 3)}\right)\lesssim \exp(-c_{\mathrm{density}}\log(n)).
    \end{align*}
(b) (Local rate) There exists $N_3\in\mathbb{N}$ depending only on $G_*$ such that for $n\geq N_3$, we have 
\begin{align*}
\mathbb{P}\left(W_1(\widehat{G}_n,G_*)>\frac{C_{\mathrm{density}}}{C_{\mathrm{global,3}}}\left(\frac{d\log(n)}{n}\right)^{\frac{1}{2}}\Delta_{\mathrm{sep}}^{-(4k_* - 3)}\right) \lesssim   
\exp(-c_{\mathrm{density}}\log(n)).
    \end{align*}
    
\end{corollary}
\begin{remark} 
\label{remark:unstructured_exact_fit}
(1) To derive the local rate, we employ similar arguments as in Part (1) of Remark \ref{remark:uni_cluster_exact_fit} by choosing 
$$N_3 = \min \left\{ n\in\mathbb{N} \colon n/\log(n) \geq d \left(4\cdot \frac{C_{\mathrm{density}}}{C_{\mathrm{global,3}}}(\pi_{\min}^{*})^{-2}\Delta_{\mathrm{sep}}^{-(4k_*-2)}\right)^2 \right\}.$$
(2) As in Remark \ref{remark:uni_cluster_exact_fit}, while minimax rate and point-wise convergence rate reflects different nature, it is informative to compare our result with the minimax rate stated in \cite{doss_optimal_2023} and reformulated in equation \eqref{eq:minimax.rate}. Recall that when $d > d_{\mathrm{thr}}:=n^{\frac{2k_*-3}{2k_*-1}}$, the mimimax rate is determined by dimension-dependent factor $\left(\frac{d}{n}\right)^{\frac{1}{4}}$, while this rate is dominated by $\left(\frac{1}{n}\right)^{\frac{1}{4k_*-2}}$ when $d < d_{\mathrm{thr}}$. If the separation $\Delta_{\mathrm{sep}}$ and the minimal probability $\pi^*_{\mathrm{min}}$ remain bounded away from zero, then the MLE achieves the standard parametric rate of order $\left(\frac{d \log(n)}{n}\right)^{\frac{1}{2}}$, which is much faster than the minimax lower bound rate in equation \eqref{eq:minimax.rate}. On the other hand, when the separation $\Delta_{\mathrm{sep}}$ vanishes at the rate of order $\Delta_{\mathrm{sep}} \asymp n^{-\frac{k_*-1}{(2k_*-1)(4k_*-3)}} d^{\frac{1}{2(4k_* - 3)}}$ when $d \leq d_{\mathrm{thr}}$ and $\Delta_{\mathrm{sep}} \asymp (\frac{d}{n})^{\frac{1}{4(4k_*-3)}}$ otherwise, then the above MLE rate matches the minimax rate. 

(3) The above corollary yields a fully general convergence rate for the MLE without requiring any assumptions on the cluster structure of the true components. Compared to the single-cluster regime, where the convergence rate depends on the factor $\Delta_{\mathrm{sep}}^{-(2k_*-1)}$, the unrestricted regime exhibits a substantially more severe dependence on the minimum separation through the factor $\Delta_{\mathrm{sep}}^{-(4k_*-3)}$. The deterioration becomes even more pronounced relative to the multi-cluster regime, where the corresponding exponent reduces further to $-(2s_{\max}-1)$.
These comparisons reveal that geometric structure plays a fundamental role in stabilizing parameter estimation in finite Gaussian mixtures. Although density estimation still achieves the standard parametric rate on the sample size under the Hellinger distance, recovering the latent mixing measure requires substantially more samples in the absence of geometric constraints.
\end{remark}
\begin{table}[t!]
\centering
\renewcommand{\arraystretch}{1.5}
\caption{Transition of the MLE convergence rates in the exact-specified setting under different scales of the minimum separation $\Delta_{\mathrm{sep}}$ in single-cluster (S), multi-cluster (M), and unstructured (U) regimes. Below, the notation (*) indicates the point-wise rate (up to a logarithmic factor) characterized in our paper, while (**) refers to the minimax rates established by Doss et al. \cite{doss_optimal_2023} and Wu et al. \cite{wu2020moment}. Note that $k_0=1$ and $\gamma=1$ for the single-cluster and unstructured regimes. Also, $d_{\mathrm{thr}}=n^{\frac{2 k_* - 3}{2k_*-1}}$.}
\begin{tabular}{cccc}
\toprule
 & $\left(\frac{d}{n}\right)^{\frac{1}{2}}$ (*)
&
$\gamma^{-\frac{2k_0-2}{2(k_*-k_0)+1}}\left(\frac{1}{n}\right)^{\frac{1}{4(k_*-k_0)+2}}$ (**)
&
$\left(\frac{d}{n}\right)^{\frac{1}{4}}$ (**)
 \\
\midrule
S&$\Delta_{\mathrm{sep}}\asymp 1$  & $\Delta_{\mathrm{sep}}\asymp n^{-\frac{k_*-1}{(2k_*-1)^2}} d^{\frac{1}{2(2k_* - 1)}}$, $d\leq d_{\mathrm{thr}}$ & $\Delta_{\mathrm{sep}}\asymp  \left( \frac{d}{n}\right)^{\frac{1}{4(2k_* - 1)}}$, $d>d_{\mathrm{thr}}$
\\
\midrule
M&$\Delta_{\mathrm{sep}}\asymp 1$  & $\Delta_{\mathrm{sep}}\asymp \left(
    n^{-\frac{k_*-k_0}{2(k_*-k_0)+1}}
    \gamma^{\frac{2k_0-2}{2(k_*-k_0)+1}}
    \right)^{\frac{1}{2s_{\max}-1}}$  & N/A
\\
\midrule
U&$\Delta_{\mathrm{sep}}\asymp 1$  & $\Delta_{\mathrm{sep}}\asymp n^{-\frac{k_*-1}{(2k_*-1)(4k_*-3)}} d^{\frac{1}{2(4k_* - 3)}}$, $d\leq d_{\mathrm{thr}}$ & $\Delta_{\mathrm{sep}}\asymp  \left( \frac{d}{n}\right)^{\frac{1}{4(4k_* - 3)}}$, $d>d_{\mathrm{thr}}$
\\
\bottomrule
\end{tabular}
\end{table}

\section{Over-specified Setting}
\label{sec:over-specified}
In this section, we study the convergence behavior of the MLE under the over-specified setting, where the number of true components $k_*$ is unknown and the MLE in equation~\eqref{eq:MLE} is computed under mis-specified setting using  $k >  k_*$ components. 
To be clear, while the number of true components $k_*$ is unknown, the geometric arrangement of the true components still plays a fundamental role in determining the MLE convergence behavior under this setting. Thus, we will investigate the same three geometric regimes introduced in Section~\ref{sec:exact-specified}. 

Compared with the exact-specified setting, the analysis becomes considerably more delicate in the over-specified setting due to the presence of extra fitted components. In particular, a single true component may be approximated by multiple fitted components, creating additional flexibility in the latent representation and fundamentally altering the relationship between density estimation and parameter estimation. As a consequence, the 1-Wasserstein distance employed in Section~\ref{sec:exact-specified} is no longer appropriate to characterize the convergence behavior of the fitted mixing measure. Instead, our analysis is based on the 2-Wasserstein distance, which is naturally tuned to capture the Euclidean geometry induced by additional fitted components. This is not a new phenomenon:  previous studies of over-specified finite mixtures already rely on the 2-Wasserstein distance; see, e.g., \cite{Ho-Nguyen-EJS-16}

The transition from $W_1$ to $W_2$ also leads to several notable differences from the exact-specified setting. First, the resulting Hellinger lower bound and MLE convergence rate no longer explicitly hinge on the minimum mixture weight $\pi_{\min}^*$. Secondly, the dependence on the minimum separation $\Delta_{\mathrm{sep}}$ -- or, more precisely, the exponents -- also changes

Note that the main ideas required for establishing local bounds have already been introduced in the exact-specified setting and are implicitly embedded in the proofs of the global bounds presented below. Thus, for brevity, we restrict our attention to only global bounds in the sequel.

\subsection{Single-cluster Regime}
We first evaluate the single-cluster case, where all the $k_*$ true components clustered together. 
The proof is in Appendix~\ref{sec:proof_theorem:one_group_multivariate_global}.

\begin{theorem}
\label{theorem:one_group_univariate}
Suppose that all true components are partitioned into one cluster of diameter $C_0\Delta_{\mathrm{sep}}$, where $C_0 \ge 1$, that is, $\max_{j, l} \|\theta_j^* - \theta_l^*\| \le C_0 \Delta_{\mathrm{sep}}$. Then there exists a universal constant $C_{\mathrm{global,4}}$ depending only on $R$, $d$, $k_{*}$, $\Sigma$, and $C_{0}$ such that
\begin{align} 
d_H(p_{G},p_{G_*}) \ge C_{\mathrm{global,4}} \cdot \Delta_{\mathrm{sep}}^{2k_* - 2} \cdot W_2^2(G, G_*),
\end{align}
uniformly for all mixing measures $G$ with at most $k$ components. As a consequence, we have
\begin{align*}
    \mathbb{P}\left(W_2(\widehat{G}_n,G_*)>\left(\frac{C_{\mathrm{density}}}{C_{\mathrm{global,4}}}\right)^{\frac{1}{2}}\left(\frac{d\log(n)}{n}\right)^{\frac{1}{4}} \Delta_{\mathrm{sep}}^{-(k_* - 1)}\right)\lesssim \exp(-c_{\mathrm{density}}\log(n)).
\end{align*}
\end{theorem}
\begin{remark} 
\label{remark:unicluster_overfit}
Below, we remark on several notable features of the above bounds.

(1) Assuming a fixed separation $\Delta_{\mathrm{min}}$, the convergence rate of the MLE $\widehat{G}_n$ to $G_*$ is of order $[d \log(n)/n]^{1/4}$, which is slower than the standard parametric rate $[d \log(n)/n]^{1/2}$ holding in the exact-specified setting. This slowdown occurs due to the existence of a true component approximated by more than one fitted component.

(2) The impact on the rate of the minimum separation parameter $\Delta_{\mathrm{min}}$ is also different. In particular, the exponent changes from $-(2k_*-1)$ in the exact-specified setting to $-(k_*-1)$ in the over-specified setting. Consequently, although the overall convergence rate becomes slower, the deterioration caused by decreasing separation is substantially less severe than in the exact-specified setting. This change reflects the different geometric structure induced by additional fitted components. Unlike the exact-specified setting, the convergence rate no longer depends on the minimum mixture weight $\pi_{\min}^*$. 

(3)  The convergence behavior of fitted components is determined by the local complexity of their associated Voronoi cells. More precisely, the rate depends on the number of fitted components assigned to the same true component. In the special case where a Voronoi cell contains a single fitted component, we obtain the parametric convergence rate $[d\log(n)/n]^{1/2}$. In particular, to distinguish the different convergence behaviors among fitted centers, we introduce the Voronoi-cell-based distance $S_1$ defined as 
\begin{align*}
    S_1(G,G_*) &= \frac{\Delta_{\mathrm{sep}}}{2}\sum_{j:|\bar{\mathcal{V}}_j|=1,\mathcal{V}_j = \{i\}} \pi_i\|\theta_i - \theta^*_{j}\| +  2\sum_{j:|\bar{\mathcal{V}_j}|>1}\sum_{i \in \bar{\mathcal{V}}_j}\pi_i\|\theta_i - \theta_j^*\|^2 \\
&+ 2\sum_{i \in \mathcal{M}_{\mathrm{far}}} \pi_i\|\theta_i - \theta^*_{c(i)}\|^2+ C_0^2 \Delta_{\mathrm{sep}}^2 \left(\sum_{j=1}^{k_*} |\Delta \pi_j| + \pi_{\mathrm{far}}\right),
\end{align*}
where 
$\bar{\mathcal{V}}_j$ denotes the sub-Voronoi cell of near point $\theta_i$ at true center $\theta_j^*$, i.e. $\|\theta_i -\theta^*_{c(i)}\|\leq \Delta_{\mathrm{sep}}/4$, and the quantities $\Delta\pi_j$ and $\pi_{\mathrm{far}}$ correspond to the 
mass allocation terms and will be formally defined later in 
Section~\ref{sec:proof_theorem:one_group_multivariate_global}. We establish the inequality 
\begin{equation*}
    d_H(p_{G},p_{G_*}) \geq  C_{\mathrm{global,4}} \cdot \Delta_{\mathrm{sep}}^{2k_* - 2} \cdot S_1(G, G_*)  \ge C_{\mathrm{global,4}} \cdot \Delta_{\mathrm{sep}}^{2k_* - 2} \cdot W_2^2(G, G_*). 
\end{equation*}
This result shows that, for any $i \in \bar{\mathcal{V}}_j$ such that 
$|\bar{\mathcal{V}}_j|=1$, the fitted parameter $\widehat{\theta}_{n,i}$ enjoys a much faster 
convergence rate of order $[d \log(n)/n]^{1/2}$. 

(4) Combining the bound $W_2(G,G_*) \geq W_1(G,G_*)$ with Theorem 1.1 in 
\cite{doss_optimal_2023}, we obtain the minimax rate for univariate situation as 
\begin{equation}
    \label{eq:minimax_rate_overfit}
    \inf_{\widetilde{G}_n \in \mathcal{G}_{k}(\Theta)}
    \sup_{G_* \in \mathcal{E}_{k_*}(\Theta)}
    \mathbb{E}_{G_*}
    \left[W_2(\widetilde{G}_n,G_*)\right]
    \gtrsim_{k^*}
    \left( \frac{d}{n} \right)^{1/4} \wedge 1
    +
    \left( \frac{1}{n} \right)^{1/(4k_{*}-2)},
\end{equation}
where the notation $\gtrsim_{k^*}$ indicates that the left-hand side is 
bounded below by the right-hand side up to a multiplicative constant depending only on $k^*$. Although the minimax rate and the pointwise convergence rate established in 
our work capture different statistical behaviors, it is still instructive to 
compare our result with equation \eqref{eq:minimax_rate_overfit}. In particular, when 
the minimum separation $\Delta_{\mathrm{sep}}$ is bounded away from zero, our  estimator achieves the rate of order $\left(d/n\right)^{1/4}$ (ignoring the logarithmic factor). On the other hand, when $\Delta_{\mathrm{sep}}$ vanishes, our rate can adapt to the more challenging regime and match the minimax lower bound. More precisely, denote $d_{\mathrm{thr}} = n^{\frac{2k_*-3}{2k_*-1}}$, the MLE estimator $\widehat{G}_n$ achieves the minimax rate in equation \eqref{eq:minimax_rate_overfit} whenever $\Delta_{\mathrm{sep}}\asymp n^{-\frac{2k_*-3}{4(2k_*-1)(k_*-1)}} d^{\frac{1}{4(k_*-1)}}$, for $d\leq d_{\mathrm{thr}}$, or 
$\Delta_{\mathrm{sep}}\asymp 1$, for
$d>d_{\mathrm{thr}}$.
\end{remark}

\subsection{Multi-cluster Regime}
We now move to the setting when  the $k_*$ true components are partitioned into $k_0$ disjoint topological clusters $\mathcal{C}_1, \dots, \mathcal{C}_{k_0}$, that is
 \begin{itemize}
    \item \textit{Intra-cluster assumption}: For any two components $\theta_j^* \neq \theta_q^* \in \mathcal{C}_m$, their spatial separation is bounded by the cluster diameter $C_0 \Delta_{\mathrm{sep}}$, where $C_0\geq 1$, that is, $|\theta_j^* - \theta_q^*| \le C_0 \Delta_{\mathrm{sep}}$.
    \item \textit{Inter-cluster assumption}: For any $\theta_j^* \in \mathcal{C}_m$ and $\theta_p^* \notin \mathcal{C}_m$, their separation is strictly bounded below by the macroscopic gap $D_0>0$, that is, $|\theta_j^* - \theta_p^*| \ge D_0$.
\end{itemize}
To guarantee that the clusters remain well separated from one another, we further assume $\Delta_{\mathrm{sep}} \le \frac{D_0}{4 C_0}$. Next, we demonstrate that in this multi-cluster regime, 
the dependence of the MLE convergence rate on the separation is governed by the densest local region of the mixture $s_{\max} = \max_{1 \le m \le k_0} s_m$ rather than by the total number of true components $k_*$.

\begin{theorem}
\label{theorem:one_group_univariate_multi_cluster}
Given the intra-cluster and inter-cluster assumptions on $G_{*}$, there exists a positive universal constant $C_{\mathrm{global,5}}$ depending only on $R$, $k_{*}$, $\Sigma$, $C_{0}$, and $D_{0}$ such that
\begin{align} 
d_H(p_{G},p_{G_*}) \ge C_{\mathrm{global,5}} \cdot \Delta_{\mathrm{sep}}^{2s_{\max} - 2} \cdot W_2^2(G, G_*).
\end{align}
uniformly for all $G$ with at most $k$ components. As a consequence, we have
\begin{align*}
    \mathbb{P}\left(W_2(\widehat{G}_n,G_*)>\left(\frac{C_{\mathrm{density}}}{C_{\mathrm{global,5}}}\right)^{\frac{1}{2}}\left(\frac{d\log(n)}{n}\right)^{\frac{1}{4}} \Delta_{\mathrm{sep}}^{-(s_{\max} - 1)}\right)\lesssim \exp(-c_{\mathrm{density}}\log(n)).
\end{align*}
\end{theorem}
The proof of Theorem~\ref{theorem:one_group_univariate_multi_cluster} can be found in Section~\ref{sec:proof:theorem:one_group_univariate_multi_cluster}. 
\begin{remark} 
\label{remark:unicluster_overfit_2}
(1) Compared to the single-cluster regime in the over-specified setting, we first observe that the MLE $\widehat{G}_n$ still converges to the true mixing measure $G_*$ at the rate of order $[\log(n)/n]^{1/4}$.

(2) The convergence rate deteriorates polynomially as the minimum separation decreases, reflecting the fact that stronger component overlap leads to a more challenging parameter estimation problem. However, the dependence on the minimum separation is significantly improved, from $\Delta_{\mathrm{sep}}^{-(k_*-1)}$ in the single-cluster regime to $\Delta_{\mathrm{sep}}^{-(s_{\max}-1)}$ in the multi-cluster regime. Consequently, the advantage of exploiting the multi-cluster structure becomes more pronounced as the number of clusters increases. For example, when $k_0=\mathcal{O}(1)$, the largest cluster size remains of order $\mathcal{O}(k_*)$, and hence the exponent associated with the minimum separation is approximately $-\mathcal{O}(k_*)$. In this case, the resulting convergence rate is comparable to that of the single-cluster regime. In contrast, when $k_0=\mathcal{O}(k_*)$ and the components are distributed approximately evenly among clusters, we have $s_{\max}=\mathcal{O}(1)$. Therefore, the dependence on the minimum separation reduces to only a constant-order exponent, leading to a substantial improvement in the convergence rate. These observations demonstrate that, under the over-specified setting, the effective complexity of parameter estimation is determined not only by the total number of components but also by the local cluster structure. A detailed proof of this result is provided in Section~\ref{sec:proof:theorem:one_group_univariate_multi_cluster}.

(3) As in Theorem~\ref{theorem:one_group_univariate}, the convergence behavior of each fitted center depends on the number of fitted centers contained in its corresponding Voronoi cell. To better characterize the asymptotic behavior of individual centers, we introduce the Voronoi-based distance $S_2$ as 
\begin{align*}
     S_2(G,G_*) &= \frac{\Delta_{\mathrm{sep}}}{2}\sum_{j:|\bar{\mathcal{V}}_j|=1,\mathcal{V}_j = \{i\}} \pi_i\|\theta_i - \theta^*_{j}\| +  2\sum_{j:|\bar{\mathcal{V}_j}|>1}\sum_{i \in \bar{\mathcal{V}}_j}\pi_i\|\theta_i - \theta_j^*\|^2\\
    &\hspace{0.5cm} + 2\sum_{i \in \mathcal{M}_{\mathrm{mac}}} \pi_i\|\theta_i - \theta^*_{c(i)}\|^2 + 2\sum_{i \in \mathcal{M}_{\mathrm{far}}} \pi_i\|\theta_i - \theta^*_{c(i)}\|^2\\
    &\hspace{0.5cm} +2\left(C_0^2\Delta^2_{\mathrm{sep}}\left( \pi_{\mathrm{far}} +\pi_{\mathrm{mac}}+\sum_{j=1}^k|\Delta\pi_j| \right) + 4R^2\sum_{m=1}^{k_0}|\Delta\Pi_m|\right).
\end{align*}
where $\mathcal{V}_j$ denotes the Voronoi cell at true center $\theta_j^*$, $\bar{\mathcal{V}}_j$ denotes the sub-Voronoi cell of near point $\theta_i$ at true center $\theta_j^*$, i.e. $\|\theta_i -\theta^*_{c(i)}\|\leq \Delta_{\mathrm{sep}}/4$,  $\Delta\pi_j$, $\Delta\Pi_m$, $\pi_{\mathrm{far}}$, and $\pi_{\mathrm{mac}}$ capture the corresponding 
mass discrepancies, whose precise definitions are deferred to 
Section~\ref{sec:proof:theorem:one_group_univariate_multi_cluster}. We show that 
\begin{equation*}
    d_H(p_{G},p_{G_*}) \ge C_{\mathrm{global,5}} \cdot \Delta_{\mathrm{sep}}^{2s_{\max} - 2} \cdot W_2^2(G, G_*) \geq C_{\mathrm{global,5}} \cdot \Delta_{\mathrm{sep}}^{2s_{\max} - 2} \cdot S_2(G,G_*), 
\end{equation*}
which implies that for any $i \in \bar{\mathcal{V}}_j$ with 
$|\bar{\mathcal{V}}_j|=1$, the fitted center $\theta_i$ exhibits a significantly 
faster convergence behavior, attaining the rate $[\log(n)/n]^{1/2}$.

(4) By combining the inequality $W_2(G,G_*) \geq W_1(G,G_*)$ with 
\cite[Theorem 2 and Remark 4]{wu2020moment}, we obtain a $W_2$-version of 
the minimax rate in equation \eqref{eqn:adaptive_minimax_rate} for the multi-cluster 
regime, which is given by
\begin{equation}
\label{eqn:multicluster_overspecified_adaptive_rate}
    \inf_{\widetilde{G}_n \in \mathcal{G}_k(\Theta)}
    \sup_{G_*\in \mathcal{E}_{k_*,k_0,\gamma,\omega}}
    W_2(\widetilde{G}_n,G_*)
    \gtrsim_{k^*}
    \gamma^{-\frac{2k_0-2}{2(k_*-k_0)+1}}
    \left(\frac{1}{n}\right)^{\frac{1}{4(k_*-k_0)+2}}.
\end{equation}
As discussed previously in Remark \ref{remark:multi_cluster_exact_fit}, 
the main difference between our analysis and that of \cite{wu2020moment} lies in 
the interpretation of the inter-cluster separation. More specifically, 
Wu et al. \cite{wu2020moment} characterize the explicit contribution of the 
inter-cluster separation $\gamma$ to the convergence rate, whereas our work 
treats this quantity as a macroscopic factor for controlling interactions between 
different clusters. 

In addition, our result reveals the role of the largest local cluster size 
$s_{\max}$ in determining the local convergence behavior. Furthermore, we observe that, 
ignoring logarithmic factors, our rate matches the minimax rate in equation~\eqref{eqn:multicluster_overspecified_adaptive_rate} when
\begin{equation*}
    \Delta_{\mathrm{sep}} 
    \asymp 
    \left(
    n^{-\frac{2(k_*-k_0)-1}{8(k_*-k_0)+4}}
    \cdot\gamma^{\frac{2k_0-2}{2(k_*-k_0)+1}}
    \right)^{\frac{1}{s_{\max}-1}}.
\end{equation*}
\end{remark}

\subsection{Unstructured Regime} We finally consider the unstructured regime, where no assumptions are imposed on the cluster structure of the true components. In contrast to the previous two regimes, the component configuration may simultaneously exhibit both highly concentrated local groups and isolated components at multiple geometric scales. Consequently, neither the single-cluster analysis nor the multi-cluster decomposition can be applied directly. In particular, the absence of a well-defined cluster structure prevents us from isolating the dominant local singularity responsible for the statistical complexity of parameter estimation.

Under the over-specified setting, this challenge becomes even more pronounced. The extra fitted components may interact with the true components across multiple scales, creating significantly more intricate geometric configurations than those encountered in the single-cluster and multi-cluster regimes. As a result, the Hellinger lower bound must account for the worst possible arrangement of component locations without exploiting any additional geometric structure. The following theorem establishes the corresponding global Hellinger lower bound and the resulting convergence rate of the MLE.

\begin{theorem}
\label{theorem:no_group_overspecified}
Under the unstructured regime, there exists a positive universal constant $C_{\mathrm{global,6}}$ depending only on $R$, $k_{*}$, and $\Sigma$ such that
\begin{align} 
d_H(p_{G},p_{G_*}) \ge C_{\mathrm{global,6}}  \cdot \Delta_{\mathrm{sep}}^{4k_{*} - 3} \cdot W_2^2(G, G_*).
\end{align}
uniformly for all $G$ with at most $k$ components. As a consequence, we have
\begin{align*}
    \mathbb{P}\left(W_2(\widehat{G}_n,G_*)>\left(\frac{C_{\mathrm{density}}}{C_{\mathrm{global,6}}}\right)^{\frac{1}{2}}\left(\frac{d\log(n)}{n}\right)^{\frac{1}{4}} \Delta_{\mathrm{sep}}^{-\left(2k_*-\frac{3}{2}\right)}\right)\lesssim \exp(-c_{\mathrm{density}}\log(n)).
\end{align*}
\end{theorem}

A complete proof of this theorem is presented in Section~\ref{sec:proof_theorem:no_group_overspecified}. 
\begin{remark}

(1) Unlike the single-cluster and multi-cluster regimes in the over-specified settings, the exponent of $\Delta_{\mathrm{sep}}$ in the Hellinger lower bound remains unchanged from that of the exact-specified unstructured regime. The primary effect of model over-specification therefore manifests itself through the transition from $W_1$ to $W_2^2$, which yields the slower MLE convergence rate of order $[d\log(n)/n]^{1/4}$, while preserving the same separation exponent $\Delta_{\mathrm{sep}}^{4k_*-3}$.

(2) However, when turning to the MLE convergence, the dependence on the minimum separation becomes substantially less favorable than in the previous geometric regimes. In particular, the exponent of the minimum separation increases from $k_*-1$ and $s_{\max}-1$ to $2k_*-\frac{3}{2}$, respectively. This deterioration reflects the absence of exploitable geometric structure in the unstructured regime. Without a well-defined cluster decomposition, the analysis must simultaneously accommodate all possible local and global interactions among components, leading to a substantially more conservative dependence on the minimum separation.

In addition, as in the previous over-specified setting, the convergence rate remains independent of the minimum mixture weight $\pi_{\min}^*$, indicating that the dominant statistical complexity is driven by geometric degeneracies rather than by low-weight components.

(3) As in the single- and multi-cluster regimes, we also observe the phenomenon that fitted centers associated with singleton Voronoi cells enjoy a faster convergence rate of order $[d\log(n)/n]^{1/2}$. 
This behavior is characterized through the Voronoi-based distance $S_3$ defined as
\begin{equation*}
    S_3(G,G_*) := \frac{\Delta_{\mathrm{sep}}}{2} \sum_{j:|\bar{\mathcal{V}}_j|=1,\bar{\mathcal{V}}_j=\{i\}}\pi_i\|\theta_i-\theta^*_{c(i)}\| + 2\sum_{j:|\bar{\mathcal{V}}_j|\neq 1}\sum_{i \in \mathcal{V}_j}\pi_i\|\theta_i - \theta_{c(i)}^*\|^2+8R^2 \sum_{i=1}^{k_*} |\tilde{\Delta} \pi_i|, 
\end{equation*}
where 
$\tilde{\Delta}\pi_j$ represents a mass reallocation quantity that will be specified later in 
Section~\ref{sec:proof_theorem:no_group_overspecified}. We finally can show that 
\begin{equation*}
    d_H(p_{G},p_{G_*}) \ge C_{\mathrm{global,6}}  \cdot \Delta_{\mathrm{sep}}^{4k_{*} - 3} \cdot W_2^2(G, G_*) \geq C_{\mathrm{global,6}}  \cdot \Delta_{\mathrm{sep}}^{4k_{*} - 3} \cdot S_3^2(G, G_*), 
\end{equation*}
which demonstrates that the fitted parameter $\widehat{\theta}_{n,i}$ admits the convergence rate of order $[\log(n)/n]^{1/2}$, for any $i \in \bar{\mathcal{V}}_j$ such that $|\bar{\mathcal{V}}_j| = 1$. 

(4) It is still instructive to compare this result with the minimax rate adapted 
from \cite{doss_optimal_2023} and reformulated in equation 
\eqref{eq:minimax_rate_overfit}. Recall that when 
$d \geq d_{\mathrm{thr}}:=n^{\frac{2k_*-3}{2k_*-1}}$, the minimax rate is 
determined by the term $(d/n)^{1/4}$, whereas it is dominated by 
$(1/n)^{1/(4k_*-2)}$ when $d<d_{\mathrm{thr}}$.

If the separation factor $\Delta_{\mathrm{sep}}$ remains bounded away from 
zero, the MLE achieves the standard parametric rate of order $(d/n)^{1/4}$, 
which is substantially faster than the minimax lower bound in 
\eqref{eq:minimax_rate_overfit}. On the other hand, when the separation 
factor vanishes at the rate $\Delta_{\mathrm{sep}}\asymp n^{-\frac{2k_*-3}{2(2k_*-1)(4k_*-3)}}d^{\frac{1}{2(4k_*-3)}}$,
for $d\leq d_{\mathrm{thr}}$ (and remains of constant order when 
$d>d_{\mathrm{thr}}$), the above MLE convergence rate matches the minimax 
rate.
\end{remark}

\begin{table}[t!]
\centering
\renewcommand{\arraystretch}{1.5}
\caption{Transition of the MLE convergence rates in the over-specified setting under different scales of the minimum separation $\Delta_{\mathrm{sep}}$ in single-cluster (S), multi-cluster (M), and unstructured (U) regimes. Below, the notation (*) indicates the point-wise rate (up to a logarithmic factor) characterized in our paper, while (**) refers to the minimax rates established by Doss et al. \cite{doss_optimal_2023} and Wu et al. \cite{wu2020moment}. Note that $k_0=1$ and $\gamma=1$ for the single-cluster and unstructured regimes. Also, $d_{\mathrm{thr}}=n^{\frac{2 k_* - 3}{2k_*-1}}$.}
\begin{tabular}{ccc}
\toprule
 & $\left(\frac{d}{n}\right)^{\frac{1}{4}}$ (*), (**)
&
$\gamma^{-\frac{2k_0-2}{2(k_*-k_0)+1}}\left(\frac{1}{n}\right)^{\frac{1}{4(k_*-k_0)+2}}$ (**)

 \\
\midrule
S&$\Delta_{\mathrm{sep}}\asymp 1$  & $\Delta_{\mathrm{sep}}\asymp n^{-\frac{2k_*-3}{4(2k_*-1)(k_*-1)}} d^{\frac{1}{4(k_* - 1)}}$, $d\leq d_{\mathrm{thr}}$ 
\\
\midrule
M&$\Delta_{\mathrm{sep}}\asymp 1$  & $\Delta_{\mathrm{sep}}\asymp \left(
    n^{-\frac{2(k_*-k_0)-1}{8(k_*-k_0)+4}}
    \gamma^{\frac{2k_0-2}{2(k_*-k_0)+1}}
    \right)^{\frac{1}{s_{\max}-1}}$  
\\
\midrule
U&$\Delta_{\mathrm{sep}}\asymp 1$  & $\Delta_{\mathrm{sep}}\asymp n^{-\frac{k_*-1}{2(2k_*-1)(4k_*-3)}} d^{\frac{1}{2(4k_* - 3)}}$, $d\leq d_{\mathrm{thr}}$ 
\\
\bottomrule
\end{tabular}
\end{table}
\section{Proof Sketches}
\label{sec:proof_sketch}
The main idea underlying all proofs is to decompose the Wasserstein distance between $G$ and the true measure $G_*$ according to the geometric location of each atom of $G$ relative to its associated reference atom. 
\begin{equation*}
    W_p^p(G,G') \leq \underbrace{\sum_{i}\pi_i \|\theta_i - \theta^*_{c(i)}\|^p}_{\text{location reallocation}} + \underbrace{\sum_{j}|\sum_{i:c(i)=j}\pi_i - \pi^*_{j}|\cdot \text{dist}^p_j}_{\text{mass reallocation}}, 
\end{equation*}
where $c(i)$ denotes the index of the nearest true center $\theta^*_{j}$ to $\theta_i$, and $\text{dist}_j$ denotes an upper bound on the transportation cost of moving mass from the center $\theta_j^*$ to another center. This decomposition corresponds to a natural transportation scheme in which each atom $\theta_i$ is first transported to its nearest reference center, after which the remaining mass discrepancy is redistributed among appropriate centers. In the multi-cluster setting (Theorems~\ref{theorem:exact_multi_group} and \ref{theorem:one_group_univariate_multi_cluster}), as transportation within clusters is less costly we prioritize transportation within each cluster, reallocating mass locally whenever possible before transporting any remaining mass across different clusters.
$$W^p_p(G, G_*) \le \sum_{i} \pi_{i}\|\theta_{i} - \theta_{c(i)}^{*}\|^p_{} + \underbrace{C^p_{0} \Delta^p_{\mathrm{sep}}\sum_{m=1}^{k_0} \sum_{j \in \mathcal{C}_m} |\sum_{i: c(i)=j}\pi_i- \pi_{j}^{*}|}_{\text{intra-cluster mass reallocation}} + \underbrace{(2R)^p \sum_{m = 1}^{k_{0}} |\Delta \Pi_{m}|}_{{\text{inter-cluster mass reallocation}}},$$
where $\Delta \Pi_{m}$ denotes the mass discrepancy between $G$ and $G'$ within cluster $\mathcal{C}_m$, denoted by
$$\Delta \Pi_{m} = \sum_{i:c(i) \in \mathcal{C}_m}\pi_i - \sum_{j\in \mathcal{C}_m}\pi^*_j.$$

To estimate the reallocation terms in a neighborhood of a given true center $\theta_j^*$, including both mass and location discrepancies, we construct a suitable multivariate polynomial $P$ in $d$ variables. For an atom $\theta_i$ that is associated with the true center $\theta^*_{c(i)} = \theta^*_j$, Taylor's expansion yields
\begin{equation*}
    P(\theta_i) = P(\theta^*_j) + \nabla P(\theta^*_j)(\theta_i-\theta^*_j) + \frac{1}{2}
    (\theta_i - \theta^*_j)^{\top}
    D^2P(\xi_i)
    (\theta_i-\theta^*_j),
\end{equation*}
where $\xi_i$ lies on the line segment joining $\theta_i$ and $\theta^*_j$.
To quantify the local discrepancy around $\theta_j^*$, we perform a second-order Taylor expansion of $P$ at $\theta_j^*$ for all centers $\theta_i$ lying in a neighborhood of $\theta_j^*$. Integrating the resulting expansion with respect to the signed measure $\nu := G-G_*$, we obtain
\begin{align*}
    \int Pd\nu &= \underbrace{\sum_{i:c(i)=j}\pi_i(\theta_i - \theta_j^*)^\top\nabla P(\theta_j^*) + P(\theta^*_{j})\left(\sum_{i:c(i)=j}\pi_i - \pi_j^*\right)}_{\text{extracting mass/center reallocation}}\\
    &+ \underbrace{\frac{1}{2}\sum_{i:c(i)=j}\pi_i(\theta_i - \theta^*_{j})^{\top}D^2P(\xi_i)(\theta_i - \theta^*_{j})}_{\text{quadratic remainder}} + \underbrace{\sum_{l:\theta_l \text{ far from } \theta_j^*}P(\theta_l) - \sum_{l\neq j}\pi^*_jP(\theta_j^*)}_{\text{far component}}. 
\end{align*}

To isolate quantities of interest, such as the mass and center reallocation errors, it suffices to construct a suitable polynomial $P$ whose values and derivatives at the true centers $\theta_l^*$ are prescribed so that the zeroth- and first-order terms in the Taylor expansion recover the desired quantity. For instance, to extract the mass reallocation error associated with $\theta_j^*$, we choose $P$ such that
$$P(\theta_l^*)=\delta_{jl},
\qquad
\nabla P(\theta_l^*)=0,$$
for all $l$. On the other hand, to extract the center reallocation error at $\theta_j^*$, we impose
$$P(\theta_l^*)=0,
\qquad
\nabla P(\theta_l^*)=0 \quad \text{for } l\neq j,$$
while prescribing $\nabla P(\theta_j^*)$ to be a suitable vector that will be specified later.

Once such a polynomial has been constructed, the remaining task is to estimate the integral $\int P\,d\nu$,
quadratic remainder, and far component. This estimate is obtained by combining three ingredients: 
Lemma~\ref{lemma:Hellinger_to_Polynomial}, which establishes a direct connection between the polynomial moment discrepancy and the Hellinger distance $d_H(p_G,p_{G_*})$ between mixture densities; the control of the quadratic remainder term arising from the Taylor expansion (which motivates the need for variance estimates); and the bound on the contribution from the far components. The main technical challenge lies in controlling the magnitude of $P$ and its Hessian matrix. Since the construction of $P$ depends on the configuration of the true centers $\{\theta_l^*\}$, these quantities are intimately related to the separation complexity $\Delta_{\mathrm{sep}}$. From an approximation-theoretic perspective, this phenomenon reflects the inflation of interpolation coefficients as interpolation points become increasingly close. From a statistical perspective, it captures the intrinsic difficulty of accurately distinguishing and estimating parameters that are nearly indistinguishable from one another.

In addition, we also observe the contribution of the minimum mass $\pi_{\min}^*$. 
This term naturally appears in the exact-fit setting when a mass imbalance occurs: one Voronoi cell is overfitted by containing multiple centers, while another Voronoi cell contains no center. 
In this situation, transporting mass to compensate for the empty Voronoi cell incurs an additional cost. Consequently, the Wasserstein distance is significantly influenced by the amount of missing mass associated with this empty cell, which is controlled by $\pi_{\min}^*$. In the local regime, where the Wasserstein distance is sufficiently small, Lemma~\ref{lemma:small_wasserstein_implies_center_in_voronoi_cell} shows that such a phenomenon cannot occur, since every center is guaranteed to lie inside its corresponding sub-Voronoi cell.

For the over-specified setting, we follow the same proof strategy as above. The main challenge in this setting arises from the fact that a single true component can be approximated by multiple fitted components. This phenomenon alters the relationship between density estimation and parameter estimation and cannot be fully captured by the $W_1$ distance. This motivates the use of the $W_2$ distance. From a technical perspective, consider a second-order Taylor expansion of a test function $U$ around a true center $\theta_j^*$:
$$U(\theta )= U(\theta_j^*)+(\theta-\theta_j^*)^\top\nabla U(\theta_j^*)+\frac{1}{2}(\theta-\theta_j^*)^\top D^2U(\xi_j)
(\theta-\theta_j^*).$$
If one works with the $W_1$ distance, the analysis essentially relies on controlling the first-order displacement terms. However, when several fitted centers collapse toward the same true center $\theta_j^*$, these first-order contributions may cancel after taking the weighted summation. More precisely, the aggregated first-order deviation can vanish even though the individual fitted centers remain separated from $\theta_j^*$. This cancellation reflects the merging behavior of multiple fitted components around a single true component, making the discrepancy between mixing measures significantly harder to quantify.

To overcome this difficulty, we exploit the variance test function
\[
P_{\mathrm{var}}(\theta)
=
\prod_{i=1}^{k_*}
\|\theta-\theta_i^*\|^2 .
\]
For a fitted center $\theta_i$ associated with the true center $\theta_j^*$, this construction yields a lower bound of the form
\[
P_{\mathrm{var}}(\theta_i)
\geq
C\,
\Delta_{\mathrm{sep}}^{\mathcal{O}(k_*)}
\|\theta_i-\theta_j^*\|^2 .
\]
where $\mathcal{O}(k_*)$ denotes a linear factor in $k_*$ whose constant depends on the underlying separation structure of the mixing measure, including the single-cluster, multi-cluster, and unstructured regimes. By integrating this test function with respect to the signed measure $\nu=G-G_*$ and applying the above lower bound, we can recover the second-order center reallocation error
\[
\sum_{i=1}^{k}
\pi_i
\|\theta_i-\theta_{c(i)}^*\|^2,
\]
which provides the necessary control of the $W_2$ discrepancy.

\subsection{Proof Sketch of Theorem~\ref{theorem:exact_one_group_univariate}}
\begin{proof}[Proof Sketch of Theorem~\ref{theorem:exact_one_group_univariate}]
For ease of presentation, we will sketch the proof for the univariate setting only, that is, when $d=1$, while full proofs of univariate and multivariate settings will be provided in the appendices. 

We proceed to derive the global bound in six main steps. For every fitted atom $\theta_i$, let $c(i) = \text{argmin}_j |\theta_i - \theta_j^*|$ be the index of true center that is closest to that fitted atom. To further facilitate the arguments, we unconditionally partition the components of $G$ into two disjoint sets based on a strictly enforced boundary of $\Delta_{\mathrm{sep}}/4$:
\begin{itemize}
    \item The near set $\mathcal{M}_{\mathrm{near}} := \{ i \in [k_*] : |\theta_{i} - \theta_{c(i)}^{*}| \le \Delta_{\mathrm{sep}} / 4 \}$. We denote $\bar{\mathcal{V}}_j = \{i \in \mathcal{M}_{\mathrm{near}} : c(i) = j\}$ as a subset of the Voronoi cells $\mathcal{V}_{j}$ and $\Delta \pi_j : = \left( \sum_{i \in \bar{\mathcal{V}}_j} \pi_i \right) - \pi_j^*$.
    \item The far set $\mathcal{M}_{\mathrm{far}} := \{ i \in [k_*] : |\theta_{i} - \theta_{c(i)}^{*}| > \Delta_{\mathrm{sep}} / 4 \}$. We denote the total far mass as $\pi_{\mathrm{far}} = \sum_{i \in \mathcal{M}_{\mathrm{far}}} \pi_i$.
\end{itemize}
\emph{Step 1 - Wasserstein decomposition.} By involving an intermediate mixing measure $\widetilde{G} = \sum_{i=1}^k \pi_i \delta_{\theta_{c(i)}^*}$ and with a suitable choice of transportation map, we can decompose the Wasserstein distance $W_1(G,G_*)$ as 
\begin{align*}
W_1(G, G_*) \le \sum_{i \in \mathcal{M}_{\mathrm{near}}} \pi_i |\theta_{i} - \theta_{c(i)}^{*}| + \sum_{i \in \mathcal{M}_{\mathrm{far}}}  \pi_i |\theta_{i} - \theta_{c(i)}^{*}| + \frac{1}{2} C_0 \Delta_{\mathrm{sep}} \left( \sum_{j=1}^{k_*} |\Delta \pi_j| + \pi_{\mathcal{F}} \right). 
\end{align*}

\emph{Step 2 - Bounding near variance and far high moment for variance. } By using test function $P_{\mathrm{var}}(\theta) : = \prod_{l=1}^{k_*} (\theta - \theta_l^*)^2$ and noting that $|\theta_i -\theta^*_{j}| \geq 3\Delta_{\mathrm{sep}}/4$ for $j \neq c(i)$, we have 
\begin{equation}
\label{eqn:dung_proof_sketch_3_2_variance_bound}
    \sum_{i\in \mathcal{M}_{\mathrm{near}}} \pi_i(\theta_i-\theta^*_{c(i)})^2 \leq C_{\mathrm{var},2}\Delta^{-(2k_*-2)}_{\mathrm{sep}}d_H(p_G,p_{G_*}).
\end{equation}

Furthermore, recall that $P_{\mathrm{var}}(\theta) : = \prod_{l=1}^{k_*} (\theta - \theta_l^*)^2$, we have
\begin{align}
    \label{eq:exact_one_sketch_7}
    \sum_{i \in \mathcal{M}_{\mathrm{far}}} \pi_i (\theta_{i} - \theta_{c(i)}^{*})^{2k_*}&\leq \sum_{i \in \mathcal{M}_{\mathrm{far}}} \pi_iP_{\mathrm{var}}(\theta_i)\nonumber\\
    &\leq \int P_{\mathrm{var}}(\theta) dG(\theta) = \int P_{\mathrm{var}}(\theta) d\nu(\theta)\leq C_{\mathrm{var}}C_{\mathrm{poly}} d_H(p_{G},p_{G_*}),
\end{align}
where the last inequality occurs due to Lemmas~\ref{lemma:Hellinger_to_Polynomial} and \ref{lemma:variance_test_function}.

\emph{Step 3 - Bounding far parameters.} For the far mean discrepancy, we have 
$$ \sum_{i \in \mathcal{M}_{\mathrm{far}}}\pi_i|\theta_{i} - \theta_{c(i)}^{*}| =  \sum_{i \in \mathcal{M}_{\mathrm{far}}}\pi_i\cdot\frac{|\theta_{i} - \theta_{c(i)}^{*}|^{2k_*}}{ |\theta_{i} - \theta_{c(i)}^{*}|^{2k_*-1}} < \left(\frac{4}{\Delta_{\mathrm{sep}}}\right)^{2k_*-1} \sum_{i \in \mathcal{M}_{\mathrm{far}}} \pi_i (\theta_{i} - \theta_{c(i)}^{*})^{2k_*}.$$

 As a result, we get
\begin{align*}
    \sum_{i \in \mathcal{M}_{\mathrm{far}}}\pi_i|\theta_{i} - \theta_{c(i)}^{*}|\leq C_{\mathrm{far,mean,1}}\Delta_{\mathrm{sep}}^{-(2k_*-1)} d_H(p_{G},p_{G_*}).
\end{align*}
For the total far mass $\pi_{\mathrm{far}}$, for any index $i \in \mathcal{M}_{\mathrm{far}}$, we have $1 = (\theta_{i} - \theta_{c(i)}^{*})^{2k_*} (\theta_{i} - \theta_{c(i)}^{*})^{-2k_*} < (\theta_{i} - \theta_{c(i)}^{*})^{2k_*} \left(\frac{4}{\Delta_{\mathrm{sep}}}\right)^{2k_*}$, which follows that
\begin{align*}
    \pi_{\mathrm{far}} = \sum_{i \in \mathcal{M}_{\mathrm{far}}} \pi_i \le \left(\frac{4}{\Delta_{\mathrm{sep}}}\right)^{2k_*} \sum_{i \in \mathcal{M}_{\mathrm{far}}} \pi_i (\theta_{i} - \theta_{c(i)}^{*})^{2k_*} \le C_{\mathrm{far,mass,1}} \Delta_{\mathrm{sep}}^{-2k_*} d_H(p_{G},p_{G_*}). 
\end{align*} 
\emph{Step 4 - Bounding the near mass discrepancy.} To extract the near mass discrepancy, we consider the Hermite polynomial function $E_j(\theta)$ defined as 
\begin{align*}
    E_j(\theta) : = \ell_j^2(\theta) \left[ \bar{A}_j + \bar{B}_j(\theta - \theta_j^*)\right], \quad \text{where} \quad \ell_j(\theta) = \prod_{q \neq j} \frac{\theta - \theta_q^*}{\theta_j^* - \theta_q^*},
\end{align*}
and $\bar{A}_j = 1$, $\bar{B}_j = - 2\ell'_j(\theta_j^*)$.
Since $E_j(\theta^*_l)=\delta_{jl}$ and $E_j'(\theta^*_l)=0$, for all $l\in[k_*]$, we find that
\begin{align*}
    \int E_j(\theta) d\nu(\theta) &=\sum_{i=1}^{k_*} \pi_{i} E_j(\theta_{i}) - \pi_j^*=\sum_{i \in \mathcal{M}_{\mathrm{near}}} \pi_i E_j(\theta_i)- \pi_j^*+\sum_{i \in \mathcal{M}_{\mathrm{far}}} \pi_i E_j(\theta_i)\\
    &=\Delta \pi_j + \frac{1}{2}\sum_{i \in \mathcal{M}_{\mathrm{near}}} \pi_{i} E_j''(\xi_i) (\theta_{i} - \theta_{c(i)}^{*})^2 + \sum_{i \in \mathcal{M}_{\mathrm{far}}} \pi_i E_j(\theta_i),
\end{align*}
where the last equality is obtained by performing the Taylor expansion for $E_j(\theta^*_i)$ around $\theta^*_{c(i)}$, for  $i\in\mathcal{M}_{\mathrm{near}}$. Using Lemma \ref{lemma:Hellinger_to_Polynomial} and Lemma \ref{lemma:mass_test_function} to bound the integral with respect to $E_j$ and Lemma \ref{lemma:mass_test_function} and estimation \eqref{eqn:dung_proof_sketch_3_2_variance_bound} to estimate $\sum_{i \in \mathcal{M}_{\mathrm{near}}} \pi_{i} E_j''(\xi_i) (\theta_{i} - \theta_{c(i)}^{*})^2$, we have 
\begin{align*}
    \left|\int E_j(\theta) d\nu(\theta)\right|+\frac{1}{2}\sum_{i \in \mathcal{M}_{\mathrm{near}}} \pi_{i} |E_j''(\xi_i)| (\theta_{i} - \theta_{c(i)}^{*})^2\leq \bar{C}_{\text{mass}}\Delta_{\mathrm{sep}}^{-2k_*} d_H(p_{G},p_{G_*}),
\end{align*}
where $\bar{C}_{\text{mass}} = 2R C_{\mathrm{poly}}C_{\mathrm{norm},E} + \frac{1}{2} C_{E,1} C_{\mathrm{var},2}$. Thus, it is sufficient to bound $$\sum_{i \in \mathcal{M}_{\mathrm{far}}} \pi_i |E_j(\theta_i)|.$$ For any $i\in\mathcal{M}_{\mathrm{far}}$, we have $\Delta_{\mathrm{sep}}<4|\theta_i - \theta_{c(i)}^*|$.
Based on this property, we find that
$$|E_j(\theta_i)|\leq C_{\mathrm{far,E,1}}\Delta_{\mathrm{sep}}^{-(2k_*-1)}|\theta_i - \theta_{c(i)}^*|^{2k_* - 1} \leq 4 C_{\mathrm{far,E,1}}\Delta_{\mathrm{sep}}^{-2k_*}|\theta_i - \theta_{c(i)}^*|^{2k_*},$$ where $C_{\mathrm{far,E,1}}>0$ is some universal constant. This leads to
\begin{align*}
\sum_{i \in \mathcal{M}_{\mathrm{far}}} \pi_i |E_j(\theta_i)| &\le 4C_{\mathrm{far,E,1}} \Delta_{\mathrm{sep}}^{-2k_*} \sum_{i \in \mathcal{M}_{\mathrm{far}}} \pi_i |\theta_i - \theta_{c(i)}^*|^{2k_*} \nonumber\\
& \leq 4 C_{\mathrm{far,E,1}} C_{\mathrm{var}}C_{\mathrm{poly}} \Delta_{\mathrm{sep}}^{-2k_*} d_H(p_{G},p_{G_*}), 
\end{align*}
where the last inequality comes from equation~\eqref{eq:exact_one_sketch_7}. As a consequence, we obtain
\begin{align}
    \label{eq:exact_one_sketch_9}
    |\Delta \pi_j| & \le  C_{\mathrm{mass},2}\Delta_{\mathrm{sep}}^{-2k_*} d_H(p_{G},p_{G_*}). 
\end{align}
\emph{Step 5 - Bounding the near mean discrepancy.} To extract the near mass discrepancy, we utilize the Hermite polynomial function $H_j(\theta)$ given by 
\begin{align*}
    H_{j}(\theta) : = \ell_j^2(\theta) (\theta - \theta_j^*),  \quad \text{where} \quad \ell_j(\theta) : = \prod_{q \neq j} \frac{\theta - \theta_q^*}{\theta_j^* - \theta_q^*}, 
\end{align*}
This polynomial satisfies $H_j(\theta_l^*) = 0$ and $H'_j(\theta_l^*) = \delta_{jl}$. By taking the integral of $H_j$ with respect to sign measure $\nu$ and using Taylor expansion for $H_j$ at $\theta_i$ near its true center $\theta^*_{c(i)}$, we have 
\begin{align*}
    \int H_j(\theta) d\nu(\theta) = \sum_{i \in \bar{\mathcal{V}}_{j}} \pi_{i}(\theta_{i} - \theta_{j}^{*}) + \frac{1}{2} \sum_{l=1}^{k_*} \sum_{i \in \bar{\mathcal{V}}_{l}} \pi_{l} H_{j}''(\xi_l) (\theta_{i} - \theta_{l}^{*})^2.
\end{align*}
Note that a sub-Voronoi cell $\bar{\mathcal{V}}_j$ in this global setting may contain more than one element. Thus, we consider two cases as follows.

\emph{Case 5.1.} There exists a sub-Voronoi cell that has more than one element. Since $G$ and $G_*$ have the same number of centers, there must exist at least one Voronoi cell of $G_*$ that contains no center of $G$. Without loss of generality, assume that this cell is $\bar{\mathcal{V}}_1$. This assumption together with equation~\eqref{eq:exact_one_sketch_9} implies
\begin{align*}
    \pi^*_{\min}\leq \pi_1^*=|\Delta\pi_1|\leq C_{\mathrm{mass},2}\Delta_{\mathrm{sep}}^{-2k_*} d_H(p_{G},p_{G_*}). 
\end{align*}
Since $|\theta_{i} - \theta_{c(i)}^{*}| \leq \Delta_{\mathrm{sep}}/4$ for all $i \in \mathcal{M}_{\mathrm{near}}$, we have
\begin{align*}
    \sum_{i \in \mathcal{M}_{\mathrm{near}}} \pi_i |\theta_{i} - \theta_{c(i)}^{*}| \leq \dfrac{\Delta_{\mathrm{sep}}}{4} & \leq \dfrac{\Delta_{\mathrm{sep}}}{4} \cdot \dfrac{C_{\mathrm{mass},2} \Delta_{\mathrm{sep}}^{-2k_*} d_H(p_{G},p_{G_*})}{\pi_{\text{min}}^{*}} \nonumber \\
    & = C_{\mathrm{near,mean,1}} (\pi_{\text{min}}^{*})^{-1}\Delta_{\mathrm{sep}}^{-(2k_* - 1)}d_H(p_{G},p_{G_*}). 
\end{align*}

\emph{Case 5.2.} There are no sub-Voronoi cells $\bar{\mathcal{V}}_{j}$ that have more than one element. If there exists some empty sub-Voronoi cell $\bar{\mathcal{V}}_{j}$, we can argue in a similar fashion to Case 5.1. Therefore, it is sufficient to assume that all the sub-Voronoi cells $\bar{\mathcal{V}}_{j}$ have exactly one element. Without loss of generality, we can suppose that $\theta_j \in \bar{\mathcal{V}}_j$ for all $j$.  In this case, we have 
\begin{equation}
\label{eqn:dung_proof_sketch_thm_3_1_mean_discrepancy}
    \pi_i(\theta_i - \theta^*_{i}) = \int H_j(\theta)d\nu(\theta) - \frac{1}{2}\sum_{j=1}^{k_*}\pi_jH''_j(\xi_j)(\theta_j-\theta_j^*)^2. 
\end{equation}

Note that we can employ the results of Lemma~\ref{lemma:Hellinger_to_Polynomial} and Part (a) of Lemma~\ref{lemma:mean_test_function} to bound the first term in RHS of \eqref{eqn:dung_proof_sketch_thm_3_1_mean_discrepancy} as
\begin{align}
    \label{eq:exact_one_sketch_2}
    \left|\int H_j(\theta) d\nu(\theta)\right| \leq C_{\mathrm{poly}}\|H_j\|_{\infty}d_H(p_{G},p_{G_*})\leq C_{\mathrm{poly}}C_{\mathrm{norm},H}\Delta_{\mathrm{sep}}^{-(2k_*-2)} d_H(p_{G},p_{G_*}).
\end{align}
where $C_{\mathrm{poly}},C_{\mathrm{norm},H}>0$ are some universal constants. Regarding the second term in RHS of \eqref{eqn:dung_proof_sketch_thm_3_1_mean_discrepancy}, we use the variance discrepancy in \eqref{eqn:dung_proof_sketch_3_2_variance_bound}. In addition, Part (a) of Lemma~\ref{lemma:mean_test_function} also gives us that $\max_{\theta \in \mathcal{V}_{l}: |\theta - \theta_{l}^{*}| \leq \Delta_{\mathrm{sep}}/4} |H_j''(\theta)|\leq \Delta_{\mathrm{sep}}^{-1} C_{H,1}$. Finally, we get 
$\sum_{i = 1}^{k_*} \pi_{i}|\theta_{i} - \theta_{i}^{*}| \leq C_{\mathrm{mean},2}\Delta_{\mathrm{sep}}^{-(2k_*-1)} d_H(p_{G},p_{G_*})$. This together with the fact that $1\leq(\pi_{\text{min}}^{*})^{-1}$ implies
\begin{align*}
    \sum_{i \in \mathcal{M}_{\mathrm{near}}} \pi_i |\theta_{i} - \theta_{c(i)}^{*}| &= \sum_{j = 1}^{k_{*}} \sum_{i \in \bar{\mathcal{V}}_{j}} \pi_{i}|\theta_{i} - \theta_{j}^{*}| \leq C_{\mathrm{near,mean,2}}(\pi_{\text{min}}^{*})^{-1} \Delta_{\mathrm{sep}}^{-(2k_* - 1)}d_H(p_{G},p_{G_*}).
\end{align*}
\emph{Step 6 - Bounding $W_1(G,G_*)$ and conclusion.} Putting the results of the above five steps, we reach the global bound.
\end{proof}

\subsection{Proof Sketch of Theorem~\ref{theorem:exact_multi_group}}
\begin{proof}[Proof Sketch of Theorem~\ref{theorem:exact_multi_group}]
    Similar to the proof sketch of Theorem~\ref{theorem:exact_one_group_univariate}, we will only sketch the proof for the univariate setting for ease of presentation. 

We proceed to prove the local bound in seven steps. Before that, we unconditionally partition the $k_*$ atoms of $G$ into three disjoint sets based on exact spatial thresholds 
\begin{itemize}
    \item The micro near set $\mathcal{M}_{\mathrm{mic}} = \{ i : |\theta_i - \theta_{c(i)}^*| \le \frac{\Delta_{\mathrm{sep}}}{4} \}$. We define the explicit Voronoi cells strictly inside this core as $\bar{\mathcal{V}}_j = \{i \in \mathcal{M}_{\mathrm{mic}} : c(i) = j\}$. 
    \item The macro near set $\mathcal{M}_{\mathrm{mac}} = \{ i : \frac{\Delta_{\mathrm{sep}}}{4} < |\theta_i - \theta_{c(i)}^*| \le \frac{D_0}{4} \}$, which includes atoms that have escaped the micro near set but still associate with the cluster $\mathcal{C}_{m(i)}$. We denote the total macro near mass as $\pi_{\mathrm{mac}} = \sum_{i \in \mathcal{M}_{\mathrm{mac}}} \pi_i$.
    \item The far void set $\mathcal{M}_{\mathrm{far}} = \{ i : |\theta_i - \theta_{c(i)}^*| > \frac{D_0}{4} \}$, which consists of atoms in the void between distinct clusters. We denote the total far mass as $\pi_{\mathrm{far}} = \sum_{i \in \mathcal{M}_{\mathrm{far}}} \pi_i$.
\end{itemize}
\emph{Step 1 - Wasserstein decomposition.} Let $\Delta \pi_j : = \left( \sum_{i \in \bar{\mathcal{V}}_j} \pi_i \right) - \pi_j^*$ and $\Delta\Pi_m = \sum_{j \in \mathcal{C}_m} \tilde{\Delta}\pi_j$, then by using intermediate measure $\tilde{G} = \sum_{j=1}^{k_*}\pi_j\delta_{\theta^*_{c(j)}}$ and suitable transportation plan, we have
\begin{align*}
\nonumber
    W_1(G, G_*) &\le \sum_{i \in \mathcal{M}_{\mathrm{mic}}} \pi_i|\theta_{i}-\theta^*_{c(i)}| \ + \sum_{i \in \mathcal{M}_{\mathrm{mac}} \cup \mathcal{M}_{\mathrm{far}}} \pi_i|\theta_{i}-\theta^*_{c(i)}|+C_0 \Delta_{\mathrm{sep}} \sum_{j=1}^{k_*} |\Delta \pi_j|\\
    &\hspace{4cm} + C_0 \Delta_{\mathrm{sep}} \big(\pi_{\mathcal{M}_{\mathrm{mac}}} + \pi_{\mathcal{M}_{\mathrm{far}}}\big) +2R \sum_{m=1}^{k_0} |\Delta \Pi_m|.
\end{align*}

\emph{Step 2 - Bounding variance.} For this purpose, let us revisit the polynomial $ P_{\mathrm{var}}(\theta) = \prod_{l=1}^{k_*} (\theta - \theta_l^*)^2$. Integrating $P_{\mathrm{var}}$ with respect to signed measure $\nu$ and using Lemma \ref{lemma:Hellinger_to_Polynomial} and Lemma \ref{lemma:variance_test_function}, we achieve 
\begin{align*}
\left|\int P_{\mathrm{var}}(\theta) dG(\theta)\right| = \left|\int P_{\mathrm{var}}(\theta) d\nu(\theta)\right| \le C_{\mathrm{poly}}C_{\mathrm{var}} d_H(p_{G},p_{G_*}).
\end{align*}

Using the relative position of $\theta_i$ with respect to its nearest true center, we obtain the following bound on $P_{\mathrm{var}}(\theta_i)$:
\begin{align*}
    P_{\mathrm{var}}(\theta_i) \geq C_{\mathrm{mic}} \cdot \Delta_{\mathrm{sep}}^{2s_{\max}-2} |\theta_i-\theta^*_{c(i)}|_2^2 \quad \text{for all } i \in \mathcal{M}_{\mathrm{mic}},
\end{align*}
\begin{align*}
P_{\mathrm{var}}(\theta_i)\geq C_{\mathrm{mac}}\cdot\Delta_{\mathrm{sep}}^{2s_{\max}-2}|\theta_i - \theta_{c(i)}^*|_2^2 \quad \text{for all } i \in \mathcal{M}_{\mathrm{mac}},
\end{align*}
\begin{align*}  
P_{\mathrm{var}}(\theta_i) \geq C_{\mathrm{far}} \cdot \Delta^{2s_{\max}-2}_{\mathrm{sep}}|\theta_i-\theta^*_{c(i)}|_2^2 \quad \text{for all } i \in \mathcal{M}_{\mathrm{far}}.
\end{align*}
Combining these estimation together, we have 
\begin{equation}
\label{eqn:dung_sketch_thm_3_2_variance_bound}
    \sum_{i=1}^{k_*} \pi_i|\theta_i - \theta^*_{c(i)}|^2\leq C_{\mathrm{var},2} \Delta^{2s_{\max}-2}_{\mathrm{sep}}d_H(p_G,p_{G_*}). 
\end{equation}

\emph{Step 3 - Bounding macro- and far-related quantities.}
For macro-related quantities, we first bound the high moment $\sum_{i\in \mathcal{M}_{\mathrm{mac}}}\pi_i|\theta_i - \theta^*_{c(i)}|^{2s_{\max}}$. For $i \in \mathcal{M}_{\mathrm{mac}}$, we first achieve a lower bound of $P_{\mathrm{var}}(\theta_i)$ based on  $|\theta_i - \theta^*_{c(i)}|^{2s_{\max}}$, before achieving 
\begin{equation*}
    \sum_{i \in \mathcal{M}_{\mathrm{mac}}} \pi_i |\theta_i - \theta^*_{c(i)}|^{2s_{\max}} \leq C_{\mathrm{mac,moment,1}}\cdot d_{H}(p_G,p_{G_*}). 
\end{equation*}
From this, as $|\theta_i - \theta^*_{c(i)}| > \Delta_{\mathrm{sep}}/4$, we have 
$$\pi_{\mathrm{mac}} \leq C_{\mathrm{mac,mass,1}}\Delta^{-2s_{\max}}_{\mathrm{sep}}d_H(p_G,p_{G_*}).$$

For $i\in\mathcal{M}_{\mathrm{far}}$, we have $|\theta_i - \theta_l^*|\geq |\theta_i-\theta^*_{c(i)}| > \frac{D_0}{4}$, for all $1\leq l\leq k_*$, which leads to $P_{\mathrm{var}}(\theta_i) \geq \left(\frac{D_0}{4}\right)^{2k_*}$. As a result, we obtain 
\begin{align}
\nonumber
    \pi_{\mathrm{far}} = \sum_{i \in \mathcal{M}_{\mathrm{far}}} \pi_i &\leq \sum_{i \in \mathcal{M}_{\mathrm{far}}}\left(\frac{D_0}{4}\right)^{-2k_*} \pi_iP_{\mathrm{var}}(\theta_i) \leq \left(\frac{D_0}{4}\right)^{-2k_*}\int P_{\mathrm{var}}(\theta)d\nu(\theta) \\
    &\leq  C_{\mathrm{far,mass,2}} \cdot d_H(p_{G},p_{G_*}).
\label{eq:exact_multi_sketch_16}
\end{align}

For any $i \in \mathcal{M}_{\mathrm{mac}}$, since $|\theta_i-\theta^*_{c(i)}| > \Delta_{\mathrm{sep}}/4$, we have
\begin{align*}
    \sum_{i\in\mathcal{M}_{\mathrm{mac}}} \pi_i \le \left( \frac{4}{\Delta_{\mathrm{sep}}} \right)^{2s_{\max}} \sum_{i\in\mathcal{M}_{\mathrm{mac}}} \pi_i |\theta_i-\theta^*_{c(i)}|^{2s_{\max}}. 
\end{align*}
Recall that $|\theta_i - \theta^*_{c(i)}|\leq 2R$, we can show that get $P_{\mathrm{var}}(\theta_i) \geq C_{\min}|\theta_i - \theta^*_{c(i)}|^{2s_{\max}}$, where $C_{\min}>0$ is some universal constant.
Thus, we have 
\begin{equation}
\label{eq:exact_multi_sketch_14}
    \sum_{i\in \mathcal{M}_{\mathrm{mac}}} \pi_i |\theta_i-\theta^*_{c(i)}|^{2s_{\max}} \leq C_{\min}^{-1}\int P_{\mathrm{var}}(\theta)d\nu(\theta)\leq C_{\min}^{-1}C_{\mathrm{poly}}C_{\mathrm{var}}d_H(p_{G},p_{G_*}).
\end{equation}
As a result, we obtain the following 
\begin{align*}
    \pi_{\mathrm{mac}} = \sum_{i\in\mathcal{M}_{\mathrm{mac}}} \pi_i \le C_{\mathrm{mac,mass,1}} \Delta_{\mathrm{sep}}^{-2s_{\max}} d_H(p_{G},p_{G_*}).
\end{align*}

\emph{Step 4 - Bounding cluster mass discrepancy.} For each cluster $1 \leq m \leq k_{0}$, we aim to build a polynomial $P_m(\theta)$ satisfying exactly $2k_*$ Hermite interpolation conditions at the true centers $\theta_j^*$, that is,
\begin{align}
    \label{eq:polynomial_conditions}
    P_m(\theta_j^*) = 1_{\left\{j \in \mathcal{C}_m\right\}}, \qquad P_m'(\theta_j^*) = 0, \qquad \forall j\in[k_*].
\end{align}

Given these conditions, for any $j\in[k_{*}]$, a second-order Taylor expansion to $P_m(\theta_j)$ around $\theta^*_j$, that is,
$P_m(\theta_j) = 1_{\left\{j \in \mathcal{C}_m\right\}} + \frac{1}{2} P_m''(\xi_j) (\theta_j - \theta_j^*)^2$, helps extract the term $\Delta\Pi_m$ as
\begin{align*}
\Delta\Pi_m =  \int P_m(\theta)d\nu(\theta) + \sum_{c(i) \in\mathcal{C}_m}\pi_{i}\mathbf{1}_{\{i \in\mathcal{M}_{\mathrm{far}}\}} &-  \frac{1}{2} \sum_{i \in \mathcal{M}_{\mathrm{mic}} \cup \mathcal{M}_{\mathrm{mac}}} \pi_i P_m''(\xi_i) |\theta_i-\theta^*_{c(i)}|^2\\
    &-  \sum_{i \in \mathcal{M}_{\mathrm{far}}} \pi_i P_m(\theta_i).
\end{align*} 

Thus, it boils down to bound four terms in the above RHS. Before that, we need to find such polynomial functions $P_m(\theta)$, for $m\in[k_0]$, which is the primary challenge in this proof. The main idea is to employ the Newton interpolation technique based on divided differences in \eqref{eqn:dung_divided_difference}. Although the complicated cluster structure of the interpolation points may initially raise concerns about the complexity of the interpolation scheme and the potential growth of the polynomial coefficients, this issue can be resolved by exploiting the connection between the interpolation polynomial and a smooth function satisfying the conditions in \eqref{eq:polynomial_conditions}.

More precisely, we construct a bump function $\Phi$ that separates the points within a cluster from those outside the cluster. The corresponding divided differences can then be controlled through the derivatives of $\Phi$, which in turn provides the desired bounds on the interpolation coefficients. The detailed construction is presented in Section~\ref{sec:cluster_wise_test_function}. 

Upon constructing desirable polynomial $P_m$, by means of Lemmas~\ref{lemma:Hellinger_to_Polynomial} and \ref{lemma:dung_bound_for_cluster_mass_extracting}, we have 
\begin{align}
\nonumber
    |\Delta \Pi_m| &\le C_{\mathrm{poly}}C_{\mathrm{norm},P}\cdot d_H(p_{G},p_{G_*})\\
    &+\frac{1}{2} C_{P,2} \left( \sum_{\mathcal{M}_{\mathrm{mic}} \cup \mathcal{M}_{\mathrm{mac}}} \pi_i |\theta_i-\theta^*_{c(i)}|^2 \right) \ + \ (C_{\mathrm{norm},P} + 1) \pi_{\mathrm{far}}.
    \label{eq:exact_multi_sketch_8}
\end{align}

Here, we can employ the variance bound in \eqref{eqn:dung_sketch_thm_3_2_variance_bound} and far mass bound in \eqref{eq:exact_multi_sketch_16} to achieve the following bound

\begin{align*}
    |\Delta \Pi_m| &\le C_{\Delta\Pi,1} \Delta_{\mathrm{sep}}^{-(2s_{\max}-1)}d_H(p_{G},p_{G_*}).
\end{align*}

\emph{Step 5 - Bounding mass discrepancy.} Let us consider the polynomial $$\bar{E}_{j}(\theta) : =  \ell_{j,\text{micro}}^2(\theta) [\bar{A}_{j} + \bar{B}_{j}(\theta - \theta_{j}^{*})] P_{\text{macro}}(\theta),$$ where $\ell_{j,\text{micro}}(\theta) = \prod_{q \in \mathcal{C}_{m}\setminus\{j\}} \frac{\theta - \theta_q^*}{\theta_j^* - \theta_q^*}$, $P_{\text{macro}}(\theta) = \prod_{p \neq m} \prod_{q \in \mathcal{C}_p} \left( \frac{\theta - \theta_q^*}{2R} \right)^{2s_{\max}}$ and
\begin{align*}
    \bar{A}_j  = 1/ P_{\text{macro}}(\theta_{j}^{*}), \quad 
    \bar{B}_j = -2\bar{A}_j \left(\sum_{q \in \mathcal{C}_m \setminus \{j\}} \frac{1}{\theta_j^* - \theta_q^*} + 1\sum_{p \neq m} \sum_{q \in \mathcal{C}_p} \frac{s_{\max}}{\theta_j^* - \theta_q^*} \right).
\end{align*}
It is obvious to check that $\bar{E}_j(\theta^*_l) = \delta_{jl}$ and $\bar{E}'_{j}(\theta_l^*) = 0$. Assume that $j\in \mathcal{C}_m$. Then, we have $\int\bar{E}_j(\theta)d\nu(\theta) = \sum_{l=1}^{k_*}\pi_l\bar{E}_{j}(\theta_l) - \pi^*_j = \left(\sum_{i\in \bar{\mathcal{V}}_j}\pi_i -\pi^*_j\right) + \sum_{l=1}^{k_*}\pi_l\bar{E}_{j}(\theta_l) - \sum_{i\in \bar{\mathcal{V}}_j}\pi_i,$
which implies that
\begin{align}
    \label{eq:exact_multi_sketch_10}
    \Delta\pi_j &= \int\bar{E}_j(\theta)d\nu(\theta)  - \sum_{l \in \mathcal{C}_m}\sum_{i\in \bar{\mathcal{V}}_l}\pi_i(\bar{E}_j(\theta_i) -\bar{E}_j(\theta^*_l)) - \sum_{\substack{i \in \mathcal{M}_{\mathrm{mic}}\\ c(i) \notin \mathcal{C}_m }}\pi_i\bar{E}_j(\theta_i) \nonumber\\
    &-\sum_{\substack{i\in \mathcal{M}_{\mathrm{mac}}\\ c(i)\in\mathcal{C}_m}}\pi_i\bar{E}_j(\theta_i)-\sum_{\substack{i\in \mathcal{M}_{\mathrm{mac}}\\ c(i)\notin\mathcal{C}_m}}\pi_i\bar{E}_j(\theta_i)- \sum_{{i\in\mathcal{M}_{\mathrm{far}}}}\pi_i\bar{E}_j(\theta_i).
\end{align}
Using Lemma \ref{lemma:Hellinger_to_Polynomial} and Lemma \ref{lemma:dung_multicluster_mass_test_function}, we have $$\left|\int \bar{E}_{j}(\theta) d\nu(\theta)\right|\leq C_{\mathrm{poly}}C_{\mathrm{norm},\bar{E}} \Delta_{\mathrm{sep}}^{-(2s_{m}-1)}d_H(p_{G},p_{G_*}).$$ Thus, it suffices to bound the last five terms in the RHS of equation~\eqref{eq:exact_multi_sketch_10}.

Regarding the first term, for each $i \in \bar{\mathcal{V}}_l$ and $l \in \mathcal{C}_m$, a Taylor-expansion of $\bar{E}_j(\theta_i)$ around $\theta^*_l$ gives 
$\bar{E}_j(\theta_i) -\bar{E}_j(\theta^*_l)  = \frac{1}{2}(\theta_i-\theta_l^*)^2\bar{E}''_{j}(\xi_i)$. This result together with Lemma~\ref{lemma:dung_multicluster_mass_test_function} yields that
\begin{align}
    \left|\sum_{l \in \mathcal{C}_m}\sum_{i\in \bar{\mathcal{V}}_l}\pi_i(\bar{E}_j(\theta_i) -\bar{E}_j(\theta^*_l))\right| &\leq C_{2,E}\Delta^{-2}_{\mathrm{sep}}\sum_{i=1}^{k_*}\pi_i|\theta_i-\theta_{c(i)}|^2\nonumber\\
    &\leq C_{2,E}C_{\mathrm{var},2}\Delta^{-2s_{\max}}_{\mathrm{sep}}d_H(p_{G},p_{G_*}), 
    \label{eq:exact_multi_sketch_11}
\end{align}
where the last inequality follows from equation~\eqref{eqn:dung_sketch_thm_3_2_variance_bound}. Regarding the rest terms, a key step is to evaluate the each term $P_{\text{macro}}(\theta_i)$, $\ell^2_{j,\text{micro}}(\theta_i)$ separately, and then obtain the bound for $|\bar{E}_{j}(\theta_i)|$ for $i$ in corresponding sets. Putting the results, we obtain 
\begin{align*}
       |\Delta\pi_j| 
    &\leq C_{\mathrm{mass,2}} \Delta^{-2s_{\max}}_{\mathrm{sep}}d_H(p_{G},p_{G_*}).
\end{align*}

\emph{Step 6 - Bounding $\sum_{j=1}^{k_*} \pi_j|\theta_j - \theta^*_{c(j)}|$.}  For this term, we consider two separate cases when $i$ belongs to $\mathcal{M}_{\mathrm{mac}}\cup \mathcal{M}_{\mathrm{far}}$ and when $i$ belongs to $\mathcal{M}_{\mathrm{near}}$. 

For $i\in\mathcal{M}_{\mathrm{mac}}\cup \mathcal{M}_{\mathrm{far}}$, note that in this case $|\theta_i-\theta^*_{c(i)}|>\Delta_{\mathrm{sep}}/4$, thus we have 
\begin{align*}
\nonumber
   \sum_{\mathcal{M}_{\mathrm{mac}}\cup\mathcal{M}_{\mathrm{far}}} \pi_i |\theta_i -\theta^*_{c(i)}| &= \sum_{\mathcal{M}_{\mathrm{mac}}\cup\mathcal{M}_{\mathrm{far}}} \pi_i |\theta_i -\theta^*_{c(i)}|^2 |\theta_i -\theta^*_{c(i)}|^{-1}\\
   &\leq \dfrac{4}{\Delta_{\mathrm{sep}}}\sum_{\mathcal{M}_{\mathrm{mac}}\cup\mathcal{M}_{\mathrm{far}}} \pi_i |\theta_i -\theta^*_{c(i)}|^2   \leq 4C_{\mathrm{var},2}\Delta_{\mathrm{sep}}^{-(2s_{\max}-1)}d_H(p_{G},p_{G_*}).
\end{align*}

For $i\in\mathcal{M}_{\mathrm{mic}}$, by streamlining similar arguments to the proof for the global bound in Theorem \ref{theorem:exact_one_group_univariate}, we only need to proof for the case that each sub-Voronoi cell contains exactly one element. Now, let us consider a polynomial of the form
$$\bar{H}_{j}(\theta) = \ell_{j,\mathrm{micro}}^2(\theta) [\bar{v}_{j}(\theta - \theta_{j}^{*})] P_{\mathrm{macro}}(\theta),$$ 
where we define $\ell_{j,\mathrm{micro}}(\theta) = \prod_{q \in \mathcal{C}_{m} \neq j} \frac{\theta - \theta_q^*}{\theta_j^* - \theta_q^*}$ and $P_{\mathrm{macro}}(\theta) = \prod_{p \neq m} \prod_{q \in \mathcal{C}_p} \left( \frac{\theta - \theta_q^*}{2R} \right)^{2s_{\max}}$ and $\bar{v}_j = 1/ P_{\mathrm{macro}}(\theta_{j}^{*}) $. We can verify that $\bar{H}_j(\theta_l^*) = 0$ and $\bar{H}_j'(\theta_l^*) = \delta_{jl}$, for all $1\leq l\leq k_*$. Assume that $j \in \mathcal{C}_{m}$, a second-order Taylor expansion of $\bar{H}_j(\theta_l)$ around $\theta^*_l$ gives 
\begin{align*}
    \int\bar{H}_{j}(\theta) d\nu(\theta) & = \sum_{l \in \mathcal{C}_m} \pi_l \bar{H}_{j}(\theta_l)+\sum_{l \notin \mathcal{C}_m} \pi_l \bar{H}_{j}(\theta_l) \\
    & = \pi_j(\theta_j - \theta_{c(j)}^*) + \frac{1}{2} \sum_{l \in \mathcal{C}_m} \pi_l \bar{H}_{j}''(\xi_l) (
    \theta_{l} - \theta^*_{c(l)})^2 + \sum_{l \notin \mathcal{C}_m} \pi_l \bar{H}_{j}(\theta_l).
\end{align*} 
Therefore, by means of the triangle inequality, we have
\begin{align}
\pi_j|\theta_j - \theta_{c(j)}^*| \leq \left|\int \bar{H}_{j}(\theta) d\nu(\theta)\right| + \frac{1}{2} \sum_{l \in \mathcal{C}_m} \pi_l |\bar{H}_{j}''(\xi_l)| (
    \theta_{l} - \theta_{l})^2 + \sum_{l \notin \mathcal{C}_m} \pi_l |\bar{H}_{j}(\theta_l)|. 
\end{align}
By means of Lemma~\ref{lemma:Hellinger_to_Polynomial} and Lemma~\ref{lemma:dung_multicluster_mean_test_function}, we have
\begin{equation*}
    \left| \int \bar{H}_{j}(\theta) d\nu(\theta) \right| \le C_{\mathrm{poly}}C_{\mathrm{norm},\bar{H}} \Delta_{\mathrm{sep}}^{-(2s_{\max}-1)} d_H(p_{G},p_{G_*}).
\end{equation*}
Furthermore, it also follows from Lemma~\ref{lemma:dung_multicluster_mean_test_function} that
\begin{align*}
        \dfrac{1}{2}\sum_{l \in \mathcal{C}_m} \pi_l |\bar{H}''_{j}(\xi_l)| (\theta_l - \theta^*_{c(l)})^2 &\leq C_{\bar{H}}\Delta^{-1}_{\mathrm{sep}} \sum_{l \in \mathcal{C}_m}\pi_l(\theta_l -\theta_{c(l)}^*)^2,\\
         \sum_{l \notin \mathcal{C}_m} \pi_l |\bar{H}_{j}(\theta_l)|&\leq C_{\mathrm{cross},\bar{H}}\sum_{l \notin \mathcal{C}_m}\pi_l(\theta_{l} - \theta_{c(l)}^{*})^{2}. 
\end{align*}
As a result, using the variance bound in \eqref{eqn:dung_sketch_thm_3_2_variance_bound}, we obtain
\begin{align*}
    \pi_j|\theta_j - \theta_{c(j)}^*| \leq C_{\mathrm{mean},1}\Delta_{\mathrm{sep}}^{-(2s_{\max}-1)}d_H(p_{G},p_{G_*}), \quad \text{.}
\end{align*}  
where $C_{\mathrm{mean},1}>0$ is some universal constant. Thus, we eventually achieve
\begin{align*}
    \sum_{j\in \mathcal{M}_{\mathrm{near}}}\pi_j|\theta_j-\theta_{c(j)}^*| \leq C_{\mathrm{mean,near}} (\pi^*_{\text{min}})^{-1}\Delta^{-(2s_{\max}-1)}_{\mathrm{sep}}d_H(p_{G},p_{G_*}).
\end{align*}
Combine the results for $i \in \mathcal{M}_{\mathrm{mac}} \cup \mathcal{M}_{\mathrm{far}}$ and $i \in  \mathcal{M}_{\mathrm{mic}}$, we obtain
\begin{align*}
    \sum_{j=1}^{k_*} \pi_j|\theta_j - \theta^*_{c(j)}|\leq C_{\mathrm{mean},2}(\pi^*_{\text{min}})^{-1}\Delta^{-(2s_{\max}-1)}_{\mathrm{sep}}d_H(p_{G},p_{G_*}).
\end{align*}
\emph{Step 7 - Conclusion.} Putting the results of the above six steps together, we reach the desired global bound.
\end{proof}

\subsection{Proof Sketch of Theorem~\ref{theorem:exact_no_group}}
\begin{proof}[Proof Sketch of Theorem~\ref{theorem:exact_no_group}]
We present the proof for the global bound in five main steps.

\emph{Step 1 - Wasserstein decomposition.} By transporting the mass through the intermediate measure $G' = \sum_{j=1}^{k_*}\pi_j \delta_{\theta^*_{c(j)}}$, we obtain a decomposition for $W_1(G, G_*)$ as 
\begin{align*}
W_1(G, G_*)
&= \sum_{j = 1}^{k_*} \sum_{i \in \mathcal{V}_j} \pi_i \, |\theta_i - \theta_j^*|
+ 2R \sum_{j = 1}^{k_*} \left| \sum_{i \in \mathcal{V}_j} \pi_i - \pi_j^* \right|.
\end{align*}
Thus, it reduces to bound two terms in the above RHS.

\emph{Step 2 - Bounding variance.} Consider the witness function 
\begin{equation*}
    P_{\mathrm{var}}(\theta) = \prod_{l=1}^{k_*}(\theta-\theta^*_{l})^2.
\end{equation*}
Since $\int P_{\mathrm{var}}(\theta) d\nu(\theta)  \ge    \left(\frac{1}{2}\right)^{2k_*-2} \Delta_{\mathrm{sep}}^{2k_*-2} \sum_{j = 1}^{k_{*}} \sum_{i \in \mathcal{V}_{j}} \pi_{i} |\theta_{i} - \theta_{j}^{*}|^2$, Lemma~\ref{lemma:Hellinger_to_Polynomial} and Lemma~\ref{lemma:variance_test_function} indicate that
\begin{align}
    \label{eq:exact_no_sketch_2}
    \sum_{j = 1}^{k_{*}} \sum_{i \in \mathcal{V}_{j}} \pi_{i} |\theta_{i} - \theta_{j}^{*}|^2 \le C_{\mathrm{var}} \Delta_{\mathrm{sep}}^{-(2k_* - 2)} d_H(p_{G},p_{G_*}).
\end{align}

\emph{Step 3 - Bounding mass discrepancy.} Let us revisit the following polynomial function:  
\begin{align*}
    E_{j}(\theta) : = \ell_j^2(\theta) \left[ \bar{A}_j + \bar{B}_j(\theta - \theta_j^*) \right], \quad \text{where} \quad \ell_j(\theta) = \prod_{q \neq j} \frac{\theta - \theta_q^*}{\theta_j^* - \theta_q^*},
\end{align*}
and $\bar{A}_j = 1$, $\bar{B}_j = - 2\ell_j'(\theta_j^*)$. By means of Taylor expansion for $E_j(\theta_l)$ exactly around its true center $\theta_l^*$, we have $E_j(\theta_l) = \delta_{jl} + \frac{1}{2} E_j''(\xi_l) (\theta_{l} - \theta_{c(l)}^{*})^2$, which implies that
$$  \left(\sum_{i \in \mathcal{V}_{j}} \pi_{i} - \pi_{j}^{*}\right)= \int E_j(\theta) d\nu(\theta)- \frac{1}{2}\sum_{l=1}^{k_*} \sum_{i \in \mathcal{V}_l} \pi_{i} E_j''(\xi_i) (\theta_{i} - \theta_{l}^{*})^2.$$ 
By means of Lemma~\ref{lemma:Hellinger_to_Polynomial} and Lemma~\ref{lemma:mass_test_function}, we have
\begin{align*}
    \left|\sum_{i \in \mathcal{V}_{j}} \pi_{i} - \pi_{j}^{*}\right|\leq C_{\mathrm{poly}}{C}_{\mathrm{norm},E} \Delta_{\mathrm{sep}}^{-(2k_*-1)}d_H(p_{G},p_{G_*})+\frac{1}{2}C_{E,2} \Delta_{\mathrm{sep}}^{-(2k_*-1)}\sum_{l=1}^{k_*} \sum_{i \in \mathcal{V}_l} \pi_{i} (\theta_{i} - \theta_{l}^{*})^2.
\end{align*}

From the variance estimation in \eqref{eq:exact_no_sketch_2}, we have 
\begin{align}
\label{eq:exact_no_sketch_1}
    \left|\sum_{i \in \mathcal{V}_{j}} \pi_{i} - \pi_{j}^{*}\right|  
& :=  C_{\mathrm{mass},2} \Delta_{\mathrm{sep}}^{-(4k_* - 3)} d_H(p_{G},p_{G_*}). 
\end{align}

\emph{Step 4 - Bounding mean discrepancy.}  We consider the Hermite polynomial given by:
\begin{align*}
    H_j(\theta) : = \ell_j^2(\theta) (\theta - \theta_j^*),  \quad \text{where} \quad \ell_j(\theta) : = \prod_{q \neq j} \frac{\theta - \theta_q^*}{\theta_j^* - \theta_q^*}.
\end{align*}
Since $H_j(\theta_l^*) = 0$ and $H_j'(\theta_l^*) = \delta_{jl}$, for all $1 \leq l \leq k_{*}$, a Taylor expansion for $H_j(\theta_i)$ around its true center $\theta_j^*$ yields
$$ \sum_{i \in \mathcal{V}_{j}} \pi_{i}(\theta_{i} - \theta_{j}^{*})=\int H_j(\theta) d\nu(\theta)  - \frac{1}{2} \sum_{l=1}^{k_*} \sum_{i \in \mathcal{V}_{l}} \pi_{l} H_j''(\xi_l) (\theta_{i} - \theta_{l}^{*})^2.$$
 
We now divide our argument into two cases:

\emph{Case 4.1.} Suppose that there exists a Voronoi cell $\mathcal{V}_j$ containing more than one element. Thus, at least one Voronoi cell is empty, and we may assume it is $\mathcal{V}_1$. In this case, it follows directly that $| \sum_{i \in \mathcal{V}_1} \pi_i - \pi_1^* | \;=\; \pi_1^*$. 
Consequently, the bound in equation~\eqref{eq:exact_no_sketch_1} implies that
\begin{align*}
    \pi_{\text{min}}^{*}\leq \pi_{1}^{*} \leq C_{\mathrm{mass},2} \Delta_{\mathrm{sep}}^{-(4k_*-3)} d_H(p_{G},p_{G_*}). 
\end{align*}
Now, since $|\theta_{i} - \theta_{c(i)}^{*}| \leq 2R$ for all $1\leq i\leq k_*$, we have
\begin{align*}
    \sum_{j = 1}^{k_{*}} \sum_{i \in \mathcal{V}_{j}} \pi_{i} |\theta_{i} - \theta_{j}^{*}| \leq 2R \sum_{i = 1}^{k_{*}} \pi_{i} = 2R & \leq 2R \cdot \dfrac{C_{\mathrm{mass},2} \Delta_{\mathrm{sep}}^{-(4k_* - 3)} d_H(p_{G},p_{G_*})}{\pi_{\text{min}}^{*}} \nonumber \\
    & := C_{\mathrm{mass,near,1}} (\pi_{\text{min}}^{*})^{-1}\Delta_{\mathrm{sep}}^{-(4k_*- 3)}d_H(p_{G},p_{G_*}).  
\end{align*}
\emph{Case 4.2.} Assume that every Voronoi cell $\mathcal{V}_j$ contains at most one element. 
When there exists an empty sub-Voronoi cell $\bar{\mathcal{V}}_j
:= \left\{
\theta:|\theta-\theta_j^*|\leq \Delta_{\mathrm{sep}}/4
\right\}$,
we can proceed using the same argument as in Case~4.1. 
It remains to consider the case where each sub-Voronoi cell contains exactly one element, which also means that each $\mathcal{V}_j$ similarly contains exactly one element. In this case, using Lemma \ref{lemma:Hellinger_to_Polynomial} for $\int H_j(\theta)d\nu(\theta)$, Lemma \ref{lemma:mean_test_function} for $H_j''$, and variance estimation in \eqref{eq:exact_no_sketch_2}, we achieve
\begin{align}
\sum_{j = 1}^{k_{*}} \sum_{i \in \mathcal{V}_{j}} \pi_{i} |\theta_{i} - \theta_{j}^{*}| \leq  C_{\mathrm{mean,near,2}} \Delta_{\mathrm{sep}}^{-(4k_* - 4)}d_H(p_{G},p_{G_*}).
\label{eq:exact_global_key_equation_thirteen} 
\end{align}
Combining the above two cases and let $C_{\mathrm{mean,2}} := \max\{C_{\mathrm{mean,near,1}}, 2RC_{\mathrm{mean,near,2}}\}$, we have 
\begin{align}
\label{eqn:dung_thm3_global_final_upper_bound_for_first_moment}
    \sum_{j = 1}^{k_{*}} \sum_{i \in \mathcal{V}_{j}} \pi_{i} |\theta_{i} - \theta_{j}^{*}| &\leq C_{\mathrm{mean,2}}(\pi_{\text{min}}^{*})^{-1}\Delta_{\mathrm{sep}}^{-(4k_* - 3)}d_H(p_{G},p_{G_*}).
\end{align}
\emph{Step 5 - Conclusion.} Put the results of Steps 1-4 together, we reach the desired global bound.
\end{proof}

\section{Discussion}
\label{sec:discussion}
In this article, we develop a geometric framework for understanding how the minimum separation and the minimum weight influence the convergence behavior of parameter estimation in finite Gaussian mixtures. By establishing lower bounds for the Hellinger distance between mixture densities in terms of Wasserstein distances between mixing measures, our analysis reveals that the impact of component separation is fundamentally determined by the spatial organization of the true components, ranging from a single compact cluster, several well-separated clusters, to an entirely unconstrained geometric configuration. We further investigate the over-specified setting, where the fitted model contains more components than the ground-truth. In this regime, a transition from first-order to second-order Wasserstein geometry alters the dependence of the convergence rates on both component separation and mixture weights: the separation complexity becomes less severe, while the minimum mixture weight no longer appears in the resulting bounds. As a consequence, our theories yield separation-dependent convergence rates that smoothly connect point-wise and uniform estimation regimes, providing a unified understanding of parameter recovery in finite Gaussian mixtures.

\subsection{Extension to other strongly identifiable mixtures}
Although our analysis is developed for finite location Gaussian mixtures, the underlying methodology is considerably more general. The key ingredient throughout the paper is Lemma~\ref{lemma:Hellinger_to_Polynomial}, which establishes a bridge between moments of signed mixing measures and the Hellinger distance between mixture densities. Once such a moment-control inequality is available, the interpolation-polynomial constructions developed in this work can be applied almost verbatim to extract geometric discrepancies and derive separation-dependent Hellinger lower bounds.

This observation suggests that the analysis of separation complexity is not unique to Gaussian mixtures. Indeed, it is plausible that analogues of Lemma 3.2 can be established by exploiting the strongly identifiable structures of several classical finite mixture models, including location mixtures of Laplace, Student-t, or mixtures of discrete distributions, such as Poisson distributions. 

Indeed, the key identity \eqref{eqn:dung_lemma_hellinger_heat_identity}, 
which relates the Gaussian convolution to the heat semigroup,
\begin{equation*}
    \mathbb{E}_{Z \sim N(0,I_d)}[g(\theta+Z)] 
    = \left(e^{\frac{1}{2}\Delta}g\right)(\theta),
\end{equation*}
can be naturally extended to a general noise distribution. More precisely, 
for a random variable $Z$ with characteristic function $\varphi_Z$, we have
\begin{equation*}
    \mathbb{E}[g(\theta+Z)] 
    = \left(\varphi_Z(-i\nabla)g\right)(\theta).
\end{equation*}
Therefore, for any test polynomial $U$, we can construct
\begin{equation*}
    g = \left(\varphi_Z(-i\nabla)\right)^{-1}U,
\end{equation*}
such that $U(\theta)=\mathbb{E}[g(\theta+Z)]$. Using this representation together with H\"{o}lder's inequality, we obtain 
an upper bound for the integral of the test function $U$ with respect to 
the signed measure $\nu$ in terms of the Hellinger distance, as in 
\eqref{eq:Hellinger_to_polynomial_1}. It remains to control the quantity
\begin{equation*}
\begin{aligned}
    &\sqrt{\int_{\mathcal{S}} g^2(x_{\mathcal{S}})
    \left(p_{G,\mathcal{S}}(x_{\mathcal{S}})
    + p_{G_{*},\mathcal{S}}(x_{\mathcal{S}})\right)
    dx_{\mathcal{S}}}   \leq
    2\sup_{\theta\in \mathcal{S}\cap B(0,R)}
    \sqrt{
    \mathbb{E}_{X\sim \mathcal{N}(\theta,I_{d_{\mathcal{S}}})}
    \left[g^2(X)\right]} .
\end{aligned}
\end{equation*}
Consequently, the extension is completed once we establish a suitable 
bound on $g=(\varphi_Z(-i\nabla))^{-1}U$ in terms of the original test 
polynomial $U$ over $B(0,R)$, which can be achieved, for instance, through 
polynomial inequalities such as Markov-type bounds. 


\subsection{Separation complexity in location-scale Gaussian mixtures}
A substantially more challenging extension concerns finite Gaussian mixtures with both unknown locations and unknown covariance matrices. Unlike the location model considered in this paper, the key moment-control inequality in Lemma 3.2 no longer holds directly in the location-scale setting. The reason is that perturbations of covariance parameters introduce higher-order interactions that cannot be captured solely through polynomial moment functionals of the mixing measure. As a result, the interpolation-polynomial arguments developed in this paper are insufficient for controlling the discrepancy between latent parameters and observed densities.

This difficulty is closely related to recent developments in the literature on location-scale mixtures, where identifiability and convergence analyses require substantially more delicate arguments than those for location mixtures. In particular, covariance perturbations generate additional cancellation mechanisms and anisotropic geometric structures that are absent in purely location-based models. Understanding how these effects interact with component separation remains largely unexplored.

We therefore conjecture that a fundamentally new family of test functions will be required to extend the notion of separation complexity to location-scale Gaussian mixtures. Developing such techniques would not only generalize the present theory but could also lead to a deeper understanding of weak identifiability phenomena in more general mixture models. 
Hence, we leave an important problem of establishing separation-dependent convergence rates in this setting for future work.

\newpage
\appendix
\addtocontents{toc}{\protect\setcounter{tocdepth}{10}}
\begin{center}
{}\textbf{\Large{Appendices for\\ \vspace{0.5em}
``On the Geometry of Separation in Finite Gaussian Mixtures''}}
\end{center}
\tableofcontents

\section{Proofs}

\subsection{Preliminary Results}
\label{sec:proof_of_preliminary_results}
\begin{proof}[Proof of Lemma~\ref{lemma:general_to_identity}]
By letting $y = \Sigma^{-1/2}x$, we have for any $i \in [1,k]$
\begin{align*}  f(x|\Sigma^{-1/2}\theta_i,I_d) &=(2\pi)^{-d/2}\exp\left(-\frac{1}{2}\|x-\Sigma^{-1/2}\theta_i\|^2\right) = (2\pi)^{-d/2}\exp\left(-\frac{1}{2}\|\Sigma^{-1/2}y-\Sigma^{-1/2}\theta_i\|^2\right)\\
&= (2\pi)^{-d/2}\exp\left(-\frac{1}{2}(y-\theta_i)^\top \Sigma^{-1}(y-\theta_i)\right) = \sqrt{\det(\Sigma)}\cdot f(y|\theta_i,\Sigma).
\end{align*}
Thus, we have $p_{\tilde{G}}(x) = \sqrt{\det\Sigma}\cdot p_{G}(y)$, thus 
\begin{align*}
    d^2_H(p_{\tilde{G}},p_{\tilde{G_*}}) &= \int_{\mathbb{R}^d}(\sqrt{p_{\tilde{G}}(y)}-\sqrt{p_{\tilde{G}_*}(y)})^{2}dy = \int_{\mathbb{R}^d}\sqrt{\det\Sigma}\cdot(\sqrt{p_{G}(x)}-\sqrt{p_{G_*}(x)})^{2}\sqrt{\det\Sigma^{-1}} dx\\
    &=d^2_H(p_{G},p_{G_*}). 
\end{align*}
As a result, $d_H(p_{\tilde{G}},p_{\tilde{G_*}}) = d_H(p_{G},p_{G_*})$. For the order preservation of Wasserstein distance, to prove the second inequality, let $\pi$ be the optimal coupling measure of $G$ and $G'$, i.e. 
\begin{equation*}
    W_p^p(G,G') = \iint_{\mathbb{R}^d\times \mathbb{R}^d}\|x-y\|^{p}d\pi(x,y). 
\end{equation*}
Then, by abusing of notation, consider the measure $\Sigma^{\#}\pi$ defined in $\mathbb{R}^d\times \mathbb{R}^d$ such that $\Sigma^{\#}\pi(A) = \pi(\Sigma^{-1}A)$, where $\Sigma^{-1}A = \{(x,y)\in \mathbb{R}^d\times\mathbb{R}^d:(\Sigma x,\Sigma y) \in A\}$. It is obvious that $\Sigma^{\#}\pi$ is a coupling measure of $\Sigma^{\#}G$ and $\Sigma^{\#}G'$. Then, we have 
\begin{align*}
W_p^p(\Sigma^{\#}G,\Sigma^{\#}G') &\leq \iint_{\mathbb{R}^d \times \mathbb{R}^d} \|x-y\|^pd\Sigma^{\#}\pi(x,y) = \iint_{\mathbb{R}^d \times \mathbb{R}^d}\|\Sigma x -\Sigma y\|^pd\pi(x,y)\\
&\leq \lambda_{\max}^p\iint_{\mathbb{R}^d \times \mathbb{R}^d}\|x-y\|^pd\pi(x,y) = W_p^p(G,G'). 
\end{align*}
Thus, $W_p^p(\Sigma^{\#}G,\Sigma^{\#}G') \leq \lambda_{\max}W_p(G,G')$. For the first inequality, we apply the result of the first quality for two measures $\Sigma^{\#}G$ and $\Sigma^{\#}G'$ and matrix $\Sigma^{-1}$. 
\end{proof}

\begin{proof}[Proof of Lemma~\ref{lemma:Hellinger_to_Polynomial}]
    Since $\mathcal{S}$ is a finite-dimensional closed linear subspace of the space $\mathbb{R}^d$, we have the decomposition $\mathbb{R}^d = \mathcal{S} \oplus \mathcal{S}^\perp$, where $\mathcal{S}^\perp$ is the orthogonal complement of $\mathcal{S}$. We denote $d_{\mathcal{S}}$ as the dimension of $\mathcal{S}$, by performing a rotation, without lost of generality, we can suppose that $\mathcal{S} = \mathbb{R}^{d_{\mathcal{S}}}$ be subspace of $\mathbb{R}^d$ with zero values for last $d-d_{\mathcal{S}}$ coordinates, while $\mathcal{S}^{\perp}$ with zero values for first $d_{\mathcal{S}}$ coordinates. Then, any vector $x \in \mathbb{R}^d$ admits a unique orthogonal projection decomposition:
$$ x = x_{\mathcal{S}} + x_{\mathcal{S}^\perp}, \quad \text{where } x_{\mathcal{S}} \in \mathcal{S} \text{ and } x_{\mathcal{S}^\perp} \in \mathcal{S}^\perp.$$

Furthermore, the Lebesgue measure  factorizes as $dx = dx_{\mathcal{S}} \otimes dx_{\mathcal{S}^{\perp}}$. For any parameter $\theta \in \mathcal{S}$, we have
$$\|x - \theta\|_2^2 = \|(x_{\mathcal{S}} - \theta) + x_{\mathcal{S}^\perp}\|_2^2 = \|x_{\mathcal{S}} - \theta\|_2^2 + \|x_{\mathcal{S}^\perp}\|_2^2.$$
Then, it is clear that $d_{\mathcal{S}} \leq \min\{k + k_{*},d\}$. Direct calculation yields that
\begin{align*}
\frac{1}{(2\pi)^{d/2}} \exp\left( -\frac{\|x - \theta\|_{2}^2}{2} \right) \
& = \left[ \frac{1}{(2\pi)^{d_{\mathcal{S}}/2}} \exp\left( -\frac{\|x_{\mathcal{S}} - \theta\|_{2}^2}{2} \right) \right] \left[ \frac{1}{(2\pi)^{(d - d_{\mathcal{S}})/2}} \exp \left( - \frac{\|x_{\mathcal{S}^{\perp}}\|_{2}^2}{2} \right) \right].
\end{align*}
Define $\phi_\perp(x_{\mathcal{S}^{\perp}}) = \frac{1}{(2\pi)^{(d - d_{\mathcal{S}})/2}} \exp \left( - \frac{\|x_{\mathcal{S}^{\perp}}\|_{2}^2}{2} \right)$. Furthermore, we define $p_{G,\mathcal{S}}(x_{\mathcal{S}}) = \sum_{i = 1}^{k} \pi_{i} f(x_{\mathcal{S}}|\theta_{i},I_{d_{\mathcal{S}}})$, $p_{G_{*},\mathcal{S}}(x_{\mathcal{S}}) = \sum_{i = 1}^{k_{*}} \pi_{i}^{*} f(x_{\mathcal{S}}|\theta_{i}^{*},I_{d_{\mathcal{S}}})$ and $$ d_H^2(p_{G,\mathcal{S}}, p_{G_*,\mathcal{S}}) = \left[ \frac{1}{2} \int_{\mathcal{S}} \left( \sqrt{p_{G,\mathcal{S}}(x_{\mathcal{S}})} - \sqrt{p_{G_*,\mathcal{S}}(x_{\mathcal{S}})} \right)^2 dx_{\mathcal{S}} \right].$$
Therefore, we obtain that
\begin{align*}
d_H^2(p_G, p_{G_*}) & = \frac{1}{2} \int_{\mathcal{S} \times \mathcal{S}^\perp} \phi_\perp(x_{\mathcal{S}^{\perp}}) \left( \sqrt{p_{G,\mathcal{S}}(x_{\mathcal{S}})} - \sqrt{p_{G_*, \mathcal{S}}(x_{\mathcal{S}})} \right)^2 dx_{\mathcal{S}} dx_{\mathcal{S}^{\perp}} \\
& = \frac{1}{2} \int_{\mathcal{S}} \left( \sqrt{p_{G,\mathcal{S}}(x_{\mathcal{S}})} - \sqrt{p_{G_*, \mathcal{S}}(x_{\mathcal{S}})} \right)^2 dx_{\mathcal{S}} = d_H^2(p_{G,\mathcal{S}}, p_{G_*,\mathcal{S}}),
\end{align*}
where the first equation is due to the fact that the term $\phi_\perp(x_{\mathcal{S}^{\perp}})$ is a Gaussian probability density function on $\mathcal{S}^\perp$, which means that $\int_{\mathcal{S}^\perp} \phi_\perp(x_{\mathcal{S}^{\perp}}) dx_{\mathcal{S}^{\perp}} = 1$.

For any polynomial function $U: \mathcal{S}  \to \mathbb{R}$, we choose the function $g:\mathcal{S} \to \mathbb{R}$ such that for any $\theta \in \mathcal{S}$
\begin{equation}
\label{eqn:dung_lemma_hellinger_heat_identity}
U(\theta) = \mathbb{E}_{X \sim N(\theta, I_{d_{\mathcal{S}}})}[g(X)] = \mathbb{E}_{Z \sim N(0,I_{d_{\mathcal{S}}})}[g(Z + \theta)].    
\end{equation}
Indeed, since $\mathcal{S} \cong \mathbb{R}^{d_{\mathcal{S}}}$ we have 
$$\mathbb{E}_{Z \sim N(0, I_{d_{\mathcal{S}}})}[g(Z + \theta)] = e^{\frac{\Delta_{L}}{2}}g(\theta) = \sum_{m = 0}^{\infty} \frac{1}{2^{m} m!} \Delta_{L}^{m} g(\theta)$$ 
where $\Delta_{L}$ denotes the Laplacian operator defined in subspace $\mathcal{V}$, namely, $\Delta_{L} = \sum_{i=1}^{d_{\mathcal{S}}}\frac{\partial^2}{\partial x_i^2}$, and $\Delta_{L}^{m} g = \underbrace{\Delta_{L} (\Delta_{L} (\dots (\Delta_{L}}_{m \text{ times}} g) \dots ))$ for any $m \geq 0$. 

It indicates that $g = e^{-\frac{\Delta_{L}}{2}} U$. Assume that $U$ has the degree of $D$. Then, this formulation means that
$$ g(x_{\mathcal{S}}) = \sum_{m=0}^{\lfloor D/2 \rfloor} \frac{(-1)^m}{2^m m!} \Delta_L^m U(x_{\mathcal{S}}).$$
Now, given the formulation of the function $U$, we have
$$\sum_{i=1}^k \pi_i U(\theta_i) - \sum_{j=1}^{k_*} \pi_j^* U(\theta_j^*) = \int_{\mathcal{S}} U(\theta) d(G-G_{*})(\theta) = \int_{\mathcal{S}} \int_{\mathcal{S}} g(x_{\mathcal{S}}) f(x_{\mathcal{S}}|\theta,I_{d_{\mathcal{S}}}) dx_{\mathcal{S}} d(G-G_{*})(\theta).$$
An application of the Fubini's theorem leads to
\begin{align*}
\int_{\mathcal{S}} \int_{\mathcal{S}} g(x_{\mathcal{S}}) f(x_{\mathcal{S}}|\theta,I_{d_{\mathcal{S}}}) dx_{\mathcal{S}} d(G-G_{*})(\theta) & = \int_{\mathcal{S}} g(x_{\mathcal{S}}) \left(\int_{\mathcal{S}} f(x_{\mathcal{S}}|\theta,I_{d_{\mathcal{s}}}) d(G-G_{*})(\theta)\right) dx_{\mathcal{S}} \\ 
& = \int_{\mathcal{S}} g(x_{\mathcal{S}}) (p_{G, \mathcal{S}}(x_{\mathcal{S}}) - p_{G_{*}, \mathcal{S}}(x_{\mathcal{S}})) dx_{\mathcal{S}}.
\end{align*}
An application of the Holder's inequality leads to
\begin{align*}
& \left(\int_{\mathcal{S}} g(x_{\mathcal{S}}) (p_{G,\mathcal{S}}(x_{\mathcal{S}}) - p_{G_{*},\mathcal{S}}(x_{\mathcal{S}})) dx_{\mathcal{S}} \right)^2 \\
& \leq \left(\int_{\mathcal{S}} g^2(x_{\mathcal{S}})(\sqrt{p_{G,\mathcal{S}}(x_{\mathcal{S}})} + \sqrt{p_{G_{*},\mathcal{S}}(x_{\mathcal{S}})})^2dx_{\mathcal{S}}\right)\left(\int_{\mathcal{S}} \frac{(p_{G, \mathcal{S}}(x_{\mathcal{S}}) - p_{G_{*}, \mathcal{S}}(x_{\mathcal{S}}))^2}{(\sqrt{p_{G, \mathcal{S}}(x_{\mathcal{S}})} + \sqrt{p_{G_{*}, \mathcal{S}}(x_{\mathcal{S}})})^2} dx_{\mathcal{S}}\right) \\
& = 2 d_H^2(p_{G,\mathcal{S}},p_{G_{*},\mathcal{S}}) \int_{\mathcal{S}} g^2(x_{\mathcal{S}})(\sqrt{p_{G,\mathcal{S}}(x_{\mathcal{S}})} + \sqrt{p_{G_{*},\mathcal{S}}(x_{\mathcal{V}})})^2dx_{\mathcal{V}} \\ 
& \leq 4 d_H^2(p_{G},p_{G_{*}}) \int_{\mathcal{S}} g^2(x_{\mathcal{S}})(p_{G,\mathcal{S}}(x_{\mathcal{S}}) + p_{G_{*},\mathcal{S}}(x_{\mathcal{S}}))dx_{\mathcal{S}}, 
\end{align*}
where the final inequality is due to Cauchy-Schwarz's inequality. 

Putting these results together leads to
\begin{align}
\left|\sum_{i=1}^k \pi_i U(\theta_i) - \sum_{j=1}^{k_*} \pi_j^* U(\theta_j^*)\right| \leq 2 d_H(p_{G}, p_{G_{*}}) \sqrt{\int_{\mathcal{S}} g^2(x_{\mathcal{S}})(p_{G,\mathcal{S}}(x_{\mathcal{S}}) + p_{G_{*},\mathcal{S}}(x_{\mathcal{S}}))dx_{\mathcal{S}}}. \label{eq:Hellinger_to_polynomial_1}
\end{align}

Now, we proceed to upper bound $\int_{\mathcal{S}} g^2(x_{\mathcal{S}})(p_{G,\mathcal{S}}(x_{\mathcal{S}}) + p_{G_{*},\mathcal{S}}(x_{\mathcal{S}}))dx_{\mathcal{S}}$. It is clear that
\begin{align}
\int_{\mathcal{S}} g^2(x_{\mathcal{S}})(p_{G,\mathcal{S}}(x_{\mathcal{S}}) + p_{G_{*},\mathcal{S}}(x_{\mathcal{S}}))dx_{\mathcal{S}} \leq 2 \sup_{\theta \in \mathcal{S} \cap B(0,R)} \mathbb{E}_{X \sim \mathcal{N}(\theta, I_{d_{\mathcal{S}}})} [g^2(X)].\label{eq:Hellinger_to_polynomial_2}
\end{align}
Recall that, $g(x_{\mathcal{S}}) = \sum_{m=0}^{\lfloor D/2 \rfloor} \frac{(-1)^m}{2^m m!} \Delta_L^m U(x_{\mathcal{S}})$. We first proceed to bound $\max_{x \in \mathcal{S} \cap B(0,R)} |g(x)|$. Since $U$ is a polynomial of degree $D$, using multivariate Markov polynomial inequalities (see \cite{duffin1941refinement}, or Theorem 1, \cite{HARRIS2008350}), we obtain that
$$\|\partial_{x_i}^2 U\|_\infty  \le \frac{T^{(2)}_{D}(1)}{R^2} \|U\|_\infty = \frac{D^4-D^2}{3R^2} \|U\|_\infty \leq \frac{D^4}{R^2}\|U\|_\infty,$$
where $T_{D}$ is the degree $D$-Chebyshev polynomial. Therefore, we have
$$\|\Delta_{L} U\|_\infty \le d_{\mathcal{V}} \frac{D^4}{R^2} \|U\|_\infty.$$
By repeating this argument $m$ times, we arrive at
$$\|\Delta_{L}^m U\|_\infty \le \left( \frac{d_{\mathcal{V}} D^4}{R^2} \right)^m \|U\|_\infty.$$
As a consequence, 
$$ \max_{x \in \mathcal{S} \cap \mathcal{B}(0, R)} |g(x)| \le \sum_{m=0}^{\lfloor D/2 \rfloor} \frac{1}{2^m m!} \left( \frac{d_{\mathcal{V}} D^4}{R^2} \right)^m \|U\|_\infty \le \exp\left( \frac{d_{\mathcal{V}} D^4}{2 R^2} \right) \|U\|_\infty := M_{f}.$$
From the extremal growth properties of Chebyshev polynomials, any polynomial bounded by $M_f$ on $[-R, R]$ cannot grow faster than $M_f |T_D(\|x\|_2/R)|$ outside the boundary (see Theorem 1, \cite{HARRIS2008350}), namely, 
$$|g(x)| \leq M_{f} \max\left\{1, |T_D(\|x\|_2/R)|\right\}.$$
Since $|T_D(y)| \le (2|y|)^D$ for all $|y| \ge 1$, we obtain further that 
$$ |g(x)| \le M_f \max\left(1, \left( \frac{2\|x\|_2}{R} \right)^D\right) \le M_f \left( 1 + \frac{2\|x\|_2}{R} \right)^D.$$
Now, for any $X = \theta + Z$ where $Z \sim \mathcal{N}(0, I_{d_{\mathcal{S}}})$, substituting the global Chebyshev bound leads to $$ |g(\theta + Z)| \le M_f \left( 1 + \frac{2(R + \|Z\|_2)}{R} \right)^D = M_f \left( 3 + \frac{2\|Z\|_2}{R} \right)^D.$$
It indicates that
$$\mathbb{E}_{X \sim \mathcal{N}(\theta, I_{d_{\mathcal{S}}})} [g^2(X)] = \mathbb{E}_{Z \sim \mathcal{N}(0, I_{d_{\mathcal{S}}})} [g^2(\theta + Z)] \le M_f . \mathbb{E}_{Z} \left[ \left( 3 + \frac{2\|Z\|_2}{R} \right)^{2D} \right].$$
Via Minskowski's inequality, we have
$$\mathbb{E} \left[ \left( 3 + \frac{2\|Z\|_2}{R} \right)^{2D} \right]  \le \left(3 + \frac{2}{R} \Big( \mathbb{E} \big[ \|Z\|_2^{2D} \big] \Big)^{\frac{1}{2D}}\right)^{2D}.$$
Direct calculation yields that
$$\mathbb{E}\big[ (\|Z\|_2^2)^D \big] = d_{\mathcal{S}}(d_{\mathcal{S}}+2)\dots(d_{\mathcal{S}}+2D-2) \le (d_{\mathcal{S}} + 2D)^D.$$
Putting all the bounds together, we arrive at
\begin{align}
\sup_{\theta \in \mathcal{S} \cap B(0,R)} \mathbb{E}_{X \sim \mathcal{N}(\theta, I_{d_{\mathcal{S}}})} [g^2(X)] \leq M_{f} \left( 3 + \frac{2 \sqrt{d_{\mathcal{S}} + 2D}}{R} \right)^D. \label{eq:Hellinger_to_polynomial_3}
\end{align}
Since $d_{\mathcal{S}} \leq \min\{k + k_{*},d\}:=\bar{d}$, by combining the results from equations~\eqref{eq:Hellinger_to_polynomial_1}-\eqref{eq:Hellinger_to_polynomial_3} we obtain
$$\left|\sum_{i=1}^k \pi_i U(\theta_i) - \sum_{j=1}^{k_*} \pi_j^* U(\theta_j^*)\right| \leq C(R,\bar{d},D) ||U||_{\infty} d_H(p_{G}, p_{G_{*}}),$$
where $C(R,\bar{d},D) = 2 \exp\left( \frac{\bar{d} D^4}{2 R^2} \right) \left( 3 + \frac{2 \sqrt{\bar{d} + 2D}}{R} \right)^D$. As a consequence, we reach the conclusion of the lemma.
\end{proof}

\subsection{Proof of Theorem~\ref{theorem:exact_one_group_univariate}}
\label{sec:proof_theorem:exact_one_group_univariate}

\subsubsection{Univariate setting - Local bound}
\label{sec:local_bound_univariate}
To shed light into the proof idea, we first consider the univariate setting, namely, $d = 1$.
To facilitate the ensuing presentation, we denote $G = \sum_{i=1}^{k_{*}} \pi_i \delta_{\theta_i}$ be a strictly positive measure and $\nu = G - G_*$ be the corresponding finite signed measure.

We first consider the local bound when $W_1(G, G_*)$ is sufficiently small such that every fitted atom $\theta_i$ belongs strictly to a Voronoi cell $\mathcal{V}_j$ anchored to its closest true center $\theta_j^*$, with local distance $|\theta_i - \theta_j^*| \le \frac{\Delta_{\mathrm{sep}}}{4}$. Since $G$ only has $k_{*}$ components, it indicates that each Voronoi cell $\mathcal{V}_{j}$ has exactly one element as pointed out in Lemma \ref{lemma:small_wasserstein_implies_center_in_voronoi_cell}. Without loss of generality, we assume that $j \in \mathcal{V}_{j}$, for any $1 \leq j \leq k_{*}$.

\emph{Step 1 - Wasserstein decomposition.}
The first step is to decompose $W_1(G,G_*)$ into moment mismatch and mass mismatch components. We consider the measure $\tilde{G} = \sum_{i=1}^{k_*} \pi_i \delta_{\theta^*_i}$. Using the triangle inequality for Wasserstein distance, we have 
\begin{align}
\label{eqn:dung_thm3_1_first_inequality_W_1}
    W_1(G,G_*) \leq W_1(G,\tilde{G}) + W_1(\tilde{G},G_*). 
\end{align}
For $W_1(G,\tilde{G})$, consider the transportation map totally transferring each $\theta_i$ into its closet center $\theta^*_{i}$. This map gives us an upper bound for $W_1(G,\tilde{G})$
\begin{equation}
\label{eqn:dung_thm3_1_inequality_W_1_G_and_intermediate}
    W_1(G,\tilde{G}) \leq \sum_{j=1}^{k_*}\pi_j|\theta_j-\theta_j^*|. 
\end{equation}
For $W_1(\tilde{G},G_*)$, we keep a mass of $\min\{\pi_i,\pi^*_i\}$ for each $\theta^*_i$, while transferring $\max\{\pi_i-\pi^*_i,0\}$ to the other center. Then, the total mass moved is exactly $\frac{1}{2}\sum_{j = 1}^{k_{*}} |\pi_j - \pi_{j}^{*}|$, and the distance between any true centers is less than $C_0\Delta_{\mathrm{sep}}$, thus 
\begin{equation}
\label{eqn:dung_thm3_1_univariate_local_inequality_W_1_true_measure_and_intermediate}
    W_1(\tilde{G},G_*) \leq C_0 \Delta_{\mathrm{sep}} \sum_{j=1}^{k_*} |\pi_j - \pi_{j}^{*}|. 
\end{equation}

Combining the result from \eqref{eqn:dung_thm3_1_first_inequality_W_1}, \eqref{eqn:dung_thm3_1_inequality_W_1_G_and_intermediate}, and \eqref{eqn:dung_thm3_1_univariate_local_inequality_W_1_true_measure_and_intermediate}, we have 
\begin{align}
W_1(G, G_*) \le \sum_{j=1}^{k_*} \pi_{j}|\theta_{j} - \theta_{j}^{*}| + C_0 \Delta_{\mathrm{sep}} \sum_{j=1}^{k_*} |\pi_j - \pi_{j}^{*}|. \label{eq:exact_key_equation_second}
\end{align}
Before bounding two terms in the above right hand side, it is necessary to bound the term $\sum_{j=1}^{k_*} \pi_j |\theta_{j} - \theta_{j}^{*}|^2$ first.
\newline 

\emph{Step 2 - Bounding variance $\sum_{j=1}^{k_*} \pi_j |\theta_{j} - \theta_{j}^{*}|^2$.} We now consider the following polynomial function, which is similarly defined in Section \ref{sec:variance_test_function}: 
\begin{align*}
    P_{\mathrm{var}}(\theta) : = \prod_{l=1}^{k_*} (\theta - \theta_l^*)^2.
\end{align*}
Since $P_{\mathrm{var}}(\theta_l^*) = 0$, it follows that $\int P_{\mathrm{var}}(\theta) dG_*(\theta) = 0$. For any probability measure $G$, we obtain that
\begin{align*} 
\int P_{\mathrm{var}}(\theta) d\nu(\theta) = \sum_{j=1}^{k_*} \pi_j (\theta_{j} - \theta_{j}^{*})^2 \prod_{l \neq j} (\theta_j - \theta_l^*)^2.
\end{align*}
For any $l \neq j$, we have
\begin{align}
    \label{eq:argument_34}
    |\theta_{j} - \theta_l^*| \ge |\theta_j^* - \theta_l^*| - |\theta_{j} - \theta_{j}^{*}| \ge \Delta_{\mathrm{sep}} - \frac{\Delta_{\mathrm{sep}}}{4} = \frac{3 \Delta_{\mathrm{sep}}}{4}.
\end{align}
Thus, we find that
\begin{align*} 
\int P_{\mathrm{var}}(\theta) d\nu(\theta) & \ge    \left(\frac{3}{4}\right)^{2k_*-2} \Delta_{\mathrm{sep}}^{2k_*-2} \sum_{j=1}^{k_*} \pi_j |\theta_{j} - \theta_{j}^{*}|^2.
\end{align*}
By the results of Lemma~\ref{lemma:Hellinger_to_Polynomial} and Lemma~\ref{lemma:variance_test_function}, we have $\int P_{\mathrm{var}}(\theta) d\nu(\theta) \le C_{\mathrm{poly}}C_{\mathrm{var}} d_H(p_{G},p_{G_*})$. These bounds then lead to
\begin{align}
    \sum_{j=1}^{k_*} \pi_j |\theta_{j} - \theta_{j}^{*}|^2 &\le \left[ \frac{C_{\mathrm{var}} C_{\mathrm{poly}}}{(3/4)^{2k_*-2}} \right]  \Delta_{\mathrm{sep}}^{-(2k_* - 2)} d_H(p_{G},p_{G_*})\nonumber\\
    &: = C_{\mathrm{var,1}} \Delta_{\mathrm{sep}}^{-(2k_* - 2)} d_H(p_{G},p_{G_*}). \label{eq:exact_key_equation_third} 
\end{align}
\newline 

\emph{Step 3 - Bounding the mean discrepancy $\pi_i|\theta_j - \theta_j^*|$.} To obtain an upper bound for $\sum_{j=1}^{k_*} \pi_j |\theta_{j} - \theta_{j}^{*}|$, we construct a  Hermite polynomial $H_j(\theta)$ such that $H_j(\theta_l^*) = 0$ and $H'_j(\theta^*_{l}) =\delta_{jl}$. In particular, we can choose the Hermite interpolation polynomial $H_j(\theta)$ of degree $2k_*-1$ given by
\begin{align*}
    H_j(\theta) : = \ell_j^2(\theta) (\theta - \theta_j^*),  \quad \text{where} \quad \ell_j(\theta) : = \prod_{q \neq j} \frac{\theta - \theta_q^*}{\theta_j^* - \theta_q^*},
\end{align*}
which is also defined in Section \ref{sec:mean_test_function}. It can be checked that $H_j(\theta_l^*) = 0$ and $H_j'(\theta_l^*) = \delta_{jl}$ being a Kronecker delta, for all $1 \leq l \leq k_{*}$.
By applying Taylor expansion for $H_j(\theta_l)$ exactly around its true center $\theta_l^*$, we have $H_j(\theta_l) = \delta_{jl}(\theta_{l} - \theta_{l}^{*}) + \frac{1}{2} H_j''(\xi_l) (\theta_{l} - \theta_{l}^{*})^2$, where $\xi_l$ is some point between $\theta_l$ and $\theta_l^*$. Thus, we arrive at $$ \int H_j(\theta) d\nu(\theta) = \pi_{j}(\theta_{j} - \theta_{j}^{*}) + \frac{1}{2} \sum_{l=1}^{k_*} \pi_{l} H_{j}''(\xi_l) (\theta_{l} - \theta_{l}^{*})^2,$$
which indicates that
\begin{align}
    \pi_{j}|\theta_{j} - \theta_{j}^{*}| &\leq \left|\int H_j(\theta) d\nu(\theta)\right| + \frac{1}{2} \sum_{l=1}^{k_*} \pi_{l} |H_{j}''(\xi_l)| (\theta_{l} - \theta_{l}^{*})^2\nonumber\\
    \label{eq:bound_2}
    &\leq C_{\mathrm{poly}}\|H_j\|_{\infty}d_H(p_{G},p_{G_*})+\frac{1}{2}\max_{\theta \in \mathcal{V}_{l}: |\theta - \theta_{l}^{*}| \leq \Delta_{\mathrm{sep}}/4} |H_j''(\theta)|\sum_{l=1}^{k_*} \pi_{l}  (\theta_{l} - \theta_{l}^{*})^2,
\end{align}
where the second inequality follows from Lemma~\ref{lemma:Hellinger_to_Polynomial}. By the results of Lemma~\ref{lemma:mean_test_function}, we have
\begin{align}
    \label{eq:exact_key_equation_third_2}
    \|H_j\|_{\infty}\leq C_{\mathrm{norm},H}\Delta_{\mathrm{sep}}^{-(2k_*-2)},
\end{align}
and
\begin{align}
    \label{eq:exact_key_equation_third_1}
    \max_{\theta \in \mathcal{V}_{l}: |\theta - \theta_{l}^{*}| \leq \Delta_{\mathrm{sep}}/4} |H_j''(\theta)| \le \Delta_{\mathrm{sep}}^{-1} C_{H,1}.
\end{align}
By plugging the bounds~\eqref{eq:exact_key_equation_third}, \eqref{eq:exact_key_equation_third_2}, and \eqref{eq:exact_key_equation_third_1} to the bound in equation~\eqref{eq:bound_2}, we obtain
\begin{align}
    \pi_{j}|\theta_{j} - \theta_{j}^{*}| & \leq C_{\mathrm{poly}}C_{\mathrm{norm},H}\Delta_{\mathrm{sep}}^{-(2k_*-2)} d_H(p_{G},p_{G_*}) + \frac{1}{2} C_{\mathrm{var},1} C_{H,1}  \Delta_{\mathrm{sep}}^{-(2k_* - 1)} d_H(p_{G},p_{G_*}) \nonumber \\
    & \leq \left(2R\cdot C_{\mathrm{poly}}C_{\mathrm{norm},H} + \frac{1}{2} C_{\mathrm{var},1} C_{H,1}\right) \Delta_{\mathrm{sep}}^{-(2k_*-1)} d_H(p_{G},p_{G_*})\nonumber\\
    &:=C_{\text{mean},1}\Delta_{\mathrm{sep}}^{-(2k_*-1)} d_H(p_{G},p_{G_*}), \label{eq:exact_key_equation_fourth_1}
\end{align}
where the second inequality occurs due to $\Delta_{\mathrm{sep}}\leq 2R$.
\newline 

\emph{Step 4 - Bounding mass discrepancy $|\pi_j - \pi_j^*|$.} We now turn to upper bound $|\pi_{j} - \pi_{j}^{*}|$. For that purpose, we build a polynomial $E_j$ such that $E_j(\theta_l^*) = \delta_{jl}$ and $E'_j(\theta^*_l) = 0$. We take into account another Hermite polynomial $E_j(\theta)$ of degree $2k_*-1$ defined as
\begin{align*}
    E_j(\theta) : = \ell_j^2(\theta) \left[ \bar{A}_j + \bar{B}_j(\theta - \theta_j^*) \right], \quad \text{where} \quad \ell_j(\theta) = \prod_{q \neq j} \frac{\theta - \theta_q^*}{\theta_j^* - \theta_q^*},
\end{align*}
and $\bar{A}_j = 1$, $\bar{B}_j = - 2\ell_j'(\theta_j^*)$. Direct calculation yields that $E_j(\theta_l^*) = \delta_{jl}$ and $E_j'(\theta_l^*) = 0$ for all $l$.
By utilizing Taylor expansion for $E_j(\theta_l)$ exactly around its true center $\theta_l^*$, we have $E_j(\theta_l) = \delta_{jl} + \frac{1}{2} \bar{H}_j''(\xi_l) (\theta_{l} - \theta_{l}^{*})^2$, where $\xi_l$ is some point between $\theta_l$ and $\theta_l^*$. Thus, we arrive at $$ \int E_j(\theta) d\nu(\theta) = \sum_{l=1}^{k_*} \pi_{l} E_j(\theta_{l}) - \pi_j^* = (\pi_{j} - \pi_{j}^{*}) + \frac{1}{2} \sum_{l=1}^{k_*} \pi_{l} E_j''(\xi_{l}) (\theta_{l} - \theta_{l}^{*})^2,$$
implying that
\begin{align} 
|\pi_{j} - \pi_{j}^{*}| & \le \left| \int E_j(\theta) d\nu(\theta) \right| +\frac{1}{2} \sum_{l=1}^{k_*} \pi_{l} |E_j''(\xi_{l})| (\theta_{l} - \theta_{l}^{*})^2 \nonumber\\
\label{eq:bound_3}
& \le C_{\mathrm{poly}} \|E_j\|_{\infty} d_H(p_{G},p_{G_*}) + \frac{1}{2}\max_{\substack{\theta \in \mathcal{V}_l: |\theta - \theta_{l}^{*}| \leq \Delta_{\mathrm{sep}}/4}} |E''_j(\theta)|\sum_{l=1}^{k_*} \pi_{l} (\theta_{l} - \theta_{l}^{*})^2.
\end{align}
By the results of Lemma~\ref{lemma:mass_test_function}, we have
\begin{align}
    \label{eq:H_bar_supnorm}
    \|E_{j}\|_{\infty}\leq C_{\mathrm{norm},E} \Delta_{\mathrm{sep}}^{-(2k_*-1)},
\end{align}
and
\begin{align}
\label{eq:H_bar_second_order}
\max_{\theta \in \mathcal{V}_l: |\theta - \theta_{l}^{*}| \leq \Delta_{\mathrm{sep}}/4} |E_j''(\theta)| &\le \Delta_{\mathrm{sep}}^{-2} C_{E,1}. 
\end{align}

By applying the results in equations~\eqref{eq:exact_key_equation_third}, \eqref{eq:H_bar_supnorm} and \eqref{eq:H_bar_second_order} to the bound in equation~\eqref{eq:bound_3}, we obtain
\begin{align*} 
|\pi_{j} - \pi_{j}^{*}| 
& \le C_{\mathrm{poly}}C_{\mathrm{norm},E} \Delta_{\mathrm{sep}}^{-(2k_*-1)} d_H(p_{G},p_{G_*}) + \frac{C_{E,1}}{2 \Delta_{\mathrm{sep}}^2} \left[ C_{\mathrm{var},1} \Delta_{\mathrm{sep}}^{-(2k_*-2)} d_H(p_{G},p_{G_*}) \right].
\end{align*}
Since $\Delta_{\mathrm{sep}}^{-(2k_*-1)} = \Delta_{\mathrm{sep}} \cdot \Delta_{\mathrm{sep}}^{-2k_*} \le 2R \cdot \Delta_{\mathrm{sep}}^{-2k_*}$, we rigorously factor out the dominant penalty:
\begin{align}
    |\pi_{j} - \pi_{j}^{*}| \le \left( 2R C_{\mathrm{poly}}C_{\text{norm}} + \frac{1}{2} C_{E,1} C_{\mathrm{var}} \right) \Delta_{\mathrm{sep}}^{-2k_*} d_H(p_{G},p_{G_*}) := C_{\mathrm{mass},1}\Delta_{\mathrm{sep}}^{-2k_*} d_H(p_{G},p_{G_*}). \label{eq:exact_key_equation_sixth}
\end{align}
\newline 

\emph{Step 5 - Bounding $W_1(G,G_*)$ and conclusion.} 
We substitute the bounds in equations~\eqref{eq:exact_key_equation_sixth} and ~\eqref{eq:exact_key_equation_fourth_1} back into the inequality~\eqref{eq:exact_key_equation_second} and obtain that 
\begin{align*}
    W_1(G, G_*) &\le k_{*} C_{\mathrm{mean}} \Delta_{\mathrm{sep}}^{-(2k_*-1)} d_H(p_{G},p_{G_*}) + k_{*} C_0 \Delta_{\mathrm{sep}} C_{\mathrm{mass}} \Delta_{\mathrm{sep}}^{-2k_*} d_H(p_{G},p_{G_*}), 
\end{align*}
or equivalently,
$$ W_1(G, G_*) \le k_{*} \left( C_{\mathrm{mean}} + C_0 C_{\mathrm{mass}} \right) \Delta_{\mathrm{sep}}^{-(2k_*-1)} d_H(p_{G},p_{G_*}). $$
Define the constant $C_{\mathrm{local},1}^{-1} = k_{*} \left(C_{\mathrm{mean}} + C_0 C_{\text{mass}}\right)$, 
then we achieve the local geometric bound:
\begin{align}
    d_H(p_{G},p_{G_*}) \ge C_{\mathrm{local},1} \cdot \Delta_{\mathrm{sep}}^{2k_* - 1} \cdot W_1(G, G_*) \label{eq:exact_local_bound}
\end{align}
as long as $W_1(G, G_*)$ is sufficiently small such that every fitted atom $\theta_i$ belongs strictly to a Voronoi cell $\mathcal{V}_j$ with local distance $|\theta_j - \theta_j^*| \le \frac{\Delta_{\mathrm{sep}}}{4}$.


\subsubsection{Univariate setting - Global bound}
\label{sec:global_bound_exact_univariate_one_cluster}
We now establish the global bound without any assumptions on the supports of $G$. For every fitted atom $\theta_i$, let $c(i) = \text{argmin}_j |\theta_i - \theta_j^*|$ be the index of true center $\theta_{j}^{*}$ that is closest to that fitted atom. While there are multiple true centers sharing the same minimum distance to $\theta_{i}$, we just pick randomly one of them. We unconditionally partition the components of $G$ into two disjoint sets based on a strictly enforced boundary of $\Delta_{\mathrm{sep}}/4$:
\begin{itemize}
    \item The near set $\mathcal{M}_{\mathrm{near}} = \{ i \in \{1, \dots, k_{*}\} : |\theta_{i} - \theta_{c(i)}^{*}| \le \Delta_{\mathrm{sep}} / 4 \}$. We denote $\bar{\mathcal{V}}_j = \{i \in \mathcal{M}_{\mathrm{near}} : c(i) = j\}$ as a subset of the Voronoi cells $\mathcal{V}_{j}$.
    \item The far set $\mathcal{M}_{\mathrm{far}} = \{ i \in \{1, \dots, k_{*}\} : |\theta_{i} - \theta_{c(i)}^{*}| > \Delta_{\mathrm{sep}} / 4 \}$. We denote the total far mass as $\pi_{\mathcal{F}} = \sum_{i \in \mathcal{M}_{\mathrm{far}}} \pi_i$.
\end{itemize}
\emph{Step 1 - Wasserstein decomposition.} We define a deterministic spatial mapping function $T: \mathbb{R} \to \Theta^* = \left\{\theta_{1}^{*}, \ldots, \theta_{k_{*}}^{*}\right\}$ such that $T(\theta_i) = \theta_{c(i)}^*$. We define $\widetilde{G}$ as the exact pushforward measure of $G$ under $T$:
$$ \widetilde{G} = T_{\#} G = \sum_{i=1}^{k_*} \pi_i \delta_{T(\theta_i)} = \sum_{i=1}^{k_*} \pi_i \delta_{\theta_{c(i)}^*}.$$
From the triangle inequality with the $W_{1}$ metric, we have
\begin{align}
    W_1(G, G_*) \leq W_1(G, \widetilde{G}) + W_{1}(\widetilde{G},G_*). \label{eq:exact_Wasserstein_bound_0}
\end{align}
We first upper bound $W_1(G, \widetilde{G})$. By specifically considering the transportation plan $\hat{\gamma} = (\text{Id} \times T)_{\#} G$, we obtain that:
\begin{align} 
W_1(G, \widetilde{G}) \le \int_{\mathbb{R} \times \mathbb{R}} |x - y| d\hat{\gamma}(x,y) & = \int_{\mathbb{R}} |\theta - T(\theta)| dG(\theta) = \sum_{i=1}^{k_*} \pi_i |\theta_i - \theta_{c(i)}^*| \nonumber \\
& = \sum_{i \in \mathcal{M}_{\mathrm{near}}} \pi_i |\theta_{i} - \theta_{c(i)}^{*}| + \sum_{i \in \mathcal{M}_{\mathrm{far}}}  \pi_i |\theta_{i} - \theta_{c(i)}^{*}|. \label{eq:exact_Wasserstein_bound_1}
\end{align}

We now move to upper bound $W_1(\widetilde{G}, G_*)$. Let us rewrite $\widetilde{G}$ on the true discrete support basis by grouping all mass mapped to the identical center $j$ as follows: $$\widetilde{G} = \sum_{j=1}^{k_*} \tilde{\pi}_j \delta_{\theta_j^*} \quad \text{where} \quad \tilde{\pi}_j = \sum_{i: c(i) = j} \pi_i.$$
Let the net mass discrepancy at center $j$ be exactly defined as $\tilde{\Delta} \pi_j = \tilde{\pi}_j - \pi_j^*$. We consider the transportation plan from $\tilde{G}$ to $G_*$ such that for each $j$, we only keep $\min\{\sum_{i \in \mathcal{V}_j}\pi_i,\pi_j^*\}$ for center $\theta_j^*$ while moving $\max\{\tilde{\Delta}\pi_j,0\}$ to other center, where $\tilde{\Delta}\pi_i:= \tilde{\Delta}\pi_i-\pi_i^*$. Then, it is obvious that the total transported mass is exactly $\frac{1}{2}\sum_{j=1}^{k_*}|\tilde{\Delta}\pi_j|$. As a result, we achieve the following bound for Wasserstein distance between $\tilde{G}$ and $G_*$ 
\begin{align}
W_1(\widetilde{G}, G_*) \leq \frac{C_{0}\Delta_{\mathrm{sep}}}{2} \sum_{j = 1}^{k_{*}} |\tilde{\Delta} \pi_j|. \label{eq:exact_Wasserstein_bound_3}
\end{align}
Putting the results from bounds~\eqref{eq:exact_Wasserstein_bound_0},~\eqref{eq:exact_Wasserstein_bound_1}, and~\eqref{eq:exact_Wasserstein_bound_3} together, we attain the following unconditional global transport bound
\begin{align}
W_1(G, G_*) \le \sum_{i \in \mathcal{M}_{\mathrm{near}}} \pi_i |\theta_{i} - \theta_{c(i)}^{*}| + \sum_{i \in \mathcal{M}_{\mathrm{far}}}  \pi_i |\theta_{i} - \theta_{c(i)}^{*}| + \frac{1}{2} C_0 \Delta_{\mathrm{sep}} \sum_{j=1}^{k_*} |\tilde{\Delta} \pi_j|. \label{eq:exact_Wasserstein_bound_4}
\end{align}
From the definition of $\tilde{\pi}_j$, we decompose the subset of atoms mapping to $j$ into near and far subsets as follows:$$\tilde{\pi}_j = \sum_{i \in \bar{\mathcal{V}}_j} \pi_i  + \sum_{i \in \mathcal{M}_{\mathrm{far}}, c(i)=j} \pi_i.$$
We define $\Delta \pi_j : = \left( \sum_{i \in \bar{\mathcal{V}}_j} \pi_i \right) - \pi_j^*$. Then, simple algebra shows that $\tilde{\Delta} \pi_j = \tilde{\pi}_j - \pi_j^* = \Delta \pi_j \ + \sum_{i \in \mathcal{M}_{\mathrm{far}}, c(i)=j} \pi_i$. An application of triangle inequality leads to $|\tilde{\Delta} \pi_j| \leq |\Delta \pi_j|  + \sum_{i \in \mathcal{M}_{\mathrm{far}}, c(i)=j} \pi_i$. Therefore, we have
\begin{align}
\sum_{j=1}^{k_*} |\tilde{\Delta} \pi_j| \le \sum_{j=1}^{k_*} |\Delta \pi_j| \ + \ \sum_{j=1}^{k_*} \sum_{i \in \mathcal{M}_{\mathrm{far}}, c(i)=j} \pi_i = \sum_{j=1}^{k_*} |\Delta \pi_j| + \sum_{i \in \mathcal{M}_{\mathrm{far}}} \pi_{i}. \label{eq:exact_Wasserstein_bound_5}
\end{align}
Combining the results from equation~\eqref{eq:exact_Wasserstein_bound_4} and~\eqref{eq:exact_Wasserstein_bound_5}, we obtain the following bound on the $W_{1}$ metric:
\begin{align}
W_1(G, G_*) \le \sum_{i \in \mathcal{M}_{\mathrm{near}}} \pi_i |\theta_{i} - \theta_{c(i)}^{*}| + \sum_{i \in \mathcal{M}_{\mathrm{far}}}  \pi_i |\theta_{i} - \theta_{c(i)}^{*}| + \frac{1}{2} C_0 \Delta_{\mathrm{sep}} \left( \sum_{j=1}^{k_*} |\Delta \pi_j| + \pi_{\mathcal{F}} \right), \label{eq:exact_global_key_equation_first}
\end{align}
where $\pi_{\mathcal{F}} : = \sum_{i \in \mathcal{M}_{\mathrm{far}}} \pi_{i}$. In order to obtain upper bounds for terms in the above right hand side, we first need to bound two terms $\sum_{i \in \mathcal{M}_{\mathrm{near}}} \pi_i (\theta_{i} - \theta_{c(i)}^{*})^2$ and $\sum_{i \in \mathcal{M}_{\mathrm{far}}} \pi_i (\theta_{i} - \theta_{c(i)}^{*})^{2k_*}$.
\newline

\emph{Step 2 - Bounding near variance and far high moment for variance.} We now evaluate the polynomial $P_{\mathrm{var}}(\theta)=\prod_{l=1}^{k_*} (\theta - \theta_l^*)^2$ unconditionally over all atoms of $G$. 
For any index $i \in \mathcal{M}_{\mathrm{near}}$, the argument with the local bound in equation~\eqref{eq:argument_34} indicates 
\begin{align}
\sum_{i \in \mathcal{M}_{\mathrm{near}}} \pi_i (\theta_{i} - \theta_{c(i)}^{*})^2 \le C_{\mathrm{var},2} \Delta_{\mathrm{sep}}^{-(2k_*-2)} d_H(p_{G},p_{G_*}). \label{eq:exact_global_key_equation_second}
\end{align}
For any index $i \in \mathcal{M}_{\mathrm{far}}$, its distance to any true center $\theta_{l}^{*}$ is strictly bounded by $|\theta_i - \theta_l^*| \ge |\theta_{i} - \theta_{c(i)}^{*}|$, implying that $P_{\mathrm{var}}(\theta_i)=\prod_{l=1}^{k_*} (\theta_i - \theta_l^*)^2 \ge (\theta_{i} - \theta_{c(i)}^{*})^{2k_*}$.

From the results of Lemma~\ref{lemma:Hellinger_to_Polynomial} and Lemma~\ref{lemma:variance_test_function} and the fact that $\int P_{\mathrm{var}}(\theta) dG_*(\theta) = 0$, we have 
\begin{align*}
\int P_{\mathrm{var}}(\theta) dG(\theta) = \int P_{\mathrm{var}}(\theta) d\nu(\theta) \le C_{\mathrm{var}}C_{\mathrm{poly}} d_H(p_{G},p_{G_*}).
\end{align*}
We partition the above integral as 
\begin{align}
    \int P_{\mathrm{var}}(\theta) dG(\theta) = \sum_{i \in \mathcal{M}_{\mathrm{near}}} \pi_i P_{\mathrm{var}}(\theta_i) + \sum_{i \in \mathcal{M}_{\mathrm{far}}} \pi_i P_{\mathrm{var}}(\theta_i). \label{eq:exact_global_key_equation_first_1}
\end{align}
Putting these results together, we have
\begin{align*}
    \sum_{i \in \mathcal{M}_{\mathrm{far}}} \pi_i (\theta_{i} - \theta_{c(i)}^{*})^{2k_*}&\leq\sum_{i \in \mathcal{M}_{\mathrm{far}}} \pi_iP_{\mathrm{var}}(\theta_i)\leq  \int P_{\mathrm{var}}(\theta) dG(\theta)\\
    &\le  C_{\mathrm{var}}C_{\mathrm{poly}}  d_H(p_{G},p_{G_*}) := C_{\mathrm{far,moment,1}} d_H(p_{G},p_{G_*}).
\end{align*}
We now systematically convert this rigid far $2k_*$-moment bound into bounds for the far transport cost $\sum_{i \in \mathcal{M}_{\mathrm{far}}} \pi_i |\theta_{i} - \theta_{c(i)}^{*}|$ and the far mass $\pi_{\mathcal{F}}$ required for the bound of $W_{1}(G,G_{*})$ in equation~\eqref{eq:exact_global_key_equation_first}. 
\newline

\emph{Step 3: Bounding far parameters.} In this step, we bound the related far parameters. 

\emph{Step 3.1: Bounding far mean discrepancy.} For any index $i \in \mathcal{M}_{\mathrm{far}}$, we explicitly have the strict geometric condition $|\theta_{i} - \theta_{c(i)}^{*}| > \Delta_{\mathrm{sep}}/4$. Direct algebra shows that $|\theta_{i} - \theta_{c(i)}^{*}| = |\theta_{i} - \theta_{c(i)}^{*}|^{2k_*} |\theta_{i} - \theta_{c(i)}^{*}|^{-(2k_*-1)} < (\theta_{i} - \theta_{c(i)}^{*})^{2k_*} \left(\frac{4}{\Delta_{\mathrm{sep}}}\right)^{2k_*-1}$. Therefore, we find that
\begin{align}
    \sum_{i \in \mathcal{M}_{\mathrm{far}}} \pi_i |\theta_{i} - \theta_{c(i)}^{*}| &\le \left(\frac{4}{\Delta_{\mathrm{sep}}}\right)^{2k_*-1} \sum_{i \in \mathcal{M}_{\mathrm{far}}} \pi_i (\theta_{i} - \theta_{c(i)}^{*})^{2k_*}\nonumber\\
    &\le  4^{2k_*-1} C_{\mathrm{far,moment,1}} \Delta_{\mathrm{sep}}^{-(2k_*-1)} d_H(p_{G},p_{G_*}) := C_{\mathrm{far,mean,1}}\Delta_{\mathrm{sep}}^{-(2k_*-1)} d_H(p_{G},p_{G_*})\label{eq:exact_global_key_equation_third}
\end{align} 

\emph{Step 3.2: Bounding far mass.} For any index $i \in \mathcal{M}_{\mathrm{far}}$, we have $1 = (\theta_{i} - \theta_{c(i)}^{*})^{2k_*} (\theta_{i} - \theta_{c(i)}^{*})^{-2k_*} < (\theta_{i} - \theta_{c(i)}^{*})^{2k_*} \left(\frac{4}{\Delta_{\mathrm{sep}}}\right)^{2k_*}$, which follows that
\begin{align}
    \pi_{\mathcal{F}} &= \sum_{i \in \mathcal{M}_{\mathrm{far}}} \pi_i \le \left(\frac{4}{\Delta_{\mathrm{sep}}}\right)^{2k_*} \sum_{i \in \mathcal{M}_{\mathrm{far}}} \pi_i (\theta_{i} - \theta_{c(i)}^{*})^{2k_*} \le 4^{2k_*} C_{\mathrm{far,moment,1}} \Delta_{\mathrm{sep}}^{-2k_*} d_H(p_{G},p_{G_*}) \nonumber\\
    &:= C_{\mathrm{far,mass,1}} \Delta_{\mathrm{sep}}^{-2k_*} d_H(p_{G},p_{G_*}). \label{eq:exact_global_key_equation_fourth}
\end{align} 
\newline

\emph{Step 4: Bounding mass discrepancy. } 
We now evaluate the mass extractor test function $E_j(x)$ such that $E_j(\theta_l^*) = \delta_{jl}$ and $E_j'(\theta_l^*) = 0$ for all $l$. A plausible choice is a  Hermite polynomial $E_j(\theta)$ of degree $2k_*-1$ defined in Section \ref{sec:point_wise_mass_extractor_non_multicluster}
\begin{align*}
    E_j(\theta) : = \ell_j^2(\theta) \left[ \bar{A}_j + \bar{B}_j(\theta - \theta_j^*) \right], \quad \text{where} \quad \ell_j(\theta) = \prod_{q \neq j} \frac{\theta - \theta_q^*}{\theta_j^* - \theta_q^*},
\end{align*}
and $\bar{A}_j = 1$, $\bar{B}_j = - 2\ell_j'(\theta_j^*)$. Direct calculation yields that $E_j(\theta_l^*) = \delta_{jl}$ and $E_j'(\theta_l^*) = 0$ for all $l$.  
Integrating $E_j$ unconditionally yields:
\begin{align*}
    \int E_j(\theta) d\nu(\theta) &=\sum_{i=1}^{k_*} \pi_{i} E_j(\theta_{i}) - \pi_j^*=\sum_{i \in \mathcal{M}_{\mathrm{near}}} \pi_i E_j(\theta_i)- \pi_j^*+\sum_{i \in \mathcal{M}_{\mathrm{far}}} \pi_i E_j(\theta_i).
\end{align*}
For $i\in\mathcal{M}_{\mathrm{near}}$, by performing the Taylor expansion on $E_j(\theta_i)$ exactly around its true center $\theta_{c(i)}^*$, we have $E_j(\theta_i) = \delta_{jc(i)} + \frac{1}{2} E_j''(\xi_i) (\theta_{i} - \theta_{c(i)}^{*})^2$. 
Thus, we arrive at 
\begin{align*}
    \int E_j(\theta) d\nu(\theta) &=\Delta \pi_j + \frac{1}{2}\sum_{i \in \mathcal{M}_{\mathrm{near}}} \pi_{i} E_j''(\xi_i) (\theta_{i} - \theta_{c(i)}^{*})^2 + \sum_{i \in \mathcal{M}_{\mathrm{far}}} \pi_i E_j(\theta_i).
\end{align*}
Applying the triangle inequality leads to
\begin{align}
    |\Delta \pi_j| \le \left| \int E_j(\theta) d\nu(\theta) \right| + \underbrace{ \frac{1}{2}\sum_{i \in \mathcal{M}_{\mathrm{near}}} \pi_{i} |E_j''(\xi_i)| (\theta_{i} - \theta_{c(i)}^{*})^2}_{\text{Near Error}} + \underbrace{ \sum_{i \in \mathcal{M}_{\mathrm{far}}} \pi_i |E_j(\theta_i)| }_{\text{Far Error}}. \label{eq:exact_global_key_equation_fifth}
\end{align}
From the results of Lemma~\ref{lemma:Hellinger_to_Polynomial} and Lemma~\ref{lemma:mass_test_function}, we have
\begin{align*}
    \left| \int E_j(\theta) d\nu(\theta) \right| \le C_{\mathrm{poly}}C_{\mathrm{norm},E} \Delta_{\mathrm{sep}}^{-(2k_*-1)} d_H(p_{G},p_{G_*}).
\end{align*}
Since $\Delta_{\mathrm{sep}} \le 2R$, we factor $\Delta_{\mathrm{sep}}^{-(2k_*-1)} = \Delta_{\mathrm{sep}} \cdot \Delta_{\mathrm{sep}}^{-2k_*} \le 2R \Delta_{\mathrm{sep}}^{-2k_*}$, leading to
\begin{align}
\left| \int E_j(\theta) d\nu(\theta) \right| \le 2R C_{\mathrm{poly}}C_{\mathrm{norm},E} \Delta_{\mathrm{sep}}^{-2k_*} d_H(p_{G},p_{G_*}). \label{eq:exact_global_key_equation_sixth}
\end{align}
Regarding the near error, Lemma~\ref{lemma:mass_test_function} indicates that it is bounded explicitly by $\frac{C_{\bar{H}}}{2 \Delta_{\mathrm{sep}}^2} \sum_{i \in \mathcal{M}_{\mathrm{near}}} \pi_{i} (\theta_{i} - \theta_{c(i)}^{*})^2$. Substituting the bound in equation~\eqref{eq:exact_global_key_equation_second}, we obtain that
\begin{align}
\frac{1}{2}\sum_{i \in \mathcal{M}_{\mathrm{near}}}  \pi_i |E_j''(\xi_i)| (\theta_{i} - \theta_{c(i)}^{*})^2 \leq \frac{1}{2} C_{E,1} C_{\mathrm{var},2} \Delta_{\mathrm{sep}}^{-2k_*} d_H(p_{G},p_{G_*}).\label{eq:exact_global_key_equation_seventh}
\end{align}
Regarding the far error, for any index $i \in \mathcal{M}_{\mathrm{far}}$, the distance from $\theta_i$ to any arbitrary true center $\theta^{*}_{q}$ is bounded by $|\theta_i - \theta_q^*| \le |\theta_i - \theta_{c(i)}^*| + |\theta_{c(i)}^* - \theta_q^*| \le |\theta_i - \theta_{c(i)}^*| + C_0 \Delta_{\mathrm{sep}}$. Because $\Delta_{\mathrm{sep}} < 4|\theta_i - \theta_{c(i)}^*|$, we bound the distance as
\begin{align*}
    |\theta_i - \theta_q^*| < |\theta_i - \theta_{c(i)}^*| + 4C_0 |\theta_i - \theta_{c(i)}^*| = (1+4C_0)|\theta_i - \theta_{c(i)}^*|.
\end{align*}
The denominator of the polynomial function $\ell_j$ is precisely bounded below by $\Delta_{\mathrm{sep}}^{k_*-1}$. Substituting the root bound yields $$\ell_j^2(\theta_i) \le (1+4C_0)^{2k_*-2} \Delta_{\mathrm{sep}}^{-(2k_*-2)} |\theta_i - \theta_{c(i)}^*|^{2k_*-2}.$$ 
From the argument in equation~\eqref{eq:bound_Bj}, we have $|\bar{B}_j| \le \Delta_{\mathrm{sep}}^{-1} C_B$. Because $|\theta_i - \theta_{c(i)}^*| > \Delta_{\mathrm{sep}}/4$, it follows that
\begin{align*}
|\bar{A}_j + \bar{B}_j(\theta_i - \theta_j^*)| & \le 1 + \Delta^{-1}_{\mathrm{sep}} C_B(1+4C_0)|\theta_i - \theta_{c(i)}^*| \\
& < \Delta_{\mathrm{sep}}^{-1} \Big[ 4  + C_B(1+4C_0) \Big] |\theta_i - \theta_{c(i)}^*| := \Delta_{\mathrm{sep}}^{-1} C_{\mathrm{lin},1} |\theta_i - \theta_{c(i)}^*|.
\end{align*}
Multiplying the bounded components, we have
\begin{align*}
    |E_j(\theta_i)| & \le \Big[ (1+4C_0)^{2k_*-2} \Delta_{\mathrm{sep}}^{-(2k_*-2)} |\theta_i - \theta_{c(i)}^*|^{2k_*-2} \Big] \Big[ C_{\mathrm{lin},1} \Delta_{\mathrm{sep}}^{-1} |\theta_i - \theta_{c(i)}^*| \Big] \\
    & := C_{\mathrm{far},E,1} \Delta_{\mathrm{sep}}^{-(2k_*-1)} |\theta_i - \theta_{c(i)}^*|^{2k_*-1}.
\end{align*}
Therefore, we obtain that
\begin{align}
\sum_{i \in \mathcal{M}_{\mathrm{far}}} \pi_i |E_j(\theta_i)| &\le C_{\mathrm{far},E,1} \Delta_{\mathrm{sep}}^{-(2k_*-1)} \sum_{i \in \mathcal{M}_{\mathrm{far}}} \pi_i |\theta_i - \theta_{c(i)}^*|^{2k_* - 1} \nonumber\\
& \leq 4 C_{\mathrm{far},E,1} \Delta_{\mathrm{sep}}^{-2k_*} \sum_{i \in \mathcal{M}_{\mathrm{far}}} \pi_i |\theta_i - \theta_{c(i)}^*|^{2k_*} \nonumber \\
& \leq 4 C_{\mathrm{far,E},1} C_{\mathrm{far,moment,1}} \Delta_{\mathrm{sep}}^{-2k_*} d_H(p_{G},p_{G_*}). \label{eq:exact_global_key_equation_eighth}
\end{align}
Plugging the upper bounds in equations~\eqref{eq:exact_global_key_equation_sixth}-\eqref{eq:exact_global_key_equation_eighth} to the inequality in equation~\eqref{eq:exact_global_key_equation_fifth} leads to
\begin{align}
|\Delta \pi_j| & \le \big[ 2R C_{\mathrm{poly}}C_{\mathrm{norm},E} + \frac{1}{2} C_{\mathrm{E},1} C_{\mathrm{var},2} + 4 C_{\mathrm{far,E}} C_{\mathrm{far,moment,1}} \big] \Delta_{\mathrm{sep}}^{-2k_*} d_H(p_{G},p_{G_*}) \nonumber \\
& :=  C_{\mathrm{mass},2}\Delta_{\mathrm{sep}}^{-2k_*} d_H(p_{G},p_{G_*}). \label{eq:exact_global_key_equation_ninth}
\end{align}
\newline

\emph{Step 5: Bounding mass discrepancy.} In the local bound in Section~\ref{sec:proof_theorem:exact_one_group_univariate}, we utilize the function $H_{j}$ to bound that term. However, the key distinction between the global setting and the local setting is that for the local setting the sub-Voronoi cell $\bar{\mathcal{V}}_{j}$ only has one element for each $1 \leq j \leq k_{*}$. In the global setting, that cell can contain more than one element. In particular, with the similar argument as that in the local bound setting, we arrive at
$$ \int H_j(\theta) d\nu(\theta) = \sum_{i \in \bar{\mathcal{V}}_{j}} \pi_{i}(\theta_{i} - \theta_{j}^{*}) + \frac{1}{2} \sum_{l=1}^{k_*} \sum_{i \in \bar{\mathcal{V}}_{l}} \pi_{l} H_{j}''(\xi_l) (\theta_{i} - \theta_{l}^{*})^2.$$
If $\bar{\mathcal{V}}_{j}$ has more than one element, the term $\sum_{i \in \bar{\mathcal{V}}_{j}} \pi_{i}(\theta_{i} - \theta_{j}^{*})$ can be zero, which means that it is not possible to bound $\sum_{i \in \bar{\mathcal{V}}_{j}} \pi_{i}|\theta_{i} - \theta_{j}^{*}|$ based on the previous argument that we use in the local bound when $\bar{\mathcal{V}}_{j}$ has only one element. However, an important insight is that since $G$ only has $k_{*}$ elements, if some $\bar{\mathcal{V}}_{j}$ has more than one element, it means that there is a Voronoi cell that has no element from $G$. That empty Voronoi cell plays an important role in overcoming the issue of the impossibility of bounding $\sum_{i \in \bar{\mathcal{V}}_{j}} \pi_{i}|\theta_{i} - \theta_{j}^{*}|$ via the function $H_{j}$ when $\bar{\mathcal{V}}_{j}$ has more than one element. 

We now divide our argument into two cases:

\emph{Case 5.1.} There exists a sub-Voronoi cell $\bar{\mathcal{V}}_{j}$ that has more than one element. Then, without loss of generality, we assume that the Voronoi cell $\mathcal{V}_{1}$ has no element. Under this case, it is clear that $|\Delta \pi_{1}| = \pi_{1}^{*}$. The bound in equation~\eqref{eq:exact_global_key_equation_ninth} indicates that
\begin{align}
    \pi_{1}^{*} \leq C_{\mathrm{mass,2}} \Delta_{\mathrm{sep}}^{-2k_*} d_H(p_{G},p_{G_*}). \label{eq:exact_global_key_equation_tenth}
\end{align}
As $\pi_{1}^{*} \geq \pi_{\text{min}}^{*}$, this bound becomes
\begin{align*}
    \pi_{\text{min}}^{*} \leq C_{\mathrm{mass,2}} \Delta_{\mathrm{sep}}^{-2k_*} d_H(p_{G},p_{G_*}).
\end{align*}
Now, since $|\theta_{i} - \theta_{c(i)}^{*}| \leq \Delta_{\mathrm{sep}}/4$ for all $i \in \mathcal{M}_{\mathrm{near}}$, we have
\begin{align}
    \sum_{i \in \mathcal{M}_{\mathrm{near}}} \pi_i |\theta_{i} - \theta_{c(i)}^{*}| \leq \dfrac{\Delta_{\mathrm{sep}}}{4} \sum_{i \in \mathcal{M}_{\mathrm{near}}} \pi_{i} \leq \dfrac{\Delta_{\mathrm{sep}}}{4}. \label{eq:exact_global_key_equation_eleventh}
\end{align}
Combining the bounds~\eqref{eq:exact_global_key_equation_tenth} and~\eqref{eq:exact_global_key_equation_eleventh} leads to
\begin{align}
    \sum_{i \in \mathcal{M}_{\mathrm{near}}} \pi_i |\theta_{i} - \theta_{c(i)}^{*}| & \leq \dfrac{\Delta_{\mathrm{sep}}}{4} \cdot \dfrac{C_{\mathrm{mass,2}}  \Delta_{\mathrm{sep}}^{-2k_*} d_H(p_{G},p_{G_*})}{\pi_{\text{min}}^{*}} \nonumber \\
    & = \dfrac{C_{\mathrm{mass,2}}  \Delta_{\mathrm{sep}}^{-(2k_* - 1)}d_H(p_{G},p_{G_*})}{4 \pi_{\text{min}}^{*}} \nonumber \\
    & = C_{\mathrm{near,mean,1}} (\pi_{\text{min}}^{*})^{-1}\Delta_{\mathrm{sep}}^{-(2k_* - 1)}d_H(p_{G},p_{G_*}). \label{eq:exact_global_key_equation_twelth}
\end{align}

\emph{Case 5.2.} There is no sub-Voronoi cell $\bar{\mathcal{V}}_{j}$ that has more than one element. It means that if $\bar{\mathcal{V}}_{j}$ is non-empty, it has only one element. If there exists some empty sub-Voronoi cell $\bar{\mathcal{V}}_{j}$, we can argue in a similar fashion to Case 1. Therefore, it is sufficient to assume that all the sub-Voronoi cells $\bar{\mathcal{V}}_{j}$ have exactly one element. Then, by using the similar arguments to the local bound along with the bound in equation~\eqref{eq:exact_key_equation_fourth_1}, we obtain
\begin{align*}
\sum_{i \in \bar{\mathcal{V}}_{j}} \pi_{i}|\theta_{i} - \theta_{j}^{*}| \leq C_{\mathrm{mean,1}}\Delta_{\mathrm{sep}}^{-(2k_*-1)} d_H(p_{G},p_{G_*}).
\end{align*}
As a consequence, 
\begin{align}
\sum_{i \in \mathcal{M}_{\mathrm{near}}} \pi_i |\theta_{i} - \theta_{c(i)}^{*}| &= \sum_{j = 1}^{k_{*}} \sum_{i \in \bar{\mathcal{V}}_{j}} \pi_{i}|\theta_{i} - \theta_{j}^{*}| \nonumber\\
& \leq k_{*}C_{\mathrm{mean,1}} \Delta_{\mathrm{sep}}^{-(2k_*-1)} d_H(p_{G},p_{G_*}) \nonumber \\
& = C_{\mathrm{near,mean,2}} \Delta_{\mathrm{sep}}^{-(2k_* - 1)}d_H(p_{G},p_{G_*})
\label{eq:dung_exact_global_key_equation_thirteen} 
\end{align}

\emph{Step 6 - Bounding $W_1(G,G_*)$ and conclusion.} 
Plugging all the bounds~\eqref{eq:exact_global_key_equation_third},~\eqref{eq:exact_global_key_equation_fourth},~\eqref{eq:exact_global_key_equation_ninth},~\eqref{eq:exact_global_key_equation_twelth}, and~\eqref{eq:dung_exact_global_key_equation_thirteen} into the upper bound of $W_{1}(G,G_{*})$ in equation~\eqref{eq:exact_global_key_equation_first}, we arrive at
\begin{align*}
W_1(G, G_*) & \le \left[C_{\mathrm{far,mean,1}}+ \max\left\{C_{\mathrm{mean,near,1}}, C_{\mathrm{mean,near,2}}\right\} + \frac{1}{2}k_*C_0C_{\mathrm{mass,2}} + \frac{1}{2}C_0C_{\mathrm{far,mass,1}} \right]\\
&\hspace{6cm}\times(\pi_{\text{min}}^{*})^{-1} \Delta_{\mathrm{sep}}^{-(2k_* - 1)} d_H(p_{G},p_{G_*}).
\end{align*}
Denote $$C_{\mathrm{global,1}}^{-1} = C_{\mathrm{far,mean,1}}+ \max\left\{C_{\mathrm{near,mean,1}}, C_{\mathrm{near,mean,2}},\right\} + \frac{1}{2}k_*C_0C_{\mathrm{mass,2}} + \frac{1}{2}C_0C_{\mathrm{far,mass,1}},$$
which is independent of $\Delta_{\mathrm{sep}}$ and $\pi_{\text{min}}^{*}$, revealing the final global bound:
$$ d_H(p_{G},p_{G_*}) \ge C_{\mathrm{global,1}} \cdot \pi_{\text{min}}^{*} \cdot \Delta_{\mathrm{sep}}^{2k_* - 1} \cdot W_{1}(G, G_*)$$
uniformly for all $G$ that has $k_{*}$ components. We achieve the conclusion of the theorem.
\subsubsection{Multivariate setting - Global bound}
\label{sec:global_bound_exact_multivariate_one_cluster}

As we consider test functions in the space $\mathcal{S}= \text{span} \{\theta_1, \dots, \theta_k\} \cup \{\theta_1^*, \dots, \theta_{k_*}^*\}$, by some rotation, we can suppose that $\mathcal{S} = \mathbb{R}^{d'}$ to be the subspace of $\mathbb{R}^d$ with zero value for lass $d-d'$ coordinate, and consider the testing function of $d'$ variables. Here, we have $d' \leq \min\{2k_*,d\}$.

For global bound, the first step, as in the univariate case, is to partition centers into scale set. For each fitted atom $\theta_i$, let $c(i) = \arg\min_j\|\theta_i-\theta^*_j\|$ be the index of true center $\theta_j^*$ which is closest to $\theta_i$ (when several points share the same minimum distance to $\theta_i$, we just randomly pick up one of them). 
\begin{itemize}
    \item The near set $\mathcal{M}_{\mathrm{near}} = \{i \in \{1,\ldots,k_*\}:\|\theta_i-\theta^*_{c(i)}\|\leq \Delta_{\mathrm{sep}}/4\}$, and we denote $\bar{\mathcal{V}}_j = \{i\in \mathcal{M}_{\mathrm{near}}:c(i) = j\}$ as a subset of the Voronoi cells $\mathcal{V}_j$. 
    \item The far set $\mathcal{M}_{\mathrm{far}} = \{i \in \{1,\ldots,k_*\}:\|\theta_i-\theta^*_{c(i)}\|>\Delta_{\mathrm{sep}}/4\}$, and we denote the total far mass as $\pi_{\mathcal{F}} = \sum_{i \in \mathcal{M}_{\mathrm{far}}}\pi_i$. 
\end{itemize}

\vspace{0.5 em}
\noindent
\emph{Step 1 - Wasserstein decomposition.}  We consider a deterministic spatial mapping function $T: \mathbb{R}^{d'} \to \Theta^* = \left\{\theta_{1}^{*}, \ldots, \theta_{k_{*}}^{*}\right\}$ such that $T(\theta_i) = \theta_{c(i)}^*$. We define $\widetilde{G}$ as the exact pushforward measure of $G$ under $T$:
$$ \widetilde{G} = T_{\#} G = \sum_{i=1}^{k_*} \pi_i \delta_{T(\theta_i)} = \sum_{i=1}^{k_*} \pi_i \delta_{\theta_{c(i)}^*}.$$
From the triangle inequality with the $W_{1}$ metric, we have
\begin{align}
    W_1(G, G_*) \leq W_1(G, \widetilde{G}) + W_{1}(\widetilde{G},G_*). \label{eq:dung_thm_3.1_multivariate_case_preliminary_estimation_of_W1}
\end{align}

We first upper bound $W_1(G, \widetilde{G})$. By specifically considering the transportation plan $\hat{\gamma} = (\text{Id} \times T)_{\#} G$, we obtain that:
\begin{align} 
W_1(G, \widetilde{G}) \le \int_{\mathbb{R}^{d'} \times \mathbb{R}^{d'}} \|x-y\| d\hat{\gamma}(x,y) & = \int_{\mathbb{R}^{d'}} \|\theta - T(\theta)\| dG(\theta) = \sum_{i=1}^k \pi_i \|\theta_i - \theta_{c(i)}^*\| \nonumber \\
& = \sum_{i \in \mathcal{M}_{\mathrm{near}}} \pi_i \|\theta_{i} - \theta_{c(i)}^{*}
\| + \sum_{i \in \mathcal{M}_{\mathrm{far}}}  \pi_i \|\theta_{i} - \theta_{c(i)}^{*}\|. \label{eq:dung_thm_3.1_multivariate_case_estimation_for_G_and_intermediate}
\end{align}

We now move to upper bound $W_1(\widetilde{G}, G_*)$. Let us rewrite $\widetilde{G}$ on the true discrete support basis by grouping all mass mapped to the identical center $j$ as follows: $$\widetilde{G} = \sum_{j=1}^{k_*} \tilde{\pi}_j \delta_{\theta_j^*} \quad \text{where} \quad \tilde{\pi}_j = \sum_{i: c(i) = j} \pi_i.$$
We consider the transporation plan from $\tilde{G}$ to $G_*$ such that for each $j$, we keep a mass of $\min\{\sum_{i\in \mathcal{V}_j}\pi_i,\pi^*_j\}$ for each $\theta^*_i$, while transferring $\max\{0,\tilde{\Delta}\pi_i\}$ to the other centers, where $\tilde{\Delta}\pi_i:= \tilde{\Delta}\pi_i-\pi_i^*$ . Also noting that 
\begin{align*}
    \tilde{\pi}_j = \sum_{i\in \bar{\mathcal{V}}_j} \pi_i + \sum_{i\in \mathcal{M}_{\mathrm{far}}, c(i)=j} \pi_i
\end{align*}
Define $\Delta\pi_j:= \left(\sum_{i\in \bar{\mathcal{V}}_j}\pi_i\right)-\pi_j^*$. Then, the total mass moved is exactly $\frac{1}{2}\sum_{j = 1}^{k_{*}} |\tilde{\Delta}\pi_j|$, and the distance between any true centers is $C_0\Delta_{\mathrm{sep}}$, thus 
\begin{align}
\nonumber
    W_1(\tilde{G},G_*) &\leq C_0 \Delta_{\mathrm{sep}} \sum_{j=1}^{k_*} |\tilde{\Delta}\pi_j| \leq C_0\Delta_{\mathrm{sep}}\left(\sum_{j=1}^{k_*}|\Delta \pi_j| + \sum_{j=1}^{k_*}\sum_{i\in \mathcal{M}_{\mathrm{far}}, c(i)=j} \pi_i\right)\\
    &\leq C_0\Delta_{\mathrm{sep}}\left(\sum_{j=1}^{k_*}|\Delta \pi_j| +\pi_{\mathcal{F}}\right)
\label{eqn:dung_thm3_1_inequality_W_1_true_measure_and_intermediate}
\end{align}

Plugging in the result from \eqref{eq:dung_thm_3.1_multivariate_case_estimation_for_G_and_intermediate} and \eqref{eqn:dung_thm3_1_inequality_W_1_true_measure_and_intermediate} into \eqref{eq:dung_thm_3.1_multivariate_case_preliminary_estimation_of_W1}, we have 
\begin{align}
    W_1(G,G_*) \leq \sum_{i \in \mathcal{M}_{\mathrm{near}}} \pi_i \|\theta_{i} - \theta_{c(i)}^{*}
\| + \sum_{i \in \mathcal{M}_{\mathrm{far}}}  \pi_i \|\theta_{i} - \theta_{c(i)}^{*}\| + C_0\Delta_{\mathrm{sep}}\left(\sum_{j=1}^{k_*}|\Delta \pi_j| +\pi_{\mathcal{F}}\right).
\label{eq:dung_thm_3.1_multivariate_case_estimation_for_W1_based_on_moment}
\end{align}

\vspace{0.5 em}
\noindent
\emph{Step 2 - Bounding near variance and far high moment for variance.} Now we bound the variance in near set and high moment for variance in far set. We use the test function in Section \ref{sec:variance_test_function}
\begin{equation*}
    P_{\mathrm{var}}(\theta) = \prod_{l=1}^{k_*} \|\theta - \theta^*_l\|_2^2. 
\end{equation*}
Using Lemma \ref{lemma:variance_test_function}, we have $\|P_{\mathrm{var}}\|_{\infty} \leq C_{\mathrm{var}}$. From Lemma \ref{lemma:Hellinger_to_Polynomial}, we achieve 
\begin{align}
\label{eqn:dung_thm_3_multivariate_variance_test_function}
\int P_{\mathrm{var}}(\theta)d\nu(\theta) \leq C_{\mathrm{poly}}C_{\mathrm{var}}d_H(p_G,p_{G_*}).
\end{align}
In addition, the variance test  polynomial satisfies $P_{\mathrm{var}}(\theta_l^*) = 0$, thus $\int P_{\mathrm{var}}(\theta)dG_*(\theta) = 0$. As a result, we have 
$$
\int P_{\mathrm{var}}(\theta)d\nu(\theta) = \underbrace{\sum_{i \in \mathcal{M}_{\mathrm{near}}}\pi_iP_{\mathrm{var}}(\theta_i)}_{\geq 0} + \underbrace{\sum_{i \in \mathcal{M}_{\mathrm{far}}} \pi_iP_{\mathrm{var}}(\theta_i)}_{\geq 0}.  
$$
Consider the case $i \in \mathcal{M}_{\mathrm{near}}$, in which 
$$P_{\mathrm{var}}(\theta_i) = \|\theta_i-\theta^*_{c(i)}\|^2\prod_{l\neq c(i)}\|\theta_i - \theta^*_{l}\|_2^2.$$
For any $j \neq c(i)$, we have 
\begin{equation*}
    \|\theta_i - \theta_j^*\|_2\geq  \|\theta^*_j - \theta_{c(i)}^*\|_2 - \|\theta_i - \theta_{c(i)}^*\|_2  \geq\Delta_{\mathrm{sep}} - \frac{\Delta_{\mathrm{sep}}}{4} = \frac{3\Delta_{\mathrm{sep}}}{4}. 
\end{equation*}
As a result, we have 
\begin{equation*}
    \int P_{\mathrm{var}}(\theta)d\nu(\theta)\geq \sum_{i \in \mathcal{M}_{\mathrm{near}}}\pi_iP_{\mathrm{var}}(\theta_i)   \geq  \left(\frac{3}{4}\right)^{2k_*-2} \Delta^{2k_*-2}_{\mathrm{sep}}\sum_{i \in \mathcal{M}_{\mathrm{near}}} \pi_i\|\theta_i - \theta^*_{c(i)}\|^2. 
\end{equation*}
Combining this result with \eqref{eqn:dung_thm_3_multivariate_variance_test_function}, we have 
\begin{align}
\sum_{i \in \mathcal{M}_{\mathrm{near}}} \pi_i\|\theta_i - \theta^*_{c(i)}\|^2&\leq \left[\frac{C_{\mathrm{var}}C_{\mathrm{poly}}}{(3/4)^{2k_*-2}}\right]\Delta_{\mathrm{sep}}^{-(2k_*-2)}d_H(p_G,p_{G_*})\nonumber\\
&:= C_{\mathrm{var},3} \Delta_{\mathrm{sep}}^{-(2k_*-2)}d_H(p_G,p_{G_*}).
\label{eqn:dung_thm_3.1_multivariate_near_variance_estimation}
\end{align}
For $i \in \mathcal{M}_{\mathrm{far}}$, we have $\|\theta_i - \theta^*_{c(i)}\|_2\leq \|\theta_i - \theta_j^*\|_2$ for any $j$, thus $P_{\mathrm{var}}(\theta_i) \geq \|\theta_i-\theta^*_{c(i)}\|_2^{2k_*}$. Thus, from estimation \eqref{eqn:dung_thm_3_multivariate_variance_test_function}, we have 
\begin{align}
\nonumber
    \sum_{i \in \mathcal{M}_{\mathrm{far}}}\pi_i\|\theta_i-\theta^*_{c(i)}\|^{2k_*} &\leq \sum_{i \in \mathcal{M}_{\mathrm{far}}}\pi_iP_{\mathrm{var}}(\theta_i)\\
    &\leq C_{\mathrm{var}}C_{\mathrm{poly}}d_H(p_G,p_{G_*}) := C_{\mathrm{far,moment,2}}d_H(p_G,p_{G_*}). 
\label{eqn:dung_thm3.1_multivariate_far_high_moment_estimation} 
\end{align}

\vspace{0.5 em}
\noindent
\emph{Step 3 - Bounding far parameter.} In this step, we bound the related far parameters.

\emph{Step 3.1 - Bounding far mean discrepancy.} For any index $i \in \mathcal{M}_{\mathrm{far}}$, we explicitly have the strict geometric condition $\|\theta_{i} - \theta_{c(i)}^{*}\| > \Delta_{\mathrm{sep}}/4$. In addition, we have $\|\theta_{i} - \theta_{c(i)}^{*}\| = \|\theta_{i} - \theta_{c(i)}^{*}\|^{2k_*} \|\theta_{i} - \theta_{c(i)}^{*}\|^{-(2k_*-1)} < \|\theta_{i} - \theta_{c(i)}^{*}\|^{2k_*} \left(\frac{4}{\Delta_{\mathrm{sep}}}\right)^{2k_*-1}$. Therefore, we find that
\begin{align}
    \sum_{i \in \mathcal{M}_{\mathrm{far}}} \pi_i \|\theta_{i} - \theta_{c(i)}^{*}\| &\le \left(\frac{4}{\Delta_{\mathrm{sep}}}\right)^{2k_*-1} \sum_{i \in \mathcal{M}_{\mathrm{far}}} \pi_i \|\theta_{i} - \theta_{c(i)}^{*}\|^{2k_*}\nonumber\\
    &\le  4^{2k_*-1} C_{\mathrm{far,moment,2}}  \Delta_{\mathrm{sep}}^{-(2k_*-1)} d_H(p_{G},p_{G_*})\nonumber\\
    &:= C_{\mathrm{far,mean,2}} \Delta_{\mathrm{sep}}^{-(2k_*-1)} d_H(p_{G},p_{G_*}).\label{eq:dung_thm3.1_multivariate_far_first_moment}
\end{align} 
\emph{Step 3.2 - Bounding far mass.} For any index $i \in \mathcal{M}_{\mathrm{far}}$, we have $$1 = \|\theta_{i} - \theta_{c(i)}^{*}\|^{2k_*} \|\theta_{i} - \theta_{c(i)}^{*}\|^{-2k_*} < \|\theta_{i} - \theta_{c(i)}^{*}\|^{2k_*} \left(\frac{4}{\Delta_{\mathrm{sep}}}\right)^{2k_*},$$ which follows that
\begin{align}
\nonumber
    \pi_{\mathcal{F}} &= \sum_{i \in \mathcal{M}_{\mathrm{far}}} \pi_i \le \left(\frac{4}{\Delta_{\mathrm{sep}}}\right)^{2k_*} \sum_{i \in \mathcal{M}_{\mathrm{far}}} \pi_i \|\theta_{i} - \theta_{c(i)}^{*}\|^{2k_*} \\
    &\le 4^{2k_*} C_{\mathrm{far,moment,2}} \Delta_{\mathrm{sep}}^{-2k_*} d_H(p_{G},p_{G_*}):= C_{\mathrm{far,mass,2}}\Delta_{\mathrm{sep}}^{-2k_*} d_H(p_{G},p_{G_*}). \label{eqn:dung_thm3_1_multivariate_far_zero_moment_bound}
\end{align} 
\emph{Step 4 - Bounding near mass. } To evaluate the total mass mismatch, we consider the test function $H_{j,j}$ such that $E_j(\theta_l^*) = \delta_{jl}$ and $\nabla E_j(\theta_l^*) = 0$. This polynomial can be chosen the one constructed in Section \ref{sec:point_wise_mass_extractor_non_multicluster}: 
\begin{equation*}
    E_j(\theta) : = \ell_j^2(\theta) \left[ \bar{A}_j + \langle\bar{B}_j,\theta - \theta_j^*\rangle\right], \quad \text{where} \quad \ell_j(\theta) = \prod_{q \neq j} \frac{\|\theta - \theta_q^*\|_2}{\|\theta_j^* - \theta_q^*\|_2},
\end{equation*}
and $\bar{A}_j = 1$, $\bar{B}_j = - 2\nabla\ell_j(\theta_j^*)$. Recall that this test function satisfies $E_j(\theta_l^*) = \delta_{jl}$ and $\nabla E_j(\theta_j^*) = 0$. Integrating $E_j$ yields:
\begin{align*}
    \int E_j(\theta)d\nu(\theta) = \sum_{i=1}^{k_*}\pi_i E_j(\theta_i) - \pi_j^* = \sum_{i \in \mathcal{M}_{\mathrm{near}}}\pi_i E_j(\theta_i) - \pi_j^*+ \sum_{i \in \mathcal{M}_{\mathrm{far}}}\pi_i E_j(\theta_i). 
\end{align*}
For $i\in\mathcal{M}_{\mathrm{near}}$, by performing the Taylor expansion with Lagrange remainder on $E_j(\theta_i)$ exactly around its true center $\theta_{c(i)}^*$, we have $$E_j(\theta_i) = \delta_{jc(i)} + \frac{1}{2} (\theta_{i} - \theta_{c(i)}^{*})^\top D^2E_j(\xi_i) (\theta_{i} - \theta_{c(i)}^{*}),$$ where $\xi_i$ is some point lying in the segment between $\theta_i$ and $\theta^*_{c(i)}$. 
Thus, we arrive at
\begin{align*}
    \int E_j(\theta) d\nu(\theta) &=\Delta \pi_j + \frac{1}{2}\sum_{i \in \mathcal{M}_{\mathrm{near}}} \pi_{i} (\theta_{i} - \theta_{c(i)}^{*})^{\top}D^2E_j(\xi_i) (\theta_{i} - \theta_{c(i)}^{*}) + \sum_{i \in \mathcal{M}_{\mathrm{far}}} \pi_i E_j(\theta_i).
\end{align*}
Applying the triangle inequality leads to
\begin{align}
    |\Delta \pi_j| \le \left| \int E_j(\theta) d\nu(\theta) \right| + \underbrace{ \frac{1}{2}\sum_{i \in \mathcal{M}_{\mathrm{near}}} \pi_{i} \|D^2E_j(\xi_i)\|_{\mathrm{op}} \|\theta_{i} - \theta_{c(i)}^{*}\|^2}_{\text{Near Error}} + \underbrace{ \sum_{i \in \mathcal{M}_{\mathrm{far}}} \pi_i |E_j(\theta_i)| }_{\text{Far Error}}. \label{eqn:dung_thm3.1_multivariate_mass_mismatch_prelim_estimation}
\end{align}
From the results of Lemma~\ref{lemma:Hellinger_to_Polynomial} and Lemma~\ref{lemma:mass_test_function}, we have
\begin{align*}
    \left| \int E_j(\theta) d\nu(\theta) \right| \le C_{\mathrm{norm,E}} \Delta_{\mathrm{sep}}^{-(2k_*-1)} d_H(p_{G},p_{G_*}).
\end{align*}
Since $\Delta_{\mathrm{sep}} \le 2R$, we factor $\Delta_{\mathrm{sep}}^{-(2k_*-1)} = \Delta_{\mathrm{sep}} \cdot \Delta_{\mathrm{sep}}^{-2k_*} \le 2R \Delta_{\mathrm{sep}}^{-2k_*}$, leading to
\begin{align}
\left| \int E_j(\theta) d\nu(\theta) \right| \le 2R C_{\mathrm{norm},E} \Delta_{\mathrm{sep}}^{-2k_*} d_H(p_{G},p_{G_*}). \label{eqn:dung_thm3_multivariate_integral_of_H_estimation}
\end{align}
Regarding the near error, Lemma~\ref{lemma:mass_test_function} indicates that it is bounded explicitly by $\frac{C_{E,1}}{2 \Delta_{\mathrm{sep}}^2} \sum_{i \in \mathcal{M}_{\mathrm{near}}} \pi_{i} \|\theta_{i} - \theta_{c(i)}^{*}\|^2$. Substituting the bound in equation~\eqref{eqn:dung_thm_3.1_multivariate_near_variance_estimation}, we obtain that
\begin{align}
\frac{1}{2}\sum_{i \in \mathcal{M}_{\mathrm{near}}}  \pi_i \|D^2E_j(\xi_i)\| \|\theta_{i} - \theta_{c(i)}^{*}\|^2 \leq \frac{1}{2} C_{E,1} C_{\mathrm{var},3} \Delta_{\mathrm{sep}}^{-2k_*} d_H(p_{G},p_{G_*}).\label{eqn:dung_thm3.1_near_variance_with_derivative_estimation}
\end{align}
Regarding the far error, for any index $i \in \mathcal{M}_{\mathrm{far}}$, the distance from $\theta_i$ to any arbitrary true center $\theta^{*}_{q}$ is bounded by $\|\theta_i - \theta_q^*\| \le \|\theta_i - \theta_{c(i)}^*\| + \|\theta_{c(i)}^* - \theta_q^*\| \le \|\theta_i - \theta_{c(i)}^*\| + C_0 \Delta_{\mathrm{sep}}$. Because $\Delta_{\mathrm{sep}} < 4\|\theta_i - \theta_{c(i)}^*\|$, we bound the distance as
\begin{align}
\label{eqn:dung_thm_3_1_multivariate_global_theta_distance_to_r_distance}
    \|\theta_i - \theta_q^*\| < \|\theta_i - \theta_{c(i)}^*\| + 4C_0 \|\theta_i - \theta_{c(i)}^*\| = (1+4C_0)\|\theta_i - \theta_{c(i)}^*\|.
\end{align}
The denominator of the polynomial function $\ell_j$ is precisely bounded below by $\Delta_{\mathrm{sep}}^{k_*-1}$. Substituting the root bound yields 
\begin{equation}
\label{eqn:dung_thm_3_1_multivariate_global_estimation_for_ell_square_for_far_case}
\ell_j^2(\theta_i) \le (1+4C_0)^{2k_*-2} \Delta_{\mathrm{sep}}^{-(2k_*-2)} \|\theta_i - \theta_{c(i)}^*\|^{2k_*-2}.
\end{equation}
From the argument in equation~\eqref{eq:bound_Bj}, $\|\bar{B}_j\| \le \Delta_{\mathrm{sep}}^{-1} C_B$. Because $\|\theta_i - \theta_{c(i)}^*\| > \Delta_{\mathrm{sep}}/4$, we have
\begin{align}
\nonumber
|\bar{A}_j + \langle \bar{B}_j, \theta_i - \theta_j^*\rangle| & \le 1 + \Delta^{-1}_{\mathrm{sep}} C_B(1+4C_0)\|\theta_i - \theta_{c(i)}^*\| \\
& < \Delta_{\mathrm{sep}}^{-1} \Big[ 4  + C_B(1+4C_0) \Big] \|\theta_i - \theta_{c(i)}^*\| = \Delta_{\mathrm{sep}}^{-1} C_{\mathrm{lin},2} \|\theta_i - \theta_{c(i)}^*\|.
\label{eqn:dung_thm3_1_multivariate_global_part_mononomial_factor_estimation}
\end{align}
Multiplying the bounded components, we have 
\begin{align}
\nonumber
    |E_j(\theta_i)| & \le \Big[ (1+4C_0)^{2k_*-2} \Delta_{\mathrm{sep}}^{-(2k_*-2)} \|\theta_i - \theta_{c(i)}^*\|^{2k_*-2} \Big] \Big[ C_{\mathrm{lin},2} \Delta_{\mathrm{sep}}^{-1} \|\theta_i - \theta_{c(i)}^*\| \Big] \\
    & := C_{\mathrm{far},E,2} \Delta_{\mathrm{sep}}^{-(2k_*-1)} \|\theta_i - \theta_{c(i)}^*\|^{2k_*-1}.
\label{eqn:dung_thm3_1_multivariate_global_estimation_of_function_far_case}
\end{align}
Therefore, we obtain from \eqref{eqn:dung_thm3.1_multivariate_far_high_moment_estimation} that
\begin{align}
\sum_{i \in \mathcal{M}_{\mathrm{far}}} \pi_i |E_j(\theta_i)| &\le C_{\mathrm{far},E,2} \Delta_{\mathrm{sep}}^{-(2k_*-1)} \sum_{i \in \mathcal{M}_{\mathrm{far}}} \pi_i \|\theta_i - \theta_{c(i)}^*\|^{2k_* - 1} \nonumber\\
& \leq 4 C_{\mathrm{far},E,2} \Delta_{\mathrm{sep}}^{-2k_*} \sum_{i \in \mathcal{M}_{\mathrm{far}}} \pi_i \|\theta_i - \theta_{c(i)}^*\|^{2k_*} \nonumber \\
& \leq 4 C_{\mathrm{far},E,2} C_{\mathrm{far,moment,2}} \Delta_{\mathrm{sep}}^{-2k_*} d_H(p_{G},p_{G_*}). \label{eqn:dung_thm3_1_multivariate_far_set_test_function_sum}
\end{align}
Plugging the upper bounds in equations~\eqref{eqn:dung_thm3_multivariate_integral_of_H_estimation}, \eqref{eqn:dung_thm3.1_near_variance_with_derivative_estimation}, \eqref{eqn:dung_thm3_1_multivariate_far_set_test_function_sum} to the inequality in equation~\eqref{eqn:dung_thm3.1_multivariate_mass_mismatch_prelim_estimation} leads to
\begin{align}
|\Delta \pi_j| & \le \big[ 2R C_{\mathrm{norm},E} + \frac{1}{2} C_{E,1} C_{\mathrm{var},3} + 4 C_{\mathrm{far},H,2} C_{\mathrm{far,moment,2}} \big] \Delta_{\mathrm{sep}}^{-2k_*} d_H(p_{G},p_{G_*}) \nonumber \\
& := C_{\mathrm{mass},3}\Delta_{\mathrm{sep}}^{-2k_*}d_H(p_{G},p_{G_*}).\label{eqn:dung_thm3_1_final_bound_for_mismatch_mass}
\end{align}

\vspace{0.5 em}
\noindent
\emph{Step 5 - Bounding near mass discrepancy. }
For this quantity, we employ similar argument as in the proof of global case for univariate setting. 
We divide our argument into two cases:

\emph{Case 5.1: } There exists a sub-Voronoi cell $\bar{\mathcal{V}}_{j}$ that has more than one element. Then, without loss of generality, we assume that the Voronoi cell $\mathcal{V}_{1}$ has no element. Under this case, it is clear that $|\Delta \pi_{1}| = \pi_{1}^{*}$. The bound in equation~\eqref{eqn:dung_thm3_1_final_bound_for_mismatch_mass} indicates that
\begin{align}
    \pi^*_{\min}\le \pi_{1}^{*} \leq C_{\mathrm{mass},3}\Delta_{\mathrm{sep}}^{-2k_*} d_H(p_{G},p_{G_*}). \label{eqn:dung_thm3_1_multivariate_exact_global_upper bound_for_pi_min_star}
\end{align}
As a result, given that $\|\theta_{i} - \theta_{c(i)}^{*}\| \leq \Delta_{\mathrm{sep}}/4$ for all $i \in \mathcal{M}_{\mathrm{near}}$, we have
\begin{align}
\nonumber
    \sum_{i \in \mathcal{M}_{\mathrm{near}}} \pi_i \|\theta_{i} - \theta_{c(i)}^{*}\| &\leq \dfrac{\Delta_{\mathrm{sep}}}{4} \sum_{i \in \mathcal{M}_{\mathrm{near}}} \pi_{i} \leq \dfrac{\Delta_{\mathrm{sep}}}{4}\leq  \dfrac{\Delta_{\mathrm{sep}}}{4}\frac{C_{\mathrm{mass},3}\Delta_{\mathrm{sep}}^{-2k_*} d_H(p_{G},p_{G_*})}{\pi^*_{\min}}\\
    &\leq \bar{C}_{\mathrm{mean,near,1}}(\pi^*_{\min})^{-1}\Delta_{\mathrm{sep}}^{-2k_*} d_H(p_{G},p_{G_*}). 
\label{eqn:dung_thm3.1_multivariate_first_case_for_near_first_moment_bound}
\end{align}

\emph{Case 5.2:} There is no sub-Voronoi cell $\bar{\mathcal{V}}_{j}$ that has more than one element. It means that if $\bar{\mathcal{V}}_{j}$ is non-empty, it has only one element. If there exists some empty sub-Voronoi cell $\bar{\mathcal{V}}_{j}$, we can argue in a similar fashion to Case 1. Therefore, it is sufficient to assume that all the sub-Voronoi cells $\bar{\mathcal{V}}_{j}$ have exactly one element. Without loss of generality, we suppose that each point $\theta_i \in \bar{\mathcal{V}}_{i}$. 

We use the mean extractor test function $H_{i,i}$ such that $H_{i,i}(\theta_l^*) = 0$ for all $l$, $\nabla H_{i,i}(\theta_l^*) = 0$ for all $l\neq i$, and $\langle \nabla H_{i,i}(\theta_i^*), \theta_i - \theta_i^*\rangle = \|\theta_i - \theta_i^*\|_2$. An eligible choice is the polynomial $H_{i,i}$ defined in Section \ref{sec:mean_test_function}, whose formula is 
\begin{align*}
    H_{j,i}(\theta) : = \ell_j^2(\theta) \left[\langle v_{j,i}, \theta - \theta_j^*\rangle\right],  \quad \text{where} \quad \ell_j(\theta) : = \prod_{q \neq j} \frac{\|\theta - \theta_q^*\|_2}{\|\theta_j^* - \theta_q^*\|_2}, 
\end{align*}
and $v_{j,i} = {(\theta_i-\theta_j^*)}/{\|\theta_i-\theta_j^*\|}$ if $\theta_i \neq \theta_j^*$ and $v_{j,i}=(1, 0,\ldots ,0)^\top$ otherwise. By applying Taylor expansion for 
$H_{i,i}(\theta_l)$ exactly around its true center $\theta_l^*$, we have 
\begin{align*}
H_{i,i}(\theta_l) &= \delta_{il}v_{i,i}^{\top}(\theta_{l} - \theta_{l}^{*}) + \frac{1}{2} (\theta_{l} - \theta_{l}^{*})^{\top}D^2H_{i,i}(\xi_l) (\theta_{l} - \theta_{l}^{*}) \\
&= \delta_{il}\|\theta_i-\theta^*_i\| + \frac{1}{2} (\theta_{l} - \theta_{l}^{*})^{\top}D^2H_{i,i}(\xi_l) (\theta_{l} - \theta_{l}^{*}).
\end{align*}
Thus, we arrive at $$ \int H_{i,i}(\theta) d\nu(\theta) = \pi_{i}\|\theta_{i} - \theta_{i}^{*}\| + \frac{1}{2} \sum_{l=1}^{k_*} \pi_{l} (\theta_{l} - \theta_{l}^{*})^{\top} D^2H_{i,i}(\xi_l)(\theta_{l} - \theta_{l}^{*}),$$
which indicates that
\begin{align*}
    \pi_{i}\|\theta_{i} - \theta_{i}^{*}\| &\leq \left|\int H_{i,i}(\theta) d\nu(\theta)\right| + \frac{1}{2} \sum_{l=1}^{k_*} \pi_{l} \|D^2H_{i,i}(\xi_l)\|_{\mathrm{op}} \|\theta_{l} - \theta_{l}^{*}\|_2^2\\
    &\leq C_{\mathrm{poly}}\|H_{i,i}\|_{\infty}d_H(p_{G},p_{G_*})+\frac{1}{2}\max_{\theta \in \mathcal{V}_{l}: \|\theta - \theta_{l}^{*}\|\leq \Delta_{\mathrm{sep}}/4} \|D^2H_{i,i}(\theta)\|_{\mathrm{op}}\sum_{l=1}^{k_*} \pi_{l}  \|\theta_{l} - \theta_{l}^{*}\|^2,
\end{align*}
where the second inequality follows from Lemma~\ref{lemma:Hellinger_to_Polynomial}. Using estimations in Lemma \ref{lemma:mean_test_function} and equation~\eqref{eqn:dung_thm_3.1_multivariate_near_variance_estimation}, we have 
\begin{align*}
\sum_{i \in \bar{\mathcal{V}}_{j}} \pi_{i}\|\theta_{i} - \theta_{i}^{*}\| \leq \left(2RC_{\mathrm{poly}}C_{\mathrm{norm},H} + \frac{1}{2} C_{\mathrm{var},3} C_{H,1}\right) \Delta_{\mathrm{sep}}^{-(2k_*-1)} d_H(p_{G},p_{G_*}).
\end{align*}
As a consequence, 
\begin{align}
\sum_{i \in \mathcal{M}_{\mathrm{near}}} \pi_i \|\theta_{i} - \theta_{c(i)}^{*}\| &= \sum_{j = 1}^{k_{*}} \sum_{i \in \bar{\mathcal{V}}_{j}} \pi_{i}\|\theta_{i} - \theta_{j}^{*}\| \nonumber\\
& \leq k_{*}\left(2C_{\mathrm{poly}}C_{\text{norm}}R + \frac{1}{2} C_{\mathrm{var}} C_{H,1}\right) \Delta_{\mathrm{sep}}^{-(2k_*-1)} d_H(p_{G},p_{G_*}) \nonumber \\
& = \bar{C}_{\mathrm{mean,near},2} \Delta_{\mathrm{sep}}^{-(2k_* - 1)}d_H(p_{G},p_{G_*})
\label{eqn:dung_thm3.1_multivariate_second_case_for_near_first_moment_bound} 
\end{align}
Combining the results from equations~\eqref{eqn:dung_thm3.1_multivariate_first_case_for_near_first_moment_bound} and \eqref{eqn:dung_thm3.1_multivariate_second_case_for_near_first_moment_bound}, we have 
\begin{align}
\label{eqn:dung_thm3.1_multivariate_final_bound_for_near_first_moment_bound} 
    \sum_{i \in \mathcal{M}_{\mathrm{near}}} \pi_i \|\theta_{i} - \theta_{c(i)}^{*}\| \leq \bar{C}_{\mathrm{mean,near}}(\pi^*_{\min})^{-1}\Delta_{\mathrm{sep}}^{-(2k_* - 1)}d_H(p_{G},p_{G_*}),
\end{align}
where $\bar{C}_{\mathrm{mean,near}} = \min\{\bar{C}_{\mathrm{mean,near},1},\bar{C}_{\mathrm{mean,near},2}\}$. 

\vspace{0.5 em}
\noindent
\emph{Step 6 - Bounding $W_1(G,G_*)$ and conclusion. } Now we put everything together. We substitute the results in equations~\eqref{eq:dung_thm3.1_multivariate_far_first_moment}, \eqref{eqn:dung_thm3_1_multivariate_far_zero_moment_bound}, \eqref{eqn:dung_thm3_1_final_bound_for_mismatch_mass}, and  \eqref{eqn:dung_thm3.1_multivariate_final_bound_for_near_first_moment_bound} back to equation~\eqref{eq:dung_thm_3.1_multivariate_case_estimation_for_W1_based_on_moment}, by defining 
\begin{equation*}
    \bar{C}_{\text{global,1}} = \left[\bar{C}_{\mathrm{mean,near}}+C_{\mathrm{far,mean,2}} + C_0(k_*\cdot C_{\mathrm{mass},3}+C_{\mathrm{far,mass,2}})\right]^{-1},
\end{equation*}
we have 
\begin{align*}
     d_{H}(p_G,p_{G_*})\geq \bar{C}_{\text{global,1}}\pi^*_{\min}\Delta_{\mathrm{sep}}^{2k_* - 1}W_1(G,G_*).
\end{align*}
This completes our proof. 

\subsection{Proof of Theorem~\ref{theorem:exact_multi_group}}
\label{sec:proof_theorem:exact_multi_group_univariate}

\subsubsection{Univariate Setting - Local Bound}
We begin by analyzing the local regime where the Wasserstein distance $W_1(G, G_*)$ is sufficiently small. According to Lemma \ref{lemma:small_wasserstein_implies_center_in_voronoi_cell}, each fitted atom $\theta_i$ lies strictly within a Voronoi cell $\mathcal{V}_j$ associated with its closest true center $\theta_j^*$ such that $|\theta_i - \theta_j^*| \le \frac{\Delta_{\mathrm{sep}}}{4}$. Because $G$ only contains $k_{*}$ atoms, each Voronoi cell $\mathcal{V}_{j}$ has exactly one element. Thus, WLOG, we may assume that $j \in \mathcal{V}_{j}$ for any $1 \leq j \leq k_{*}$. In addition, let $m(j)$ denote the index of the cluster containing the center $\theta^{*}_{j}$. Finally, we define $\Delta\Pi_m = \sum_{j \in \mathcal{C}_m} (\pi_{j} - \pi_{j}^{*})$, for any $1 \leq m \leq k_{0}$.  

\vspace{0.5 em}
\noindent
\emph{Step 1 - Wasserstein decomposition.} We define $G' = \sum_{j=1}^{k_*} \pi_j \delta_{\theta_j^*}$. By the triangle inequality with the $W_{1}$ metric, we have
$$ W_1(G, G_*) \le W_1(G, G') + W_1(G', G_*).$$

For $W_1(G, G')$, by choosing the transportation plan $\rho_{jj} = \pi_{j}$ for all $1 \leq j \leq k_{*}$ and $\rho_{ij} = 0$ as $1 \leq i \neq j \leq k_{*}$ we obtain that
\begin{align}
    W_1(G, G') \leq \sum_{j=1}^{k_*} \pi_{j}|\theta_{j} - \theta_{j}^{*}|. \label{eq:local_bound_Wasserstein_multi_group_1} 
\end{align}

To bound $W_1(G', G_*)$, we consider the following transportation plan from $G'$ to $G_*$
\begin{enumerate}
    \item For each center $\theta_j^*$, we keep $\min\{\pi_j,\pi^*_j\}$. Then, the remaining mass in each center for $G'$ is exactly $\max\{0, \Delta \pi_j\}$. 
    \item For each cluster $\mathcal{C}_m$, we move the remaining mass in each center to other center \textit{within} this cluster. Then, as the transportation distance is less than or equal to $C_0\Delta_{\mathrm{sep}}$, the total transportation cost is not exceed $\sum_{j=1}^k C_0\Delta_{\mathrm{sep}}|\Delta\pi_j|$, while the total remaining mass for $G'$ in each cluster $\mathcal{C}_m$ is exactly $\max\{0, \Delta \Pi_m\}$.
    \item Finally, we transport each remaining mass in each cluster $\max\{0, \Delta \Pi_m\}$ to other center in other cluster. Then, as the transportation distance is less than or equal to $2R$, the total transportation cost is not exceed $\sum_{m=1}^{k_0}R|\Delta\Pi_m|$. 
\end{enumerate}

Combing these estimations, we arrive at
\begin{align}
    W_1(G', G_*) \leq C_{0} \Delta_{\mathrm{sep}}\sum_{m=1}^{k_0} \sum_{j \in \mathcal{C}_m} |\pi_j - \pi_{j}^{*}| + 2R \sum_{m = 1}^{k_{0}} |\Delta \Pi_{m}|. \label{eq:local_bound_Wasserstein_multi_group_5}  
\end{align}

In light of the bounds on $W_{1}(G,G')$ in equation~\eqref{eq:local_bound_Wasserstein_multi_group_1} and on $W_{1}(G',G_{*})$ in equation~\eqref{eq:local_bound_Wasserstein_multi_group_5}, we eventually obtain that
\begin{align}
W_1(G, G_*) \le \sum_{j=1}^{k_*} \pi_{j}|\theta_{j} - \theta_{j}^{*}| + C_{0} \Delta_{\mathrm{sep}}\sum_{m=1}^{k_0} \sum_{j \in \mathcal{C}_m} |\pi_j - \pi_{j}^{*}| + 2R \sum_{m = 1}^{k_{0}} |\Delta \Pi_{m}|. \label{eq:exact_multi_group_key_equation_second}
\end{align}
\emph{Step 2 - Bounding variance.} In this step, we consider the following witness function defined in Section \ref{sec:variance_test_function}
\begin{align*}
    P_{\mathrm{var}}(\theta) = \prod_{l=1}^{k_*} (\theta - \theta_l^*)^2.
\end{align*} 
Using Lemma~\ref{lemma:Hellinger_to_Polynomial} and Lemma~\ref{lemma:variance_test_function} about the supremum norm of $P_{\mathrm{var}}$, we can bound the absolute value of integral of $P_{\mathrm{var}}$ as
\begin{align}
    \label{eqn:dung_thm_3_2_univariate_local_integral_bound_for_variance_test_function}
    \left|\int P_{\mathrm{var}}d\nu\right|  \leq  C_{\mathrm{poly}}\|P_{\mathrm{var}}\|_{\infty}d_{H}(p_G,p_{G_*})\leq C_{\mathrm{poly}}C_{\mathrm{var}}d_{H}(p_G,p_{G_*}). 
\end{align}

From the hypothesis, we have $|\theta_{j} - \theta_{j}^{*}| \leq \frac{\Delta_{\mathrm{sep}}}{4}$, for any $1 \leq j \leq k_{*}$. For other centers $\theta_{q}^{*}$ lying in the same cluster $\mathcal{C}_{m(j)}$ of $\theta_{j}^{*}$, by the reverse triangle inequality, $$|\theta_j - \theta_q^*| \ge |\theta_j^* - \theta_q^*| - |\theta_{j} - \theta_{j^{*}}| \ge \Delta_{\mathrm{sep}} - \frac{\Delta_{\mathrm{sep}}}{4} = \frac{3}{4}\Delta_{\mathrm{sep}}.$$
Additionally, the distance of $\theta_{j}$ to centers $\theta_{q}^{*}$ in external clusters, i.e., $\theta_{q}^{*} \notin \mathcal{C}_m(j)$, is lowered bounded as follows:
 $$|\theta_j - \theta_q^*| \ge |\theta_j^* - \theta_q^*| -|\theta_{j} - \theta_{j}^{*}| \ge D_0 - \frac{\Delta_{\mathrm{sep}}}{4}.$$
 As $\Delta_{\mathrm{sep}} \le \frac{D_0}{4C_0} \le \frac{D_0}{4}$, this bound becomes $|\theta_i - \theta_q^*| \geq \frac{15}{16} D_{0}$. Putting all these bounds together, we obtain that 
 \begin{align*}
     P_{\mathrm{var}}(\theta_j) & \ge  (\theta_{j} - \theta_{j}^{*})^2 \left( \frac{3}{4}\Delta_{\mathrm{sep}}\right)^{2s_{m(j)} - 2} \left( \frac{15}{16}D_0 \right)^{2k_* - 2s_{m(j)}}, 
 \end{align*} 
Therefore, we arrive at
\begin{align} 
\int P_{\mathrm{var}}(\theta) d\nu(\theta) &= \sum_{j=1}^{k_*} \pi_j P_{\mathrm{var}}(\theta_{j}) \geq \sum_{j = 1}^{k_{*}} \pi_{j} (\theta_{j} - \theta_{j}^{*})^2 \left( \frac{3}{4}\Delta_{\mathrm{sep}} \right)^{2s_{m(j)} - 2} \left( \frac{15}{16}D_0 \right)^{2k_* - 2s_{m(j)}}\nonumber\\
&\geq \min\{(2R)^{-2k_{*}},1\} \left(\frac{3}{4}\right)^{2k_*-2}\min\left\{\left( \frac{15}{16}D_0 \right)^{2k_* - 2},1 \right\}\Delta_{\mathrm{sep}}^{2s_{\max} - 2} \sum_{j = 1}^{k_{*}} \pi_{j} (\theta_{j} - \theta_{j}^{*})^2.
\label{eqn:dung_thm_3_2_P_var_greater_than_quadratic_for_near_case}
\end{align}
As a consequence, by letting $$C^{-1}_{\mathrm{var},1} = C^{-1}_{\mathrm{poly}}C^{-1}_{\mathrm{var}}\min\{(2R)^{-2k_{*}},1\} \left(\frac{3}{4}\right)^{2k_*-2}\min\left\{\left( \frac{15}{16}D_0 \right)^{2k_* - 2},1 \right\},$$
we can derive from estimation equation~\eqref{eqn:dung_thm_3_2_univariate_local_integral_bound_for_variance_test_function} that
\begin{align}
\label{eqn:dung_thm_3_2_univariate_local_variance_final_bound}
\sum_{j=1}^{k_*}  \pi_{j} (\theta_{j} - \theta_{j}^{*})^2 \leq  C_{\mathrm{var},1}\Delta_{\mathrm{sep}}^{-(2s_{\max} - 2)}d_{H}(p_G,p_{G_*}). 
\end{align}

\vspace{0.5 em}
\noindent
\emph{Step 3 - Bounding cluster-wise mass discrepancy.} Recall that
$\Delta\Pi_m = \sum_{j \in \mathcal{C}_m} \pi_j - \sum_{j \in \mathcal{C}_m} \pi_j^*$. For each $1 \leq m \leq k_{0}$, we choose the witness function $P_m(\theta)$ satisfying the condition that $ P_m(\theta_j^*) = \mathbf{1}_{\{j\in \mathcal{C}_m\}}$ and $P'_m(\theta_j^*) = 0$ for all $j \in [k_*]$. A possible choice is the polynomial $P_m$ defined in Section \ref{sec:cluster_wise_test_function}.
Now, we move to evaluate $\int_{\mathbb{R}} P_m(\theta) d\nu(\theta)$. Direct computation leads to
$$ \int P_m(\theta) d\nu(\theta) = \int P_m(\theta) dG(\theta) - \int P_m(\theta) dG_*(\theta).$$
From the formulation of the polynomial $P_{m}$, we have $P_m(\theta_l^*) = 1_{\left\{l \in \mathcal{C}_m\right\}}$, which implies that
$$\int P_m(\theta) dG_*(\theta) = \sum_{l=1}^{k_*} \pi_l^* 1_{\{l \in \mathcal{C}_m\}} = \sum_{l \in \mathcal{C}_m} \pi_l^* = \Pi_m^*.$$
For any $1 \leq j \leq k_{*}$, an application of Taylor expansion up to the second order to $P_m(\theta_j)$ around $\theta^*_j$ leads to
$$P_m(\theta_j) = P_m(\theta_j^*) + P_m'(\theta_j^*)(\theta_j - \theta_j^*) + \frac{1}{2} P_m''(\xi_j)(\theta_j - \theta_j^*)^2.$$
Since $P_m(\theta_j^*) = 1_{\left\{j \in \mathcal{C}_m\right\}}$ and $P_m'(\theta_j^*) = 0$, the above equation becomes 
$$ P_m(\theta_j) = 1_{\left\{j \in \mathcal{C}_m\right\}} + \frac{1}{2} P_m''(\xi_j) (\theta_j - \theta_j^*)^2.$$
Therefore, we obtain that
$$\int P_m(\theta) dG(\theta) = \sum_{j=1}^{k_*} \pi_j \left[ 1_{\left\{j \in \mathcal{C}_m\right\}} + \frac{1}{2} P_m''(\xi_j) (\theta_j-\theta_j^*)^2 \right] = \sum_{j \in \mathcal{C}_m} \pi_j + \frac{1}{2} \sum_{j=1}^{k_*} \pi_j P_m''(\xi_j) (\theta_j - \theta_j^*)^2.$$
Putting the above results together, we have
$$\int P_m(\theta) d\nu(\theta) = \Delta\Pi_m + \frac{1}{2} \sum_{j=1}^{k_*} \pi_j P_m''(\xi_j)(\theta_j - \theta_j^*)^2.$$
Hence, an application of triangle inequality leads to
\begin{align*}
    |\Delta\Pi_m| & \le \left| \int_{\mathbb{R}} P_m d\nu \right| + \frac{1}{2} \sum_{j=1}^{k_*} \pi_j |P_m''(\xi_j)| (\theta_j - \theta_j^*)^2  \\
    & \leq  C_{\mathrm{poly}}\|P_m\|_{\infty} d_H(p_{G},p_{G_*}) + \frac{1}{2} \sum_{j=1}^{k_*} \pi_j |P_m''(\xi_j)| (\theta_j - \theta_j^*)^2.
\end{align*}
An application of Lemma \ref{lemma:dung_bound_for_cluster_mass_extracting} leads to 
\begin{align*}
|\Delta\Pi_m|\leq  C_{\mathrm{poly}}C_{\mathrm{norm},P} d_H(p_{G},p_{G_*}) + \frac{1}{2}C_{P,2} \sum_{j=1}^{k_*} \pi_j (\theta_j - \theta_j^*)^2. 
\end{align*}

Since $1 \le (2R)^{2s_{\max}-2} \Delta_{\mathrm{sep}}^{-(2s_{\max}-2)} \leq \max\{1,(2R)^{2k_*-2}\}\Delta_{\mathrm{sep}}^{-(2s_{\max}-2)}$, using the variance estimation in equation~\eqref{eqn:dung_thm_3_2_univariate_local_variance_final_bound}, by setting 
$$C^{-1}_{\mathrm{cluster},1} = \max\{1,(2R)^{2k_*-2}\}C_{\mathrm{poly}}C_{\mathrm{norm},P}+ \frac{1}{2} C_{P,2}  C_{\mathrm{var},1},$$
we finally establish
\begin{align}
    |\Delta\Pi_m| \leq C_{\mathrm{cluster},1}\Delta_{\mathrm{sep}}^{-(2s_{\max}-2)} d_H(p_{G},p_{G_*}). \label{eq:exact_multi_group_key_equation_seventh}
\end{align}

\vspace{0.5 em}
\noindent
\emph{Step 4 - Bounding mass discrepancy.} We now move to upper bound $\sum_{j = 1}^{k_{*}} |\pi_{j} - \pi_{j}^{*}|$. We aim to construct $\bar{E}_j(\theta)$ such that $\bar{E}_j(\theta_l^*) = \delta_{jl}$ and $\bar{E}_j'(\theta_l^*) = 0$ for all $1\leq l\leq k_*$. An eligible choice is the polynomial $\bar{E}_j$ defined in Section \ref{sec:dung_mass_extractor_multicluster}
$$\bar{E}_{j}(\theta) = \ell_{j,\text{micro}}^2(\theta) [\bar{A}_j +  \bar{B}_j(\theta - \theta_{j}^{*})]P_{\text{macro}}(\theta),$$ 
where we define $\ell_{j,\text{micro}}(\theta) = \prod_{q \in \mathcal{C}_{m} \neq j} \frac{\theta - \theta_q^*}{\theta_j^* - \theta_q^*}$, $P_{\text{macro}}(\theta) = \prod_{p \neq m} \prod_{q \in \mathcal{C}_p} \left( \frac{\theta - \theta_q^*}{2R} \right)^{2s_{\max}}$, and $\bar{A}_j = 1/ P_{\text{macro}}(\theta_{j}^{*})$, and
\begin{align*}
    \bar{B}_j & = -\bar{A}_j[2\ell'_{j,\text{micro}}(\theta_j^*) + (P_{\text{macro}}(\theta^*_j))^{-1} P'_{\text{macro}}(\theta^*_j)]\\
    &= -\bar{A}_j \left(2\sum_{q \in \mathcal{C}_m \setminus \{j\}} \frac{1}{\theta_j^* - \theta_q^*}+ 2s_{\max}\sum_{p \neq m} \sum_{q \in \mathcal{C}_p} \frac{1}{\theta_j^* - \theta_q^*} \right).
\end{align*}

Assume without loss of generality that $j \in \mathcal{C}_{m}$, for some $1\leq m\leq k_0$. Then, direct calculation yields that $$\int \bar{E}_{j}(\theta) d\nu(\theta) = \int \bar{E}_{j}(\theta) dG(\theta) - \int \bar{E}_{j}(\theta) dG_*(\theta).$$
From the properties of the function $\bar{E}_{j}$, we have
$$\int \bar{E}_{j}(\theta) dG_*(\theta) = \sum_{l=1}^{k_*} \pi_l^* \bar{E}_{j}(\theta_l^*) = \sum_{l=1}^{k_*} \pi_l^* \delta_{jl} = \pi_j^*.$$
It leads to $\int \bar{E}_{j}(\theta) d\nu(\theta) = \sum_{l=1}^{k_*} \pi_l \bar{E}_{j}(\theta_l) - \pi_j^* = (\pi_j - \pi_j^*) + \sum_{l=1}^{k_*} \pi_l \bar{E}_{j}(\theta_l) - \pi_j$. It is equivalent to
\begin{align*}
(\pi_j - \pi_j^*)& = \int \bar{E}_{j}(\theta) d\nu(\theta) -  \left(\sum_{l=1}^{k_*} \pi_l \bar{E}_{j}(\theta_l) - \pi_j\right) = \int \bar{E}_{j}(\theta) d\nu(\theta) -  \sum_{l=1}^{k_*} \pi_l (\bar{E}_{j}(\theta_l) - \delta_{jl}) \\
& = \int \bar{E}_{j}(\theta) d\nu(\theta) -  \sum_{l=1}^{k_*} \pi_l (\bar{E}_{j}(\theta_l) - \bar{E}_{j}(\theta_l^{*})).
\end{align*}
The above equation can be rewritten as
\begin{align}
(\pi_j - \pi_j^*) = \int_{\mathbb{R}} \bar{E}_{j}(\theta) d\nu(\theta) -  \sum_{l \in \mathcal{C}_{m}} \pi_l (\bar{E}_{j}(\theta_l) - \bar{E}_{j}(\theta_l^{*})) - \sum_{l \notin \mathcal{C}_{m}} \pi_l (\bar{E}_{j}(\theta_l) - \bar{E}_{j}(\theta_l^{*})). \label{eq:exact_multi_group_key_equation_sixth_1}
\end{align}

For the integral term, by invoking Lemma \ref{lemma:Hellinger_to_Polynomial} about its connection with Hellinger distance and Lemma \ref{lemma:dung_multicluster_mass_test_function} about the supremum norm of $\bar{E}_j$, also noting that $\Delta^{-(2s_m-1)}_{\mathrm{sep}} \leq \Delta^{-(2s_{\max}-1)}_{\mathrm{sep}}(2R)^{2s_{\max}-2s_m} \leq \Delta^{-(2s_{\max}-1)}_{\mathrm{sep}}\max\{1,(2R)^{2k_*-2}\}$, we have 
\begin{align}
    \left|\int_{\mathbb{R}} \bar{E}_{j}(\theta) d\nu(\theta)\right|&\leq C_{\mathrm{poly}}\|\bar{E}_j\|_{\infty}d_H(p_{G},p_{G_*}) \leq C_{\mathrm{poly}}C_{\mathrm{norm},\bar{E}}\Delta^{-(2s_m-1)}_{\mathrm{sep}}d_H(p_{G},p_{G_*})\nonumber\\
    &\leq \max\{1,(2R)^{2k_*-2}\}C_{\mathrm{poly}}C_{\mathrm{norm},\bar{E}}\Delta^{-(2s_{\max}-1)}_{\mathrm{sep}}d_H(p_{G},p_{G_*}). 
    \label{eq:f_bar_moment}
\end{align}
We now proceed to bound the term $ \sum_{l \in \mathcal{C}_{m}} \pi_l (\bar{E}_{j}(\theta_l) - \bar{E}_{j}(\theta_l^{*}))$. By utilizing Taylor expansion for $\bar{E}_j(\theta_l)$ exactly around its true center $\theta_l^*$, we have $\bar{E}_j(\theta_l) = \bar{E}_j(\theta_l^{*}) + \frac{1}{2} \bar{E}_j''(\xi_l) (\theta_{l} - \theta_{l}^{*})^2$, for any $l \in \mathcal{C}_{m}$, implying that
$$\sum_{l \in \mathcal{C}_{m}} \pi_l (\bar{E}_{j}(\theta_l) - \bar{E}_{j}(\theta_l^{*})) = \sum_{l \in \mathcal{C}_{m}} \frac{\pi_l}{2} \bar{E}_j''(\xi_l) (\theta_{l} - \theta_{l}^{*})^2.$$
An application of the triangle inequality leads to
\begin{align*}
\left|\sum_{l \in \mathcal{C}_{m}} \pi_l (\bar{E}_{j}(\theta_l) - \bar{E}_{j}(\theta_l^{*}))\right| \leq \sum_{l \in \mathcal{C}_{m}} \pi_l |\bar{E}_{j}(\theta_l) - \bar{E}_{j}(\theta_l^{*})| \leq \sum_{l \in \mathcal{C}_{m}} \frac{\pi_l}{2} |\bar{E}_j''(\xi_l)| (\theta_{l} - \theta_{l}^{*})^2. 
\end{align*}

Using Lemma \ref{lemma:dung_multicluster_mass_test_function} for the bound of second derivative of $\bar{E}_j$ and the estimation in equation~\eqref{eqn:dung_thm_3_2_univariate_local_variance_final_bound} about the variance, we have 
\begin{align}
    \left|\sum_{l \in \mathcal{C}_{m}} \pi_l (\bar{E}_{j}(\theta_l) - \bar{E}_{j}(\theta_l^{*}))\right| &\leq \frac{1}{2}C_{\bar{E},2}\Delta^{-2}_{\mathrm{sep}}\sum_{l \in \mathcal{C}_m}\pi_l(\theta_l-\theta_l^*)^2 \nonumber\\
    &\leq \frac{1}{2}C_{\bar{E},2}C_{\mathrm{var},1}\Delta^{-2s_{\max}}_{\mathrm{sep}}d_H(p_G,p_{G_*}). 
    \label{eqn:dung_thm_3_2_univariate_local_sum_of_difference_E_inside_cluster}
\end{align}

We now move on deriving an upper bound for the term $|\sum_{l \notin \mathcal{C}_{m}} \pi_l (\bar{E}_{j}(\theta_l) - \bar{E}_{j}(\theta_l^{*}))|$. Recall that $\bar{E}_j(\theta_l^{*}) = 0$ for $l \notin \mathcal{C}_m$, using Lemma \ref{lemma:dung_multicluster_mass_test_function}, we have 
\begin{align*}
\left|\sum_{l \notin \mathcal{C}_{m}} \pi_l (\bar{E}_{j}(\theta_l) - \bar{E}_{j}(\theta_l^{*}))\right|  & \leq \sum_{l \notin \mathcal{C}_{m}} \pi_l |\bar{E}_{j}(\theta_l)|  \leq C_{\mathrm{cross},\bar{E}} \Delta_{\mathrm{sep}}^{-1} \sum_{l \notin \mathcal{C}_{m}} \pi_l (\theta_{l} - \theta_{l}^{*})^{2}. 
\end{align*}

Thus, it follows from estimation in equation~\eqref{eqn:dung_thm_3_2_univariate_local_variance_final_bound} that 
\begin{align}
    \left|\sum_{l \notin \mathcal{C}_{m}} \pi_l (\bar{E}_{j}(\theta_l) - \bar{E}_{j}(\theta_l^{*}))\right|\leq C_{\mathrm{cross},\bar{E}}C_{\mathrm{var},1}\Delta^{-(2s_{\max}-1)}_{\mathrm{sep}}d_H(p_G,p_{G_*}). 
\label{eqn:dung_thm_3_2_univariate_local_sum_of_difference_E_outside_cluster}
\end{align}
Combining the estimations in equations~\eqref{eq:f_bar_moment}, \eqref{eqn:dung_thm_3_2_univariate_local_sum_of_difference_E_inside_cluster}, and \eqref{eqn:dung_thm_3_2_univariate_local_sum_of_difference_E_outside_cluster} with equation~\eqref{eq:exact_multi_group_key_equation_sixth_1} yields
\begin{align}
\label{eqn:dung_thm_3_2_univariate_local_final_bound_for_mass_discrepancy}
    |\pi_j - \pi_j^*| \leq C_{\mathrm{mass},1}\Delta^{-2s_{\max}}_{\mathrm{sep}}d_H(p_G,p_{G_*}),
\end{align}
where 
\begin{equation*}
    C^{-1}_{\mathrm{mass},1} = \max\{2R,(2R)^{2k_*-1}\}C_{\mathrm{poly}}C_{\mathrm{norm},\bar{E}} +\frac{1}{2}C_{\bar{E},2}C_{\mathrm{var},1} + 2R C_{\mathrm{cross},\bar{E}}C_{\mathrm{var},1}. 
\end{equation*}
{
}

\vspace{0.5 em}
\noindent
\emph{Step 5: Bounding mean discrepancy.} To obtain an upper bound on $\sum_{j=1}^{k_*} \pi_j |\theta_{j} - \theta_{j}^{*}|$, we aim to construct $\bar{H}_j(\theta)$ such that $\bar{H}_j(\theta_l^*) = 0$ and $\bar{H}_j'(\theta_l^*) = \delta_{jl}$, for all $1\leq l\leq k_*$. A suitable choice for this is the polynomial $\bar{H}_j$ defined in Section \ref{sec:dung_mean_extractor_multicluster} as 
$$\bar{H}_{j}(\theta) = \ell_{j,\text{micro}}^2(\theta) [ \bar{v}_{j}(\theta - \theta_{j}^{*})]P_{\text{macro}}(\theta),$$ 
where we define $\ell_{j,\text{micro}}(\theta) = \prod_{q \in \mathcal{C}_{m} \neq j} \frac{\theta - \theta_q^*}{\theta_j^* - \theta_q^*}$, $P_{\text{macro}}(\theta) = \prod_{p \neq m} \prod_{q \in \mathcal{C}_p} \left( \frac{\theta - \theta_q^*}{2R} \right)^{2s_{\max}}$, and $$\bar
v_{j,i} = P^{-1}_{\text{macro}}(\theta_j^*),$$
when $\theta_i\neq \theta_j^*$ and $\bar{v}_{j} = P^{-1}_{\text{macro}}(\theta_j^*)$ otherwise.

Assume that $j \in \mathcal{C}_{m}$, then we have
$$\int \bar{H}_{j}(\theta) d\nu(\theta) = \int \bar{H}_{j}(\theta) dG(\theta) - \int \bar{H}_{j}(\theta) dG_*(\theta).$$
From the properties of the function $\bar{H}_{j}$, we have
$$\int_{\mathbb{R}} \bar{H}_{j}(\theta) dG_*(\theta) = \sum_{l=1}^{k_*} \pi_l^* \bar{H}_{j}(\theta_l^*) = 0.$$
It leads to
\begin{align*}
    \int \bar{H}_{j}(\theta) d\nu(\theta) & = \sum_{l=1}^{k_*} \pi_l \bar{H}_{j}(\theta_l) \\
    &= \sum_{l \in \mathcal{C}_{m}} \pi_{l}(\underbrace{\bar{H}_{j}(\theta_l^*)}_{=0} + \underbrace{\bar{H}_{j}'(\theta_l^*)}_{=\delta_{jl}}(\theta_l - \theta_l^*) + \frac{1}{2} \bar{H}_{j}''(\xi_l)(\theta_l - \theta_l^*)^2) + \sum_{l \notin \mathcal{C}_m} \pi_l \bar{H}_{j}(\theta_l) \\
    & = \pi_j(\theta_j - \theta_j^*) + \frac{1}{2} \sum_{l \in \mathcal{C}_m} \pi_l \bar{H}_{j}''(\xi_l) (
    \theta_{l} - \theta^*_{l})^2 + \sum_{l \notin \mathcal{C}_m} \pi_l \bar{H}_{j}(\theta_l).
\end{align*} 
Therefore, by means of the triangle inequality, we have
\begin{align}
\pi_j|\theta_j - \theta_j^*| \leq \left|\int_{\mathbb{R}} \bar{H}_{j}(\theta) d\nu(\theta)\right| + \frac{1}{2} \sum_{l \in \mathcal{C}_m} \pi_l |\bar{H}_{j}''(\xi_l)| (
    \theta_{l} - \theta_{l})^2 + \sum_{l \notin \mathcal{C}_m} \pi_l |\bar{H}_{j}(\theta_l)|,  \label{eq:exact_multi_group_key_equation_twevth}
\end{align}
For the integral term, an application of Lemma \ref{lemma:Hellinger_to_Polynomial} and Lemma \ref{lemma:dung_multicluster_mass_test_function} for the supremum norm of $\bar{H}_j$ gives us 
\begin{align}
    \left|\int \bar{H}_{j}(\theta) d\nu(\theta)\right|&\leq C_{\mathrm{poly}}\|\bar{H}_j\|_{\infty}d_H(p_{G},p_{G_*}) \leq C_{\mathrm{poly}}C_{\mathrm{norm},\bar{H}}\Delta^{-(2s_m-2)}_{\mathrm{sep}}d_H(p_{G},p_{G_*})\nonumber\\
    &\leq \max\{1,(2R)^{2k_*-2}\}C_{\mathrm{poly}}C_{\mathrm{norm},\bar{H}}\Delta^{-(2s_{\max}-2)}_{\mathrm{sep}}d_H(p_{G},p_{G_*}). 
    \label{eqn:dung_thm_3_2_univariate_local_mean_integral_bound}
\end{align}
where the last inequality holds thanks to $$1 \le (2R)^{2s_{\max}-2} \Delta_{\mathrm{sep}}^{-(2s_{\max}-2)} \leq \max\{1,(2R)^{2k_*-2}\}\Delta_{\mathrm{sep}}^{-(2s_{\max}-2)}.$$

We now consider the term $\frac{1}{2}\sum_{l \in \mathcal{C}_m}|\bar{H}''(\xi_l)|(\theta_l-\theta_l^*)^2$. As a consequence of Lemma \ref{lemma:dung_multicluster_mean_test_function} and the estimation for variance in equation~\eqref{eqn:dung_thm_3_2_univariate_local_variance_final_bound}, the following estimation can be established 
\begin{align}
\nonumber
\frac{1}{2}\sum_{l \in \mathcal{C}_m}\pi_l|\bar{H}''(\xi_l)|(\theta_l-\theta_l^*)^2 &\leq C_{\bar{H}}\Delta^{-1}_{\mathrm{sep}}\sum_{l \in \mathcal{C}_m}\pi_l(\theta_l-\theta_l^*)^2 \\
&\leq C_{\bar{H}}C_{\mathrm{var},1}\Delta^{-(2s_{\max}-1)}_{\mathrm{sep}}d_H(p_G,p_{G_*}).
\label{eqn:dung_thm_3_2_univariate_local_bound_for_sum_second_derivative_and_variance}
\end{align}
Let move to $\sum_{l \notin \mathcal{C}_m} \pi_l |\bar{H}_{j}(\theta_l)|$, from Lemma \ref{lemma:dung_multicluster_mass_test_function} and variance estimation in equation~\eqref{eqn:dung_thm_3_2_univariate_local_variance_final_bound}, we have 
\begin{align}
    \sum_{l \in \mathcal{C}_m} \pi_l|\bar{H}_j(\theta_l)| &\leq C_{\mathrm{cross},\bar{H}}\sum_{l \in \mathcal{C}_m}\pi_l(\theta_l - \theta^*_{l})^2\nonumber\\
    &\leq C_{\mathrm{cross},\bar{H}}C_{\mathrm{var},1}\Delta^{-(2s_{\max}-2)}_{\mathrm{sep}}d_H(p_{G},p_{G_*}). 
\label{eqn:dung_thm_3_2_univariate_local_bound_for_sum_of_H_for_j_notin_cluster}
\end{align}

Putting together the result from equations~\eqref{eqn:dung_thm_3_2_univariate_local_mean_integral_bound}, \eqref{eqn:dung_thm_3_2_univariate_local_bound_for_sum_second_derivative_and_variance}, \eqref{eqn:dung_thm_3_2_univariate_local_bound_for_sum_of_H_for_j_notin_cluster} and from the estimation in equation~\eqref{eq:exact_multi_group_key_equation_twevth}, we derive the bound 
\begin{align}
    \pi_j|\theta_j - \theta_j^*| \leq C_{\mathrm{mean},1}\Delta^{-(2s_{\max}-2)}_{\mathrm{sep}}d_H(p_{G},p_{G_*}),
\label{eqn:dung_thm_3_2_univariate_local_final_bound_for_mean_discrepancy} 
\end{align}
where 
\begin{equation*}
    C_{\mathrm{mean},1} = \max\{1,(2R)^{2k_*-2}\}C_{\mathrm{poly}}C_{\mathrm{norm},\bar{H}} + 2RC_{\bar{H}}C_{\mathrm{var},1} + C_{\mathrm{cross},\bar{H}}C_{\mathrm{var},1}. 
\end{equation*}

\vspace{0.5 em}
\noindent
\emph{Step 6 - Bounding $W_1(G,G_*)$ and conclusion.} Upon substituting  equations~\eqref{eq:exact_multi_group_key_equation_seventh}, \eqref{eqn:dung_thm_3_2_univariate_local_final_bound_for_mass_discrepancy}, and \eqref{eqn:dung_thm_3_2_univariate_local_final_bound_for_mean_discrepancy} into equation~\eqref{eq:exact_multi_group_key_equation_second}, we achieve 
\begin{equation*}
        d_H(p_{G},p_{G_*}) \geq C_{\mathrm{local},2}\cdot \Delta_{\mathrm{sep}}^{2(s_{\max}-1)}\cdot W_1(G,G_*),
\end{equation*}
where 
\begin{equation*}
    C^{-1}_{\mathrm{local},2} = k_02RC_{\mathrm{cluster},1} + k_*(C_0C_{\mathrm{mass},1} + 2RC_{\mathrm{mean},1}). 
\end{equation*}
    
This completes our proof for the local part. 
\subsubsection{Univariate Setting - Global Bound}
Throughout the proof, we denote $G = \sum_{i=1}^{k_{*}} \pi_i \delta_{\theta_i}$ to be a probability measure with exactly $k_{*}$ components. For every fitted atom $\theta_i$, let $c(i) = \text{argmin}_{1 \le j \le k_*} |\theta_i - \theta_j^*|$ be its absolute closest true center (while there exists several centers sharing the minimal distance, we randomly choose one of them). Let $m(i)$ denote the cluster containing the center $\theta^{*}_{c(i)}$. We unconditionally partition the $k_*$ atoms of $G$ into three disjoint sets based on exact spatial thresholds:
\begin{itemize}
    \item The micro near set $\mathcal{M}_{\mathrm{mic}} = \{ i : |\theta_i - \theta_{c(i)}^*| \le \frac{\Delta_{\mathrm{sep}}}{4} \}$. We define the explicit Voronoi cells strictly inside this core as $\bar{\mathcal{V}}_j = \{i \in \mathcal{M}_{\mathrm{mic}} : c(i) = j\}$. 
    \item The macro near set $\mathcal{M}_{\mathrm{mac}} = \{ i : \frac{\Delta_{\mathrm{sep}}}{4} < |\theta_i - \theta_{c(i)}^*| \le \frac{D_0}{4} \}$, which includes atoms that have escaped the micro near set but still associate with the cluster $\mathcal{C}_{m(i)}$. We denote the total macro near mass as $\pi_{{\mathrm{mac}}} = \sum_{i \in \mathcal{M}_{\mathrm{mac}}} \pi_i$.
    \item The far void set $\mathcal{M}_{\mathrm{far}} = \{ i : |\theta_i - \theta_{c(i)}^*| > \frac{D_0}{4} \}$, which consists of atoms in the void between distinct clusters. We denote the total far mass as $\pi_{\mathrm{far}} = \sum_{i \in \mathcal{M}_{\mathrm{far}}} \pi_i$.
\end{itemize}

\vspace{0.5 em}
\noindent
\emph{Step 1 - Wasserstein decomposition. } We first provide an upper bound on $W_{1}(G,G_{*})$. We define a deterministic spatial mapping function $T: \mathbb{R} \to \Theta^* = \left\{\theta_{1}^{*}, \ldots, \theta_{k_{*}}^{*}\right\}$ such that $T(\theta_i) = \theta_{c(i)}^*$. We define $\widetilde{G}$ as the exact pushforward measure of $G$ under $T$:
$$ \widetilde{G} = T_{\#} G = \sum_{i=1}^k \pi_i \delta_{T(\theta_i)} = \sum_{i=1}^k \pi_i \delta_{\theta_{c(i)}^*}.$$
From the triangle inequality with the $W_{1}$ metric, we have
\begin{align}
    W_1(G, G_*) \leq W_1(G, \widetilde{G}) + W_{1}(\widetilde{G},G_*). \label{eq:exact_Wasserstein_global_bound_multi_0}
\end{align}
We first upper bound $W_1(G, \widetilde{G})$. By specifically considering the transportation plan $\hat{\gamma} = (\text{Id} \times T)_{\#} G$, we obtain that:
\begin{align} 
W_1(G, \widetilde{G}) \le \int_{\mathbb{R} \times \mathbb{R}} |x - y| d\hat{\gamma}(x,y) & = \int_{\mathbb{R}} |\theta - T(\theta)| dG(\theta) = \sum_{i=1}^k \pi_i |\theta_i - \theta_{c(i)}^*| \nonumber \\
& = \sum_{i \in \mathcal{M}_{\mathrm{mic}}} \pi_i |\theta_{i} - \theta_{c(i)}^{*}| + \sum_{i \in \mathcal{M}_{\mathrm{mac}} \cup \mathcal{M}_{\mathrm{far}}}  \pi_i |\theta_{i} - \theta_{c(i)}^{*}|. \label{eq:exact_Wasserstein_global_bound_multi_1}
\end{align}
We now move to upper bound $W_1(\widetilde{G}, G_*)$. Firstly, we define some notation to be utilized later in the proof. Let $\tilde{\pi}_j = \sum_{i:c(i) = j}\pi_i$ to be the total mass of $\tilde{G}$ at $\theta_j^*$, then we can write 
$\tilde{G} = \sum_{j=1}^{k_*} \tilde{\pi}_j\delta_{\theta^*_{j}}$. Let the net mass discrepancy at center $j$ be exactly defined as $\tilde{\Delta} \pi_j = \tilde{\pi}_j - \pi_j^*$, while 
$\Delta\pi_j:= \left(\sum_{i\in \bar{\mathcal{V}}_j}\pi_i\right)-\pi_j^*$. While the cluster discrepancy $\Delta\Pi_m$ follows the analogous intuition as in univariate local situation, we need to redefine it as $\Delta\Pi_m = \sum_{j \in \mathcal{C}_{m}}\tilde{\Delta}\pi_{j}$. We also consider the transportation plan from $\widetilde{G}$ to $G_*$ as in univariate local part, but with a more careful handle with the transportation cost.
\begin{enumerate}
    \item For each center $\theta_j^*$, we keep $\min\{\tilde{\pi}_j,\pi^*_j\}$. Then, the remaining mass in each center for $G'$ is exactly $\max\{0, \tilde{\Delta} \pi_j\}$. 
    \item For each cluster $\mathcal{C}_m$, we move the remaining mass in each center to other center \textit{within} this cluster. Then, as the transportation distance is less than or equal to $C_0\Delta_{\mathrm{sep}}$, the total transportation cost is not exceed $\sum_{j=1}^k C_0\Delta_{\mathrm{sep}}|\Delta\pi_j|$, while the total remaining mass for $G'$ in each cluster $\mathcal{C}_m$ is exactly $\max\{0, \Delta \Pi_m\}$.
    \item Finally, we transport each remaining mass in each cluster $\max\{0, \Delta \Pi_m\}$ to other center in other cluster. Then, as the transportation distance is less than or equal to $2R$, the total transportation cost is not exceed $\sum_{m=1}^{k_0}R|\Delta\Pi_m|$. 
\end{enumerate}

Let us rewrite $\widetilde{G}$ on the true discrete support basis by grouping all mass mapped to the identical center $j$ as follows: $$\widetilde{G} = \sum_{j=1}^{k_*} \tilde{\pi}_j \delta_{\theta_j^*} \quad \text{where} \quad \tilde{\pi}_j = \sum_{i: c(i) = j} \pi_i.$$

Combing these estimations, we arrive at
\begin{align}
    W_1(G', G_*) \leq C_{0} \Delta_{\mathrm{sep}}\sum_{m=1}^{k_0} \sum_{j \in \mathcal{C}_m} |\tilde{\Delta}\pi_j| + 2R \sum_{m = 1}^{k_{0}} |\Delta \Pi_{m}|. 
\label{eqn:dung_thm_3_2_univariate_global_W_1_G'_G_star}
\end{align}
In addition, we also note that $$\tilde{\Delta}\pi_j = \Delta\pi_j + \sum_{i \notin \mathcal{M}_{\mathrm{mic}}, \, c(i)=j} \pi_i.$$
Therefore, we have $|\tilde{\Delta}\pi_j| \leq |\Delta\pi_j| + \sum_{i \notin \mathcal{M}_{\mathrm{mic}}, \, c(i)=j} \pi_i$. Plugging this inequality into the bound \eqref{eqn:dung_thm_3_2_univariate_global_W_1_G'_G_star} and using the bound \eqref{eq:exact_Wasserstein_global_bound_multi_1}, from the estimation in equation~\eqref{eq:exact_Wasserstein_global_bound_multi_0}, we have 
\begin{align}
\nonumber
    W_1(G, G_*) &\le \sum_{i \in \mathcal{M}_{\mathrm{mic}}} \pi_i|\theta_{i}-\theta_{c(i)}| \ + \sum_{i \in \mathcal{M}_{\mathrm{mac}} \cup \mathcal{M}_{\mathrm{far}}} \pi_i|\theta_{i}-\theta_{c(i)}|+C_0 \Delta_{\mathrm{sep}} \sum_{j=1}^{k_*} |\Delta \pi_j|\\
    &\hspace{4cm} + C_0 \Delta_{\mathrm{sep}} \big(\pi_{\mathcal{M}_{\mathrm{mac}}} + \pi_{\mathcal{M}_{\mathrm{far}}}\big) +2R \sum_{m=1}^{k_0} |\Delta \Pi_m|.
    \label{eqn:dung_thm2_global_W_1_expansion}
\end{align}

\vspace{0.5 em}
\noindent
\emph{Step 2 - Bounding variance. } Similar to the proof of the local bound, we consider the witness function $$P_{\mathrm{var}}(\theta) = \prod_{l=1}^{k_*} (\theta - \theta_l^*)^2.$$ According to Lemma~\ref{lemma:variance_test_function}, the supremum norm of $P_{\mathrm{var}}$ can be upper bounded as follows:
$$ \|P_{\mathrm{var}}\|_{\infty} \le C_{\mathrm{var}}:=(2R)^{2k_*}.$$
Since $\int P_{\mathrm{var}}(\theta) dG_*(\theta) = 0$, Lemma~\ref{lemma:Hellinger_to_Polynomial} indicates that 
\begin{align}
\label{eqn:dung_first_estimation_about_f_var_multicluster}
\left|\int P_{\mathrm{var}}(\theta) dG(\theta)\right| = \left|\int P_{\mathrm{var}}(\theta) d\nu(\theta)\right| \le C_{\mathrm{poly}}C_{\mathrm{var}} d_H(p_{G},p_{G_*}).
\end{align}
We divide the integral into three components
\begin{align}
    \int P_{\mathrm{var}}(\theta) dG(\theta) = \sum_{i \in \mathcal{M}_{\mathrm{mic}}} \pi_i P_{\mathrm{var}}(\theta_i) + \sum_{i \in \mathcal{M}_{\mathrm{mac}}} \pi_i P_{\mathrm{var}}(\theta_i) + \sum_{i \in \mathcal{M}_{\mathrm{far}}} \pi_i P_{\mathrm{var}}(\theta_i). \label{eq:exact_global_multi_key_equation_first_1}
\end{align}
Now we estimate the value of $P_{\mathrm{var}}(\theta_i)$, for $i$ in each set $\mathcal{M}_{\mathrm{mic}}$, $\mathcal{M}_{\mathrm{mac}}$, and $\mathcal{M}_{\mathrm{far}}$. 

\emph{Case 2.1 - In micro-near set ($\mathcal{M}_{\mathrm{mic}}$). } For $i \in \mathcal{M}_{\mathrm{mic}}$,
we define $C_{\mathrm{mic},1} := \left(\frac{3}{4}\right)^{2k_*-2}\cdot \min\left\{\left(\frac{15}{16}D_0\right)^{2k_*-2}, 1\right\}\cdot \min\{(2R)^{-2k_*},1\}$, from equation \eqref{eqn:dung_thm_3_2_P_var_greater_than_quadratic_for_near_case}, we secure
\begin{align}
    P_{\mathrm{var}}(\theta_i) \ge C_{\mathrm{mic,1}} \cdot \Delta_{\mathrm{sep}}^{2s_{\max}-2} |\theta_i-\theta^*_{c(i)}|^2, \quad \text{for all } i \in \mathcal{M}_{\mathrm{mic}}. \label{eqn:dung_thm_3_2_univariate_global_P_var_greater_than_second_moment_micro_case}
\end{align}

\emph{Case 2.2 - In macro-near set ($\mathcal{M}_{\mathrm{mac}}$).} For $i \in \mathcal{M}_{\mathrm{mac}}$, from the definition of $\mathcal{M}_{\mathrm{mac}}$, we have $\frac{\Delta_{\mathrm{sep}}}{4} < |\theta_i - \theta^*_{j}| \le \frac{D_0}{4}$. For any center indexed by $q \in \mathcal{C}_m$ within $\theta_i$'s cluster, because $j$ is the absolute closest center globally, the Voronoi property guarantees that $|\theta_i - \theta_q^*| \ge |\theta_i - \theta_j^*|\geq\frac{\Delta_{\mathrm{sep}}}{4}$. For any center indexed by $p \notin \mathcal{C}_m$ outside $\theta_i$'s cluster, we have $|\theta_i - \theta_p^*| \ge |\theta_j^* - \theta_p^*| - |\theta_i-\theta_j^*| \ge D_0 - \frac{D_0}{4} = \frac{3}{4}D_0$. Multiplying these bounds with a note that $|\theta_i-\theta_j^*| \leq 2R$, we have
\begin{align*}
    P_{\mathrm{var}}(\theta_i) &\ge  |\theta_i - \theta_{c(i)}^*|^{2}\left(\dfrac{\Delta_{\mathrm{sep}}}{4}\right)^{2s_{m}-2} \left(\frac{3}{4}D_0\right)^{2k_* - 2s_m}\\
    &\geq  |\theta_i -\theta_{c(i)}^*|^{2}\left(\dfrac{1}{4}\right)^{2s_m-2}\Delta_{\mathrm{sep}}^{2s_{\max}-2} \left(\frac{3}{4}D_0\right)^{2k_* - 2s_m}(2R)^{-2(s_{\max}-s_m)}.
\end{align*}
In addition, we also have 
\begin{align*}
\left(\frac{1}{4}\right)^{2s_m-2}\left(\frac{3}{4}D_0\right)^{2k_* - 2s_m}(2R)^{-2(s_{\max}-s_m)} &\geq 4^{-2(k_*-1)}\min\left\{\left(\frac{3}{4}D_0\right)^{2k_*-2},1\right\} \cdot\min\{(2R)^{-2k_*},1\}\\
&:=C_{\mathrm{mac,1}}. 
\end{align*}
As a consequence, we achieve the following estimation for $P_{\mathrm{var}}(\theta_i)$
\begin{align}
    P_{\mathrm{var}}(\theta_i) \geq C_{\mathrm{mac},1}\cdot\Delta_{\mathrm{sep}}^{2s_{\max}-2}|\theta_i - \theta_{c(i)}^*|^{2}\quad \text{for all } i \in \mathcal{M}_{\mathrm{mac}}.
\label{eqn:dung_thm_3_2_univariate_global_P_var_greater_than_second_moment_macro_case}
\end{align}

\emph{Case 2.3 - In far void set ($\mathcal{M}_{\mathrm{far}}$)}. For $i \in \mathcal{M}_{\mathrm{far}}$, we have $|\theta_i-\theta^*_{c(i)}| > \frac{D_0}{4}$. Because $c(i)$ is the absolute closest center of $i$, it follows that $|\theta_i - \theta_l^*| \ge |\theta_i-\theta^*_{c(i)}| > \frac{D_0}{4}$ for all $k_*$ centers $\theta^*_l$. Thus, as $\Delta_{\mathrm{sep}} \leq 2R$, we have 
\begin{align}    
P_{\mathrm{var}}(\theta_i) &\ge  |\theta_i-\theta^*_{c(i)}|^{2k_*} \geq  |\theta_i-\theta^*_{c(i)}|^2\left(\frac{D_0}{4}\right)^{2k_*-2} \geq \left(\frac{D_0}{4}\right)^{2k_*-2s_{\max}}\Delta^{2s_{\max}-2}_{\mathrm{sep}}|\theta_i-\theta^*_{c(i)}|^2\nonumber\\
&\geq C_{\mathrm{far},1} \Delta^{2s_{\max}-2}_{\mathrm{sep}}|\theta_i-\theta_{c(i)}|^2,
\label{eqn:dung_thm_3_2_univariate_global_P_var_greater_than_second_moment_far_case}
\end{align}
where $C_{\mathrm{far},1}:= \min\left\{\left(\frac{D_0}{4}\right)^{2k_*-2},1\right\}$.

Plugging these bounds of $P_{\mathrm{var}}(\theta_i)$ from equations~\eqref{eqn:dung_thm_3_2_univariate_global_P_var_greater_than_second_moment_micro_case}, \eqref{eqn:dung_thm_3_2_univariate_global_P_var_greater_than_second_moment_macro_case}, and \eqref{eqn:dung_thm_3_2_univariate_global_P_var_greater_than_second_moment_far_case} into equation \eqref{eq:exact_global_multi_key_equation_first_1}, we have
\begin{equation*}
    \int P_{\mathrm{var}}(\theta)dG(\theta) \geq \min\{C_{\mathrm{mic},1},C_{\mathrm{mac},1},C_{\mathrm{far},1}\}\cdot \Delta_{\mathrm{sep}}^{2s_{\max}-2}\sum_{i\in\mathcal{M}_{\mathrm{mic}} \cup \mathcal{M}_{\mathrm{mac}}\cup \mathcal{M}_{\mathrm{far}}} \pi_i |\theta_i-\theta_{c(i)}|^2.
\end{equation*}
Combining this result with equation \eqref{eqn:dung_first_estimation_about_f_var_multicluster}, we have 
\begin{equation}
\label{eqn:dung_thm2_global_variation_bound}
    \sum_{i=1}^{k_*} \pi_i |\theta_i-\theta^*_{c(i)}|^2 \leq C_{\mathrm{var},2}\Delta_{\mathrm{sep}}^{-(2s_{\max}-2)}d_H(p_{G},p_{G_*}), 
\end{equation}
where $C_{\mathrm{var},2} = C_{\mathrm{poly}}C_{\mathrm{var}}\cdot \max\{C^{-1}_{\mathrm{mic},1},C^{-1}_{\mathrm{mac},1},C^{-1}_{\mathrm{far},1}\}$.

\emph{Step 3 - Bounding macro- and far-related quantities. } We now bound the quantities related to $\mathcal{M}_{\mathrm{mac}}$ and $\mathcal{M}_{\mathrm{far}}$ set. 

\emph{Step 3.1 - Bounding macro high-moment $\sum_{i\in\mathcal{M}_{\mathrm{mac}}} \pi_i |\theta_i-\theta^*_{c(i)}|^{2s_{\max}}$: } Recall that from the estimation of $P_{\mathrm{var}}(\theta_i)$ ($i \in \mathcal{M}_{\mathrm{mac}}$), we have 
\begin{align*}
    P_{\mathrm{var}}(\theta_i) &\geq |\theta_i - \theta^*_{c(i)}|^{2s_m}\left(\frac{3}{4}D_0\right)^{2k_*-2s_m}\geq \min\left\{1,\left(\frac{3}{4}D_0\right)^{2k_*}\right\}|\theta_i - \theta^*_{c(i)}|^{2s_m}
\end{align*}
Moreover, since $|\theta_i - \theta^*_{c(i)}|\leq 2R$, we can lower bound $|\theta_i - \theta^*_{c(i)}|^{2s_m}$ as 
\begin{align*}
    |\theta_i - \theta^*_{c(i)}|^{2s_m} &= |\theta_i - \theta^*_{c(i)}|^{2s_{\max}}|\theta_i - \theta^*_{c(i)}|^{2(s_m-s_{\max})}\\
    &\geq \min\{(2R)^{-2(k_*-1)},1\}|\theta_i - \theta^*_{c(i)}|^{2s_{\max}}.
\end{align*}
Consequently, by defining 
$$C_{\mathrm{mac,moment,1}} = \max\{(2R)^{2(k_*-1)},1\}\cdot\max\left\{1,\left(\frac{3}{4}D_0\right)^{-2k_*}\right\}C_{\mathrm{poly}}C_{\mathrm{var}}$$
we achieve the following bound 
\begin{equation}
\label{eqn:macro_high_moment_estimate}
    \sum_{i\in \mathcal{M}_{\mathrm{mac}}} \pi_i |\theta_i-\theta^*_{c(i)}|^{2s_{\max}} \leq C_{\mathrm{mac,moment,1}} \cdot d_H(p_{G},p_{G_*}).
\end{equation}

\emph{Step 3.2 - Bounding far mass $\pi_{\mathrm{far}}$. } Recall that for $i\in\mathcal{M}_{\mathrm{far}}$, we have $|\theta_i - \theta_l^*|\geq |\theta_i-\theta^*_{c(i)}| > \frac{D_0}{4}$ for all $1\leq l\leq k_*$, which leads to $P_{\mathrm{var}}(\theta_i) \geq \left(\frac{D_0}{4}\right)^{2k_*}$. As a consequence of estimation in equation~\eqref{eqn:dung_first_estimation_about_f_var_multicluster}, we obtain 
\begin{align}
\nonumber
    \pi_{\mathrm{far}} = \sum_{i \in \mathcal{M}_{\mathrm{far}}} \pi_i &\leq \sum_{i \in \mathcal{M}_{\mathrm{far}}}\left(\frac{D_0}{4}\right)^{-2k_*} \pi_iP_{\mathrm{var}}(\theta_i) \leq  \left(\frac{D_0}{4}\right)^{-2k_*} \cdot C_{\mathrm{poly}}C_{\mathrm{var}} d_H(p_{G},p_{G_*}) \\
    &\leq C_{\mathrm{far,mass,2}} \cdot d_H(p_{G},p_{G_*}),
    \label{eqn:dung_thm2_global_far_mass}
\end{align}
where $C_{\mathrm{far,mass,2}} = \left(\frac{D_0}{4}\right)^{-2k_*} \cdot C_{\mathrm{poly}}C_{\mathrm{var}}$.

\emph{Step 3.3 - Bounding macro mass $\pi_{\mathrm{macro}}$.} For any $i \in \mathcal{M}_{\mathrm{mac}}$, since $|\theta_i-\theta^*_{c(i)}| > \Delta_{\mathrm{sep}}/4$, equation \eqref{eqn:macro_high_moment_estimate} implies
\begin{align}
\pi_{\mathrm{mac}} = \sum_{\mathcal{M}_{\mathrm{mac}}} \pi_i &\le \left( \frac{4}{\Delta_{\mathrm{sep}}} \right)^{2s_{\max}} \sum_{\mathcal{M}_{\mathrm{mac}}} \pi_i |\theta_i-\theta^*_{c(i)}|^{2s_{\max}}\nonumber \\
    &\le C_{\mathrm{mac,mass,1}} \Delta_{\mathrm{sep}}^{-2s_{\max}} d_H(p_{G},p_{G_*}),
    \label{eqn:dung_sum_pi_M_mac}
\end{align}
where $C_{\mathrm{mac,mass,1}} := 4^{2k_*}C_{\mathrm{mac,moment,1}}$. 

\emph{Step 4 - Bounding cluster mass discrepancy.} Recall that in univariate global case, 
$\Delta\Pi_m = \sum_{j \in \mathcal{C}_m} \tilde{\Delta}\pi_j$. As in the proof of univariate local case, for each $1 \leq m \leq k_{0}$, we choose the witness function $P_m(\theta)$ satisfying the condition that $ P_m(\theta_j^*) = \mathbf{1}_{\{j\in \mathcal{C}_m\}}$ and $P'_m(\theta_j^*) = 0$ for all $j \in [k_*]$. A possible choice guaranteeing these conditions is the polynomial $P_m$ defined in Section \ref{sec:cluster_wise_test_function}.
By the definition of $P_{m}$, the integration over $\nu$ gives us 
\begin{align*}
    \int P_m(\theta)d\nu(\theta) = {\int P_m(\theta)dG(\theta)} - \int P_m(\theta)dG_*(\theta).
\end{align*}
From the formulation of $P_m$, we have $P_m(\theta_l^*) = \mathbf{1}_{\{l \in \mathcal{C}_m\}}$, which implies 
\begin{equation*}
    \int P_mdG_* = \sum_{l=1}^{k_*}\pi^*_{l}\mathbf{1}_{\{l\in\mathcal{C}_m\}} = \sum_{l\in\mathcal{C}_m} \pi^*_l. 
\end{equation*}
For every $i \in \mathcal{M}_{\mathrm{mic}} \cup \mathcal{M}_{\mathrm{mac}}$, we consider the Taylor expansion around its absolute closest true center $\theta_{c(i)}^*$, which implies that there exists $\xi_i$ between $\theta_i$ and  $\theta^*_{c(i)}$ 
\begin{align*}
P_m(\theta_i) &= P_m(\theta^*_{c(i)}) + \underbrace{P'_m(\theta^*_{c(i)})}_{=0}(\theta_{i}-\theta^*_{c(i)}) + \frac{1}{2} P_m''(\xi_i) (\theta_{i}-\theta^*_{c(i)})^2\\
&= \boldsymbol{1}_{\{c(i) \in \mathcal{C}_m\}}+ \frac{1}{2} P_m''(\xi_i) (\theta_{i}-\theta^*_{c(i)})^2. 
\end{align*}
Therefore,
\begin{align*} \int P_m dG &= \sum_{i \in \mathcal{M}_{\mathrm{mic}} \cup \mathcal{M}_{\mathrm{mac}}} \pi_i \mathbf{1}_{\{c(i) \in \mathcal{C}_m\}} + \frac{1}{2} \sum_{i \in \mathcal{M}_{\mathrm{mic}} \cup \mathcal{M}_{\mathrm{mac}}} \pi_i P_m''(\xi_i)|\theta_i-\theta^*_{c(i)}|^2 \ + \ \sum_{i \in \mathcal{M}_{\mathrm{far}}} \pi_i P_m(\theta_i)
\end{align*}
Putting the above results together, we have 
\begin{equation*}
    \int P_m(\theta)d\nu(\theta) = \Delta\Pi_m - \sum_{c(i) \in\mathcal{C}_m}\pi_{i}\mathbf{1}_{\{i \in\mathcal{M}_{\mathrm{far}}\}} +  \frac{1}{2} \sum_{i \in \mathcal{M}_{\mathrm{mic}} \cup \mathcal{M}_{\mathrm{mac}}} \pi_i P_m''(\xi_i) |\theta_i-\theta^*_{c(i)}|^2 \ + \ \sum_{i \in \mathcal{M}_{\mathrm{far}}} \pi_i P_m(\theta_i),
\end{equation*} 
which implies 
\begin{align*}
    |\Delta\Pi_m|&\leq \left|\int P_md\nu\right|  + \frac{1}{2} \sum_{i \in \mathcal{M}_{\mathrm{mic}} \cup \mathcal{M}_{\mathrm{mac}}} \pi_i |P_m''(\xi_i)| |\theta_i-\theta^*_{c(i)}|^2 \\
    &\hspace{1cm}+ \left|\sum_{c(i) \in\mathcal{C}_m}\pi_{i}\mathbf{1}_{\{i \in\mathcal{M}_{\mathrm{far}}\}}\right| +\left|\sum_{i \in \mathcal{M}_{\mathrm{far}}} \pi_i P_m(\theta_i)\right|,
\end{align*} 

Utilizing the fact that $\sum_{\mathcal{M}_{\mathrm{far}}, c(i) \in \mathcal{C}_m} \pi_i \le \pi_{\mathrm{far}}$, Lemma \ref{lemma:Hellinger_to_Polynomial}, and Lemma \ref{lemma:dung_bound_for_cluster_mass_extracting}, we have 
\begin{align*} |\Delta \Pi_m| &\le C_{\mathrm{poly}}C_{\mathrm{norm},P} \cdot d_H(p_{G},p_{G_*})\\
&\hspace{1cm}+ \frac{1}{2} C_{P,2} \left( \sum_{\mathcal{M}_{\mathrm{mic}} \cup \mathcal{M}_{\mathrm{mac}}} \pi_i |\theta_i-\theta^*_{c(i)}|^2 \right) \ + \ (C_{\mathrm{norm},P} + 1) \pi_{\mathrm{far}}.
\end{align*}

We utilize the estimation for global variance in equation~\eqref{eqn:dung_thm2_global_variation_bound} and far mass in equation~\eqref{eqn:dung_thm2_global_far_mass}, noting that  $1 \le (2R)^{2s_{\max}-2} \Delta^{-(2s_{\max}-2)}$ with $\Delta_{\mathrm{sep}} \leq 2R$, we obtain  
\begin{align}
    |\Delta \Pi_m| \le C_{\Delta\Pi,1} \Delta_{\max}^{-(2s_{\max}-2)}d_H(p_{G},p_{G_*}),\label{eqn:dung_thm2_sum_delta_pi_estimation}
\end{align}  
where $C_{\Delta\Pi,1}$ is defined as 
\begin{align*}
    C_{\Delta\Pi,1} &= C_{\mathrm{poly}}C_{\mathrm{norm},P}\max\{(2R)^{2k_*-2},1\} + \frac{1}{2} C_{P,2} C_{\mathrm{var},2}\\
    &\hspace{1cm} + (C_{\mathrm{norm},P} + 1) C_{\mathrm{far,mass,2}} \max\{(2R)^{2k_*-2},1\}. 
\end{align*}

\emph{Step 5 - Bounding mass discrepancy. } We now evaluate mass discrepancy based on mass extractor function $\bar{E}_j$ such that $\bar{E}_j(\theta^*_l) = \delta_{jl}$ and $\bar{E}'_{j}(\theta_l^*) = 0$. We can choose the Hermite interpolated polynomial $\bar{E}_j$ defined in Section \ref{sec:dung_mass_extractor_multicluster} as 
$$\bar{E}_{j}(\theta) = \ell_{j,\text{micro}}^2(\theta) [\bar{A}_j +  \bar{B}_j(\theta - \theta_{j}^{*})]P_{\text{macro}}(\theta),$$ 
where we define $\ell_{j,\text{micro}}(\theta) = \prod_{q \in \mathcal{C}_{m} \neq j} \frac{\theta - \theta_q^*}{\theta_j^* - \theta_q^*}$, $P_{\text{macro}}(\theta) = \prod_{p \neq m} \prod_{q \in \mathcal{C}_p} \left( \frac{\theta - \theta_q^*}{2R} \right)^{2s_{\max}}$, and the coefficients $\bar{A}_j = 1/ P_{\text{macro}}(\theta_{j}^{*})$,
\begin{align*}
    \bar{B}_j & = -\bar{A}_j[2\ell'_{j,\text{micro}}(\theta_j^*) + (P_{\text{macro}}(\theta^*_j))^{-1} P'_{\text{macro}}(\theta^*_j)].
\end{align*}
Without loss of generality, we assume that $j\in \mathcal{C}_m$. Then, we have 
\begin{align}
\label{eqn:dung_delta_pi_as_sum_of_test_function}
     \int \bar{E}_j(\theta)d\nu(\theta) = \sum_{l=1}^{k_*}\pi_l\bar{E}_{j}(\theta_l) - \pi^*_j = \left(\sum_{i\in \bar{\mathcal{V}}_j}\pi_i -\pi^*_j\right) + \sum_{l=1}^{k_*}\pi_l\bar{E}_{j}(\theta_l) - \sum_{i\in \bar{\mathcal{V}}_j}\pi_i,
\end{align}
which implies 
\begin{align*}
    \Delta\pi_j &= \int \bar{E}_j(\theta)d\nu(\theta)  - \underbrace{\left(\sum_{l \in \mathcal{C}_m}\sum_{i\in \bar{\mathcal{V}}_l}\pi_i(\bar{E}_j(\theta_i) -\bar{E}_j(\theta^*_l))\right)}_{\text{Case 5.1}} - \underbrace{\sum_{\substack{i \in \mathcal{M}_{\mathrm{mic}}\\ c(i) \notin \mathcal{C}_m }}\pi_i\bar{E}_j(\theta_i)}_{\text{Case 5.2}}\\
    &\hspace{1cm } -\underbrace{\sum_{\substack{i \in \mathcal{M}_{\mathrm{mac}}\\ c(i) \in \mathcal{C}_m }}\pi_i\bar{E}_j(\theta_i)}_{\text{Case 5.3}} - \underbrace{\sum_{\substack{i \in \mathcal{M}_{\mathrm{mac}}\\ c(i) \notin \mathcal{C}_m }}\pi_i\bar{E}_j(\theta_i)}_{\text{Case 5.4}} -  \underbrace{\sum_{{i\in \mathcal{M}_{\mathrm{far}}}}\pi_i\bar{E}_j(\theta_i)}_{\text{Case 5.5}}. 
\end{align*}
For integral term, an application of Lemma \ref{lemma:Hellinger_to_Polynomial} and \ref{lemma:dung_multicluster_mass_test_function} gives us 
\begin{align*}
    \left|\int \bar{E}_jd\nu\right| \leq C_{\mathrm{poly}}\|\bar{E}_j\|_{\infty}d_H(p_G,p_{G_*}) \leq C_{\mathrm{poly}} C_{\mathrm{norm},\bar{E}}\Delta^{-(2s_m-1)}_{\mathrm{sep}}d_H(p_G,p_{G_*}).
\end{align*}
In addition, since $\Delta_{\mathrm{sep}}\leq 2R$, we have 
\begin{equation*}
    \Delta^{-(2s_m-1)}_{\mathrm{sep}} \leq  \Delta^{-(2s_{\max}-1)}_{\mathrm{sep}}(2R)^{2s_{\max}-2s_m} \leq \max\{1,(2R)^{2k_*-2}\}\Delta^{-(2s_{\max}-1)}_{\mathrm{sep}}. 
\end{equation*}
As a result, we achieve 
\begin{align}
    \left|\int \bar{E}_j(\theta)d\nu(\theta)\right| \leq C_{\bar{E},\mathrm{int},1}\Delta^{-(2s_{\max}-1)}_{\mathrm{sep}}d_H(p_G,p_{G_*}).  
    \label{eqn:dung_thm_3_2_univariate_global_integral_of_E_bound}
\end{align}
Now we move to estimate the sum of weighted $\bar{E}_j(\theta_i)$ based on its geometrical location with respect to its nearest point and related cluster.

\emph{Case 5.1 - Near set inside $\mathcal{C}_m$.} For each $i \in\mathcal{V}_l$ such that $l\in\mathcal{C}_m$, using Taylor expansion, for each $i \in \bar{\mathcal{V}}_l$ and $l \in \mathcal{C}_m$, there exists $\xi_i$ between $\theta_i$ and $\theta^*_l$ such that 
$$\bar{E}_j(\theta_i) -\bar{E}_j(\theta^*_l)  =  (\theta_i-\theta_l^*)\bar{E}'_j(\theta^*_l) + \frac{1}{2}(\theta_i-\theta_l^*)^2\bar{E}''_{j}(\xi_i) = \frac{1}{2}(\theta_i-\theta_l^*)^2\bar{E}''_{j}(\xi_i).$$
Using estimation in Lemma \ref{lemma:dung_multicluster_mass_test_function} for $\bar{E}''$ and equation \eqref{eqn:dung_thm2_global_variation_bound} for global variance, noting that we have 
\begin{align}
\nonumber
    \left|\sum_{l \in \mathcal{C}_m}\sum_{i\in \bar{\mathcal{V}}_l}\pi_i(\bar{E}_j(\theta_i) -\bar{E}_j(\theta^*_l))\right| &\leq C_{\bar{E},2}\Delta^{-2}_{\mathrm{sep}}\sum_{i\in\bar{\mathcal{V}}_l, l\in\mathcal{C}_m}\pi_{i}|\theta_i-\theta^*_l|^2\\
    &\leq C_{\mathrm{near,in,mass,1}}\Delta^{-2s_{\max}}_{\mathrm{sep}}d_H(p_{G},p_{G_*}),
    \label{eqn:thm_2_sum_near_f-f}
\end{align}
where $C_{\mathrm{near,in,mass,1}} =C_{\bar{E},2}C_{\mathrm{var},2}$. 

\emph{Case 5.2 - Near set outside $\mathcal{C}_m$.} For centers indexed by  $ i \in \mathcal{M}_{\mathrm{near}}$ but $c(i) \notin \mathcal{C}_m$, we need to estimate the corresponding value of $\bar{E}_j(\theta_i)$. Indeed, using Lemma \ref{lemma:dung_multicluster_mass_test_function}, noting that $|\theta_i - \theta^*_{c(i)}|\leq \Delta_{\mathrm{sep}}/4$ we have 
\begin{equation*}
|\bar{E}_j(\theta_i)|
    \leq C_{\text{cross},\bar{E}} \Delta_{\mathrm{sep}}^{-1}|\theta_i-\theta^*_{c(i)}|^2. 
\end{equation*}
By combining with equation~\eqref{eqn:dung_thm2_global_variation_bound}, this bound yields 
\begin{align}
\nonumber
|\sum_{\substack{i \in \mathcal{M}_{\mathrm{mic}}\\ c(i) \notin \mathcal{C}_m }}\pi_i\bar{E}_j(\theta_i)|& \leq \Delta_{\mathrm{sep}}^{- 1}C_{\text{cross},\bar{E}}\sum_{i=1}^{k_*} \pi_i  |\theta_l-\theta^*_{c(l)}|^2 \\
&\leq C_{\mathrm{near,out,mass,1}}\Delta_{\mathrm{sep}}^{-( 2s_{\max} - 1)}d_H(p_{G},p_{G_*}),
\label{eqn:dung_near_not_m_value_estimation}
\end{align}
where $C_{\mathrm{near,out,mass,1}} = C_{\text{cross},\bar{E}}C_{\mathrm{var,2}}$.

\emph{Case 5.3 - Macro set inside $\mathcal{C}_m$.} For $i \in \mathcal{M}_{\text{macro}}$ such that $c(i) \in \mathcal{C}_m$, we evaluate each term $P_{\text{macro}}(\theta_i)$, $\ell^2_{j,\text{micro}}(\theta_i)$, and monomial component separately. Indeed, we have for $|P_{\text{macro}}(\theta_i)| \leq 1$ due to the fact that $|\theta_i - \theta^*_{q}| \leq 2R$ for any true center $\theta^*_{q}$. For $\ell_{j,\text{micro}}$, noting that for any $q \in \mathcal{C}_m$, 
\begin{align}
\label{eqn:dung_thm_3_2_univariate_global_bound_for_distance_inside_cluster}
    |\theta_i-\theta^*_q| \leq |\theta_i-\theta^*_{c(i)}| + |\theta^*_{c(i)} - \theta^*_q| \leq  |\theta_i-\theta^*_{c(i)}| + C_0\Delta_{\mathrm{sep}} \leq (1+4C_0) |\theta_i-\theta^*_{c(i)}|. 
\end{align}
As a result, we have 
\begin{equation}
\label{eqn:ell_j_micro_macro_in_cluster}
    |\ell_{j,\text{micro}}(\theta_i)|^2 \leq \Delta^{-2(s_m-1)}_{\mathrm{sep}}(1+4C_0)^{2(s_m-1)}|\theta_i - \theta^*_{c(i)}|^{2(s_m-1)}. 
\end{equation}

For the monomial component, using the estimation of $\bar{A}_j$ and $\bar{B}_j$ at the beginning of the proof of Lemma \ref{lemma:dung_multicluster_mass_test_function}, we have 
\begin{equation*}
    |\bar{A}_j + \bar{B}_j(\theta_i -\theta_j^*)|\leq \bar{C}_A +\bar{C}_B\Delta^{-1}_{\mathrm{sep}}(1+4C_0)|\theta_i-\theta^*_{c(i)}| \leq (4\bar{C}_A +\bar{C}_B(1+4C_0)) \Delta^{-1}_{\mathrm{sep}}|\theta_i-\theta^*_{c(i)}|. 
\end{equation*}

Combining these estimations for all components of $\bar{E}_j$, we achieve 
\begin{align*}
|\bar{E}_j(\theta_i)| \leq   (4\bar{C}_A +\bar{C}_B(1+4C_0)) (1+4C_0)^{2(s_m-1)}\Delta^{-2s_m-1}_{\mathrm{sep}}|\theta_i - \theta^*_{c(i)}|^{2s_m-1}.
\end{align*}
Moreover, since $|\theta_i - \theta^*_{c(i)}|\geq \Delta_{\mathrm{sep}}/4$, we achieve 
\begin{align*}
1&\leq 4^{2s_{\max}-2s_m+1}\Delta^{-2s_{\max}+2s_m-1}_{\mathrm{sep}}|\theta_i - \theta^*_{c(i)}|^{2s_{\max}-2s_m+1} \\
&\leq 4^{2k_*+1}\Delta^{-2s_{\max}+2s_m-1}_{\mathrm{sep}}|\theta_i - \theta^*_{c(i)}|^{2s_{\max}-2s_m+1}. 
\end{align*}

Thus, the following estimation of $\bar{E}_j(\theta_i)$ is established
\begin{align}
\label{eqn:dung_thm_3_2_univariate_local_estimation_for_E_in_macro_mass_case}
    |\bar{E}_j(\theta_i)| \leq   C_{\bar{E},\mathrm{mac,in,mass,1}}\Delta^{-2s_{\max}}_{\mathrm{sep}}|\theta_i - \theta^*_{c(i)}|^{2s_{\max}},
\end{align}
where 
\begin{equation*}
C_{\bar{E},\mathrm{mac,in,mass,1}} = 4^{2k_*+1} \cdot (4\bar{C}_A +\bar{C}_B(1+4C_0)) (1+4C_0)^{2(k_*-1)}.
\end{equation*}
Therefore, thanks to the high moment estimation in equation~\eqref{eqn:macro_high_moment_estimate}, we have 
\begin{align}
\nonumber
\left|\sum_{i \in \mathcal{M}_{\mathrm{mac}},c(i)\in\mathcal{C}_m} \pi_i \bar{E}_j(\theta_i)\right| &\leq C_{\bar{E},\mathrm{mac,in,mass,1}}\Delta^{-2s_{\max}}_{\mathrm{sep}}\sum_{i \in \mathcal{M}_{\mathrm{mac}},c(i)\in\mathcal{C}_m}\pi_{i}|\theta_i - \theta_{c(i)}|^{2s_{\max}}\\
&\leq C_{\mathrm{mac,in,mass,1}} \Delta^{-2s_{\max}}_{\mathrm{sep}}d_H(p_{G},p_{G_*}),
\label{eqn:macro_m_set_sum_estimation}
\end{align}
where $C_{\mathrm{mac,in,mass,1}}  = C_{\bar{E},\mathrm{mac,in,mass,1}}\cdot C_{\mathrm{var},2}$. 

\emph{Case 5.4 - Macro set outside $\mathcal{C}_m$.} For $i\in\mathcal{M}_{\text{macro}}$ such that $c(i)\notin\mathcal{C}_m$, we similarly evaluate each term $P_{\text{macro}}(\theta_i)$, $l^2_{j,\text{micro}}(\theta_i)$, and monomial component separately. Indeed, we have for $|\ell^2_{j,\text{micro}}(\theta_i)| \leq (2R)^{2(s_m-1)}\Delta^{-2(s_m-1)}_{\mathrm{sep}}$ due to the fact that $|\theta_i - \theta^*_{q}| \leq 2R$. For $P_{\text{macro}}$, noting that $|\theta_i - \theta^*_{q}| \leq 2R$, we have 
\begin{equation*}
    |P_{\text{macro}}(\theta_i)|\leq (2R)^{-2s_{\max}}|\theta_i-\theta^*_{c(i)}|^{2s_{\max}}.
\end{equation*}
For monomial component, using the estimation of $\bar{A}_j$ and $\bar{B}_j$ at the beginning of the proof of Lemma \ref{lemma:dung_multicluster_mass_test_function}, we have 
\begin{equation*}
    |\bar{A}_j + \bar{B}_j(\theta_i -\theta_j^*)| \leq \bar{C}_A + \bar{C}_{B}\Delta^{-1}_{\mathrm{sep}}\cdot 2R \leq 2R(\bar{C}_A + \bar{C}_{B})\Delta^{-1}_{\mathrm{sep}}. 
\end{equation*}
Combining all these estimations above, we have 
\begin{align*}
    |\bar{E}_j(\theta_i)| \leq (\bar{C}_{A} + \bar{C}_{B}) (2R)^{2s_{m} - 2s_{\max} - 1} \Delta_{\mathrm{sep}}^{-(2s_{m} - 1)} (\theta_{l} - \theta_{c(l)}^{*})^{2s_{\max}}.
\end{align*}
Moreover, since $\Delta_{\mathrm{sep}}\leq 2R$, we achieve 
\begin{align*}
1&\leq \Delta^{2(s_m - s_{\max})}_{\mathrm{sep}}(2R)^{2(s_{\max}-s_m)},
\end{align*}
thus, we obtain 
\begin{align*}
    |\bar{E}_j(\theta_i)| \leq  C_{\bar{E},\mathrm{mac,out,mass,1}}\Delta_{\mathrm{sep}}^{-(2s_{\max} - 1)} (\theta_{l} - \theta_{c(l)}^{*})^{2s_{\max}}. 
\end{align*}
where $C_{\bar{E},\mathrm{mac,out,mass,1}} = (\bar{C}_{A} + \bar{C}_{B}) (2R)^{-1}$. Therefore, thanks to the high moment estimation in equation~\eqref{eqn:macro_high_moment_estimate}, we have 
\begin{align}
\nonumber
\left|\sum_{i \in \mathcal{M}_{\mathrm{mac}},c(i)\notin\mathcal{C}_m} \pi_i \bar{E}_j(\theta_i)\right| &\leq C_{\bar{E},\mathrm{mac,out,mass,1}}\Delta^{-(2s_{\max}-1)}_{\mathrm{sep}}\sum_{i \in \mathcal{M}_{\mathrm{mac}},c(i)\notin\mathcal{C}_m}\pi_{i}|\theta_i - \theta^*_{c(i)}|^{2s_{\max}}\\
&\leq C_{\mathrm{mac,out,mass,1}}\Delta^{-(2s_{\max}-1)}_{\mathrm{sep}}d_H(p_{G},p_{G_*}),
\label{eqn:macro_not_m_set_sum_estimation}
\end{align}
where $ C_{\mathrm{mac,out,mass,1}}= C_{\bar{E},\mathrm{mac,out,mass,1}}\cdot C_{\mathrm{mac,moment,1}}$. 

\emph{Case 5.5 - Far set. } 
For the centers $i \in \mathcal{M}_{\mathrm{far}}$, as  $|\theta_i -\theta^*_{c(i)}|\leq 2R$, it is straightforward to verify that $|\ell^2_{j,\text{macro}}(\theta_i)| \leq (2R)^{2(s_m-1)}\Delta^{-(2s_m-2)}_{\mathrm{sep}}$, $|P_{\text{macro}}(\theta_i)| \leq 1$. In addition, for the monomial component, as in Case 5.4, we have 
\begin{equation*}
    |\bar{A}_j + \bar{B}_j(\theta_i -\theta_j^*)| \leq 2R(\bar{C}_A + \bar{C}_{B})\Delta^{-1}_{\mathrm{sep}}. 
\end{equation*}
Using all the above results together, we have 
\begin{align*}
    |\bar{E}_j(\theta_i)| \leq (2R)^{2s_m-1}(\bar{C}_A + \bar{C}_{B})\Delta^{-(2s_m-1)}_{\mathrm{sep}}. 
\end{align*}
Moreover, since $\Delta_{\mathrm{sep}}\leq 2R$, we achieve 
\begin{align*}
1&\leq \Delta^{2(s_m - s_{\max})}_{\mathrm{sep}}(2R)^{2(s_{\max}-s_m)},
\end{align*}
thus, we obtain 
\begin{align*}
    |\bar{E}_j(\theta_i)| \leq  C_{\bar{E},\mathrm{far,mass,1}}\Delta_{\mathrm{sep}}^{-(2s_{\max} - 1)} (\theta_{l} - \theta_{c(l)}^{*})^{2s_{\max}}. 
\end{align*}
where $C_{\bar{E},\mathrm{far,mass,1}} = (\bar{C}_{A} + \bar{C}_{B})\max\{1,(2R)^{2k_*-1}\}$. 
Using estimation for global far mass in equation~\eqref{eqn:dung_thm2_global_far_mass}, we have 
\begin{align}
\nonumber
\left|\sum_{i \in \mathcal{M}_{\mathrm{far}}} \pi_i \bar{E}_j(\theta_i)\right| &\leq C_{\bar{E},\mathrm{far,mass,1}}\Delta^{-(2s_{\max}-1)}_{\mathrm{sep}}\pi_{\mathrm{far}}\\
&\leq C_{\mathrm{far,mass,test}}\Delta^{-(2s_{\max}-1)}_{\mathrm{sep}}d_H(p_{G},p_{G_*})
\label{eqn:dung_far_set_sum_estimation}
\end{align}
where $C_{\mathrm{far,mass,test,1}} = C_{\bar{E},\mathrm{far,mass,1}} \cdot C_{\mathrm{far,mass},2}$.

Now we can put all the estimations in separated cases in equations~\eqref{eqn:dung_thm_3_2_univariate_global_integral_of_E_bound}, \eqref{eqn:thm_2_sum_near_f-f}, \eqref{eqn:dung_near_not_m_value_estimation}, \eqref{eqn:macro_m_set_sum_estimation}, \eqref{eqn:macro_not_m_set_sum_estimation}, and \eqref{eqn:dung_far_set_sum_estimation} together to achieve the bound 
\begin{align}
 |\Delta\pi_j| \leq C_{\mathrm{mass},2} \Delta^{-2s_{\max}}_{\mathrm{sep}}d_H(p_{G},p_{G_*})
\label{eqn:dung_delta_pi_j_estimation}
\end{align}
where 
\begin{align*}
    C_{\mathrm{mass},2} &= 2R\left[C_{\bar{E},\mathrm{int,1}} +  C_{\mathrm{near,out,mass,1}} + C_{\mathrm{mac,in,mass,1}} + C_{\mathrm{mac,out,mass,1}} + C_{\mathrm{far,mass,test,1}}\right] \\
    &\hspace{1cm}  + C_{\mathrm{near,in,mass,1}}. 
\end{align*}

\emph{Step 6 - Bounding mean.} For this term, we consider two separate cases for the sum of weighted mean discrepancy for  $i \in \mathcal{M}_{\mathrm{mac}}\cup \mathcal{M}_{\mathrm{far}}$ and when $i \in \mathcal{M}_{\mathrm{near}}$. 

\emph{Step 6.1 - Bounding macro and far mean discrepancy.} Note that in the case when $i\in \mathcal{M}_{\mathrm{mac}}\cup \mathcal{M}_{\mathrm{far}}$,  $|\theta_i-\theta^*_{c(i)}|>\Delta_{\mathrm{sep}}/4$, thus we have 
\begin{align*}
\sum_{\mathcal{M}_{\mathrm{mac}}\cup\mathcal{M}_{\mathrm{far}}} \pi_i |\theta_i -\theta^*_{c(i)}| &= \sum_{\mathcal{M}_{\mathrm{mac}}\cup\mathcal{M}_{\mathrm{far}}} \pi_i |\theta_i -\theta^*_{c(i)}|^2 |\theta_i -\theta^*_{c(i)}|^{-1}\leq \dfrac{4}{\Delta_{\mathrm{sep}}}\sum_{\mathcal{M}_{\mathrm{mac}}\cup\mathcal{M}_{\mathrm{far}}} \pi_i |\theta_i -\theta^*_{c(i)}|^2. 
\end{align*}
The estimation in equation~\eqref{eqn:dung_thm2_global_variation_bound} implies that 
\begin{align}
\sum_{\mathcal{M}_{\mathrm{mac}}\cup\mathcal{M}_{\mathrm{far}}} \pi_i |\theta_i -\theta^*_{c(i)}| \leq C_{\mathrm{far,mac,mean},1}\Delta_{\mathrm{sep}}^{-(2s_{\max}-1)}d_H(p_{G},p_{G_*}),
   \label{eqn:dung_thm_2_first_moment_far_and_mac} 
\end{align}
where $C_{\mathrm{far,mac,mean},1} = 4C_{\mathrm{var},2}$. 

\emph{Step 6.2 - Bounding near mean discrepancy.} For this case, we utilize the similar argument  as in the proof of global case in Theorem \ref{theorem:exact_one_group_univariate}. Indeed, we consider two specific cases:

\emph{Case 6.2.1.} There exists a sub-Voronoi cell $\bar{\mathcal{V}}_j$ containing more than one element. Then, it means that there exists one Voronoi cell, without loss of generality, $\mathcal{V}_1$ which contains no element. In this case, $|\Delta\pi_1| = \pi_1^*$. The bound in equation~\eqref{eqn:dung_delta_pi_j_estimation} gives us 
\begin{align*}
    \pi^*_{\text{min}} \leq \pi_1^* \leq C_{\mathrm{mass,2}} \Delta^{-2s_{\max}}_{\mathrm{sep}}d_H(p_{G},p_{G_*}). 
\end{align*}
As a result, noting that $|\theta_i -\theta^*_{c(i)}|\leq \frac{\Delta_{\mathrm{sep}}}{4}$, we have 
\begin{align}
\nonumber
    \sum_{i\in \mathcal{M}_{\mathrm{near}}} \pi_i|\theta_i-\theta^*_{c(i)}| &\leq \frac{\Delta_{\mathrm{sep}}}{4}\sum_{i \in \mathcal{M}_{\mathrm{near}}} \pi_i \leq \frac{\Delta_{\mathrm{sep}}}{4} \leq \frac{\Delta_{\mathrm{sep}}}{4} \cdot \frac{C_{\mathrm{mass,2}} \Delta^{-2s_{\max}}_{\mathrm{sep}}d_H(p_{G},p_{G_*})}{\pi^*_{\text{min}}}\\
    &\leq C_{\mathrm{near,mean},1}(\pi^*_{\text{min}})^{-1}\Delta^{-(2s_{\max}-1)}_{\mathrm{sep}}d_H(p_{G},p_{G_*}).
\label{eqn:dung_thm2_global_near_first_moment_first_case}
\end{align}
where $C_{\mathrm{mean,near},1} = C_{\mathrm{mass,2}}/4$.  

\emph{Case 6.2.2.} There is no sub-Voronoi cell $\bar{\mathcal{V}}_j$ that has more than one element, which means that an non-empty $\bar{\mathcal{V}}_j$ contains exactly one element. If there exists some empty sub-Voronoi cell $\bar{\mathcal{V}}_j$, then we can use the similar argument as in Case 1. Therefore, it is sufficient to consider the situation that each sub-Voronoi cell $\bar{\mathcal{V}}_j$ contains exactly one element. We use similar argument in Step 5 in univariate local situation, bearing in mind that $(\pi^*_{\text{min}})^{-1} \geq 1$, we achieve 
\begin{align}
          \sum_{j\in \mathcal{M}_{\mathrm{near}}}\pi_j|\theta_j-\theta_{c(j)}^*| &= \sum_{j=1}^{k_*}\pi_j|\theta_j-\theta_{c(j)}^*| \leq C_{\mathrm{near,mean,2}}(\pi^*_{\text{min}})^{-1}\Delta^{-(2s_{\max}-1)}_{\mathrm{sep}}d_H(p_{G},p_{G_*}). \label{eqn:dung_thm_3_2_global_near_case_2_sum_of_first_discrepancy}
\end{align}
Combining the result in equations \eqref{eqn:dung_thm2_global_near_first_moment_first_case} and \eqref{eqn:dung_thm_3_2_global_near_case_2_sum_of_first_discrepancy}, we achieve 
\begin{align}
\label{eqn:dung_thm2_global_near_first_moment_overall}
    \sum_{j\in \mathcal{M}_{\mathrm{near}}}\pi_j|\theta_j-\theta_{c(j)}^*| \leq C_{\mathrm{near,mean}} (\pi^*_{\text{min}})^{-1}\Delta^{-(2s_{\max}-1)}_{\mathrm{sep}}d_H(p_{G},p_{G_*}),
\end{align}
where $C_{\mathrm{near,mean}}=\max\{C_{\mathrm{near,mean,1}},C_{\mathrm{near,mean,2}}\}$. Then, combining the result from equations \eqref{eqn:dung_thm_2_first_moment_far_and_mac} and \eqref{eqn:dung_thm2_global_near_first_moment_overall}, we have for $C_{\mathrm{mean},2} = C_{\mathrm{far,mac,mean,1}} + C_{\mathrm{near,mean}}$
\begin{align}
\label{eqn:dung_thm2_global_first_moment_overall}
    \sum_{j=1}^{k_*} \pi_j|\theta_j - \theta^*_{c(j)}|\leq C_{\mathrm{mean},2} (\pi^*_{\text{min}})^{-1}\Delta^{-(2s_{\max}-1)}_{\mathrm{sep}}d_H(p_{G},p_{G_*})
\end{align}

\emph{Step 7 - Bounding $W_1(G,G_*)$ and conclusion.} Now we put everything together. Plugging in the estimations from equations \eqref{eqn:dung_thm2_global_far_mass}, \eqref{eqn:dung_sum_pi_M_mac}, \eqref{eqn:dung_thm2_sum_delta_pi_estimation}, \eqref{eqn:dung_delta_pi_j_estimation}, \eqref{eqn:dung_thm2_global_first_moment_overall} into \eqref{eqn:dung_thm2_global_W_1_expansion}, by defining $$C^{-1}_{\mathrm{global},2}:= \max\{1,(2R)^{2k_*}\}C_{\mathrm{far,mass},2} + C_{\mathrm{mac,mass,1}} + 4R^2k_0\cdot C_{\Delta\Pi,1} + C_0C_{\mathrm{mass,2}} + C_{\mathrm{mean},2},$$ we have 
\begin{align*}
    W_1(G,G_*) \leq  C^{-1}_{\mathrm{global},2}\cdot (\pi^*_{\text{min}})^{-1}\Delta^{-(2s_{\max}-1)}_{\mathrm{sep}}d_H(p_{G},p_{G_*}),
\end{align*}
which implies 
\begin{align*}
   d_H(p_{G},p_{G_*}) \geq C_{\mathrm{global},2}\cdot \pi^*_{\text{min}}\cdot \Delta^{2s_{\max}-1}_{\mathrm{sep}}W_1(G,G_*). 
\end{align*}
This completes our proof for the global part. 

\subsubsection{Multivariate setting - Global bound}
\label{sec:global_bound_exact_multivariate_multi_cluster}

In this part, we follow the same strategy as in univariate global part. For every fitted atom $\theta_i$, let $c(i) = \text{argmin}_{1 \le j \le k_*} |\theta_i - \theta_j^*|$ be its absolute closest true center (while there exists several centers sharing the minimal distance, we randomly choose one of them). Let $m(i)$ denote the cluster containing the center $\theta^{*}_{c(i)}$. We partition the $k$ atoms of $G$ into three mutually disjoint regions according to their exact spatial proximity to the reference atoms:
\begin{itemize}
    \item The \emph{micro-near set} $\mathcal{M}_{\mathrm{mic}}=\left\{i:\|\theta_i - \theta_{c(i)}^*\|_2\le\frac{\Delta_{\mathrm{sep}}}{4}\right\},$
    consisting of atoms lying in the immediate neighborhood of their associated reference atoms. For each $j$, we further define the corresponding Voronoi-type cell inside this core region by $\bar{\mathcal{V}}_j=\{ i \in\mathcal{M}_{\mathrm{mic}} : c(i)=j\}.$

    \item The \emph{macro-near set} $\mathcal{M}_{\mathrm{mac}}=\left\{i :\frac{\Delta_{\mathrm{sep}}}{4}<\|\theta_i-\theta_{c(i)}^*\|_2\le\frac{D_0}{4}\right\},$
    which contains atoms that are no longer in the micro-near regime but still remain associated with the cluster $\mathcal{C}_{m(i)}$. We denote the total macro near mass as $\pi_{{\mathrm{mac}}} = \sum_{i \in \mathcal{M}_{\mathrm{mac}}} \pi_i$.

    \item The \emph{far set} $\mathcal{M}_{\mathrm{far}}=\left\{i :\|\theta_i-\theta_{c(i)}^*\|_2>\frac{D_0}{4}\right\},$
    consisting of atoms located in the inter-cluster void region separating distinct clusters. We denote the total mass contained in this region by $\pi_{\mathrm{far}}=
        \sum_{i \in\mathcal{M}_{\mathrm{far}}} \pi_i.$
\end{itemize}
As the proof of this part follows arguments nearly identical to those used in the univariate global case, we only present the main estimation steps and omit the repetitive technical details.

\vspace{0.5 em}
\noindent
\emph{Step 1 - Wasserstein decomposition.} Using exactly the same argument as in univariate global case, we achieve a preliminary decomposition for $W_1(G,G_*)$ to the first moment variance and mass mismatch per center and per cluster 
\begin{align}
\nonumber
    W_1(G, G_*) &\le \sum_{i \in \mathcal{M}_{\mathrm{mic}}} \pi_i\|\theta_{i}-\theta_{c(i)}\|_2 \ + \sum_{i \in \mathcal{M}_{\mathrm{mac}} \cup \mathcal{M}_{\mathrm{far}}} \pi_i\|\theta_{i}-\theta_{c(i)}\|_2+C_0 \Delta_{\mathrm{sep}} \sum_{j=1}^{k_*} |\Delta \pi_j|\\
    &\hspace{2cm} + C_0 \Delta_{\mathrm{sep}} \big(\pi_{\mathrm{mac}} + \pi_{\mathrm{far}}\big) +2R \sum_{m=1}^{k_0} |\Delta \Pi_m|,
\label{eqn:dung_thm2_global_multivariate_W_1_expansion}
\end{align}
where $\Delta \pi_j : = \left( \sum_{i \in \bar{\mathcal{V}}_j} \pi_i \right) - \pi_j^*$, $\tilde{\pi}_j = \sum_{i:c(i) = j}\pi_i$, and $\Delta\Pi_m = \sum_{j \in \mathcal{C}_m} \tilde{\Delta}\pi_j$.

\vspace{0.5 em}
\noindent
\emph{Step 2 -Bounding variance.} In this part, we utilize the variance extractor test function define in Section \ref{sec:variance_test_function}
\begin{equation*}
    P_{\mathrm{var}}(\theta) =  \prod_{l=1}^{k_*}\|\theta - \theta^*_{l}\|^2.
\end{equation*}

Since $\int P_{\mathrm{var}}(\theta) dG_*(\theta) = 0$, Lemma~\ref{lemma:Hellinger_to_Polynomial} and Lemma \ref{lemma:variance_test_function} indicates that 
\begin{align}
\label{eqn:dung_thm_3_2_mulivariate_global_first_estimation_about_P_var_multicluster}
\left|\int P_{\mathrm{var}}(\theta) dG(\theta)\right| = \left|\int P_{\mathrm{var}}(\theta) d\nu(\theta)\right| \le C_{\mathrm{poly}}C_{\mathrm{var}} d_H(p_{G},p_{G_*}).
\end{align}

We divide the integral into three components
\begin{align*}
    \sum_{i \in \mathcal{M}_{\mathrm{mic}}} \pi_i P_{\mathrm{var}}(\theta_i) + \sum_{i \in \mathcal{M}_{\mathrm{mac}}} \pi_i P_{\mathrm{var}}(\theta_i) + \sum_{i \in \mathcal{M}_{\mathrm{far}}} \pi_i P_{\mathrm{var}}(\theta_i)
\end{align*}
and 
estimate the value of $P_{\mathrm{var}}$ in each set $\mathcal{M}_{\mathrm{mic}}$, $\mathcal{M}_{\mathrm{mac}}$, and $\mathcal{M}_{\mathrm{far}}$. 

\emph{Case 2.1 - In micro-near set ($\mathcal{M}_{\mathrm{mic}}$)}: Using the same argument as in univariate global setting in equation~\eqref{eqn:dung_thm_3_2_univariate_global_P_var_greater_than_second_moment_micro_case}, we have 
\begin{align}
\label{eqn:dung_thm3.2_multivariate_global_f_var_at_micro_estimation}
    P_{\mathrm{var}}(\theta_i) \geq C_{\mathrm{mic},2} \cdot \Delta_{\mathrm{sep}}^{2s_{\max}-2} \|\theta_i-\theta^*_{c(i)}\|_2^2 \quad \text{for all } i \in \mathcal{M}_{\mathrm{mic}},
\end{align}
where $C_{\mathrm{mic},2} := \left(\frac{3}{4}\right)^{2k_*-2}\cdot \min\left\{\left(\frac{15}{16}D_0\right)^{2k_*-2}, 1\right\}\cdot \min\{(2R)^{-2k_*},1\}$. 

\emph{Case 2.2 - In macro-near set ($\mathcal{M}_{\mathrm{mac}}$). }  For $i \in \mathcal{M}_{\mathrm{mac}}$, using the same argument as in the univariate global setting, we have $\|\theta_i-\theta_q^*\|_2\geq \frac{\Delta_{\mathrm{sep}}}{4}$ for other $s_m-1$ centers in $\mathcal{C}_m$ and $\|\theta_i - \theta_p^*\|_2\geq \frac{3}{4}D_0$, thus 
\begin{align}
\label{eqn:dung_thm3.2_multivariate_global_f_var_at_macro_estimation}
 P_{\mathrm{var}}(\theta_i)\geq C_{\mathrm{mac},2}\cdot\Delta_{\mathrm{sep}}^{2s_{\max}-2}\|\theta_i - \theta_{c(i)}^*\|_2^2,
\end{align}
where $C_{\mathrm{mac},2} =4^{2-2k_*}\cdot \min\{(3D_0/4)^{2k_*-2},1\}\cdot\min\{(2R)^{-2k_*},1\}$. 

\emph{Case 2.3 - In far near set ($\mathcal{M}_{\mathrm{far}}$).} Using the same idea as in univariate global setting, for $i \in \mathcal{M}_{\mathrm{far}}$, we have $\|\theta_i - \theta_l^*\|_2 \ge \|\theta_i-\theta_{c(i)}\|_2 > \frac{D_0}{4}$ for all $k_*$ centers. Thus, as $\Delta_{\mathrm{sep}} \leq 2R$, we have
\begin{align}  
\label{eqn:dung_thm3.2_multivariate_global_f_var_at_far_estimation}
P_{\mathrm{var}}(\theta_i) \geq C_{\mathrm{far},2} \Delta^{-2s_{\max}}_{\mathrm{sep}}\|\theta_i-\theta^*_{c(i)}\|_2^2,
\end{align}
where $C_{\mathrm{far},2}=\min\{1,(2R)^{2k_*-2}\}\left(\frac{D_0}{4}\right)^{2k_*-2}$. 

Putting together these estimations of $P_{\mathrm{var}}$ in equations~\eqref{eqn:dung_thm3.2_multivariate_global_f_var_at_micro_estimation}, \eqref{eqn:dung_thm3.2_multivariate_global_f_var_at_macro_estimation}, and \eqref{eqn:dung_thm3.2_multivariate_global_f_var_at_far_estimation}, we have
\begin{equation*}
    \int P_{\mathrm{var}}(\theta)dG(\theta) \geq \min\{C_{\mathrm{mic},2},C_{\mathrm{mac},2},C_{\mathrm{far},2}\}\cdot \Delta_{\mathrm{sep}}^{-(2s_{\max}-2)}\sum_{\mathcal{M}_{\mathrm{mic}} \cup \mathcal{M}_{\mathrm{mac}}\cup \mathcal{M}_{\mathrm{far}}} \pi_i \|\theta_i-\theta_{c(i)}\|^2.
\end{equation*}
Combining this result with equation~\eqref{eqn:dung_thm_3_2_mulivariate_global_first_estimation_about_P_var_multicluster}, we have 
\begin{equation}
\label{eqn:dung_thm3.2_multivarite_global_variation_bound}
\sum_{i=1}^{k_*} \pi_i \|\theta_i-\theta_{c(i)}\|_2^2 \leq C_{\mathrm{var},3}\Delta_{\mathrm{sep}}^{-(2s_{\max}-2)}d_H(p_{G},p_{G_*}), 
\end{equation}
where $C_{\mathrm{var},3} = C_{\mathrm{poly}}C_{\mathrm{var}}\cdot \max\{C^{-1}_{\mathrm{mic},2},C^{-1}_{\mathrm{mac},2},C^{-1}_{\mathrm{far},2}\}$.
\newline 

\vspace{0.5 em}
\noindent
\emph{Step 3 - Bounding macro- and far-related quantities.} Now we bound the quantities related to $\mathcal{M}_{\mathrm{mac}}$ and $\mathcal{M}_{\mathrm{far}}$ set. 

\emph{Step 3.1 - Macro High-Moment $\sum_{\mathcal{M}_{\mathrm{mac}}} \pi_i \|\theta_i-\theta^*_{c(i)}\|_2^{2s_{\max}}$.} For each $i \in \mathcal{M}_{\mathrm{mac}}$, using the same argument as in the univariate global setting, we have $P_{\mathrm{var}}(\theta_i) \geq C^{-1}_{\mathrm{mac,moment,2}}\|\theta_i -\theta^*_{c(i)}\|^{2s_{\max}}_2$. Furthermore, by arguing in similar fashion to equation \eqref{eqn:macro_high_moment_estimate}, it follows that 
\begin{equation}
\label{eqn:dung_thm3.2_multivariate_global_macro_high_moment_estimate_1}
    \sum_{i\in \mathcal{M}_{\mathrm{mac}}} \pi_i \|\theta_i-\theta^*_{c(i)}\|_2^{2s_{\max}} \leq C_{\mathrm{mac,moment,2}} \cdot d_H(p_{G},p_{G_*}),
\end{equation}
where $$C_{\mathrm{mac,moment,2}} = \max\{(2R)^{2(k_*-1)},1\}\cdot\max\left\{1,\left(\frac{3}{4}D_0\right)^{-2k_*}\right\}C_{\mathrm{poly}}C_{\mathrm{var}}.$$ 

\emph{Step 3.2 - Global Far Mass $\pi_{\mathrm{far}}$. } We have for each $i \in \mathcal{M}_{\mathrm{far}}$, $\|\theta_i - \theta_l^*\|_2 \geq \|\theta_i-\theta^*_{c(i)}\|_2 > \frac{D_0}{4}$, thus we have $P_{\mathrm{var}}(\theta_i)\geq \left(\frac{D_0}{4}\right)^{2k_*}$. Using the same argument as in univariate global setting to establish equation~\eqref{eqn:dung_thm2_global_far_mass}, we have 
\begin{align}
    \pi_{\mathrm{far}} = \sum_{i \in \mathcal{M}_{\mathrm{far}}} \pi_i \leq C_{\mathrm{far,mass},3} d_H(p_{G},p_{G_*}),
\label{eqn:dung_thm3.2_multivariate_global_far_mass}
\end{align}
where $C_{\mathrm{far,mass},3} = \left(\frac{D_0}{4}\right)^{-2k_*} \cdot C_{\mathrm{poly}}C_{\mathrm{var}}$.   

\emph{Step 3.3 - Bounding macro mass $\pi_{\mathrm{mac}}$.} For any $i \in \mathcal{M}_{\mathrm{mac}}$, since $\|\theta_i-\theta^*_{c(i)}\|_2 > \Delta_{\mathrm{sep}}/4$, as in univariate global part, we have
\begin{align}
\label{eqn:dung_thm3.2_multivariate_global_macro_mass}
    \pi_{\mathrm{mac}} = \sum_{\mathcal{M}_{\mathrm{mac}}} \pi_i \le C_{\mathrm{mac,mass,2}}\Delta_{\mathrm{sep}}^{-(2s_{\max})} d_H(p_{G},p_{G_*}),
\end{align}
where $C_{\mathrm{mac,mass,2}}:=4^{2k_*}C_{\mathrm{mac,moment,1}}$. 

\vspace{0.5 em}
\noindent
\emph{Step 4 - Bounding cluster mass discrepancy $\sum_{m = 1}^{k_{0}} |\Delta \Pi_{m}|$. } As in the proof of univariate global case, we utilize the witness function $P_{m}$ defined in Section \ref{sec:cluster_wise_test_function}. Using the same argument with Taylor expansion as in univariate global part, there exists $\xi_i \in \bar{B}(0,R)$ such that 
\begin{align*}
    \int P_m(\theta)d\nu(\theta) &= \Delta\Pi_m - \sum_{c(i) \in\mathcal{C}_m}\pi_{i}\mathbf{1}_{\{i \in\mathcal{M}_{\mathrm{far}}\}} \\
    &\hspace{1cm}+  \frac{1}{2} \sum_{i \in \mathcal{M}_{\mathrm{mic}} \cup \mathcal{M}_{\mathrm{mac}}} \pi_i (\theta_i-\theta^*_{c(i)})^{\top}D^2P_m(\xi_i)(\theta_i-\theta^*_{c(i)})  + \sum_{i \in \mathcal{M}_{\mathrm{far}}} \pi_i P_m(\theta_i),
\end{align*} 
which implies 
\begin{align*}
    |\Delta\Pi_m|&\leq \left|\int P_m(\theta)d\nu(\theta)\right|  + \frac{1}{2} \sum_{i \in \mathcal{M}_{\mathrm{mic}} \cup \mathcal{M}_{\mathrm{mac}}} \pi_i \|D^2P_m(\xi_i)\|_{\mathrm{op}} \|\theta_i-\theta^*_{c(i)}\|^2 \\
    &\hspace{1cm}+ \left|\sum_{c(i) \in\mathcal{C}_m}\pi_{i}\mathbf{1}_{\{i \in\mathcal{M}_{\mathrm{far}}\}}\right| +\left|\sum_{i \in \mathcal{M}_{\mathrm{far}}} \pi_i P_m(\theta_i)\right|. 
\end{align*} 
Using the bound from Lemma \ref{lemma:dung_bound_for_cluster_mass_extracting}, we have 
\begin{align*}
    |\Delta \Pi_m| &\le C_{\mathrm{poly}}C_{\mathrm{norm},P}\cdot d_H(p_{G},p_{G_*})\\
    &\hspace{1cm}+ \frac{1}{2} C_{P,2} \left( \sum_{\mathcal{M}_{\mathrm{mic}} \cup \mathcal{M}_{\mathrm{mac}}} \pi_i \|\theta_i-\theta^*_{c(i)}\|_2^2 \right)+(C_{\mathrm{norm},P} + 1) \pi_{\mathrm{far}}.
\end{align*}
As in the proof of univariate global case, by invoking the estimation  in equations~\eqref{eqn:dung_thm3.2_multivariate_global_far_mass} and \eqref{eqn:dung_thm3.2_multivarite_global_variation_bound}, and using the same argument as in multivariate global part, we have 
\begin{align}
|\Delta \Pi_m| &\leq C_{\Delta\Pi,2} \Delta_{\max}^{-(2s_{\max}-1)}d_H(p_{G},p_{G_*}),
\label{eqn:dung_thm2_multivariate_global_sum_delta_pi_estimation}
\end{align}
where $C_{\Delta\Pi,2}$ is defined as 
\begin{align*}
    C_{\Delta\Pi,2} &= C_{\mathrm{poly}}C_{\mathrm{norm},P}\max\{(2R)^{2k_*-2},1\} + \frac{1}{2} C_{P,2} C_{\mathrm{var},2}\\
    &\hspace{1cm} + (C_{\mathrm{norm},P} + 1) C_{\mathrm{far,mass,2}} \max\{(2R)^{2k_*-2},1\}. 
\end{align*}

\vspace{0.5 em}
\noindent
\emph{Step 5 - Bounding mass discrepancy $\sum_{j=1}^{k_*} |\Delta \pi_j|$.} As in the univariate global part, in this section, we use the function $\bar{E}_j(\theta)$ constructed in Section \ref{sec:dung_mass_extractor_multicluster} as 
$$\bar{E}_{j}(\theta) = \ell_{j,\text{micro}}^2(\theta) [\bar{A}_j + \langle \bar{B}_j,\theta - \theta_{j}^{*}\rangle]P_{\text{macro}}(\theta),$$ 
where we define $\ell_{j,\text{micro}}(\theta) = \prod_{q \in \mathcal{C}_{m} \neq j} \frac{\|\theta - \theta_q^*\|_2}{\|\theta_j^* - \theta_q^*\|_2}$, $P_{\text{macro}}(\theta) = \prod_{p \neq m} \prod_{q \in \mathcal{C}_p} \left( \frac{\|\theta - \theta_q^*\|_2}{2R} \right)^{2s_{\max}}$ as in Section \ref{sec:dung_mean_extractor_multicluster}, and $\bar{A}_j = 1/ P_{\text{macro}}(\theta_{j}^{*})$,
\begin{align*}
    \bar{B}_j & = -\bar{A}_j[2\nabla\ell_{j,\text{micro}}(\theta_j^*) + (P_{\text{macro}}(\theta^*_j))^{-1}\nabla P_{\text{macro}}(\theta^*_j)].
\end{align*}
It is straightforward to verify that $\bar{E}_j(\theta^*_l) = \delta_{jl}$ and $\nabla \bar{E}'_{j}(\theta_l^*) = 0$. We assume that $j\in \mathcal{C}_m$. Then, using the same argument as in univariate global part, we have 
\begin{align*}
    \Delta\pi_j &= \int \bar{E}_j(\theta)d\nu(\theta)  - \underbrace{\left(\sum_{l \in \mathcal{C}_m}\sum_{i\in \bar{\mathcal{V}}_l}\pi_i(\bar{E}_j(\theta_i) -\bar{E}_j(\theta^*_l))\right)}_{\text{Case 5.1}} - \underbrace{\sum_{\substack{i \in \mathcal{M}_{\mathrm{mic}}\\ c(i) \notin \mathcal{C}_m }}\pi_i\bar{E}_j(\theta_i)}_{\text{Case 5.2}}\\
    &\hspace{1cm } -\underbrace{\sum_{\substack{i \in \mathcal{M}_{\mathrm{mac}}\\ c(i) \in \mathcal{C}_m }}\pi_i\bar{E}_j(\theta_i)}_{\text{Case 5.3}} - \underbrace{\sum_{\substack{i \in \mathcal{M}_{\mathrm{mac}}\\ c(i) \notin \mathcal{C}_m }}\pi_i\bar{E}_j(\theta_i)}_{\text{Case 5.4}} -  \underbrace{\sum_{{i\in \mathcal{M}_{\mathrm{far}}}}\pi_i\bar{E}_j(\theta_i)}_{\text{Case 5.5}}. 
\end{align*}
For the integral term, by the same argument employing Lemma \ref{lemma:Hellinger_to_Polynomial} and Lemma \ref{lemma:dung_multicluster_mass_test_function}, we achieve the bound 
\begin{align}
    \left|\int \bar{E}_j(\theta)d\nu(\theta)\right| \leq C_{\bar{E},\mathrm{int},2}\Delta^{-(2s_{\max}-1)}_{\mathrm{sep}}d_H(p_G,p_{G_*}).  
\label{eqn:dung_thm_3_2_multivariate_global_integral_of_E_bound}
\end{align}
Now we move to estimate the sum of weighted $\bar{E}_j(\theta_i)$ based on its geometrical location with respect to its nearest point and related cluster. 

\emph{Case 5.1 - Near set inside $\mathcal{C}_m$.} For each $i \in \bar{\mathcal{V}}_l$ and $l \in \mathcal{C}_m$, using Taylor expansion with Lagrange remainder, Lemma \ref{lemma:dung_multicluster_mass_test_function}, and estimation in equation~\eqref{eqn:dung_thm3.2_multivarite_global_variation_bound}, with the same argument as in the proof of univariate global part, we have 
\begin{align}
\nonumber
    \left|\sum_{l \in \mathcal{C}_m}\sum_{i\in \bar{\mathcal{V}}_l}\pi_i(\bar{E}_j(\theta_i) -\bar{E}_j(\theta^*_l))\right| &\leq C_{\bar{E},2}\Delta^{-2}_{\mathrm{sep}}\sum_{i\in\bar{\mathcal{V}}_l, l\in\mathcal{C}_m}\pi_{i}\|\theta_i-\theta^*_l\|_2^2\\
    &\leq C_{\mathrm{near,in,mass},2}\Delta^{-2s_{\max}}_{\mathrm{sep}}d_H(p_{G},p_{G_*}). 
    \label{eqn:dung_thm3.2_multivariate_global_sum_near_f-f}
\end{align}
where $C_{\mathrm{near,in,mass},2} = C_{\bar{E},2}C_{\mathrm{var},3}$.

\emph{Case 5.2 - Near set outside $\mathcal{C}_m$.} 
For index $i \in \mathcal{M}_{\mathrm{mic}}$ that does not in $\mathcal{C}_m$, using Lemma \ref{lemma:dung_bound_for_cluster_mass_extracting}, we have
\begin{equation*}
    |\bar{E}_j(\theta_l)|
    \leq C_{\text{cross},\bar{E}}\Delta_{\mathrm{sep}}^{- 1}\|\theta_l-\theta^*_{c(l)}\|_2^2. 
\end{equation*}
Using the same argument as in univariate global part, the above bound yields
\begin{align}
\nonumber
|\sum_{\substack{i \in \mathcal{M}_{\mathrm{mic}}\\ c(i) \notin \mathcal{C}_m }}\pi_i\bar{E}_j(\theta_i)|& \leq \Delta_{\mathrm{sep}}^{-1}C_{\text{cross},\bar{E}}\sum_{i=1}^{k_*} \pi_i  \|\theta_l-\theta_{c(l)}\|^2 \\
&\leq C_{\mathrm{near,out,mass},2}\Delta_{\mathrm{sep}}^{-( 2s_{\max} - 1)}d_H(p_{G},p_{G_*}),
\label{eqn:dung_thm3_2_multivariate_global_near_not_m_value_estimation}
\end{align}
where $C_{\mathrm{near,out,mass},2} = C_{\text{cross}}C_{\mathrm{var},3}$. 

\emph{Case 5.3 - Macro set inside $\mathcal{C}_m$.} For index $i \in \mathcal{M}_{\text{macro}}$ such that $c(i) \in \mathcal{C}_m$, using the same argument as in Case 5.3 of univariate global part, we have 
\begin{equation*}
    |\bar{E}_j(\theta_l)| \leq C_{\bar{E},\mathrm{mac,in,mass},2}\Delta^{-2s_{\max}}_{\mathrm{sep}}\|\theta_l - \theta^*_{c(l)}\|_2^{2s_{\max}}.
\end{equation*}
Using the high moment estimation in equation~\eqref{eqn:dung_thm3.2_multivariate_global_macro_high_moment_estimate_1}, we have
\begin{align}
\nonumber
\left|\sum_{i \in \mathcal{M}_{\mathrm{mac}},c(i)\in\mathcal{C}_m} \pi_i \bar{E}_j(\theta_i)\right| &\leq C_{\bar{E},\mathrm{mac,in,mass},2}\Delta^{-2s_{\max}}_{\mathrm{sep}}\sum_{i \in \mathcal{M}_{\mathrm{mac}},c(i)\in\mathcal{C}_m}\pi_{i}|\theta_i - \theta_{c(i)}|^{2s_{\max}}\\
&\leq C_{\mathrm{mac,in,mass},2} \Delta^{-2s_{\max}}_{\mathrm{sep}}d_H(p_{G},p_{G_*}),
\label{eqn:dung_thm3_2_multivariate_global_macro_m_set_sum_estimation}
\end{align}

\emph{Case 5.4 - Macro set outside $\mathcal{C}_m$.} For centers $\theta_i$ in $\mathcal{C}_m \cap \mathcal{M}_{\text{macro}}$, using the same argument as in univariate global part, we have $\ell^2_{j,\text{micro}}(\theta_i) \leq (2R)^{2(s_m-1)}\Delta^{-2(s_m-1)}_{\mathrm{sep}}$, $|P_{\text{macro}}(\theta_i)|\leq (2R)^{-2s_{\max}}|\theta_i-\theta^*_{c(i)}|^{2s_{\max}}$, and for degree-one component, $|\bar{A}_j + \langle \bar{B}_j,\theta_i - \theta_j^*\rangle| \leq 2R(\bar{C}_A +\bar{C}_B)\Delta^{-1}_{\mathrm{sep}}$. Consequently, we have
\begin{equation*}
    |\bar{E}_j(\theta_l)| \leq C_{\bar{E},\mathrm{mac,out,mass},2}\Delta_{\mathrm{sep}}^{-(2s_{m} - 1)} \|\theta_{i} - \theta_{c(i)}^{*}\|_2^{2s_{\max}}.
\end{equation*}
Therefore, thanks to the high moment estimation in equation~\eqref{eqn:dung_thm3.2_multivariate_global_macro_high_moment_estimate_1}, we have 
\begin{align}
\nonumber
\left|\sum_{i \in \mathcal{M}_{\mathrm{mac}},c(i)\notin\mathcal{C}_m} \pi_i \bar{E}_j(\theta_i)\right| &\leq C_{\bar{E},\mathrm{mac,out,mass},2}\Delta^{-(2s_{\max}-1)}_{\mathrm{sep}}\sum_{i \in \mathcal{M}_{\mathrm{mac}},c(i)\notin\mathcal{C}_m}\pi_{i}\|\theta_i - \theta_{c(i)}\|^{2s_{\max}}\\
&\leq C_{\mathrm{mac,out,mass},2}\Delta^{-(2s_{\max}-1)}_{\mathrm{sep}}d_H(p_{G},p_{G_*}),
\label{eqn:dung_thm3_2_multivariate_global_macro_not_m_set_sum_estimation}
\end{align}
where $C_{\mathrm{mac,out,mass},2}  = C_{\bar{E},\mathrm{mac,out,mass},2}\cdot C_{\mathrm{mac,moment},2}$. 

\emph{Case 5.5 - Far set.}
For the centers $i$ in $\mathcal{M}_{\mathrm{far}}$, as $\|\theta_i -\theta^*_{c(i)}\|_2\leq 2R$, using the same argument as in univariate global part, we have 
\begin{align*}
|\bar{E}_j(\theta_i)| \leq (2R)^{2s_m-1}(\bar{C}_A + \bar{C}_B)\Delta^{-(2s_m-1)}_{\mathrm{sep}}.
\end{align*}
Using the estimation for global far mass in equation~\eqref{eqn:dung_thm2_global_far_mass} and the same argument as in univariate global part, we have 
\begin{align}
\nonumber
\left|\sum_{i \in \mathcal{M}_{\mathrm{far}}} \pi_i \bar{E}_j(\theta_i)\right| &\leq (2R)^{2s_m-1}(\bar{C}_A + \bar{C}_B)\Delta^{-(2s_{m}-1)}_{\mathrm{sep}}\sum_{i \in \mathcal{M}_{\mathrm{far}}}\pi_{i}\\
&\leq C_{\mathrm{far,mass,test,2}} \Delta^{-(2s_{\max}-1)}_{\mathrm{sep}}d_H(p_{G},p_{G_*}).
\label{eqn:dung_thm3_2_multivariate_global_far_set_sum_estimation}
\end{align}
Combining the results from the integral estimation in equation~\eqref{eqn:dung_thm_3_2_multivariate_global_integral_of_E_bound}, and the sum $\bar{E}_j$ based on geometric location in equations~\eqref{eqn:dung_thm3.2_multivariate_global_sum_near_f-f}, \eqref{eqn:dung_thm3_2_multivariate_global_near_not_m_value_estimation}, \eqref{eqn:dung_thm3_2_multivariate_global_macro_m_set_sum_estimation}, \eqref{eqn:dung_thm3_2_multivariate_global_macro_not_m_set_sum_estimation}, and \eqref{eqn:dung_thm3_2_multivariate_global_far_set_sum_estimation}, as in the proof of the univariate global part, we have 
\begin{align}
 |\Delta\pi_j| \leq C_{\mathrm{mass},3} \Delta^{-2s_{\max}}_{\mathrm{sep}}d_H(p_{G},p_{G_*})
\label{eqn:dung_thm3_2_multivariate_global_delta_pi_j_estimation}
\end{align}
where 
\begin{align*}
    C_{\mathrm{mass},3} &= 2R\left[C_{\bar{E},\mathrm{int,2}} +  C_{\mathrm{near,out,mass,2}} + C_{\mathrm{mac,in,mass,2}} + C_{\mathrm{mac,out,mass,2}} + C_{\mathrm{far,mass,2}}\right] \\
    &\hspace{1cm}  + C_{\mathrm{near,in,mass,2}}. 
\end{align*}

\vspace{0.5 em}
\noindent
\emph{Step 6 - Bounding mean $\sum_{j=1}^{k_*} \pi_j\|\theta_j - \theta^*_{c(j)}\|$.} For this term, we consider two separate cases when $i$ belongs to $\mathcal{M}_{\mathrm{mac}}\cup \mathcal{M}_{\mathrm{far}}$ and when $i$ belongs to $\mathcal{M}_{\mathrm{near}}$. 

\emph{Step 6.1 - Bounding macro and far mean discrepancy $\mathcal{M}_{\mathrm{mac}}\cup \mathcal{M}_{\mathrm{far}}$:} Note that in this case $\|\theta_i-\theta^*_{c(i)}\|_2>\Delta_{\mathrm{sep}}/4$, thus we have 
\begin{align}
\nonumber
\sum_{\mathcal{M}_{\mathrm{mac}}\cup\mathcal{M}_{\mathrm{far}}} \pi_i \|\theta_i -\theta^*_{c(i)}\|_2 
   &\leq \dfrac{4}{\Delta_{\mathrm{sep}}}\sum_{\mathcal{M}_{\mathrm{mac}}\cup\mathcal{M}_{\mathrm{far}}} \pi_i \|\theta_i -\theta^*_{c(i)}\|^2 \\
   &\leq C_{\mathrm{far,mac,mean,2}}\Delta_{\mathrm{sep}}^{-(2s_{\max}-1)}d_H(p_{G},p_{G_*}).
\label{eqn:dung_thm_2_multivariate_global_first_moment_far_and_mac} 
\end{align}

\emph{Step 6.2 - Bounding near set mean discrepancy $\mathcal{M}_{\mathrm{mic}}$:} For this case, we utilize the similar argument  as in the proof of univariate global case. Indeed, we consider two specific cases. 

\emph{Case 6.2.1.} There exists a sub-Voronoi cell $\bar{\mathcal{V}}_j$ containing more than one element. Then, it means that there exists one Voronoi cell, without loss of generality, $\mathcal{V}_1$ which contains no element. In this case, $|\Delta\pi_1| = \pi_1^*$. The bound in equation~\eqref{eqn:dung_thm3_2_multivariate_global_delta_pi_j_estimation} gives us 
\begin{align}
\pi^*_{\text{min}} \leq \pi_1^* \leq C_{\mathrm{mass,3}} \Delta^{-2s_{\max}}_{\mathrm{sep}}d_H(p_{G},p_{G_*}). 
\end{align}
As a result, noting that $\|\theta_i -\theta^*_{c(i)}\|\leq \frac{\Delta_{\mathrm{sep}}}{4}$, we have 
\begin{align}
    \sum_{i\in \mathcal{M}_{\mathrm{near}}} \pi_i\|\theta_i-\theta^*_{c(i)}\| 
    \leq \bar{C}_{\mathrm{near,mean},1}(\pi^*_{\text{min}})^{-1}\Delta^{-(2s_{\max}-1)}_{\mathrm{sep}}d_H(p_{G},p_{G_*}).
\label{eqn:dung_thm2_multivariate_global_near_first_moment_first_case}
\end{align}

\emph{Case 6.2.2.} There is no sub-Voronoi cell $\bar{\mathcal{V}}_j$ that has more than one element, which means that an non-empty $\bar{\mathcal{V}}_j$ contains exactly one element. If there exists some empty sub-Voronoi cell $\bar{\mathcal{V}}_j$, then we can use the similar argument as in Case 6.2.1. Therefore, it is sufficient to consider the situation that each sub-Voronoi cell $\bar{\mathcal{V}}_j$ contains exactly one element. We use similar argument in local situation in equation \eqref{eqn:dung_thm_3_2_univariate_local_final_bound_for_mean_discrepancy}, bearing in mind that $(\pi^*_{\text{min}})^{-1} \geq 1$, we achieve 
\begin{align}
\sum_{j\in \mathcal{M}_{\mathrm{near}}}\pi_j\|\theta_j-\theta_{c(j)}^*\| &= \sum_{j=1}^{k_*}\pi_j\|\theta_j-\theta_{c(j)}^*\| \leq \bar{C}_{\mathrm{near,mean},2} \Delta_{\mathrm{sep}}^{-(2s_{\max}-1)}d_H(p_{G},p_{G_*}).\label{eqn:dung_thm_3_2_multivariate_global_near_case_2_sum_of_first_discrepancy}
\end{align}
Combining the result in equations \eqref{eqn:dung_thm2_multivariate_global_near_first_moment_first_case} and \eqref{eqn:dung_thm_3_2_multivariate_global_near_case_2_sum_of_first_discrepancy}, we achieve 
\begin{align}
    \sum_{j\in \mathcal{M}_{\mathrm{near}}}\pi_j\|\theta_j-\theta_{c(j)}^*\| \leq \bar{C}_{\mathrm{near,mean}} (\pi^*_{\text{min}})^{-1}\Delta^{-(2s_{\max}-1)}_{\mathrm{sep}}d_H(p_{G},p_{G_*}),
\label{eqn:dung_thm_3_2_multivariate_global_near_first_moment_overall}
\end{align}
where $\bar{C}_{\mathrm{near,mean}}=\max\{\bar{C}_{\mathrm{near,mean},1} ,\bar{C}_{\mathrm{near,mean},2} \}$. Then, combining the result from equations \eqref{eqn:dung_thm_2_multivariate_global_first_moment_far_and_mac} and \eqref{eqn:dung_thm_3_2_multivariate_global_near_first_moment_overall}, we have for $C_{\mathrm{mean},3} = C_{\mathrm{far,mac,mean,2}} + \bar{C}_{\mathrm{near,mean}}$
\begin{align}
\label{eqn:dung_thm_3_2_multivariate_global_first_moment_overall}
    \sum_{j=1}^{k_*} \pi_j\|\theta_j - \theta^*_{c(j)}\|\leq C_{\mathrm{mean},3} (\pi^*_{\text{min}})^{-1}\Delta^{-(2s_{\max}-1)}_{\mathrm{sep}}d_H(p_{G},p_{G_*})
\end{align}

\vspace{0.5 em}
\noindent
\emph{Step 7 - Bounding $W_1(G,G_*)$ and conclusion.} Now we put everything together. Plugging the estimations from equations \eqref{eqn:dung_thm3.2_multivariate_global_far_mass}, \eqref{eqn:dung_thm3.2_multivariate_global_macro_mass}, \eqref{eqn:dung_thm2_multivariate_global_sum_delta_pi_estimation}, \eqref{eqn:dung_thm3_2_multivariate_global_delta_pi_j_estimation}, \eqref{eqn:dung_thm_3_2_multivariate_global_first_moment_overall} into equation~\eqref{eqn:dung_thm2_global_multivariate_W_1_expansion}, by defining $$\bar{C}^{-1}_{\text{global},2}:= \max\{1,(2R)^{2k_*}\}C_{\mathrm{far,mass},3} + C_{\mathrm{mac,mass,2}} + 4R^2k_0\cdot C_{\Delta\Pi,2} + C_0C_{\mathrm{mass,3}} + C_{\mathrm{mean},3},$$ we have 
\begin{align*}
    W_1(G,G_*) \leq  \bar{C}^{-1}_{\mathrm{global},2}\cdot (\pi^*_{\text{min}})^{-1}\Delta^{-(2s_{\max}-1)}_{\mathrm{sep}}d_H(p_{G},p_{G_*}),
\end{align*}
which implies 
\begin{align*}
   d_H(p_{G},p_{G_*}) \geq \bar{C}_{\mathrm{global},2}\cdot \pi^*_{\text{min}}\cdot \Delta^{2s_{\max}-1}_{\mathrm{sep}}W_1(G,G_*). 
\end{align*}
This completes our proof for the global part.

\subsection{Proof of Theorem \ref{theorem:exact_no_group}}
\label{sec:proof_of_exact_no_group}
\subsubsection{Univariate Setting - Local Bound}
\label{sec:local_bound_exact_univariate_no_group}
We now focus on the local regime where the Wasserstein distance $W_1(G, G_*)$ is sufficiently small. Lemma \ref{lemma:small_wasserstein_implies_center_in_voronoi_cell} implies that each estimated atom $\theta_i$ lies strictly within a Voronoi cell $\mathcal{V}_j$ associated with its nearest true parameter $\theta_j^*$, and satisfies
$\lvert \theta_i - \theta_j^* \rvert \leq \frac{\Delta_{\mathrm{sep}}}{4}.$
Since the fitted measure \( G \) has exactly \( k_* \) components, it follows that each Voronoi cell \( \mathcal{V}_j \) contains exactly one atom. Without loss of generality, we may relabel the indices so that \( \theta_j \in \mathcal{V}_j \) for all \( j = 1, \dots, k_* \).

\vspace{0.5 em}
\noindent
\emph{Step 1 - Wasserstein decomposition.} Let $G' = \sum_{j=1}^{k_*} \pi_j \delta_{\theta_j^*}$, we would like to decompose $W_1(G, G_*)$ through the intermediate discrete measure $G'$. By the triangle inequality with the $W_{1}$ metric, we have
\begin{equation}
\label{eqn:dung:thm_3_3_univariate_local_first_decomposition_of_W1}
    W_1(G, G_*) \le W_1(G, G') + W_1(G', G_*).
\end{equation}

For $W_1(G, G')$, by choosing the transportation plan $\rho_{jj} = \pi_{j}$ for all $1 \leq j \leq k_{*}$ and $\rho_{ij} = 0$ as $1 \leq i \neq j \leq k_{*}$ we obtain that
\begin{align}
    W_1(G, G') \leq \sum_{j=1}^{k_*} \pi_{j}|\theta_{j} - \theta_{j}^{*}|. \label{eq:local_bound_Wasserstein_no_group_1} 
\end{align}

To bound $W_1(G', G_*)$,  we consider the following transportation plan from $G'$ to $G_*$. We keep a mass of $\min\{\pi_i,\pi^*_i\}$ for each $\theta^*_i$, while transferring $\max\{\pi_i-\pi^*_i,0\}$ to the other center. Then, the total mass moved is exactly $\frac{1}{2}\sum_{j = 1}^{k_{*}} |\pi_j - \pi_{j}^{*}|$, and the distance between any true centers is less than $2R$, thus 
\begin{equation}
\label{eqn:dung_thm_3_3_univariate_local_inequality_W_1_true_measure_and_intermediate}
    W_1(\tilde{G},G_*) \leq 2R \sum_{j=1}^{k_*} |\pi_j - \pi_{j}^{*}|. 
\end{equation}

By plugging the results from equations \eqref{eq:local_bound_Wasserstein_no_group_1} and \eqref{eqn:dung_thm_3_3_univariate_local_inequality_W_1_true_measure_and_intermediate} into equation~\eqref{eqn:dung:thm_3_3_univariate_local_first_decomposition_of_W1}, we achieve a decomposition of $W_1(G,G')$ into mass and mean discrepancies 
\begin{align}
W_1(G, G_*) \le \sum_{j=1}^{k_*} \pi_{j}|\theta_{j} - \theta_{j}^{*}| + 2R\sum_{j = 1}^{k_{*}} |\pi_j - \pi_{j}^{*}|. \label{eq:no_group_Wasserstein_local_bound}
\end{align}
\emph{Step 2 - Bounding variance $\sum_{j=1}^{k_*} \pi_j |\theta_{j} - \theta_{j}^{*}|^2$.} We now consider the globally positive witness function defined in Section \ref{sec:variance_test_function}: 
\begin{align*}
    P_{\mathrm{var}}(\theta) = \prod_{l=1}^{k_*} (\theta - \theta_l^*)^2.
\end{align*}
Since $P_{\mathrm{var}}(\theta_l^*) = 0$, we have $\int P_{\mathrm{var}}(\theta) dG_*(\theta) = 0$. As a result, we have
\begin{align*} 
\int P_{\mathrm{var}}(\theta) d\nu(\theta)=\int P_{\mathrm{var}}(\theta) d G(\theta) = \sum_{j=1}^{k_*} \pi_j (\theta_{j} - \theta_{j}^{*})^2 \prod_{l \neq j} (\theta_j - \theta_l^*)^2.
\end{align*}
For any $l \neq j$, we have
\begin{align*}
    |\theta_{j} - \theta_l^*| \ge |\theta_j^* - \theta_l^*| - |\theta_{j} - \theta_{j}^{*}| \ge \Delta_{\mathrm{sep}} - \frac{\Delta_{\mathrm{sep}}}{4} = \frac{3 \Delta_{\mathrm{sep}}}{4}.
\end{align*}
Thus, we find that
\begin{align*} 
\int P_{\mathrm{var}}(\theta) d\nu(\theta) & \ge    \left(\frac{3}{4}\right)^{2k_*-2} \Delta_{\mathrm{sep}}^{2k_*-2} \sum_{j=1}^{k_*} \pi_j |\theta_{j} - \theta_{j}^{*}|^2.
\end{align*}
Recall from Lemma~\ref{lemma:Hellinger_to_Polynomial} and Lemma~\ref{lemma:variance_test_function} that $$|\int P_{\mathrm{var}}(\theta) d\nu(\theta)| \le C_{\mathrm{poly}}\|P_{\mathrm{var}}\|_{\infty}d_H(p_{G},p_{G_*}) \leq C_{\mathrm{poly}}C_{\mathrm{var}} d_H(p_{G},p_{G_*}).$$ Combining these bounds leads to
\begin{align}
\nonumber
    \sum_{j=1}^{k_*} \pi_j |\theta_{j} - \theta_{j}^{*}|^2 &\le \left[ \frac{C_{\mathrm{poly}}C_{\mathrm{var}}}{(3/4)^{2k_*-2}} \right]  \Delta_{\mathrm{sep}}^{-(2k_* - 2)} d_H(p_{G},p_{G_*}) \\
    &: = C_{\mathrm{var},1} \Delta_{\mathrm{sep}}^{-(2k_* - 2)} d_H(p_{G},p_{G_*}). \label{eq:exact_no_group_third} 
\end{align}
\emph{Step 3 - Bounding mean discrepancy $\pi_j |\theta_{j} - \theta_{j}^{*}|$.} A necessary ingredient to achieve a bound for the mean discrepancy is a polynomial $H_{j}$ such that $H_{j,i}(\theta_l^*) = 0$ and $H_j'(\theta_l^*) = \delta_{jl}$, for all $1 \leq l \leq k_{*}$. We can choose the Hermite interpolation polynomial $H_j$ of degree $2k_*-1$ defined in Section \ref{sec:mean_test_function} given by 
\begin{align*}
    H_{j}(\theta) : = \ell_j^2(\theta) (\theta - \theta_j^*)  \quad \text{where} \quad \ell_j(\theta) : = \prod_{q \neq j} \frac{\theta - \theta_q^*}{\theta_j^* - \theta_q^*},
\end{align*}
By utilizing Taylor expansion $H_j(\theta_l)$ exactly around its true center $\theta_l^*$, we have $H_j(\theta_l) = \delta_{jl}(\theta_{l} - \theta_{l}^{*}) + \frac{1}{2} H_j''(\xi_l) (\theta_{l} - \theta_{l}^{*})^2$. Thus, we arrive at $$ \int H_j(\theta) d\nu(\theta) = \pi_{j}(\theta_{j} - \theta_{j}^{*}) + \frac{1}{2} \sum_{l=1}^{k_*} \pi_{l} H_j''(\xi_l) (\theta_{l} - \theta_{l}^{*})^2.$$
This equation leads to
\begin{align*}
    \pi_{j}|\theta_{j} - \theta_{j}^{*}| &\leq \left|\int H_j(\theta) d\nu(\theta)\right| + \frac{1}{2} \sum_{l=1}^{k_*} \pi_{l} |H_j''(\xi_l)| (\theta_{l} - \theta_{l}^{*})^2\nonumber\\
    &\leq C_{\mathrm{poly}}\|H_j\|_{\infty}d_H(p_{G},p_{G_*})+ \frac{1}{2}\max_{\theta \in [-R,R]} |H_{j}''(\theta)| \sum_{l=1}^{k_*} \pi_{l} (\theta_{l} - \theta_{l}^{*})^2.
\end{align*}
where the second inequality follows from Lemma~\ref{lemma:Hellinger_to_Polynomial}. Using Lemma \ref{lemma:mass_test_function} and the estimation in equation~\eqref{eq:exact_no_group_third}, we have $\|H_j\|_{\infty}\leq C_{\mathrm{norm},H}\Delta^{-(2k_*-2)}_{\mathrm{sep}}$ and $\max_{\theta \in [-R,R]}|H_j''(\theta)| \leq C_{H,2}\Delta^{-(2k_*-2)}_{\mathrm{sep}}$. Consequently, thanks to variance bound \eqref{eq:exact_no_group_third} and the fact that $\Delta_{\mathrm{sep}}\leq 2R$, we achieve a bound 
\begin{align}
    \pi_{j}|\theta_{j} - \theta_{j}^{*}|\leq C_{\mathrm{mean},1} \Delta_{\mathrm{sep}}^{-(4k_*-4)} d_H(p_{G},p_{G_*}), \label{eqn:dung_thm_3.3_univariate_local_final_estimation_for_mean_discrepancy_ohlala}
\end{align}
where 
\begin{equation*}
    C_{\mathrm{mean},1} = (2R)^{2k_*-2}C_{\mathrm{poly}}C_{\mathrm{norm},H} + \frac{1}{2}C_{H,2}C_{\mathrm{var},2}. 
\end{equation*}

\vspace{0.5 em}
\noindent
\emph{Step 4 - Bounding mass discrepancy $|\pi_{j} - \pi_{j}^{*}|$.} To derive an upper bound for $|\pi_{j} - \pi_{j}^{*}|$, we consider a polynomial $E_j$ such that $E_j(\theta_l^*) = \delta_{jl}$ and $E_j'(\theta_l^*) = 0$ for all $1\leq l\leq k_*$. We can choose the Hermite interpolation polynomial of degree $2k_* - 1$ defined in Section \ref{sec:prior_transition_test_function} given by
\begin{align*}
    E_j(\theta) : = \ell_j^2(\theta) \left[ \bar{A}_j + \bar{B}_j(\theta - \theta_j^*) \right]:= \ell_j^2(\theta)[u_0 + u_1\theta] \quad \text{where} \quad \ell_j(\theta) = \prod_{q \neq j} \frac{\theta - \theta_q^*}{\theta_j^* - \theta_q^*},
\end{align*}
and $\bar{A}_j = 1$, $\bar{B}_j =  - 2\ell_j'(\theta_j^*)$, $u_0 = \bar{A}_j - \bar{B}_j \theta_j^*$ and $u_1 = \bar{B}_j$.

Now we employ Taylor expansion for $E_j(\theta_l)$ exactly around its true center $\theta_l^*$ to achieve $E_j(\theta_l) = \delta_{jl} + \frac{1}{2} E_j''(\xi_l) (\theta_{l} - \theta_{l}^{*})^2$. Thus, we arrive at $$ \int E_j(\theta) d\nu(\theta) = \sum_{l=1}^{k_*} \pi_{l} E_j(\theta_{l}) - \pi_j^* = (\pi_{j} - \pi_{j}^{*}) + \frac{1}{2} \sum_{l=1}^{k_*} \pi_{l} E_j''(\xi_{l}) (\theta_{l} - \theta_{l}^{*})^2.$$
Isolating $|\pi_{j} - \pi_{j}^{*}|$ and applying the results of Lemma~\ref{lemma:Hellinger_to_Polynomial}, we obtain
\begin{align*} 
|\pi_{j} - \pi_{j}^{*}| & \le \left| \int E_j(\theta) d\nu(\theta) \right| + \frac{1}{2} \sum_{l=1}^{k_*} \pi_{l} |E_j''(\xi_{l})| (\theta_{l} - \theta_{l}^{*})^2\\
&\leq C_{\mathrm{poly}}\|E_j\|_{\infty}d_H(p_{G},p_{G_*}) + \frac{1}{2}\max_{\theta \in [-R,R]} |E_{j}''(\theta)|  \sum_{l=1}^{k_*} \pi_{l}  (\theta_{l} - \theta_{l}^{*})^2.
\end{align*}
Lemma ~\ref{lemma:mass_test_function} implies that $\|E_j\|_{\infty} \leq C_{\mathrm{norm},E}\Delta^{-(2k_*-1)}_{\mathrm{sep}}$ and $\max_{\theta \in [-R,R]}|E''_j(\theta)| \leq C_{E,2}\Delta^{-(2k_*-1)}_{\mathrm{sep}}$. Thus, by substituting these estimations along with variance bound in equation~\eqref{eq:exact_no_group_third}, we achieve  
\begin{align}
    |\pi_{j} - \pi_{j}^{*}|\leq  C_{\mathrm{mass},1}\Delta_{\mathrm{sep}}^{-(4k_* - 3)} d_H(p_{G},p_{G_*}),
    \label{eqn:dung_thm3_3_univariate_local_mass_discrepancy_bound}
\end{align}
where 
\begin{equation*}
    C_{\mathrm{mass},1} = (2R)^{2k_*-2}C_{\mathrm{poly}}C_{\mathrm{norm},E} + \frac{1}{2}C_{E,2}C_{\mathrm{var},2}.
\end{equation*}
\emph{Step 5 - Bounding for $W_1(G,G_*)$ and conclusion.} We substitute the bounds in equations~\eqref{eqn:dung_thm_3.3_univariate_local_final_estimation_for_mean_discrepancy_ohlala} and ~\eqref{eqn:dung_thm3_3_univariate_local_mass_discrepancy_bound} back into the inequality~\eqref{eq:no_group_Wasserstein_local_bound}, and by defining 
\begin{equation*}
    C^{-1}_{\mathrm{local},3} = k_*\left[2RC_{\mathrm{mean},1} + C_{\mathrm{mass},1}\right]
\end{equation*}
and obtain that 
\begin{align*} W_1(G, G_*) \leq  C^{-1}_{\mathrm{local},3}\Delta_{\mathrm{sep}}^{-(4k_* - 3)} d_H(p_{G},p_{G_*})
\end{align*}
which implies
\begin{align*}
    d_H(p_{G},p_{G_*}) \ge C_{\mathrm{local},3} \cdot \Delta_{\mathrm{sep}}^{4k_* - 3} \cdot W_1(G, G_*).
\end{align*}

\subsubsection{Univariate Setting - Global Bound}
For the global bound, its Voronoi cell can contain more than one elements. Therefore, we start with the following upper bound on Wasserstein distance to take into account that setting. Throughout this proof, for each atom $\theta_i$, let $c(i) = \arg\min_{1\leq j \leq k_*}|\theta_i - \theta_j^*|$ be it absolute closest true center (while there exists several centers sharing minimal distance, we randomly choose one of them). For each $j \in [k_*]$, let $\tilde{\pi}_j = \sum_{i: c(i) = j} \pi_i$ and $\Delta\tilde{\pi}_j = \tilde{\pi}_j - \pi_j$. 

\vspace{0.5 em}
\noindent
\emph{Step 1 - Wasserstein decomposition.} Let $T: \mathbb{R} \to \Theta^* = \left\{\theta_{1}^{*}, \ldots, \theta_{k_{*}}^{*}\right\}$ be a mapping function defined as $T(\theta_i) = \theta_{c(i)}^*$. We define $\widetilde{G}$ as the exact push-forward measure of $G$ under $T$:
$$ \widetilde{G} = T_{\#} G = \sum_{i=1}^k \pi_i \delta_{T(\theta_i)} = \sum_{i=1}^k \pi_i \delta_{\theta_{c(i)}^*}.$$
We rewrite $\widetilde{G}$ on the true discrete support basis by grouping all mass mapped to the identical center $j$ as follows: $$\widetilde{G} = \sum_{j=1}^{k_*} \tilde{\pi}_j \delta_{\theta_j^*} \quad \text{where} \quad \tilde{\pi}_j = \sum_{i: c(i) = j} \pi_i.$$
The triangle inequality $W_{1}$ metric implies 
\begin{align}
    W_1(G, G_*) \leq W_1(G, \widetilde{G}) + W_{1}(\widetilde{G},G_*). \label{eq:exact_no_group_global_Wasserstein_bound_0}
\end{align} 
Firstly, we establish an upper bound $W_1(G, \widetilde{G})$. By considering the transportation plan $\hat{\gamma} = (\text{Id} \times T)_{\#} G$, we obtain that:
\begin{align} 
W_1(G, \widetilde{G}) \le \int_{\mathbb{R} \times \mathbb{R}} |x - y| d\hat{\gamma}(x,y) & = \int_{\mathbb{R}} |\theta - T(\theta)| dG(\theta) = \sum_{i=1}^{k_*} \pi_i |\theta_i - \theta_{c(i)}^*|. \label{eq:exact_no_group_global_Wasserstein_bound_1}
\end{align}
We now move to upper bound $W_1(\widetilde{G}, G_*)$. As in univariate local case, we consider the following transportation plan from $G'$ to $G_*$. While keeping a mass of $\min\{\tilde{\pi}_i,\pi^*_i\}$ for each $\theta^*_i$, we transfer a mass of quantity $\max\{\tilde{\pi}_i-\pi^*_i,0\}$ to the other center. Consequently, the total mass moved is exactly $\frac{1}{2}\sum_{j = 1}^{k_{*}} | \Delta\tilde{\pi}_{j}|$. As the distance between any true centers is less than $2R$, we achieve the following bound
\begin{equation}
\label{eqn:dung_thm_3_3_univariate_global_inequality_W_1_true_measure_and_intermediate}
    W_1(\tilde{G},G_*) \leq 2R \sum_{j=1}^{k_*} |\Delta\tilde{\pi}_{j}|. 
\end{equation}
Combining this result with equations~\eqref{eq:exact_no_group_global_Wasserstein_bound_0} and~\eqref{eq:exact_no_group_global_Wasserstein_bound_1}, we obtain the following unconditional global transportation bound:
\begin{align}
W_1(G, G_*)
\le \sum_{i=1}^{k_*} \pi_i \, |\theta_i - \theta_{c(i)}^*|
+ 2R\cdot \sum_{j=1}^{k_*} |\tilde{\Delta} \pi_j| 
\label{eq:exact_no_group_global_key_equation_first}
\end{align}

\vspace{0.5 em}
\noindent
\emph{Step 2 - Bound variance $\sum_{i=1}^{k_*}\pi_i|\theta_i-\theta^*_{c(i)}|^2$.}
We next consider the function $$P_{\mathrm{var}}(\theta) = \prod_{l=1}^{k_*} (\theta - \theta_l^*)^2$$ which is defined in Section \ref{sec:variance_test_function} and evaluate it globally over all atoms of $G$.
Moreover, Lemma~\ref{lemma:Hellinger_to_Polynomial} and Lemma~\ref{lemma:variance_test_function} indicate that
\begin{align*}
 \left|\int P_{\mathrm{var}}(\theta) d\nu(\theta)\right|\le C_{\mathrm{poly}}\|P_{\mathrm{var}}\|_{\infty}d_H(p_G,p_{G_*}) \leq C_{\mathrm{poly}}C_{\mathrm{var}}d_H(p_G,p_{G_*}).
\end{align*}
For any index $1 \le i \le k_*$ and $j \ne c(i)$, we have
\begin{align*}
    |\theta_{c(i)}^* - \theta_j^*|\leq |\theta_i - \theta_{c(i)}^*|+|\theta_i - \theta_j^*|\leq 2 |\theta_i - \theta_j^*|,
\end{align*}
which further implies that
$$
|\theta_i - \theta_j^*|
\;\ge\;
\frac{1}{2} |\theta_{c(i)}^* - \theta_j^*|
\;\ge\;
\frac{\Delta_{\mathrm{sep}}}{2}.
$$
Therefore, thanks to $\int P_{\mathrm{var}}(\theta)dG_*(\theta) = 0$, we obtain
\begin{align*} 
\int P_{\mathrm{var}}(\theta) d\nu(\theta) &= \sum_{i=1}^{k_*}\pi_iP_{\mathrm{var}}(\theta_i)  \ge  \left(\frac{1}{2}\Delta_{\mathrm{sep}}\right)^{2k_*-2} \sum_{i=1}^{k_*} \pi_i |\theta_{i} - \theta_{c(i)}^{*}|^2 \\
& \geq  \left(\frac{1}{2}\right)^{2k_*-2} \Delta_{\mathrm{sep}}^{2k_*-2} \sum_{j = 1}^{k_{*}} \sum_{i \in \mathcal{V}_{j}} \pi_{i} |\theta_{i} - \theta_{j}^{*}|^2.
\end{align*}
As a result, we arrive at
\begin{align}
    \sum_{j = 1}^{k_{*}} \sum_{i \in \mathcal{V}_{j}} \pi_{i} |\theta_{i} - \theta_{j}^{*}|^2 &\le \left[ \frac{C_{\mathrm{var}}C_{\mathrm{poly}}}{(1/2)^{2k_*-2}} \right]  \Delta_{\mathrm{sep}}^{-(2k_* - 2)} d_H(p_{G},p_{G_*}) \nonumber\\
    &: = C_{\mathrm{var},2} \Delta_{\mathrm{sep}}^{-(2k_* - 2)} d_H(p_{G},p_{G_*}). \label{eq:exact_no_group_global_third} 
\end{align}

\vspace{0.5 em}
\noindent
\emph{Step 3 - Bounding mass discrepancy $|\tilde{\Delta} \pi_j|$.} To obtain an upper bound for $|\tilde{\Delta} \pi_j|$ with $j \in [k_*]$, we consider a polynomial $E_j$ such that $E_j(\theta_l^*) = \delta_{jl}$ and $E'_j(\theta_l^*) = 0$. We can choose the Hermite interpolation polynomial $E_j$ of degree $2k_*-1$ in Section  \ref{sec:point_wise_mass_extractor_non_multicluster} given by 
\begin{align*}
    E_{j}(\theta) : = \ell_j^2(\theta) \left[ \bar{A}_j + \bar{B}_j(\theta - \theta_j^*) \right], \quad \text{where} \quad \ell_j(\theta) = \prod_{q \neq j} \frac{\theta - \theta_q^*}{\theta_j^* - \theta_q^*},
\end{align*}
and $\bar{A}_j = 1$, $\bar{B}_j = - 2\ell_j'(\theta_j^*)$. By means of Taylor expansion for $E_j(\theta_l)$ exactly around its true center $\theta_l^*$, we have $E_j(\theta_l) = \delta_{jl} + \frac{1}{2} E_j''(\xi_l) (\theta_{l} - \theta_{c(l)}^{*})^2$, where $\xi_l$ is between $\theta_l$ and $\theta_{c(l)}^*$ which implies that
$$ \int E_j(\theta) d\nu(\theta) = \left(\sum_{i \in \mathcal{V}_{j}} \pi_{i} - \pi_{j}^{*}\right) + \frac{1}{2}\sum_{l=1}^{k_*} \sum_{i \in \mathcal{V}_l} \pi_{i} E_j''(\xi_i) (\theta_{i} - \theta_{l}^{*})^2.$$ 
Applying the triangle inequality and Lemma \ref{lemma:Hellinger_to_Polynomial} leads to
\begin{align}
    |\tilde{\Delta} \pi_j| &\le \left| \int E_j(\theta) d\nu(\theta) \right| + \frac{1}{2}\sum_{l=1}^{k_*} \sum_{i \in \mathcal{V}_l} \pi_{i} |E_j''(\xi_i)| (\theta_{i} - \theta_{l}^{*})^2\nonumber\\
    &\leq C_{\mathrm{poly}}\|E_j\|_{\infty}d_H(p_{G},p_{G_*})+ \frac{1}{2}\max_{\theta \in [-R,R]} |E_{j}''(\theta)| \sum_{l=1}^{k_*} \sum_{i \in \mathcal{V}_l} \pi_{i}  (\theta_{i} - \theta_{l}^{*})^2, \label{eq:exact_no_group_global_key_equation_fifth}
\end{align}
Using Lemma \ref{lemma:mass_test_function}, we have $\|E_j\|_{\infty} \leq C_{\mathrm{norm},E}$, and $\max_{\theta \in [-R,R]}|E''_j(\theta)| \leq C_{E,2}\Delta^{-(2k_*-1)}_{\mathrm{sep}}$. Combining these estimations with the bound for variance in equation~\eqref{eq:exact_no_group_global_third}, we have
\begin{align}
 |\tilde{\Delta} \pi_j|\leq  C_{\mathrm{mass},2} \Delta_{\mathrm{sep}}^{-(4k_* - 3)} d_H(p_{G},p_{G_*}), \label{eq:exact_no_group_global_key_equation_ninth}
\end{align}
where 
\begin{equation*}
    C_{\mathrm{mass},2} = (2R)^{2k_*-2}C_{\mathrm{poly}}C_{\mathrm{norm},E} + \frac{1}{2}C_{E,2}C_{\mathrm{var},2}.
\end{equation*}

\vspace{0.5 em}
\noindent
\emph{Step 4 - Upper bound on $\sum_{j = 1}^{k_{*}} \sum_{i \in \mathcal{V}_{j}} \pi_{i} |\theta_{i} - \theta_{j}^{*}|$.} We utilize a test polynomial $H_j$ satisfying $H_j(\theta_l^*) = 0$ and $H_j'(\theta_l^*) = \delta_{jl}$, for all $1 \leq l \leq k_{*}$. We can choose the polynomial $H_j$ of degree $2k_*-1$ which is defined in Section \ref{sec:non_multi_cluster_mean_test_function} given by:
\begin{align*}
    H_j(\theta) : = \ell_j^2(\theta) (\theta - \theta_j^*)  \quad \text{where} \quad \ell_j(\theta) : = \prod_{q \neq j} \frac{\theta - \theta_q^*}{\theta_j^* - \theta_q^*},
\end{align*}

By utilizing Taylor expansion $H_j(\theta_l)$ exactly around its true center $\theta_l^*$, we arrive at
$$ \int H_j(\theta) d\nu(\theta) = \sum_{i \in \mathcal{V}_{j}} \pi_{i}(\theta_{i} - \theta_{j}^{*}) + \frac{1}{2} \sum_{l=1}^{k_*} \sum_{i \in \mathcal{V}_{l}} \pi_{l} H_j''(\xi_l) (\theta_{i} - \theta_{l}^{*})^2.$$
 
We now divide our argument into two cases:

\emph{Case 4.1.} Suppose that there exists a Voronoi cell $\mathcal{V}_j$ containing more than one element. Without loss of generality, we may assume that the cell $\mathcal{V}_1$ is empty. In this case, it follows directly that $| \sum_{i \in \mathcal{V}_1} \pi_i - \pi_1^* | \;=\; \pi_1^*$. 
Consequently, the bound in equation~\eqref{eq:exact_no_group_global_key_equation_ninth} implies that
\begin{align}
    \pi_{\text{min}}^{*}\leq \pi_{1}^{*} \leq C_{\mathrm{mass},2} \Delta_{\mathrm{sep}}^{-(4k_*-3)} d_H(p_{G},p_{G_*}). \label{eq:exact_no_group_global_key_equation_tenth}
\end{align}
Now, since $|\theta_{i} - \theta_{c(i)}^{*}| \leq 2R$ for all $1\leq i\leq k_*$, we have
\begin{align}
    \sum_{j = 1}^{k_{*}} \sum_{i \in \mathcal{V}_{j}} \pi_{i} |\theta_{i} - \theta_{j}^{*}| \leq 2R \sum_{i = 1}^{k_{*}} \pi_{i} = 2R. \label{eq:exact_no_group_global_key_equation_eleventh}
\end{align}
Combining the bounds in equations~\eqref{eq:exact_no_group_global_key_equation_tenth} and~\eqref{eq:exact_no_group_global_key_equation_eleventh} leads to
\begin{align}
    \sum_{j = 1}^{k_{*}} \sum_{i \in \mathcal{V}_{j}} \pi_{i} |\theta_{i} - \theta_{j}^{*}| & \leq 2R \cdot \dfrac{C'_{\text{mass}} \Delta_{\mathrm{sep}}^{-(4k_* - 3)} d_H(p_{G},p_{G_*})}{\pi_{\text{min}}^{*}} \nonumber \\
    & := C_{\mathrm{mean,near},1} (\pi_{\text{min}}^{*})^{-1}\Delta_{\mathrm{sep}}^{-(4k_*- 3)}d_H(p_{G},p_{G_*}). \label{eq:exact_no_group_global_key_equation_twelth}
\end{align}
\emph{Case 4.2.} Assume that no Voronoi cell $\mathcal{V}_j$ contains more than one element, so that each cell $\mathcal{V}_j$ contains exactly one element. In this case, by applying an argument analogous to that used in the local bound together with the estimate in equation~\eqref{eq:exact_no_group_global_third}, we obtain
\begin{align*}
\sum_{i \in \mathcal{V}_{j}} \pi_{i}|\theta_{i} - \theta_{j}^{*}| \leq \left( (2R)^{2k_{*} - 2}C_{\mathrm{poly}}C_{\mathrm{norm},E} + \frac{1}{2} C_{E,2}C_{\mathrm{var},2} \right)\Delta_{\mathrm{sep}}^{-(4k_*-4)} d_H(p_{G},p_{G_*}).
\end{align*}
which implies
\begin{align}
\sum_{j = 1}^{k_{*}} \sum_{i \in \mathcal{V}_{j}} \pi_{i} |\theta_{i} - \theta_{j}^{*}| & \leq C_{\mathrm{mean,near,2}}  \Delta_{\mathrm{sep}}^{-(4k_* - 4)}d_H(p_{G},p_{G_*})
\label{eqn:dung_thm_3_3_univariate_local_mean_bound_with_one_element_each_cell}
\end{align}
Combining the results in equations \eqref{eq:exact_no_group_global_key_equation_twelth} and \eqref{eqn:dung_thm_3_3_univariate_local_mean_bound_with_one_element_each_cell} and let $C_{\mathrm{mean},2} := \max\{C_{\mathrm{mean,near},1}, 2R\cdot C_{\mathrm{mean,near},2}\}$, we have 
\begin{align}
\label{eqn:dung_thm_3_3_univariate_global_final_upper_bound_for_first_moment}
    \sum_{j = 1}^{k_{*}} \sum_{i \in \mathcal{V}_{j}} \pi_{i} |\theta_{i} - \theta_{j}^{*}| &\leq C_{\mathrm{mean},2}(\pi_{\text{min}}^{*})^{-1}\Delta_{\mathrm{sep}}^{-(4k_* - 3)}d_H(p_{G},p_{G_*}).
\end{align}
\emph{Step 5 - Bounding $W_1(G,G_*)$ and conclusion.} Now we put everything together. Plugging the bounds in equations \eqref{eq:exact_no_group_global_key_equation_ninth} and \eqref{eqn:dung_thm_3_3_univariate_global_final_upper_bound_for_first_moment} into equation~\eqref{eq:exact_no_group_global_key_equation_first}, we achieve
\begin{align}
\label{eqn:dung_thm3_global_final_upper_bound_W1}
d_H(p_{G},p_{G_*}) \geq C_{\mathrm{global},3}\cdot \pi_{\text{min}}^{*}\cdot\Delta_{\mathrm{sep}}^{4k_* - 3}W_1(G,G_*),
\end{align}
where $C_{\mathrm{global},3} := (2R\cdot k_*\cdot C_{\mathrm{mass},2} + C_{\mathrm{mean},2})^{-1}$. 
\subsubsection{Multivariate setting - Global bound}
\label{sec:global_bound_exact_multivariate_no_cluster}
The proof of this part follows identically the arguments used in the univariate global case, except that all test functions are now multivariate polynomials. Thus, for brevity, we present only the core arguments and omit the repetitive technical details. Before proceeding with the proof, for each atom $\theta_i$, we denote $c(i) = \arg\min_{j \in [k_*]}\|\theta_i - \theta_j^*\|$ (in the case where several centers attain the minimal distance, we choose one of them arbitrarily). For each $j \in [k_*]$, let $\tilde{\pi}_j = \sum_{i: c(i) = j} \pi_i$ and $\Delta\tilde{\pi}_j = \tilde{\pi}_j - \pi_j$. 

\vspace{0.5 em}
\noindent
\emph{Step 1 - Wasserstein decomposition.} By splitting the transportation through intermediate measure $\tilde{G} = \sum_{j=1}^{k_*}\tilde{\pi}_j\delta_{\theta_j^*}$, 
we have the following estimation 
\begin{align}
W_1(G, G_*)
&\le \sum_{i=1}^{k_*} \pi_i \, \|\theta_i - \theta_{c(i)}^*\|+2R\cdot  \sum_{j=1}^{k_*} |\tilde{\Delta} \pi_j| \notag \\
&= \sum_{j = 1}^{k_*} \sum_{i \in \mathcal{V}_j} \pi_i \, \|\theta_i - \theta_j^*\|
+2R\cdot\sum_{j = 1}^{k_*} |\tilde{\Delta} \pi_j|.
\label{eq:dung_thm_3_3_multivariate_global_exact_no_group_global_key_equation_first}
\end{align}

\vspace{0.5 em}
\noindent
\emph{Step 2 - Bounding variance $\sum_{j=1}^{k_*}\sum_{i\in \mathcal{V}_j}\pi_i\|\theta_i-\theta_j^*\|^2$: } In this part, we use the test function  $P_{\mathrm{var}}$ defined in Section \ref{sec:variance_test_function} given by 
\begin{equation*}
    P_{\mathrm{var}}(\theta) : = \prod_{l=1}^{k_*} \|\theta - \theta_l^*\|^2.
\end{equation*}
By using the same argument as in univariate global part, we have 
\begin{align}
\label{eqn:dung_thm_3.3_multivariate_global_sum_of_variance}
\sum_{j=1}^{k_*}\sum_{i\in \mathcal{V}_j}\pi_i \|\theta_i-\theta_j^*\|^2\leq C_{\mathrm{var},3}\Delta^{-(2k_*-2)}_{\mathrm{sep}}d_H(p_G,p_{G_*}). 
\end{align}
\emph{Step 3 - Bounding mass discrepancy $|\tilde{\Delta} \pi_j|$.} In this part, we would like to pick a test function $E_j$ such that $E_j(\theta_i) = \delta_{ji}$ and $\nabla E_j(\theta_i) = 0$ for all $i \in [k_*]$. We can choose the mass extracting function $E_j$ defined in Section \ref{sec:point_wise_mass_extractor_non_multicluster} given by
\begin{align*}
    E_j(\theta) : = \ell_j^2(\theta) \left[ \bar{A}_j + \langle\bar{B}_j,\theta - \theta_j^*\rangle\right], \quad \text{where} \quad \ell_j(\theta) = \prod_{q \neq j} \frac{\|\theta - \theta_q^*\|_2}{\|\theta_j^* - \theta_q^*\|_2},
\end{align*}
and $\bar{A}_j = 1$, $\bar{B}_j = - 2\nabla\ell_j(\theta_j^*)$. From the property of $E_{j}$, we can express 
\begin{equation*}
    \tilde{\Delta}\pi_j = \int E_j(\theta)d\nu(\theta) - \frac{1}{2}\sum_{l=1}^{k_*}\sum_{i \in \mathcal{V}_l}\pi_i(\theta_i-\theta^*_{l})^\top D^2E_j(\xi_i)(\theta_i-\theta^*_{l})
\end{equation*}
Using the results from Lemma~\ref{lemma:Hellinger_to_Polynomial}, Lemma~\ref{lemma:mass_test_function} and the estimation in equation~\eqref{eqn:dung_thm_3.3_multivariate_global_sum_of_variance} as in proof of univariate global part, we have 
\begin{equation}
    |\tilde{\Delta}\pi_j| \leq C_{\mathrm{mass},3}\Delta_{\mathrm{sep}}^{-(4k_* - 3)} d_H(p_{G},p_{G_*}). \label{eq:dung_thm3_3_multivariate_global_exact_no_group_global_key_equation_ninth}
\end{equation}

\vspace{0.5 em}
\noindent
\emph{Step 4 - Bounding mass discrepancy $\sum_{j = 1}^{k_{*}} \sum_{i \in \mathcal{V}_{j}} \pi_{i} \|\theta_{i} - \theta_{j}^{*}\|$.} For an index $i \in \mathcal{V}_j$, we choose a Hermite interpolation polynomial $H_{j,i}$ such that $H_{j,i}(\theta_l^*) = 0$ and $\nabla H_{j,i}(\theta_l^*) = v_{j,i}\delta_{lj}$, where $\langle v_{j,i}, \theta_i - \theta_j^*\rangle = \|\theta_i - \theta_j^*\|$. We can choose $H_{j,i}$ defined in Section \ref{sec:non_multi_cluster_mean_test_function} given by 
\begin{align*}
    H_{j,i}(\theta) : = \ell_j^2(\theta) \left[\langle v_{j,i}, \theta - \theta_j^*\rangle\right],  \quad \text{where} \quad \ell_j(\theta) : = \prod_{q \neq j} \frac{\|\theta - \theta_q^*\|_2}{\|\theta_j^* - \theta_q^*\|_2}, 
\end{align*}
and $v_{j,i} = {(\theta_i-\theta_j^*)}/{\|\theta_i-\theta_j^*\|}$ if $\theta_i \neq \theta_j^*$ and $v_{j,i}=(1, 0,\ldots ,0)^\top$ otherwise. By utilizing Taylor expansion $H_{j,i}(\theta_l)$ exactly around its true center $\theta_l^*$, we arrive at
$$ \int H_{j,i}(\theta) d\nu(\theta) = \sum_{i \in \mathcal{V}_{j}} \pi_{i}\langle v_{j,i},\theta_{i} - \theta_{j}^{*}\rangle + \frac{1}{2} \sum_{l=1}^{k_*} \sum_{i \in \mathcal{V}_{l}} \pi_{l} (\theta_{i} - \theta_{l}^{*})^{\top}D^2H_{j,i}(\xi_l) (\theta_{i} - \theta_{l}^{*}).$$

We now divide our argument into two cases:

\emph{Case 4.1.} Suppose that there exists a Voronoi cell $\mathcal{V}_j$ containing more than one element. Without loss of generality, we may assume that the cell $\mathcal{V}_1$ is empty. In this case, it follows directly that $| \sum_{i \in \mathcal{V}_1} \pi_i - \pi_1^* | \;=\; \pi_1^*$. 
Consequently, the bound in equation~\eqref{eq:dung_thm3_3_multivariate_global_exact_no_group_global_key_equation_ninth} implies that
\begin{align}
    \pi_{\text{min}}^{*}\leq \pi_{1}^{*} \leq C_{\mathrm{mass,3}} \Delta_{\mathrm{sep}}^{-(4k_*-3)} d_H(p_{G},p_{G_*}). \label{eq:dung_thm_3_3_multivariate_global_exact_no_group_global_key_equation_tenth}
\end{align}
Using the same argument as in univariate global part, we have 
\begin{align}
    \sum_{j = 1}^{k_{*}} \sum_{i \in \mathcal{V}_{j}} \pi_{i} \|\theta_{i} - \theta_{j}^{*}\| \leq  \bar{C}_{\mathrm{mean,near,1}} (\pi_{\text{min}}^{*})^{-1}\Delta_{\mathrm{sep}}^{-(4k_*- 3)}d_H(p_{G},p_{G_*}).\label{eq:dung_thm3_3_multivariate_global_exact_no_group_global_key_equation_tenth}
\end{align}
\emph{Case 4.2.} Assume that no Voronoi cell $\mathcal{V}_j$ contains more than one element, so that each cell $\mathcal{V}_j$ contains exactly one element. Without loss of generality, we can suppose that $\theta_j \in \mathcal{V}_j$. In this case, by applying an argument analogous to that used in the local bound together with the estimation in equation~\eqref{eqn:dung_thm_3.3_multivariate_global_sum_of_variance}, we obtain
\begin{align*}
\|\theta_{i} - \theta_{i}^{*}\|_2 \leq \left((2R)^{2k_{*} - 2}C_{\mathrm{poly}}C_{\mathrm{norm},E} + \frac{1}{2} C_{\mathrm{var},E} C_{E,2}\right) \Delta_{\mathrm{sep}}^{-(4k_*-4)} d_H(p_{G},p_{G_*}). 
\end{align*}
As a consequence, we get
\begin{align}
\sum_{i = 1}^{k_{*}} \pi_{i} \|\theta_{i} - \theta_{i}^{*}\| \leq  \bar{C}_{\mathrm{mean,near,2}} \Delta_{\mathrm{sep}}^{-(4k_* - 4)}d_H(p_{G},p_{G_*})
\label{eq:exact_global_key_equation_thirteen_2} 
\end{align}
Combining the results in equations \eqref{eq:dung_thm3_3_multivariate_global_exact_no_group_global_key_equation_tenth} and \eqref{eq:exact_global_key_equation_thirteen_2} and let $C_{\mathrm{mean},3} := \max\{\bar{C}_{\mathrm{mean,near,1}}, 2R\bar{C}_{\mathrm{mean,near,2}}\}$, we have 
\begin{align}
\label{eqn:dung_thm3_global_final_upper_bound_for_first_moment_2}
    \sum_{j = 1}^{k_{*}} \sum_{i \in \mathcal{V}_{j}} \pi_{i} \|\theta_{i} - \theta_{j}^{*}\| &\leq C_{\mathrm{mean},3}(\pi_{\text{min}}^{*})^{-1}\Delta_{\mathrm{sep}}^{-(4k_* - 3)}d_H(p_{G},p_{G_*}).
\end{align} 
\emph{Step 5 - Bounding $W_1(G,G_*)$ and conclusion.} Now we put everything together. Plugging in the bounds in equations \eqref{eq:dung_thm3_3_multivariate_global_exact_no_group_global_key_equation_ninth} and \eqref{eqn:dung_thm3_global_final_upper_bound_for_first_moment_2} into equation~\eqref{eq:dung_thm_3_3_multivariate_global_exact_no_group_global_key_equation_first}, we achieve
\begin{align}
\label{eqn:dung_thm3_global_final_upper_bound_W1_2}
d_H(p_{G},p_{G_*}) \geq \bar{C}_{\mathrm{global},3}\cdot \pi_{\text{min}}^{*}\cdot\Delta_{\mathrm{sep}}^{4k_* - 3}W_1(G,G_*),
\end{align}
where $\bar{C}_{\mathrm{global},3} := (2R\cdot k_*\cdot C_{\mathrm{mass},3} + C_{\mathrm{mean},3}')^{-1}$. 

\subsection{Proof of Theorem~\ref{theorem:one_group_univariate}}
\label{sec:proof_theorem:one_group_multivariate_global}
In this proof, we only consider the multivariate global situation. For each fitted atom $\theta_i$, we define $c(i) = \text{argmin}_j \|\theta_i - \theta_j^*\|_2$ as the index of true center $\theta_{j}^{*}$ that is closest to that fitted atom (while there exists several centers sharing minimal distance, we randomly choose one of them). We then partition the atoms of $G$ into two disjoint groups according to a fixed threshold of $\Delta_{\mathrm{sep}}/4$:
\begin{itemize}
    \item The near set $\mathcal{M}_{\mathrm{near}} = \{ i \in \{1, \dots, k\} : \|\theta_i-\theta^*_{c(i)}\|_2 \le \Delta_{\mathrm{sep}} / 4 \}$. Let $\bar{\mathcal{V}}_j = \{i \in \mathcal{M}_{\mathrm{near}} : c(i) = j\}$ be a subset of the Voronoi cells $\mathcal{V}_{j}$.
    \item The far set $\mathcal{M}_{\mathrm{far}} = \{ i \in \{1, \dots, k\} : \|\theta_i-\theta^*_{c(i)}\| > \Delta_{\mathrm{sep}} / 4 \}$. The total far mass is given by $\pi_{\mathrm{far}} = \sum_{i \in \mathcal{M}_{\mathrm{far}}} \pi_i$.
\end{itemize}
In this proof, for each $j \in [k_*]$, we denote $\tilde{\pi}_j = \sum_{i: c(i) = j} \pi_i$, $\Delta\tilde{\pi}_j = \tilde{\pi}_j - \pi_j$, and $\Delta\pi_j = \left(\sum_{i\in \bar{\mathcal{V}}_j}\pi_i\right) - \pi_j$.

\vspace{0.5 em}
\noindent
\emph{Step 1 - Wasserstein decomposition.} To upper bound $W_2^2(G,G_*)$, we consider the intermediate measure $G' = \sum_{j=1}^k\pi_j\delta_{\theta^*_{c(j)}} = \sum_{j=1}^{k_*}\tilde{\pi}_j\delta_{\theta^*_j}$, then triangle inequality for Wasserstein distance gives us 
\begin{align}
\label{eqn:dung_thm4_multivariate_global_W2_prelim_split}
    W_2^2(G, G_*) \leq(W_2(G, G') + W_2(G', G_*))^2 \leq 2(W_2^2(G, G') + W^2_2(G', G_*)).
\end{align}
For $W_2^2(G, G')$, by taking transportation plan $\rho_{jj} = \pi_j$ for all $1\leq j\leq k$, we achieve 
\begin{align}
\label{eqn:dung_thm4_multivariate_global_W2_prelim_first_bound}
    W_2^2(G, G') \leq \sum_{i=1}^{k} \pi_i\|\theta_i-\theta^*_{c(i)}\|^2 = \sum_{j=1}^{k_*}\sum_{i \in \bar{\mathcal{V}}_j}\pi_i\|\theta_i - \theta_j^*\|^2 + \sum_{i \in \mathcal{M}_{\mathrm{far}}} \pi_i\|\theta_i - \theta^*_{c(i)}\|^2
\end{align}
For $W_2^2(G', G_*)$, we consider the transportation plan such that for each $j$, we keep $\min\{\sum_{i\in\mathcal{V}_j}\pi_i,\pi^*_j\}$ for the center $\theta_j^*$ while moving $|\tilde{\Delta}\pi_j|$. Then, the total mass transported is exactly $\frac{1}{2} \sum |\Delta \pi_j|$, and largest squared distance between any true centers is $(C_0 \Delta_{\mathrm{sep}})^2$, we have 
\begin{align}
\label{eqn:dung_thm4_local_W2_prelim_second_bound}
    W_2^2(G', G_*) \leq \frac{1}{2} C_0^2 \Delta_{\mathrm{sep}}^2 \sum_{j=1}^{k_*} |\tilde{\Delta}\pi_j| \leq \frac{1}{2} C_0^2 \Delta_{\mathrm{sep}}^2 \left(\sum_{j=1}^{k_*} |\Delta\pi_j| + \pi_{\mathrm{far}}\right). 
\end{align}
By plugging the results from equations \eqref{eqn:dung_thm4_multivariate_global_W2_prelim_first_bound} and \eqref{eqn:dung_thm4_local_W2_prelim_second_bound} into equation \eqref{eqn:dung_thm4_multivariate_global_W2_prelim_split}, we have 
\begin{align}
\nonumber
W_2^2(G, G_*) &\le 2\sum_{j=1}^{k_*}\sum_{i \in \bar{\mathcal{V}}_j}\pi_i\|\theta_i - \theta_j^*\|^2 + 2\sum_{i \in \mathcal{M}_{\mathrm{far}}} \pi_i\|\theta_i - \theta^*_{c(i)}\|^2 + C_0^2 \Delta_{\mathrm{sep}}^2 \left(\sum_{j=1}^{k_*} |\Delta \pi_j| + \pi_{\mathrm{far}}\right)\\
\nonumber
&\leq \frac{\Delta_{\mathrm{sep}}}{2}\sum_{j:|\bar{\mathcal{V}_j}|=1,\mathcal{V}_j = \{i\}} \pi_i\|\theta_i - \theta^*_{j}\| +  2\sum_{j:|\bar{\mathcal{V}_j}|>1}\sum_{i \in \bar{\mathcal{V}}_j}\pi_i\|\theta_i - \theta_j^*\|^2 + 2\sum_{i \in \mathcal{M}_{\mathrm{far}}} \pi_i\|\theta_i - \theta^*_{c(i)}\|^2 \\
&\hspace{2cm}+ C_0^2 \Delta_{\mathrm{sep}}^2 \left(\sum_{j=1}^{k_*} |\Delta \pi_j| + \pi_{\mathrm{far}}\right)
\label{eqn:dung_thm4_multivariate_global_W2_prelim_bound}
\end{align}
Let 
\begin{align}
\label{eqn:dung_thm_4_1_Voronoi_cell_distance}
    S_1(G,G_*) &= \frac{\Delta_{\mathrm{sep}}}{2}\sum_{j:|\bar{\mathcal{V}_j}|=1,\mathcal{V}_j = \{i\}} \pi_i\|\theta_i - \theta^*_{j}\| +  2\sum_{j:|\bar{\mathcal{V}_j}|>1}\sum_{i \in \bar{\mathcal{V}}_j}\pi_i\|\theta_i - \theta_j^*\|^2 + 2\sum_{i \in \mathcal{M}_{\mathrm{far}}} \pi_i\|\theta_i - \theta^*_{c(i)}\|^2 \\
&\hspace{2cm}+ C_0^2 \Delta_{\mathrm{sep}}^2 \left(\sum_{j=1}^{k_*} |\Delta \pi_j| + \pi_{\mathrm{far}}\right),
\end{align}
we prove a more stronger result
\begin{equation*}
    d_H(p_G,p_{G_*})\geq C_{\mathrm{global},4}\Delta_{\mathrm{sep}}^{2k_*-2}S_1(G,G_*).
\end{equation*}

\vspace{0.5 em}
\noindent
\emph{Step 2 - Bounding near variance and far high moment for variance.} In this part, we utilize the test function $P_{\mathrm{var}}$ define in Section \ref{sec:variance_test_function}: 
\begin{equation*}
    P_{\mathrm{var}}(\theta) = \prod_{l=1}^{k_*}\|\theta-\theta_l^*\|^2_2
\end{equation*}
Using Lemma \ref{lemma:variance_test_function}, we achieve $\|P_{\mathrm{var}}\|_{\infty} \leq C_{\mathrm{var}}$. Using Lemma \ref{lemma:Hellinger_to_Polynomial}, we have 
\begin{equation}
    \label{eqn:dung_thm_4.1_integral_variance_test_function_bound}
    \left|\int P_{\mathrm{var}}(\theta) d\nu(\theta)\right| \leq C_{\mathrm{poly}}C_{\mathrm{var}}d_H(p_G,p_{G_*}). 
\end{equation}

We partition the integral 
\begin{align}
    \int P_{\mathrm{var}}(\theta) dG(\theta) = \underbrace{\sum_{i \in \mathcal{M}_{\mathrm{near}}} \pi_i P_{\mathrm{var}}(\theta_i)}_{\geq 0} + \underbrace{\sum_{i \in \mathcal{M}_{\mathrm{far}}} \pi_i P_{\mathrm{var}}(\theta_i)}_{\geq 0}. \label{eqn:dung_thm_4.1_multivariate_global_integral_of_P_var_split}
\end{align}
For any index $i \in \mathcal{M}_{\mathrm{near}}$, we have 
\begin{align*}
    P_{\mathrm{var}}(\theta_i) =\|\theta_i-\theta^*_{c(i)}\|_2^2\cdot \prod_{l\neq c(i)} \|\theta_i-\theta_l^*\|_2^2. 
\end{align*}
For any $j\neq c(i)$, we have 
\begin{equation*}
    \|\theta_i - \theta_j^*\|_2\geq  \|\theta_j - \theta_{c(i)}^*\|_2 - \|\theta_i - \theta_{c(i)}^*\|_2  \geq\Delta_{\mathrm{sep}} - \frac{\Delta_{\mathrm{sep}}}{4} = \frac{3\Delta_{\mathrm{sep}}}{4}. 
\end{equation*}
As a consequence, the estimation in equation~\eqref{eqn:dung_thm_4.1_multivariate_global_integral_of_P_var_split} implies
\begin{equation*}
    \int P_{\mathrm{var}}(\theta)d\nu(\theta)\geq \sum_{i \in \mathcal{M}_{\mathrm{near}}}\pi_iP_{\mathrm{var}}(\theta_i)   \geq  \left(\frac{3}{4}\right)^{2k_*-2} \Delta^{2k_*-2}_{\mathrm{sep}}\sum_{i \in \mathcal{M}_{\mathrm{near}}} \pi_i\|\theta_i - \theta_{c(i)}\|^2. 
\end{equation*}
Combining this result with equation~\eqref{eqn:dung_thm_4.1_integral_variance_test_function_bound}, we have 
\begin{align}
\nonumber
\sum_{i \in \mathcal{M}_{\mathrm{near}}} \pi_i\|\theta_i - \theta_{c(i)}\|^2 &\leq \left[\frac{C_{\mathrm{var}}C_{\mathrm{poly}}}{(3/4)^{2k_*-2}}\right]\Delta_{\mathrm{sep}}^{-(2k_*-2)}d_H(p_G,p_{G_*})\\
&:=C_{\mathrm{var},1} \Delta_{\mathrm{sep}}^{-(2k_*-2)}d_H(p_G,p_{G_*}). 
\label{eqn:dung_thm_4.1_multivariate_global_near_variance_estimation}
\end{align}

\vspace{0.5 em}
\noindent
\emph{Step 3 - Bounding far mass $\pi_{\mathrm{far}}$.} For any index $i \in \mathcal{M}_{\mathrm{far}}$, we have $\|\theta_i - \theta^*_{c(i)}\|_2\leq \|\theta_i - \theta_j^*\|_2$ for any $j$, thus $P_{\mathrm{var}}(\theta_i) \geq \|\theta_i-\theta^*_{c(i)}\|_2^{2k_*}$. Thus, from the estimations in equations \eqref{eqn:dung_thm_4.1_integral_variance_test_function_bound} and \eqref{eqn:dung_thm_4.1_multivariate_global_integral_of_P_var_split}, we have 
\begin{align}
\nonumber
    \sum_{i \in \mathcal{M}_{\mathrm{far}}}\pi_i\|\theta_i-\theta^*_{c(i)}\|^{2k_*} &\leq \sum_{i \in \mathcal{M}_{\mathrm{far}}}\pi_iP_{\mathrm{var}}(\theta_i)\\
    &\leq C_{\mathrm{var}}C_{\mathrm{poly}}d_H(p_G,p_{G_*}) := C_{\mathrm{far,moment},1}d_H(p_G,p_{G_*}). 
\label{eqn:dung_thm_4.1_multivariate_far_high_moment_estimation} 
\end{align}
As $\|\theta_i - \theta^*_{c(i)}\|\geq \Delta_{\mathrm{sep}}/4$, we have $\|\theta_i - \theta^*_{c(i)}\|^{2k_*} \geq 4^{-(2k_*-2)}\Delta^{2k_*-2}_{\mathrm{sep}}\|\theta_i - \theta^*_{c(i)}\|^2$. A byproduct of this result and equation \eqref{eqn:dung_thm_4.1_multivariate_far_high_moment_estimation} is the following estimation 
\begin{align}
     \sum_{i \in \mathcal{M}_{\mathrm{far}}}\pi_i\|\theta_i-\theta^*_{c(i)}\|^{2} \leq C_{\mathrm{far,var}}\Delta^{-(2k_*-2)}_{\mathrm{sep}}d_H(p_G,p_{G_*}),
\label{eqn:dung_thm_4.1_multivariate_far_second_moment_estimation}
\end{align}
where $C_{\mathrm{far,var}} = 4^{-(2k_*-2)} C_{\mathrm{far,moment},1}$.
In addition, as $i \in \mathcal{M}_{\mathrm{far}}$, we have $\|\theta_i - \theta^*_{c(i)}\| \geq \Delta_{\mathrm{sep}}/4$, thus $\|\theta_i - \theta^*_{c(i)}\|^{2k_*} \leq 4^{2k_*}\Delta^{-2k_*}_{\mathrm{sep}}$. Thus, the high moment estimation for far centers in equation~\eqref{eqn:dung_thm_4.1_multivariate_far_high_moment_estimation} implies 
\begin{align}
    \pi_{\mathrm{far}} = \sum_{i \in \mathcal{M}_{\mathrm{far}}} \pi_i \le \left(\frac{4}{\Delta_{\mathrm{sep}}}\right)^{2k_*} \sum_{i \in \mathcal{M}_{\mathrm{far}}} \pi_i \|\theta_i - \theta^*_{c(i)}\|^{2k_*} \le C_{\mathrm{far,mass}} \Delta_{\mathrm{sep}}^{-2k_*} d_H(p_{G},p_{G_*}), \label{eqn:dung_thm_4_1_multivariate_global_pi_far_estimation}
\end{align} 
where $C_{\mathrm{far,mass}} := 4^{2k_*}C_{\mathrm{far,moment,1}}.$

\vspace{0.5 em}
\noindent

\emph{Step 4 -  Bounding mass discrepancy $|\Delta\pi_j|$.}
We now obtain an upper bound for $|\Delta \pi_j|$. To do that, we consider the mass extracting function $E_j$ such that $E_j(\theta_i) = \delta_{ji}$ and $\nabla E_j(\theta_i) = 0$ for all $i \in [k_*]$. We can choose the mass extracting function $E_j$ defined in Section \ref{sec:point_wise_mass_extractor_non_multicluster} given by
\begin{align*}
    E_j(\theta) : = \ell_j^2(\theta) \left[ \bar{A}_j + \langle\bar{B}_j,\theta - \theta_j^*\rangle\right], \quad \text{where} \quad \ell_j(\theta) = \prod_{q \neq j} \frac{\|\theta - \theta_q^*\|_2}{\|\theta_j^* - \theta_q^*\|_2},
\end{align*}
and $\bar{A}_j = 1$, $\bar{B}_j = - 2\nabla\ell_j(\theta_j^*)$. By integrating this polynomial with respect to the measure $\nu$, applying Taylor expansion with Lagrangian remainder, and using similar arguments as in Step 4 in proof of univariate local part in Theorem \ref{theorem:exact_one_group_univariate}, we have
$$ \int E_j(\theta) d\nu(\theta) = \Delta \pi_j + \frac{1}{2}\sum_{l=1}^{k_*} \sum_{i \in \mathcal{V}_l} \pi_{i} (\theta_i-\theta_l^*)^\top D^2E_j(\xi_i)(\theta_i-\theta_l^*)+ \sum_{i \in \mathcal{M}_{\mathrm{far}}} \pi_i E_j(\theta_i). $$
Isolating the absolute near mass mismatch via the triangle inequality:
\begin{align}
    |\Delta \pi_j| \le \underbrace{ \left| \int E_j(\theta) d\nu(\theta) \right| }_{\text{Test Function Dual Bound}} + \underbrace{ \frac{1}{2}\sum_{l=1}^{k_*} \sum_{i \in \mathcal{V}_l} \pi_i \|D^2E_j(\xi_i)\|_{\mathrm{op}}\|\theta_i-\theta^*_l\|^2_2 }_{\text{Near Error}} + \underbrace{ \sum_{i \in \mathcal{M}_{\mathrm{far}}} \pi_i |E_j(\theta_i)| }_{\text{Far Error}}. \label{eqn:dung_thm4_1_multivariate_global_delta_pi_j_split}
\end{align}

For the test function dual bound,  using Lemma \ref{lemma:mass_test_function} and Lemma \ref{lemma:Hellinger_to_Polynomial}, we have
\begin{align}
\left| \int E_j(\theta) d\nu(\theta) \right| \le C_{\mathrm{poly}}\|E_j\|_{\infty} d_H(p_G,p_{G_*}) \
\leq C_{\text{norm},E}C_{\mathrm{poly}}\Delta_{\mathrm{sep}}^{-(2k_*-1)} d_H(p_{G},p_{G_*}). \label{eqn:dung_thm_4.1_multivariate_global_test_function_dual_bound}
\end{align}

For the near error, using Lemma \ref{lemma:mass_test_function} and the estimationin equation~\eqref{eqn:dung_thm_4.1_multivariate_global_near_variance_estimation}, we obtain that
\begin{align}
\frac{1}{2}\sum_{l=1}^{k_*} \sum_{i \in \bar{\mathcal{V}}_l} \pi_i \|D^2E_j(\xi_i)\|_{\mathrm{op}}\|\theta_i-\theta_l^*\|^2 \leq \frac{1}{2} C_{2,E} C_{\mathrm{var,1}} \Delta_{\mathrm{sep}}^{-2k_*} d_H(p_{G},p_{G_*}).\label{eqn:dung_thm_4_1_global_multivariate_sum_of_second_derivative}
\end{align}

For the far error, by arguing similarly to equation~\eqref{eqn:dung_thm3_1_multivariate_global_estimation_of_function_far_case} in the proof of the multivariate global part of Theorem \ref{theorem:exact_one_group_univariate}, we have 
\begin{align*}
    |E_j(\theta_i)| \leq  C_{\mathrm{far},E} \Delta_{\mathrm{sep}}^{-(2k_*-1)} \|\theta_i - \theta_{c(i)}^*\|^{2k_*-1},
\end{align*}
which leads to, by combining with estimation in equation~\eqref{eqn:dung_thm_4.1_multivariate_far_high_moment_estimation}, and using the same argument as in  equation~\eqref{eqn:dung_thm3_1_multivariate_far_set_test_function_sum}, 
\begin{align}
\label{eqn:dung_thm_4.1_multivariate_sum_of_far_high_moment_estimation}
    \sum_{i \in \mathcal{M}_{\mathrm{far}}} \pi_i |E_j(\theta_i)| \leq 4 C_{\mathrm{far},E} C_{\mathrm{far,moment,1}} \Delta_{\mathrm{sep}}^{-(2k_*-1)} d_H(p_{G},p_{G_*}).
\end{align}
Plugging the upper bounds of the test function dual bound, near error, and far error in equations~\eqref{eqn:dung_thm_4.1_multivariate_global_test_function_dual_bound}, \eqref{eqn:dung_thm_4_1_global_multivariate_sum_of_second_derivative}, \eqref{eqn:dung_thm_4.1_multivariate_sum_of_far_high_moment_estimation} to the inequality in equation~\eqref{eqn:dung_thm4_1_multivariate_global_delta_pi_j_split} leads to 
\begin{align}
|\Delta \pi_j| & \le C_{\mathrm{mass}} \Delta_{\mathrm{sep}}^{-2k_*} d_H(p_{G},p_{G_*}).
\label{eqn:dung_thm4_1_multivariate_global_delta_pi_j_final_bound}
\end{align}
where $C_{\mathrm{mass}} = 2RC_{\text{norm},E}C_{\mathrm{poly}} + \frac{1}{2} C_{E,2} C_{\mathrm{var,1}} + 4 C_{\mathrm{far},E} C_{\mathrm{far,moment,1}}$. 

\vspace{0.5 em}
\noindent
\emph{Step 5 - Bounding sparse mean discrepancy $\sum_{j=1:|\bar{\mathcal{V}}_j| = 1}^{k_*}\sum_{i\in \bar{\mathcal{V}}_j}\pi_i\|\theta_i - \theta^*_{j}\|$.} Consider an index $j \in [k_*]$ such that $\bar{\mathcal{V}}_j$ contains exactly one element $\theta_i$. We find a function $H_{j,i}$ such that $H_{j,i}(\theta_l^*) = 0$ for all $l$, $\nabla H_{j,i}(\theta_l^*) = 0$ for all $l\neq i$, and $\langle \nabla H_{j,i}(\theta_j^*), \theta_i - \theta_j^*\rangle = \|\theta_i - \theta_j^*\|_2$. For this, we consider the mean extractor test function $H_{j,i}$ defined in Section \ref{sec:non_multi_cluster_mean_test_function}
\begin{align*}
    H_{j,i}(\theta) : = \ell_j^2(\theta) \left[\langle v_{j,i}, \theta - \theta_j^*\rangle\right],  \quad \text{where} \quad \ell_j(\theta) : = \prod_{q \neq j} \frac{\|\theta - \theta_q^*\|_2}{\|\theta_j^* - \theta_q^*\|_2}, 
\end{align*}
and $v_{j,i} = {(\theta_i-\theta_j^*)}/{\|\theta_i-\theta_j^*\|}$ if $\theta_i \neq \theta_j^*$ and $v_{j,i}=(1, 0,\ldots ,0)^\top$ otherwise.
For any center $\theta_l \in \mathcal{M}_{\mathrm{near}}$, by applying Taylor expansion for 
$H_{j,i}(\theta_l)$ exactly around its true center $\theta_{c(l)}^*$, we have 
\begin{align*}
H_{j,i}(\theta_l) &= \delta_{jl}v_{j,i}^{\top}(\theta_{l} - \theta_{c(l)}^{*}) + \frac{1}{2} (\theta_{l} - \theta_{c(l)}^{*})^{\top}D^2H_{j,i}(\xi_l) (\theta_{l} - \theta_{c(l)}^{*}) \\
&= \delta_{jl}\|\theta_i-\theta^*_{c(i)}\| + \frac{1}{2} (\theta_{l} - \theta_{c(l)}^{*})^{\top}D^2H_{j,i}(\xi_l) (\theta_{l} - \theta_{c(l)}^{*}).
\end{align*}
Thus, we arrive at $$ \int H_{j,i}(\theta) d\nu(\theta) = \pi_{i}\|\theta_{i} - \theta_{c(i))}^{*}\| + \frac{1}{2} \sum_{l\in \mathcal{M}_{\mathrm{near}}} \pi_{l} (\theta_{l} - \theta_{c(l)}^{*})^{\top} D^2H_{i,i}(\xi_l)(\theta_{l} - \theta_{c(l)}^{*})+ \sum_{l \in \mathcal{M}_{\mathrm{far}}} \pi_l H_{j,i}(\theta_l). ,$$
which indicates that
\begin{align}
\label{eqn:dung_thm_4_1_first_moment_estimation}
    \pi_{i}\|\theta_{i} - \theta_{c(i)}^{*}\| &\leq \underbrace{\left|\int H_{j,i}(\theta) d\nu(\theta)\right|}_{\text{Test Function Dual Bound}} + \frac{1}{2} \underbrace{\sum_{l=1}^{k_*} \pi_{l} \|D^2H_{j,i}(\xi_l)\|_{\mathrm{op}} \|\theta_{l} - \theta_{c(l)}^{*}\|_2^2}_{\text{Near Error}} +\underbrace{\sum_{l \in \mathcal{M}_{\mathrm{far}}} \pi_l |H_{j,i}(\theta_l)|}_{\text{Far Error}}.
\end{align}

For the test function dual bound, using Lemma \ref{lemma:mean_test_function} and Lemma \ref{lemma:Hellinger_to_Polynomial}, we have
\begin{align}
\left| \int H_{j,i}(\theta) d\nu(\theta) \right| \le C_{\mathrm{poly}}\|H_{j,i}\|_{\infty}d_H(p_G,p_{G_*})\leq C_{\mathrm{poly}}C_{\text{norm},H} \Delta_{\mathrm{sep}}^{-(2k_*-1)} d_H(p_{G},p_{G_*}). \label{eqn:dung_thm_4.1_multivariate_global_mean_unique_point_test_function_dual_bound}
\end{align}

For the near error, using Lemma \ref{lemma:mass_test_function} and the estimation in equation~\eqref{eqn:dung_thm_4.1_multivariate_global_near_variance_estimation}, we obtain that
\begin{align}
\frac{1}{2}\sum_{l=1}^{k_*} \sum_{i \in \bar{\mathcal{V}}_l} \pi_l \|D^2H_{j,i}(\xi_i)\|_{\mathrm{op}}\|\theta_i-\theta_l^*\|^2 \leq \frac{1}{2} C_{2,H} C_{\mathrm{var,1}} \Delta_{\mathrm{sep}}^{-(2k_*-1)} d_H(p_{G},p_{G_*}).\label{eqn:dung_thm_4_1_global_multivariate_sum_of_second_derivative_of_H}
\end{align}

For the far error, when $\|\theta_l - \theta_{c(l)}^*\| \geq \Delta_{\mathrm{sep}}/4$, using the estimation in equation \eqref{eqn:dung_thm_3_1_multivariate_global_theta_distance_to_r_distance}, we have 
\begin{equation*}
    |\langle v_{j,i},\theta_l-\theta_j^* \rangle| \leq \|v_{j,i}\|_2\cdot \|\theta_l-\theta_j^*\|_2 \leq (1+4C_0)\|\theta_{l} - \theta^*_{c(l)}\|_2
\end{equation*}
Combining this result with equation~\eqref{eqn:dung_thm_3_1_multivariate_global_estimation_for_ell_square_for_far_case}, and noting that $\|\theta_l -\theta^*_{c(l)}\|\geq \Delta_{\mathrm{sep}}/4$, we have 
\begin{align*}
    |H_{j,i}(\theta_l)| \leq  C_{\mathrm{far},H} \Delta_{\mathrm{sep}}^{-(2k_*-2)} \|\theta_l - \theta_{c(l)}^*\|^{2k_*-1} \leq 4C_{\mathrm{far},H} \Delta_{\mathrm{sep}}^{-(2k_*-1)} \|\theta_l - \theta_{c(l)}^*\|^{2k_*},
\end{align*}
which leads to, by combining with the estimation in equation~\eqref{eqn:dung_thm_4.1_multivariate_far_high_moment_estimation},
\begin{align}
\label{eqn:dung_thm_4.1_multivariate_sum_of_far_high_moment_estimation_of_H}
    \sum_{i \in \mathcal{M}_{\mathrm{far}}} \pi_i |H_{j,i}(\theta_i)| \leq 4 C_{\mathrm{far},H} C_{\mathrm{far,moment},1} \Delta_{\mathrm{sep}}^{-(2k_*-1)} d_H(p_{G},p_{G_*}).
\end{align}
Combining the results from equations \eqref{eqn:dung_thm_4.1_multivariate_global_mean_unique_point_test_function_dual_bound}, \eqref{eqn:dung_thm_4_1_global_multivariate_sum_of_second_derivative_of_H}, \eqref{eqn:dung_thm_4.1_multivariate_sum_of_far_high_moment_estimation_of_H} and plugging them into equation \eqref{eqn:dung_thm_4_1_first_moment_estimation}, we have 
\begin{align}
\label{eqn:dung_thm4_1_multivariate_global_unique_near_first_moment_bound}
    \pi_i\|\theta_i-\theta_{c(i)}^*\|\leq C_{\mathrm{mean}}\Delta_{\mathrm{sep}}^{-(2k_*-1)} d_H(p_{G},p_{G_*})
\end{align}
where 
\begin{equation*}
C_{\mathrm{mean}}= C_{\mathrm{poly}}C_{\text{norm},H} + \frac{1}{2} C_{2,H} C_{\mathrm{var,1}} +  4 C_{\mathrm{far},H} C_{\mathrm{far,moment},1}. 
\end{equation*}
\emph{Step 6 - Bounding $W^2_2(G,G_*)$ and conclusion.} Combining the results from equations~\eqref{eqn:dung_thm_4.1_multivariate_global_near_variance_estimation}, \eqref{eqn:dung_thm_4.1_multivariate_far_second_moment_estimation}, \eqref{eqn:dung_thm_4_1_multivariate_global_pi_far_estimation}, \eqref{eqn:dung_thm4_1_multivariate_global_delta_pi_j_final_bound} and  \eqref{eqn:dung_thm4_1_multivariate_global_unique_near_first_moment_bound}, we have 
\begin{align}
    S_1(G,G_*)\leq C^{-1}_{\mathrm{global},4}\Delta_{\mathrm{sep}}^{-(2k_* - 2)} d_H(p_{G},p_{G_*}), 
\end{align}
where 
\begin{equation*}
C^{-1}_{\mathrm{global},4} = C_{\mathrm{var,1}} +C_{\mathrm{far,var}}  +  C_0^2(C_{\mathrm{far,mass}} + k_*C_{\mathrm{mass}}) +\frac{1}{2}k_* C_{\mathrm{mean}}.
\end{equation*}
As a consequence, we can conclude that 
$$ d_H(p_G,p_{G_*})\geq C_{\mathrm{global},4}\Delta_{\mathrm{sep}}^{2k_*-2}S_1(G,G_*)\geq C_{\mathrm{global},4}\Delta_{\mathrm{sep}}^{2k_*-2}W_2^2(G,G_*).$$

\subsection{Proof of Theorem~\ref{theorem:one_group_univariate_multi_cluster}}
\label{sec:proof:theorem:one_group_univariate_multi_cluster}

Throughout the proof, we denote $G = \sum_{i=1}^{k} \pi_i \delta_{\theta_i}$ to be a probability measure with exactly $k$ components. For every fitted atom $\theta_i$, let $c(i) = \text{argmin}_{1 \le j \le k_*} \|\theta_i - \theta_j^*\|$ be its absolute closest true center. Let $m(i)$ denote the cluster containing the center $\theta^{*}_{c(i)}$.We unconditionally partition the $k_*$ atoms of $G$ into three disjoint sets based on exact spatial thresholds:
\begin{itemize}
    \item The micro near set $\mathcal{M}_{\mathrm{mic}} = \{ i : \|\theta_i - \theta_{c(i)}^*\|_2 \le \frac{\Delta_{\mathrm{sep}}}{4} \}$. We define the explicit Voronoi cells strictly inside this core as $\bar{\mathcal{V}}_j = \{i \in \mathcal{M}_{\mathrm{mic}} : c(i) = j\}$. 
    \item The macro near set $\mathcal{M}_{\mathrm{mac}} = \{ i : \frac{\Delta_{\mathrm{sep}}}{4} < \|\theta_i - \theta_{c(i)}^*\|_2 \le \frac{D_0}{4} \}$, which includes atoms that have escaped the micro near set but still associate with the cluster $\mathcal{C}_{m(i)}$. We denote the total macro near mass as $\pi_{\mathrm{mac}} = \sum_{i \in \mathcal{M}_{\mathrm{mac}}} \pi_i$.
    \item The far void set $\mathcal{M}_{\mathrm{far}} = \{ i : \|\theta_i - \theta_{c(i)}^*\| > \frac{D_0}{4} \}$, which consists of atoms in the void between distinct clusters. We denote the total far mass as $\pi_{\mathrm{far}} = \sum_{i \in \mathcal{M}_{\mathrm{far}}} \pi_i$.
\end{itemize}

Before delving into details of proof, let us define some notation to be utilized later in the proof. Let $\tilde{\pi}_j = \sum_{i:c(i) = j}\pi_i$ to be the total mass of $\tilde{G}$ at $\theta_j^*$. Let the net mass discrepancy at center $j$ be exactly defined as $\tilde{\Delta} \pi_j = \tilde{\pi}_j - \pi_j^*$, while 
$\Delta\pi_j:= \left(\sum_{i\in \bar{\mathcal{V}}_j}\pi_i\right)-\pi_j^*$. We define cluster discrepancy as $\Delta\Pi_m = \sum_{j \in \mathcal{C}_{m}}\tilde{\Delta}\pi_{j}$.

\vspace{0.5 em}
\noindent
\emph{Step 1 - Wasserstein Decomposition.} To upper bound $W_2^2(G,G_*)$, we consider measure $G' = \sum_{j=1}^k\pi_j\delta_{\theta^*_{c(j)}} = \sum_{j=1}^{k_*}(\sum_{i\in \mathcal{V}_j}\pi_{i})\delta_{\theta^*_j} = \sum_{j=1}^{k_*} \tilde{\pi}_j\delta_{\theta^*_{j}}$, then triangle inequality for Wasserstein distance gives us 
\begin{align}
\label{eqn:dung_thm4_2_multivariate_global_W2_prelim_split}
    W_2^2(G, G_*) \leq(W_2(G, G') + W_2(G', G_*))^2 \leq 2(W_2^2(G, G') + W^2_2(G', G_*)).
\end{align}

For $W_2^2(G, G')$, we can consider the transportation plan $\rho_{jj} = \pi_j$ for all $1\leq j\leq k$, which implies
\begin{align}
\nonumber
    W_2^2(G, G')&\leq \sum_{i=1}^k \pi_i \|\theta_i-\theta^*_{c(i)}\|^2\\
    &\leq \sum_{i \in \mathcal{M}_{\mathrm{mic}}} \pi_i\|\theta_i - \theta^*_{c(i)}\|^2 + \sum_{i \in \mathcal{M}_{\mathrm{mac}}} \pi_i\|\theta_i - \theta^*_{c(i)}\|^2 + \sum_{i \in \mathcal{M}_{\mathrm{far}}} \pi_i\|\theta_i - \theta^*_{c(i)}\|^2
\label{eqn:dung_thm_4_2_multivariate_global_W2_prelim_first_bound_source_and_intermediate}
\end{align}

For $W_2^2(G', G_*)$, we consider the following transportation plan from $G'$ to $G_*$
\begin{enumerate}
    \item For each center $\theta_j^*$, we keep $\min\{\sum_{i\in\mathcal{V}_j}\pi_i,\pi^*_j\}$. Then, the remaining mass in each center for $G'$ is exactly $\max\{0, \tilde{\Delta} \pi_j\}$. 
    \item For each cluster $\mathcal{C}_m$, we move the remaining mass in each center to other center \textit{within} this cluster. Then, as the transportation distance is less than or equal to $C_0\Delta_{\mathrm{sep}}$, the total transportation cost is not exceed $\sum_{j=1}^k C_0^2\Delta^2_{\mathrm{sep}}|\tilde{\Delta}\pi_j|$, while the total remaining mass for $G'$ in each cluster $\mathcal{C}_m$ is exactly $\max\{0, \Delta \Pi_m\}$.
    \item Finally, we transport each remaining mass in each cluster $\max\{0, \Delta \Pi_m\}$ to other center in other cluster. Then, as the transportation distance is less than or equal to $2R$, the total transportation cost is not exceed $\sum_{m=1}^{k_0}4R^2|\Delta\Pi_m|$. 
\end{enumerate}
Combining these estimations, noting that $|\tilde{\Delta}\pi_j| \leq |\Delta\pi_j| + \sum_{i \notin \mathcal{M}_{\mathrm{mic}}, \, c(i)=j} \pi_i$, we achieve an upper bound for $W_2^2(G',G_*)$  
\begin{align}
\nonumber
    W_2^2(G', G_*) &\leq \sum_{j=1}^{k_*} C_0^2\Delta^2_{\mathrm{sep}}|\tilde{\Delta}\pi_j| + \sum_{m=1}^{k_0}4R^2|\Delta\Pi_m|\\
    &\leq C_0^2\Delta^2_{\mathrm{sep}}\left( \pi_{\mathrm{far}} +\pi_{\mathrm{mac}}+\sum_{j=1}^{k_*}|\Delta\pi_j| \right) + 4R^2\sum_{m=1}^{k_0}|\Delta\Pi_m|. 
\label{eqn:dung_thm_4_2_local_W2_intermediate_and_target_prelim_second_bound}
\end{align}
By plugging the result from equations~\eqref{eqn:dung_thm_4_2_multivariate_global_W2_prelim_first_bound_source_and_intermediate} and \eqref{eqn:dung_thm_4_2_local_W2_intermediate_and_target_prelim_second_bound} into \eqref{eqn:dung_thm4_2_multivariate_global_W2_prelim_split}, we have 
\begin{align}
\nonumber
    W_2^2(G,G_*) &\leq 2\left(\sum_{i \in \mathcal{M}_{\mathrm{mic}}} \pi_i\|\theta_i - \theta^*_{c(i)}\|^2 + \sum_{i \in \mathcal{M}_{\mathrm{mac}}} \pi_i\|\theta_i - \theta^*_{c(i)}\|^2 + \sum_{i \in \mathcal{M}_{\mathrm{far}}} \pi_i\|\theta_i - \theta^*_{c(i)}\|^2\right)\\
    &\nonumber\hspace{1cm} +2\left(C_0^2\Delta^2_{\mathrm{sep}}\left( \pi_{\mathcal{M}_{\mathrm{far}}} +\pi_{\mathcal{M}_{\mathrm{mac}}}+\sum_{j=1}^k|\Delta\pi_j| \right) + 4R^2\sum_{m=1}^{k_0}|\Delta\Pi_m|\right)\\
    \nonumber
    &\leq \frac{\Delta_{\mathrm{sep}}}{2}\sum_{j:|\bar{\mathcal{V}_j}|=1,\mathcal{V}_j = \{i\}} \pi_i\|\theta_i - \theta^*_{j}\| +  2\sum_{j:|\bar{\mathcal{V}_j}|>1}\sum_{i \in \bar{\mathcal{V}}_j}\pi_i\|\theta_i - \theta_j^*\|^2\\
    \nonumber
    &\hspace{1cm} + 2\sum_{i \in \mathcal{M}_{\mathrm{mac}}} \pi_i\|\theta_i - \theta^*_{c(i)}\|^2 + 2\sum_{i \in \mathcal{M}_{\mathrm{far}}} \pi_i\|\theta_i - \theta^*_{c(i)}\|^2\\
    &\hspace{1cm} +2\left(C_0^2\Delta^2_{\mathrm{sep}}\left( \pi_{\mathcal{M}_{\mathrm{far}}} +\pi_{\mathcal{M}_{\mathrm{mac}}}+\sum_{j=1}^k|\Delta\pi_j| \right) + 4R^2\sum_{m=1}^{k_0}|\Delta\Pi_m|\right)
\label{eqn:dung_thm_4_2_multivariate_global_final_W2_prelim_bound}
\end{align}
Let 
\begin{align}
\nonumber
    S_2(G,G_*) &= \frac{\Delta_{\mathrm{sep}}}{2}\sum_{j:|\bar{\mathcal{V}_j}|=1,\mathcal{V}_j = \{i\}} \pi_i\|\theta_i - \theta^*_{j}\| +  2\sum_{j:|\bar{\mathcal{V}_j}|>1}\sum_{i \in \bar{\mathcal{V}}_j}\pi_i\|\theta_i - \theta_j^*\|^2\\
    \nonumber
    &\hspace{0.5cm} + 2\sum_{i \in \mathcal{M}_{\mathrm{mac}}} \pi_i\|\theta_i - \theta^*_{c(i)}\|^2 + 2\sum_{i \in \mathcal{M}_{\mathrm{far}}} \pi_i\|\theta_i - \theta^*_{c(i)}\|^2\\
    &\hspace{0.5cm} +2\left(C_0^2\Delta^2_{\mathrm{sep}}\left( \pi_{\mathrm{far}} +\pi_{\mathrm{mac}}+\sum_{j=1}^k|\Delta\pi_j| \right) + 4R^2\sum_{m=1}^{k_0}|\Delta\Pi_m|\right)
    \label{eqn:dung_thm_4_2_voronoi_loss_def}
\end{align}
we prove a stronger result
\begin{equation*}
    d_H(p_G,p_{G_*})\geq C_{\mathrm{global},5}\Delta_{\mathrm{sep}}^{2s_{\max}-2}S_1(G,G_*)
\end{equation*}
\emph{Step 2 - Bounding variance $\sum_{\mathcal{M}_{\mathrm{mic}} \cup \mathcal{M}_{\mathrm{mac}} \cup \mathcal{M}_{\mathrm{far}}} \pi_i \|\theta_i-\theta_{c(i)}\|^2$: } In this part, we utilize the variance extractor test function define in Section \ref{sec:variance_test_function}:
\begin{equation*}
     P_{\mathrm{var}}(\theta) : = \prod_{l=1}^{k_*} \|\theta - \theta_l^*\|_2^2.
\end{equation*}

Using Lemma \ref{lemma:Hellinger_to_Polynomial} and Lemma \ref{lemma:variance_test_function} and the fact that $\int P_{\mathrm{var}}dG_*(\theta) =0$, we have 
\begin{align*}
    \left|\int P_{\mathrm{var}}dG(\theta)\right| = \left|\int P_{\mathrm{var}}d\nu(\theta)\right| \leq C_{\mathrm{poly}}\|P_{\mathrm{var}}\|_{\infty} d_H(p_G,p_{G_*})\leq C_{\mathrm{var}}C_{\mathrm{poly}}d_H(p_G,p_{G_*}),
\end{align*}
which implies
\begin{align}
\sum_{i \in \mathcal{M}_{\mathrm{mic}}} \pi_i P_{\mathrm{var}}(\theta_i) + \sum_{i \in \mathcal{M}_{\mathrm{mac}}} \pi_i P_{\mathrm{var}}(\theta_i) + \sum_{i \in \mathcal{M}_{\mathrm{far}}} \pi_i P_{\mathrm{var}}(\theta_i) \leq C_{\mathrm{var}}C_{\mathrm{poly}} d_H(p_{G},p_{G_*}). \label{eqn:dung_thm_4_2_exact_global_multivariate_multi_key_equation_first_1}
\end{align}
Using the similar argument as in Proof of Theorem \ref{theorem:exact_multi_group} for multivariate global setting, we have 
\begin{align}
\label{eqn:dung_thm_4_2_multivariate_global_p_var_at_micro_estimation}
    P_{\mathrm{var}}(\theta_i) \geq C_{\mathrm{mic}} \cdot \Delta_{\mathrm{sep}}^{2s_{\max}-2} \|\theta_i-\theta^*_{c(i)}\|_2^2 \quad \text{for all } i \in \mathcal{M}_{\mathrm{mic}},
\end{align}
\begin{align}
\label{eqn:dung_thm_4_2_multivariate_global_p_var_at_macro_estimation}
P_{\mathrm{var}}(\theta_i)\geq C_{\mathrm{mac}}\cdot\Delta_{\mathrm{sep}}^{2s_{\max}-2}\|\theta_i - \theta_{c(i)}^*\|_2^2 \quad \text{for all } i \in \mathcal{M}_{\mathrm{mac}},
\end{align}
\begin{align}  
\label{eqn:dung_thm_4_2_multivariate_global_f_var_at_far_estimation}
P_{\mathrm{var}}(\theta_i) \geq C_{\mathrm{far}} \cdot \Delta^{-2s_{\max}}_{\mathrm{sep}}\|\theta_i-\theta^*_{c(i)}\|_2^2 \quad \text{for all } i \in \mathcal{M}_{\mathrm{far}}.
\end{align}
Plugging in these estimations of $P_{\mathrm{var}}$ in equations~\eqref{eqn:dung_thm_4_2_multivariate_global_p_var_at_micro_estimation}, \eqref{eqn:dung_thm_4_2_multivariate_global_p_var_at_macro_estimation}, and \eqref{eqn:dung_thm_4_2_multivariate_global_f_var_at_far_estimation} into equation \eqref{eqn:dung_thm_4_2_exact_global_multivariate_multi_key_equation_first_1}, we have
\begin{equation}
\label{eqn:dung_thm_4_2_multivarite_global_variation_bound}
    \sum_{\mathcal{M}_{\mathrm{mic}} \cup \mathcal{M}_{\mathrm{mac}}\cup \mathcal{M}_{\mathrm{far}}} \pi_i \|\theta_i-\theta^*_{c(i)}\|_2^2 \leq C_{\mathrm{var},1}\Delta_{\mathrm{sep}}^{-(2s_{\max}-2)}d_H(p_{G},p_{G_*}),
\end{equation}
where $C_{\mathrm{var},1} = C_{\mathrm{var}}\cdot C_{\mathrm{poly}}\cdot \max\{C^{-1}_{\mathrm{mic}},C^{-1}_{\mathrm{mac}},C^{-1}_{\mathrm{far}}\}$.

\vspace{0.5 em}
\noindent
\emph{Step 3 - Bounding macro- and far-related quantities.} At this step, the bounds for the quantities related to $\mathcal{M}_{\mathrm{mac}}$ and $\mathcal{M}_{\mathrm{far}}$ are established. 

\emph{Step 3.1 -  Macro High Moment $\sum_{\mathcal{M}_{\mathrm{mac}}} \pi_i \|\theta_i-\theta^*_{c(i)}\|_2^{2s_{\max}}$.} Using the same argument as in univariate global setting for Theorem \ref{theorem:exact_multi_group}, we have $P_{\mathrm{var}}(\theta_i) \geq C^{-1}_{\mathrm{mac,moment}}\|\theta_i -\theta^*_{c(i)}\|^{2s_{\max}}_2$. As a result, we have 
\begin{equation}
\label{eqn:dung_thm_4_2_multivariate_global_macro_high_moment_estimate}
    \sum_{i\in \mathcal{M}_{\mathrm{mac}}} \pi_i \|\theta_i-\theta^*_{c(i)}\|_2^{2s_{\max}} \leq C_{\mathrm{mac,moment}} \cdot d_H(p_{G},p_{G_*}),
\end{equation}
where $C_{\mathrm{mac,moment}} = C_{\mathrm{var}}\cdot C_{\mathrm{poly}}\cdot \max\left\{1,\left(\frac{3}{4}D_0\right)^{-2k_*}\right\}\cdot\max\{1,(2R)^{k_*}\}$. 

\emph{Step 3.2 - Global Far Mass $\pi_{\mathrm{far}}$.} We have for each $l$, $\|\theta_i - \theta_l^*\|_2 \geq \|\theta_i-\theta^*_{c(i)}\|_2 > \frac{D_0}{4}$, thus we have $P_{\mathrm{var}} \geq \left(\frac{D_0}{4}\right)^{2k_*}$. Using the same argument as in univariate global setting in Theorem \ref{theorem:exact_multi_group}, we have 
\begin{align}
    \sum_{i \in \mathcal{M}_{\mathrm{far}}} \pi_i \leq C_{\mathrm{far,mass}} d_H(p_{G},p_{G_*}),
    \label{eqn:dung_thm_4_2_multivariate_global_far_mass}
\end{align}
where $C_{\mathrm{far,mass}} = \left(\frac{D_0}{4}\right)^{-2k_*} \cdot C_{\mathrm{var}}\cdot C_{\mathrm{poly}}$.   

\emph{Step 3.3 - Macro Mass $\pi_{\mathrm{mac}}$:} For any $i \in \mathcal{M}_{\mathrm{mac}}$, since $\|\theta_i-\theta^*_{c(i)}\|_2 > \Delta_{\mathrm{sep}}/4$, using the estimation in equation~\eqref{eqn:dung_thm_4_2_multivariate_global_macro_high_moment_estimate}, we have
\begin{align}
\label{eqn:dung_thm_4_2_multivariate_global_macro_mass}
    \sum_{\mathcal{M}_{\mathrm{mac}}} \pi_i \le C_{\mathrm{mac,mass}}\Delta_{\mathrm{sep}}^{-(2s_{\max})} d_H(p_{G},p_{G_*}).
\end{align}
where $C_{\mathrm{mac,mass}} = 4^{2k_*} C_{\mathrm{mac,moment}}$.

\emph{Step 3.4 - Bounding cluster mass discrepancy $|\Delta \Pi_{m}|$:} The proof utilizes the cluster-wise mass extractor function $P_{m}$ in Section \ref{sec:cluster_wise_test_function}. Using the same argument with Taylor expansion as in multivariate global part, there exists $\xi_i \in \bar{B}(0,R)$ such that 
\begin{align*}
    \int P_m(x)d\nu(x) &= \Delta\Pi_m - \sum_{c(i) \in\mathcal{C}_m}\pi_{i}\mathbf{1}_{\{i \in\mathcal{M}_{\mathrm{far}}\}}\\
    & \hspace{1cm }+ \frac{1}{2} \sum_{i \in \mathcal{M}_{\mathrm{mic}} \cup \mathcal{M}_{\mathrm{mac}}} \pi_i (\theta_i-\theta^*_{c(i)})^{\top}D^2P_m(\xi_i)(\theta_i-\theta^*_{c(i)}) + \sum_{i \in \mathcal{M}_{\mathrm{far}}} \pi_i P_m(\theta_i).
\end{align*} 
As a result, we have 
\begin{align*}
    |\Delta\Pi_m| &\leq \left|\int P_m(x)d\nu(x)\right| + \sum_{c(i) \in\mathcal{C}_m}\pi_{i}\mathbf{1}_{\{i \in\mathcal{M}_{\mathrm{far}}\}} \\ &\hspace{2cm}+ \frac{1}{2} \sum_{i \in \mathcal{M}_{\mathrm{mic}} \cup \mathcal{M}_{\mathrm{mac}}} \pi_i \|D^2P_m(\xi_i)\|_{\mathrm{op}}\|\theta_i-\theta^*_{c(i)}\|_2^2  + \sum_{i \in \mathcal{M}_{\mathrm{far}}} \pi_i P_m(\theta_i). 
\end{align*}
Using Lemma \ref{lemma:Hellinger_to_Polynomial}  and the bound from Lemma \ref{lemma:dung_bound_for_cluster_mass_extracting}, we have 
\begin{align*}
    |\Delta \Pi_m| &\le C_{\mathrm{norm},P}C_{\mathrm{poly}}d_H(p_{G},p_{G_*})\\&\hspace{1cm}+ \frac{1}{2} C_{P,2} \left( \sum_{\mathcal{M}_{\mathrm{mic}} \cup \mathcal{M}_{\mathrm{mac}}} \pi_i \|\theta_i-\theta^*_{c(i)}\|_2^2 \right)+(C_{\mathrm{norm},P} + 1) \pi_{\mathrm{far}}.
\end{align*}
Using the estimations in equations \eqref{eqn:dung_thm_4_2_multivarite_global_variation_bound} and \eqref{eqn:dung_thm_4_2_multivariate_global_far_mass}, and applying the same argument as in the multivariate global part, we have 
\begin{align}
|\Delta \Pi_m| &\leq C_{\Delta\Pi} \Delta_{\max}^{-(2s_{\max}-2)}d_H(p_{G},p_{G_*}),
\label{eqn:dung_thm_4_2_multivariate_global_sum_delta_pi_estimation}
\end{align}
where $C_{\Delta\Pi}$ is defined as 
\begin{align*}
    C_{\Delta\Pi} &= C_{\mathrm{norm},P}C_{\mathrm{poly}}\max\{(2R)^{2k_*-2},1\}+ \frac{1}{2} C_{P,2}C_{\mathrm{var},1} \\
    &\hspace{1cm} + (C_{\mathrm{norm},P} + 1) C_{\mathrm{far,mass}} \max\{(2R)^{2k_*-2},1\}. 
\end{align*}

\emph{Step 4 - Bounding mass discrepancy $ |\Delta \pi_j|$.} In this section, we use the function $\bar{E}_j(\theta)$ such that $\bar{E}_j(\theta^*_l) = \delta_{jl}$ and $\nabla \bar{E}'_{j}(\theta_l^*) = 0$, which is constructed in Section \ref{sec:dung_mass_extractor_multicluster}: 
\begin{equation*}
    \bar{E}_{j}(\theta) = \ell_{j,\text{micro}}^2(\theta) [\bar{A}_j + \langle \bar{B}_j,\theta - \theta_{j}^{*}\rangle]P_{\text{macro}}(\theta),
\end{equation*}
where we define $\ell_{j,\text{micro}}(\theta) = \prod_{q \in \mathcal{C}_{m} \neq j} \frac{\|\theta - \theta_q^*\|_2}{\|\theta_j^* - \theta_q^*\|_2}$, $P_{\text{macro}}(\theta) = \prod_{p \neq m} \prod_{q \in \mathcal{C}_p} \left( \frac{\|\theta - \theta_q^*\|_2}{2R} \right)^{2s_{\max}}$ as in Section \ref{sec:dung_mean_extractor_multicluster}, and $\bar{A}_j = 1/ P_{\text{macro}}(\theta_{j}^{*})$,
\begin{align*}
    \bar{B}_j & = -\bar{A}_j[2\nabla\ell_{j,\text{micro}}(\theta_j^*) + (P_{\text{macro}}(\theta^*_j))^{-1}\nabla P_{\text{macro}}(\theta^*_j)].
\end{align*}

We assume that $j\in \mathcal{C}_m$. Then, using the same argument as in univariate global part of Theorem \ref{theorem:exact_multi_group}, we have 
\begin{align}
\nonumber
    \Delta\pi_j &= \int \bar{E}_j(\theta)d\nu(\theta) - \underbrace{\left(\sum_{l \in \mathcal{C}_m}\sum_{i\in \bar{\mathcal{V}}_l}\pi_i(\bar{E}_j(\theta_i) -\bar{E}_j(\theta^*_l))\right)}_{\text{Case 4.1}} - \underbrace{\sum_{\substack{i \in \mathcal{M}_{\mathrm{mic}}\\ c(i) \notin \mathcal{C}_m }}\pi_i\bar{E}_j(\theta_i)}_{\text{Case 4.2}}\\
    &\hspace{1cm } -\underbrace{\sum_{\substack{i \in \mathcal{M}_{\mathrm{mac}}\\ c(i) \in \mathcal{C}_m }}\pi_i\bar{E}_j(\theta_i)}_{\text{Case 4.3}} - \underbrace{\sum_{\substack{i \in \mathcal{M}_{\mathrm{mac}}\\ c(i) \notin \mathcal{C}_m }}\pi_i\bar{E}_j(\theta_i)}_{\text{Case 4.4}} -  \underbrace{\sum_{{i\in \mathcal{M}_{\mathrm{far}}}}\pi_i\bar{E}_j(\theta_i)}_{\text{Case 4.5}}.
\label{eqn:dung_thm_4_2_multivariate_global_delta_pi_as_sum_of_test_function}
\end{align}
For the integral term, applying the same argument based on Lemma \ref{lemma:Hellinger_to_Polynomial} together with Lemma \ref{lemma:dung_multicluster_mass_test_function}, we obtain the estimate
\begin{align}
    \left|\int \bar{E}_j(\theta)\, d\nu(\theta)\right|
    \leq
    C_{\bar{E},\mathrm{int}}
    \Delta_{\mathrm{sep}}^{-(2s_{\max}-1)}
    d_H(p_G,p_{G_*}).
\label{eqn:dung_thm_4_2_multivariate_global_integral_of_E_bound}
\end{align}

We now estimate the weighted sum of the terms $\bar{E}_j(\theta_i)$ according to the geometric location of $\theta_i$ relative to its nearest reference point and associated cluster.

\emph{Case 4.1 - Near set outside $\mathcal{C}_m$.} For each $i \in \bar{\mathcal{V}}_l$ with $l \in \mathcal{C}_m$, by Taylor's expansion with Lagrange remainder, Lemma \ref{lemma:dung_multicluster_mass_test_function}, and the estimation in equation \eqref{eqn:dung_thm3.2_multivarite_global_variation_bound}, we obtain, by the same argument as in the proof of the univariate global case,
\begin{align}
\nonumber
    \left|\sum_{l \in \mathcal{C}_m}\sum_{i\in \bar{\mathcal{V}}_l}\pi_i(\bar{E}_j(\theta_i) -\bar{E}_j(\theta^*_l))\right| &\leq C_{\bar{E},2}\Delta^{-2}_{\mathrm{sep}}\sum_{i\in\bar{\mathcal{V}}_l, l\in\mathcal{C}_m}\pi_{i}\|\theta_i-\theta^*_l\|_2^2 \\
    &\leq C_{\mathrm{near,in,mass}}\Delta^{-2s_{\max}}_{\mathrm{sep}}d_H(p_{G},p_{G_*})
\label{eqn:dung_thm_4_2_multivariate_global_sum_near_f-f}
\end{align}
where $C_{\mathrm{near,in,mass}} = C_{\bar{E},2}C_{\mathrm{var},1}$.

\emph{Case 4.2 - Near set outside $\mathcal{C}_m$.} For index $i \in \mathcal{M}_{\mathrm{mic}}$ that do not belong to the cluster $\mathcal{C}_m$, Lemma \ref{lemma:dung_bound_for_cluster_mass_extracting} gives
\begin{equation*}
    |\bar{E}_j(\theta_i)|
    \leq
    C_{\mathrm{cross},\bar{E}}
    \Delta_{\mathrm{sep}}^{-1}
    \|\theta_i-\theta^*_{c(i)}\|_2^2 .
\end{equation*}
Following the same argument as in the univariate and multivariate global part in Theorem \ref{theorem:exact_multi_group}, this bound implies
\begin{align}
|
\sum_{\substack{i \in \mathcal{M}_{\mathrm{mic}}\\ c(i) \notin \mathcal{C}_m}}
\pi_i \bar{E}_j(\theta_i)
|
&\leq
C_{\mathrm{cross},\bar{E}}
\Delta_{\mathrm{sep}}^{-1}
\sum_{\substack{i \in \mathcal{M}_{\mathrm{mic}}\\ c(i) \notin \mathcal{C}_m}}
\pi_i \|\theta_i-\theta^*_{c(i)}\|_2^2
\nonumber\\
&\leq
C_{\mathrm{near,out,mass}}
\Delta_{\mathrm{sep}}^{-(2s_{\max}-1)}
d_H(p_G,p_{G_*}),
\label{eqn:dung_thm_4_2_multivariate_global_near_not_m_value_estimation}
\end{align}
where $C_{\mathrm{near,out,mass}}
=
C_{\mathrm{cross},\bar{E}}C_{\mathrm{var},1}$.

\emph{Case 4.3 - Macro set inside $\mathcal{C}_m$.} For index $i$ in $\mathcal{M}_{\text{macro}}$ such that $c(i) \in \mathcal{C}_m$, using the same argument as in univariate and multivariate global part in Theorem \ref{theorem:exact_multi_group}, we have 
\begin{equation*}
    |\bar{E}_j(\theta_l)| \leq C_{\bar{E},\mathrm{mac,in,mass}}\Delta^{-2s_{\max}}_{\mathrm{sep}}\|\theta_l - \theta^*_{c(l)}\|_2^{2s_{\max}}.
\end{equation*}
Using the high moment estimation in equation \eqref{eqn:dung_thm_4_2_multivariate_global_macro_high_moment_estimate}, we have
\begin{align}
\nonumber
\left|\sum_{i \in \mathcal{M}_{\mathrm{mac}},c(i)\in\mathcal{C}_m} \pi_i \bar{E}_j(\theta_i)\right| &\leq C_{\bar{E},\mathrm{mac,in,mass}}\Delta^{-2s_{\max}}_{\mathrm{sep}}\sum_{i \in \mathcal{M}_{\mathrm{mac}},c(i)\in\mathcal{C}_m}\pi_{i}\|\theta_i - \theta_{c(i)}\|^{2s_{\max}}\\
&\leq C_{\mathrm{mac,in,mass}} \Delta^{-2s_{\max}}_{\mathrm{sep}}d_H(p_{G},p_{G_*}),
\label{eqn:dung_thm_4_2_multivariate_global_macro_m_set_sum_estimation}
\end{align}
where $C_{\mathrm{mac,in,mass}} = C_{\bar{E},\mathrm{mac,in,mass}}\cdot C_{\mathrm{mac,moment}}$.

\emph{Case 4.4 - Macro set outside $\mathcal{C}_m$.} For index $i \in \mathcal{M}_{\text{macro}}$ but $i \notin \mathcal{C}_m$, using the same argument as in univariate and multivariate global part in Theorem \ref{theorem:exact_multi_group}, we have $\ell^2_{j,\text{micro}}(\theta_l) \leq (2R)^{2(s_m-1)}\Delta^{-2(s_m-1)}_{\mathrm{sep}}$, and $|P_{\text{macro}}(\theta_l)|\leq (2R)^{-2s_{\max}}\|\theta_l-\theta^*_{c(l)}\|^{2s_{\max}}$, and for degree-one component, $|\bar{A}_j + \langle \bar{B}_j,\theta_i - \theta_j^*\rangle| \leq 2R(\bar{C}_A +\bar{C}_B)\Delta^{-1}_{\mathrm{sep}}$. Consequently, we have
\begin{equation*}
    |\bar{E}_j(\theta_l)| \leq C_{\bar{E},\mathrm{mac,out,mass}}\Delta_{\mathrm{sep}}^{-(2s_{m} - 1)} \|\theta_{l} - \theta_{c(l)}^{*}\|_2^{2s_{\max}},
\end{equation*}
Therefore, thanks to the high moment estimation in equation \eqref{eqn:dung_thm_4_2_multivariate_global_macro_high_moment_estimate}, we have 
\begin{align}
\nonumber
\left|\sum_{i \in \mathcal{M}_{\mathrm{mac}},c(i)\notin\mathcal{C}_m} \pi_i \bar{E}_j(\theta_i)\right| &\leq C_{\bar{E},\mathrm{mac,out,mass}}\Delta^{-(2s_{\max}-1)}_{\mathrm{sep}}\sum_{i \in \mathcal{M}_{\mathrm{mac}},c(i)\notin\mathcal{C}_m}\pi_{i}\|\theta_i - \theta_{c(i)}\|^{2s_{\max}}\\
&\leq C_{\mathrm{mac,out,mass}} \Delta^{-(2s_{\max}-1)}_{\mathrm{sep}}d_H(p_{G},p_{G_*}),
\label{eqn:dung_thm_4_2_multivariate_global_macro_not_m_set_sum_estimation}
\end{align}
where $C_{\mathrm{mac,out,mass}}  = C_{\bar{E},\mathrm{mac,out,mass}}\cdot C_{\mathrm{mac,moment}}$. 

\emph{Case 4.5 - Far set $\mathcal{M}_{\mathrm{far}}$}. 
For the centers $i$ in $\mathcal{M}_{\mathrm{far}}$, as $\|\theta_i -\theta^*_{c(i)}\|_2\leq 2R$, using the same argument as in univariate and multivariate global part in Theorem \ref{theorem:exact_multi_group}, we have 
\begin{align*}
|\bar{E}_j(\theta_i)| \leq (2R)^{2s_m-1}(\bar{C}_A + \bar{C}_B)\Delta^{-(2s_m-1)}_{\mathrm{sep}}.
\end{align*}
Using estimation for global far mass in equation~\eqref{eqn:dung_thm_4_2_multivariate_global_far_mass} and the same argument as in univariate and multivariate global part in Theorem \ref{theorem:exact_multi_group}, we have 
\begin{align}
\nonumber
\left|\sum_{i \in \mathcal{M}_{\mathrm{far}}} \pi_i \bar{E}_j(\theta_i)\right| &\leq (2R)^{2s_m-1}(\bar{C}_A + \bar{C}_B)\Delta^{-(2s_{m}-1)}_{\mathrm{sep}}\sum_{i \in \mathcal{M}_{\mathrm{far}}}\pi_{i}\\
&\leq C_{\mathrm{far,mass,test}} \Delta^{-(2s_{\max}-1)}_{\mathrm{sep}}d_H(p_{G},p_{G_*}).
    \label{eqn:dung_thm_4_2_multivariate_global_far_set_sum_estimation_of_E}
\end{align}
Combining the estimations from equations 
\eqref{eqn:dung_thm_4_2_multivariate_global_integral_of_E_bound}, \eqref{eqn:dung_thm_4_2_multivariate_global_sum_near_f-f}, \eqref{eqn:dung_thm_4_2_multivariate_global_near_not_m_value_estimation}, \eqref{eqn:dung_thm_4_2_multivariate_global_macro_m_set_sum_estimation}, \eqref{eqn:dung_thm_4_2_multivariate_global_macro_not_m_set_sum_estimation}, and \eqref{eqn:dung_thm_4_2_multivariate_global_far_set_sum_estimation_of_E}, as in the proof of the univariate and multivariate global part in Theorem~\ref{theorem:exact_multi_group}, we have 
\begin{align}
\label{eqn:dung_thm_4_2_multivariate_global_delta_pi_j_estimation}
    |\Delta\pi_j| \leq C_{\mathrm{mass}}\cdot \Delta^{-2s_{\max}}_{\mathrm{sep}}d_H(p_G,p_{G_*}),
\end{align}
where 
\begin{align*}
    C_{\mathrm{mass}} &= 2R\left[C_{\bar{E},\mathrm{int}} +  C_{\mathrm{near,out,mass}} + C_{\mathrm{mac,in,mass}} + C_{\mathrm{mac,out,mass}} + C_{\mathrm{far,mass,test}}\right] \\
    &\hspace{1cm}  + C_{\mathrm{near,in,mass}}. 
\end{align*}

\emph{Step 5 - Bounding spare mass discrepancy $\sum_{j=1:|\bar{\mathcal{V}}_j| = 1}^{k_*}\sum_{i\in \bar{\mathcal{V}}_j}\pi_i\|\theta_i - \theta^*_{j}\|$.}  Consider an index $j \in [k_*]$ such that $\bar{\mathcal{V}}_j$ contains exactly one element $\theta_i$. For this term, we consider the mean extractor test function $\bar{H}_{j,i}$ defined in Section \ref{sec:non_multi_cluster_mean_test_function} given by 
$$\bar{H}_{j,i}(\theta) = \ell_{j,\text{micro}}^2(\theta) [\langle \bar{v}_{j,i},\theta - \theta_{j}^{*}\rangle]P_{\text{macro}}(\theta),$$ 
where we define $\ell_{j,\text{micro}}(\theta) = \prod_{q \in \mathcal{C}_{m} \neq j} \frac{\|\theta - \theta_q^*\|_2}{\|\theta_j^* - \theta_q^*\|_2}$, $P_{\text{macro}}(\theta) = \prod_{p \neq m} \prod_{q \in \mathcal{C}_p} \left( \frac{\|\theta - \theta_q^*\|_2}{2R} \right)^{2s_{\max}}$, and $$\bar
v_{j,i} = P^{-1}_{\text{macro}}(\theta_j^*)\frac{\theta_i-\theta_j^*}{\|\theta_i-\theta_j^*\|},$$
when $\theta_i\neq \theta_j^*$ and $\bar{v}_{j,i} = P^{-1}_{\text{macro}}(\theta_j^*)(1,0,\ldots,0)^{\top}$ otherwise. This test function satisfies $\bar{H}_{j,i}(\theta_l^*) = 0$ for all $l$, $\nabla \bar{H}_{j,i}(\theta_l^*) = 0$ for all $l\neq i$, and $\langle \nabla \bar{H}_{j,i}(\theta_i^*), \theta_i - \theta_j^*\rangle = \|\theta_i - \theta_j^*\|_2$.

For any center $\theta_l \in \mathcal{M}_{\mathrm{near}}$, by applying Taylor expansion for 
$\bar{H}_{j,i}(\theta_l)$ exactly around its true center $\theta_{c(l)}^*$, we have 
\begin{align*}
\bar{H}_{j,i}(\theta_l) &= \delta_{jl}v_{j,i}^{\top}(\theta_{l} - \theta_{c(l)}^{*}) + \frac{1}{2} (\theta_{l} - \theta_{c(l)}^{*})^{\top}D^2\bar{H}_{j,i}(\xi_l) (\theta_{l} - \theta_{c(l)}^{*}) \\
&= \delta_{jl}\|\theta_i-\theta^*_{c(i)}\| + \frac{1}{2} (\theta_{l} - \theta_{c(l)}^{*})^{\top}D^2\bar{H}_{j,i}(\xi_l) (\theta_{l} - \theta_{c(l)}^{*}).
\end{align*}
Thus, we arrive at $$ \int \bar{H}_{j,i}(\theta) d\nu(\theta) = \pi_{i}\|\theta_{i} - \theta_{c(i))}^{*}\| + \frac{1}{2} \sum_{l\in \mathcal{M}_{\mathrm{near}}} \pi_{l} (\theta_{l} - \theta_{c(l)}^{*})^{\top} D^2\bar{H}_{j,i}(\xi_l)(\theta_{l} - \theta_{c(l)}^{*})+ \sum_{l \in \mathcal{M}_{\mathrm{far}}} \pi_l \bar{H}_{j,i}(\theta_l). ,$$
which indicates that
\begin{align}
\nonumber
    \pi_{i}\|\theta_{i} - \theta_{c(i)}^{*}\| &= \int \bar{H}_{j,i}(\theta)d\nu(\theta) - \underbrace{\left(\sum_{n \in \mathcal{C}_m}\sum_{l\in \bar{\mathcal{V}}_n}\pi_i(\bar{H}_{j,i}(\theta_l) -\bar{H}_{j,i}(\theta^*_l))\right)}_{\text{Case 5.1}}\\
    &\hspace{1cm } -\underbrace{\sum_{\substack{l \in \mathcal{M}_{\mathrm{mac}}\\ c(l) \in \mathcal{C}_m }}\pi_l\bar{H}_{j,i}(\theta_l)}_{\text{Case 5.2}} - \underbrace{\sum_{\substack{l \in \mathcal{M}_{\mathrm{mac}}\\ c(l) \notin \mathcal{C}_m }}\pi_l\bar{H}_{j,i}(\theta_l)}_{\text{Case 4.3}} -  \underbrace{\sum_{{l\in \mathcal{M}_{\mathrm{far}}}}\pi_l\bar{H}_{j,i}(\theta_l)}_{\text{Case 4.5}}.
\label{eqn:dung_thm_4_2_first_moment_estimation}
\end{align}
For the integral function, using Lemma \ref{lemma:mean_test_function} and Lemma \ref{lemma:Hellinger_to_Polynomial}, we have
\begin{align}
\nonumber
\left| \int \bar{H}_{j,i}(\theta) d\nu(\theta) \right| &\le C_{\mathrm{poly}}\|\bar{H}_{j,i}\|_{\infty}d_{H}(p_G,p_{G_*})\\
&\leq C_{\bar{H},\mathrm{int}} \Delta_{\mathrm{sep}}^{-(2s_{\max}-1)} d_H(p_{G},p_{G_*}),\label{eqn:dung_thm_4.2_multivariate_global_mean_unique_point_test_function_dual_bound}
\end{align}
where $C_{\bar{H},\mathrm{int}} = C_{\mathrm{poly}}C_{\text{norm},\bar{H}}$. 

We now estimate the weighted sum of the terms $\bar{H}_{j,i}(\theta_i)$ by distinguishing the geometric location of $\theta_i$ relative to its nearest reference point and its associated cluster.

\emph{Case 5.1 - Near error.} For the near error, using Lemma \ref{lemma:dung_multicluster_mean_test_function} and the estimationin equation~\eqref{eqn:dung_thm_4_2_multivarite_global_variation_bound}, we obtain that
\begin{align}
\frac{1}{2}\sum_{n=1}^{k_*} \sum_{l \in \bar{\mathcal{V}}_n} \pi_l \|D^2\bar{H}_{j,i}(\xi_l)\|_{\mathrm{op}}\|\theta_l-\theta_{c(l)}^*\|^2 \leq C_{\mathrm{near,mean}}\Delta_{\mathrm{sep}}^{-(2s_{\max}-1)} d_H(p_{G},p_{G_*}),\label{eqn:dung_thm_4_2_global_multivariate_sum_of_second_derivative_of_H}
\end{align}
where $C_{\mathrm{near,mean}} = \frac{1}{2} C_{\bar{H},2} C_{\mathrm{var,1}}$
\emph{Case 5.2 - Macro set inside $\mathcal{C}_m$.} For center $\theta_l$ indexed by  $l \in \mathcal{M}_{\text{macro}}$ such that $c(l) \in \mathcal{C}_m$, using the same argument as for the derivation of 
equation~\eqref{eqn:ell_j_micro_macro_in_cluster}, we have 
\begin{equation*}
    |\ell_{j,\text{micro}}(\theta_l)|^2 \leq \Delta^{-2(s_m-1)}_{\mathrm{sep}}(1+4C_0)^{2(s_m-1)}\|\theta_l - \theta^*_{c(l)}\|^{2(s_m-1)}. 
\end{equation*}
For the polynomial of degree 1, as 
$\|\theta_l - \theta_j^*\|_2\leq (1+4C_0)\|\theta_l-\theta^*_{c(l)}\|_2$ which is proved in equation~\eqref{eqn:dung_thm_3_2_univariate_global_bound_for_distance_inside_cluster}, we have 
\begin{equation*}
|\langle\bar{v}_{j,i},\theta_{l} - \theta^*_{j}\rangle| \leq \bar{C}_{\bar{v}}(1+4C_0)\|\theta_l-\theta^*_{c(l)}\|. 
\end{equation*}
Thus, as $|P_{\mathrm{macro}}(\theta_l)| \leq 1$, we achieve the following estimation for $|\bar{H}_{j,i}(\theta_l)|$. 
\begin{align*}
    |\bar{H}_{j,i}(\theta_l)| \leq \bar{C}_{\bar{v}}(1+4C_0)^{2s_m-1}\Delta^{-2(s_{m}-1)}_{\mathrm{sep}}\|\theta_l - \theta^*_{c(l)}\|^{2s_{m}-1}.
\end{align*}
Moreover, as $\|\theta_l - \theta^*_{c(l)}\|\geq \Delta_{\mathrm{sep}}/4$, we have 
\begin{align*}
    1&\leq 4^{2s_{\max}-2s_m+1}\|\theta_l - \theta^*_{c(l)}\|^{2s_{\max}-2s_m+1}\cdot \Delta^{-(2s_{\max}-2s_m+1)}_{\mathrm{sep}}\\
    &\leq 4^{2k_*+1}\|\theta_l - \theta^*_{c(l)}\|^{2s_{\max}-2s_m+1}\cdot \Delta^{-(2s_{\max}-2s_m+1)}_{\mathrm{sep}}. 
\end{align*}
Thus, we have 
\begin{equation*}
    |\bar{H}_{j,i}(\theta_l)| \leq C_{\bar{H},\mathrm{mac,in,mean}}\Delta^{-(2s_{\max}-1)}_{\mathrm{sep}}\|\theta_l - \theta^*_{c(l)}\|^{2s_{\max}},
\end{equation*}
where $C_{\bar{H},\mathrm{mac,in,mean}} = 4^{2k_*+1}\bar{C}_{\bar{v}}(1+4C_0)^{2k_*-1}$
Therefore, thanks to the high moment estimation in equation \eqref{eqn:dung_thm_4_2_multivariate_global_macro_high_moment_estimate}, we have 
\begin{align}
\nonumber
\left|\sum_{i \in \mathcal{M}_{\mathrm{mac}},c(i)\in\mathcal{C}_m} \pi_i \bar{H}_{j,i}(\theta_i)\right| &\leq C_{\bar{H},\mathrm{mac,in,mean}}\Delta^{-(2s_{\max}-1)}_{\mathrm{sep}}\sum_{i \in \mathcal{M}_{\mathrm{mac}},c(i)\in\mathcal{C}_m}\pi_{i}\|\theta_i - \theta_{c(i)}\|^{2s_{\max}}\\
&\leq C_{\mathrm{mac,in,mean}} \Delta^{-(2s_{\max}-1)}_{\mathrm{sep}}d_H(p_{G},p_{G_*}),
\label{eqn:dung_thm_4_2_macro_m_set_sum_estimation}
\end{align}
where $c  = C_{\bar{H},\mathrm{mac,in,mean}}\cdot C_{\mathrm{mac,moment}}$.

\emph{Case 5.3 - Macro set outside $\mathcal{C}_m$.} For index $l \in\mathcal{M}_{\text{macro}}$ such that $c(l)\notin \mathcal{C}_m$, we similarly evaluate each term $P_{\text{macro}}(\theta_l)$, $l^2_{j,\text{micro}}(\theta_l)$, and component of degree 1 separately. Indeed, we have for $|\ell^2_{j,\text{micro}}(\theta_l)| \leq (2R)^{2(s_m-1)}\Delta^{-2(s_m-1)}_{\mathrm{sep}}$ due to the fact that $\|\theta_l - \theta^*_{q}\| \leq 2R$. For $P_{\text{macro}}$, noting that $\|\theta_l - \theta^*_{q}\| \leq 2R$, we have 
\begin{equation*}
    |P_{\text{macro}}(\theta_l)|\leq (2R)^{-2s_{\max}}\|\theta_l-\theta^*_{c(l)}\|^{2s_{\max}}.
\end{equation*}
For component of degree 1, we have 
\begin{equation*}
    |\langle \bar{v}_{j,i},\theta_l - \theta^*_{i}\rangle| \leq 2R\bar{C}_{\bar{v}}. 
\end{equation*}
Consequently, as $\Delta_{\mathrm{sep}} \leq 2R$, we have
\begin{align*}
    |\bar{H}_{j,i}(\theta_l)| &\leq \bar{C}_{\bar{v}} (2R)^{2(s_{m} - 1) - 2s_{\max}+ 1} \Delta_{\mathrm{sep}}^{-(2s_{m} - 2)} \|\theta_{l} - \theta_{c(l)}^{*}\|^{2s_{\max}}\\
    &\leq \bar{C}_{\bar{v}} (2R)^{2(s_{m} - 1) - 2s_{\max}+ 1} (2R)^{2s_{\max}-2s_m}\Delta_{\mathrm{sep}}^{-(2s_{\max} - 2)} \|\theta_{l} - \theta_{c(l)}^{*}\|^{2s_{\max}}.
\end{align*}
As a result, we obtain the following bound for $ |\bar{H}_{j,i}(\theta_l)|$: 
\begin{equation*}
    |\bar{H}_{j,i}(\theta_l)| \leq C_{\bar{H},\mathrm{mac,out,mean}}\Delta_{\mathrm{sep}}^{-(2s_{\max} - 2)} \|\theta_{l} - \theta_{c(l)}^{*}\|^{2s_{\max}},
\end{equation*}
where $C_{\bar{H},\mathrm{mac,out,mean}} = \bar{C}_{\bar{v}}\cdot (2R)^{-1}$. Therefore, thanks to the high moment estimation in equation \eqref{eqn:dung_thm_4_2_multivariate_global_macro_high_moment_estimate}, we have 
\begin{align}
\nonumber
\left|\sum_{i \in \mathcal{M}_{\mathrm{mac}},c(i)\notin\mathcal{C}_m} \pi_i \bar{H}_{j,i}(\theta_l)\right| &\leq C_{\bar{H},\mathrm{mac,out,mean}}\Delta^{-(2s_{\max}-2)}_{\mathrm{sep}}\sum_{i \in \mathcal{M}_{\mathrm{mac}},c(i)\notin\mathcal{C}_m}\pi_{i}\|\theta_i - \theta^*_{c(i)}\|^{2s_{\max}}\\
&\leq C_{\mathrm{mac,out,mean}}\Delta^{-(2s_{\max}-2)}_{\mathrm{sep}}d_H(p_{G},p_{G_*}),
\label{eqn:dung_thm_4_2_macro_not_m_set_sum_estimation}
\end{align}
where $C_{\mathrm{mac,out,mean}}  = C_{\bar{H},\mathrm{mac,out,mean}}C_{\mathrm{mac,moment}}$. 

\emph{Case 5.4 - Far set.} For the centers $i \in \mathcal{M}_{\mathrm{far}}$, as  $|\theta_i -\theta^*_{c(i)}|\leq 2R$, it is straightforward to verify that $|\ell^2_{j,\text{macro}}(\theta_i)| \leq (2R)^{2(s_m-1)}\Delta^{-(2s_m-2)}_{\mathrm{sep}}$, $|P_{\text{macro}}(\theta_i)| \leq 1$. In addition, for the component of degree 1, as in Case 5.3, we have 
\begin{equation*}
|\langle\bar{v}_{j,i},\theta_l - \theta^*_{j}\rangle| \leq 2R\bar{C}_{\bar{v}}.
\end{equation*}
Using all the above results together, we have 
\begin{align*}
    |\bar{H}_{j,i}(\theta_i)| \leq (2R)^{2s_m-1}\bar{C}_{\bar{v}}\Delta^{-2(s_m-1)}_{\mathrm{sep}} \leq (2R)^{2s_{\max}-1}\bar{C}_{\bar{v}}\Delta^{-2(s_{\max}-1)}_{\mathrm{sep}}. 
\end{align*}
Using the estimation for far mass in equation~\eqref{eqn:dung_thm_4_2_multivariate_global_far_mass}, we have  
\begin{align}
\nonumber
\left|\sum_{l \in \mathcal{M}_{\mathrm{far}}} \pi_l \bar{H}_{j,i}(\theta_l)\right| &\leq (2R)^{2s_{\max}-1}\bar{C}_{\bar{v}}\Delta^{-2(s_{\max}-1)}_{\mathrm{sep}}\sum_{i \in \mathcal{M}_{\mathrm{far}}}\pi_{i}\\
&\leq C_{\mathrm{far,mean}} \Delta^{-(2s_{\max}-1)}_{\mathrm{sep}}d_H(p_{G},p_{G_*}),
    \label{eqn:dung_thm_4_2_multivariate_global_far_set_sum_estimation}
\end{align}
where $C_{\mathrm{far,mean}} = \bar{C}_{\bar{v}}C_{\mathrm{far,mass}}\cdot\max\{1,(2R)^{2k_*-1}\}$. 

Combining the estimation in equations~\eqref{eqn:dung_thm_4.2_multivariate_global_mean_unique_point_test_function_dual_bound}, \eqref{eqn:dung_thm_4_2_global_multivariate_sum_of_second_derivative_of_H}, \eqref{eqn:dung_thm_4_2_macro_m_set_sum_estimation}, \eqref{eqn:dung_thm_4_2_macro_not_m_set_sum_estimation} and \eqref{eqn:dung_thm_4_2_multivariate_global_far_set_sum_estimation}, we have 
\begin{align}
\label{eqn:dung_thm4_1_multivariate_global_unique_near_first_moment_bound_2}
    \pi_i\|\theta_i-\theta_{c(i)}^*\|\leq C_{\mathrm{mean}}\Delta_{\mathrm{sep}}^{-(2s_{\max}-1)} d_H(p_{G},p_{G_*})
\end{align}
where 
\begin{equation*}
C_{\mathrm{mean}}= C_{\bar{H},\mathrm{int}}  +C_{\mathrm{near,mean}} +  C_{\mathrm{mac,in,mean}} + C_{\mathrm{far,mean}} + 2R\cdot C_{\mathrm{mac,out,mean}}. 
\end{equation*}
\newline 

\emph{Step 6 - Bounding $W_2^2(G,G_*)$ and conclusion.} Combining the result from equations~\eqref{eqn:dung_thm_4_2_multivarite_global_variation_bound}, \eqref{eqn:dung_thm_4_2_multivariate_global_sum_delta_pi_estimation}, \eqref{eqn:dung_thm_4_2_multivariate_global_delta_pi_j_estimation} and \eqref{eqn:dung_thm4_1_multivariate_global_unique_near_first_moment_bound_2} and plugging back into equation~\eqref{eqn:dung_thm_4_2_multivariate_global_final_W2_prelim_bound}, we have 
\begin{align*}
    W_2^2(G,G_*)&\leq S_2(G,G_*)\leq C^{-1}_{\mathrm{global},5} \Delta_{\mathrm{sep}}^{-(2s_{\max}-2)} d_H(p_{G},p_{G_*}),
\end{align*}
where 
\begin{align*}
C^{-1}_{\mathrm{global},5} = &\frac{1}{2}C_{\mathrm{mean}} + C_{\mathrm{var},1} + 8R^2k_*\cdot C_{\Delta\Pi}\\
    &\hspace{1cm}+ 2C_0^2(C_{\mathrm{mac,mass}} +\max\{1,(2R)^{2k_*-2}\}C_{\mathrm{far,mass}} +  C_{\text{zero\_moment}}). 
\end{align*}
As a result, we have 
\begin{equation*}
    d_H(p_{G},p_{G_*}) \geq C_{\mathrm{global},5}\Delta_{\mathrm{sep}}^{2s_{\max}-2} S_2(G,G_*) \geq C_{\mathrm{global},5}\Delta_{\mathrm{sep}}^{2s_{\max}-2} W_2^2(G,G_*). 
\end{equation*}

\subsection{Proof of Theorem~\ref{theorem:no_group_overspecified}}
\label{sec:proof_theorem:no_group_overspecified}
In this proof, for each atom $\theta_i$, we denote $$c(i) = \arg\min_{1\leq j\leq k_*}\|\theta_i-\theta_j^*\|$$ be its absolute closest true center (we pick arbitrarily one of the centers when there exists several centers sharing minimal distance). For each $j \in [k_*]$, let $\tilde{\pi}_j = \sum_{i:c(i)=j}\pi_i$ and $\Delta\tilde{\pi}_j = \tilde{\pi}_j-\pi_j$. Let $\mathcal{V}_j = \{l: c(k) = \theta_j^*\}$ and $\bar{\mathcal{V}}_j = \{l: c(l) = \theta_j^*, \text{ and } \|\theta_l - \theta_j^*\|\leq \Delta_{\mathrm{sep}}/4 \}$. 

\emph{Step 1 - Wasserstein decomposition.} Consider the intermediate measure $$G' = \sum_{j=1}^k\pi_j\delta_{\theta^*_{c(j)}} = \sum_{j=1}^{k_*}\tilde{\pi}_j\delta_{\theta^*_{j}},$$ using triangle inequality for Wasserstein distance, we have 
\begin{align}
\label{eqn:dung_thm_4_3_multivariate_global_W2_prelim_split}
    W_2^2(G, G_*) \leq(W_2(G, G') + W_2(G', G_*))^2 \leq 2W_2^2(G, G') + 2W^2_2(G', G_*).
\end{align}
For $W_2^2(G, G')$, by considering the transportation plan $\rho_{jj} = \pi_j$ for all $1\leq j\leq k$, which implies
\begin{align}
\label{eqn:dung_thm_4_3_multivariate_global_W2_prelim_first_bound}
    W_2^2(G, G') \leq \sum_{i=1}^{k} \pi_i\|\theta_i-\theta^*_{c(i)}\|^2 = \sum_{j=1}^{k_*}\sum_{i \in \bar{\mathcal{V}}_j}\pi_i\|\theta_i - \theta_j^*\|^2.
\end{align}

For $W_2^2(G', G_*)$, we consider the transportation plan such that for each $j$, we keep $\min\{\sum_{i\in\mathcal{V}_j}\pi_i,\pi^*_j\}$ for the center $\theta_j^*$ while moving $|\Delta\pi_j|$. Then, the total mass transported is exactly $\frac{1}{2} \sum |\Delta \pi_j|$, and  largest squared distance between any true centers does not exceed $4R^2$, we have 
\begin{align}
\label{eqn:dung_thm_4_3_local_W2_prelim_second_bound}
    W_2^2(G', G_*) \leq 4R^2 \sum_{j=1}^{k_*} |\tilde{\Delta} \pi_j|. 
\end{align}
Plugging the results in equations~\eqref{eqn:dung_thm_4_3_multivariate_global_W2_prelim_first_bound} and \eqref{eqn:dung_thm_4_3_local_W2_prelim_second_bound} into equation~\eqref{eqn:dung_thm_4_3_multivariate_global_W2_prelim_split}, we have 
\begin{align}
W_2^2(G, G_*) &\leq  2\sum_{i=1}^{k}\pi_i\|\theta_i - \theta_{c(i)}^*\|^2 + 8R^2 \sum_{i=1}^{k_*} |\tilde{\Delta} \pi_i|\\
&\leq \frac{\Delta_{\mathrm{sep}}}{2} \sum_{j:|\bar{\mathcal{V}}_j|=1,\bar{\mathcal{V}}_j=\{i\}}\pi_i\|\theta_i-\theta^*_{c(i)}\| + 2\sum_{j:|\bar{\mathcal{V}}_j|\neq 1}\sum_{i \in \mathcal{V}_j}\pi_i\|\theta_i - \theta_{c(i)}^*\|^2+8R^2 \sum_{i=1}^{k_*} |\tilde{\Delta} \pi_i|.
\label{eqn:dung_thm_4_3_multivariate_global_W2_decomposition}
\end{align}
Let 
\begin{equation}
\label{eqn:dung_thm_4_3_voronoi_loss}
    S_3(G,G_*) := \frac{\Delta_{\mathrm{sep}}}{2} \sum_{j:|\bar{\mathcal{V}}_j|=1,\bar{\mathcal{V}}_j=\{i\}}\pi_i\|\theta_i-\theta^*_{c(i)}\| + 2\sum_{j:|\bar{\mathcal{V}}_j|\neq 1}\sum_{i \in \mathcal{V}_j}\pi_i\|\theta_i - \theta_{c(i)}^*\|^2+8R^2 \sum_{i=1}^{k_*} |\tilde{\Delta} \pi_i|, 
\end{equation}
we prove a stronger result
\begin{equation*}
    d_H(p_G,p_{G_*})\geq C_{\mathrm{global},6}\Delta_{\mathrm{sep}}^{4k_*-3}S_3(G,G_*). 
\end{equation*}
\newline 

\emph{Step 2 - Bounding variance $\sum_{i=1}^{k_*}\pi_i\|\theta_i-\theta^*_{c(i)}\|^2$.} We consider the globally positive witness function defined in Section \ref{sec:variance_test_function}: 
\begin{align*}
    P_{\mathrm{var}}(\theta) = \prod_{l=1}^{k_*} \|\theta - \theta_l^*\|^2.
\end{align*}
Using the same argument as in the univariate global and multivariate global part of Theorem \ref{theorem:exact_no_group}, we arrive at 
\begin{align}
    \sum_{j = 1}^{k_{*}} \sum_{i \in \mathcal{V}_{j}} \pi_{i} \|\theta_{i} - \theta_{j}^{*}\|^2 \leq C_{\mathrm{var},1} \Delta_{\mathrm{sep}}^{-(2k_* - 2)} d_H(p_{G},p_{G_*}). \label{eqn:dung_thm__4_3_exact_no_group_global_second_moment_bound} 
\end{align}

\emph{Step 3 - Bounding mass discrepancy $|\tilde{\Delta} \pi_j|$. } To obtain an upper bound for $\sum_{j = 1}^{k_{*}} |\tilde{\Delta} \pi_j|$, we consider the witness function which is defined in Section \ref{sec:point_wise_mass_extractor_non_multicluster}: 
\begin{align*}
    E_{j}(\theta) : = \ell_j^2(\theta) \left[ \bar{A}_j + \langle \bar{B}_j, \theta - \theta_j^*\rangle\right], \quad \text{where} \quad \ell_j(\theta) = \prod_{q \neq j} \frac{\|\theta - \theta_q^*\|}{\|\theta_j^* - \theta_q^*\|},
\end{align*}
and $\bar{A}_j = 1$, $\bar{B}_j = - 2\nabla\ell_j(\theta_j^*)$. This function satisfies the condition that $E_j(\theta^*_i) = \delta_{jl}$ and $\nabla E_j(\theta^*_i) = 0$. By means of Taylor expansion for $E_j(\theta_l)$ exactly around its true center $\theta_l^*$, we have $E_j(\theta_l) = \delta_{jl} + \frac{1}{2}(\theta_{l} - \theta_{c(l)}^{*})^{\top}D^2E_j(\xi_l) (\theta_{l} - \theta_{c(l)}^{*})$, where $\xi_l$ is between $\theta_l$ and $\theta_{c(l)}^*$ which implies that
$$ \int E_j(\theta) d\nu(\theta) = \left(\sum_{i \in \mathcal{V}_{j}} \pi_{i} - \pi_{j}^{*}\right) + \frac{1}{2}\sum_{l=1}^{k_*} \sum_{i \in \mathcal{V}_l} \pi_{i} (\theta_{l} - \theta_{c(l)}^{*})^{\top}D^2E_j(\xi_l) (\theta_{l} - \theta_{c(l)}^{*}).$$ 
Applying the triangle inequality and Lemma \ref{lemma:Hellinger_to_Polynomial} leads to
\begin{align}
    |\tilde{\Delta} \pi_j| &\le \left| \int E_j(\theta) d\nu(\theta) \right| + \frac{1}{2}\sum_{l=1}^{k_*} \sum_{i \in \mathcal{V}_l} \pi_{i} \left|(\theta_{i} - \theta_{l}^{*})^{\top}D^2E_j(\xi_i) (\theta_{i} - \theta_{l}^{*})^2\right|\nonumber\\
    &\leq C_{\mathrm{poly}}\|E_j\|_{\infty}d_H(p_{G},p_{G_*})+ \frac{1}{2}\max_{\theta \in [-R,R]} \|D^2E_{j}(\theta)\|_{\mathrm{op}} \sum_{l=1}^{k_*} \sum_{i \in \mathcal{V}_l} \pi_{i}  \|\theta_{i} - \theta_{l}^{*}\|^2, \label{eqn:dung_thm_4_3_exact_no_group_global_first_moment_bound}
\end{align}
Using Lemma \ref{lemma:mass_test_function} and the estimation in equation~\eqref{eqn:dung_thm__4_3_exact_no_group_global_second_moment_bound}, we have 
\begin{align}
|\tilde{\Delta} \pi_j| & \le \big[ (2R)^{2k_{*} -2} C_{\mathrm{poly}}{C}_{\text{norm},E} + \frac{1}{2} C_{E,2}C_{\mathrm{var},1} \big] \Delta_{\mathrm{sep}}^{-(4k_* - 3)} d_H(p_{G},p_{G_*}) \nonumber \\
& :=  C_{\mathrm{mass}} \Delta_{\mathrm{sep}}^{-(4k_* - 3)} d_H(p_{G},p_{G_*}). \label{eqn:dung_thm_4_3_multivariate_global_delta_small_pi_étimation}
\end{align}

\emph{Step 4 - Bounding $\sum_{j=1:|\bar{\mathcal{V}}_j| = 1}^{k_*}\sum_{i\in \bar{\mathcal{V}}_j}\pi_i\|\theta_i - \theta^*_{j}\|$. } Consider an index $j \in [k_*]$ such that $\bar{\mathcal{V}}_j$ contains exactly one element $\theta_i$. For this term, we consider the mean extractor test function $H_{j,i}$ defined in Section \ref{sec:non_multi_cluster_mean_test_function}: 
\begin{align*}
    H_{j,i}(\theta) : = \ell_j^2(\theta) \left[\langle v_{j,i}, \theta - \theta_j^*\rangle\right],  \quad \text{where} \quad \ell_j(\theta) : = \prod_{q \neq j} \frac{\|\theta - \theta_q^*\|_2}{\|\theta_j^* - \theta_q^*\|_2}, 
\end{align*}
and $v_{j,i} = {(\theta_i-\theta_j^*)}/{\|\theta_i-\theta_j^*\|}$ if $\theta_i \neq \theta_j^*$ and $v_{j,i}=(1, 0,\ldots ,0)^\top$ otherwise. We can point out that $H_{j,i}(\theta_l^*) = 0$ and $\langle \nabla H_{j,i}(\theta_j^*), \theta_i - \theta_j^*\rangle = \|\theta_i - \theta_j^*\|_2$ 
For any center $\theta_l$, by applying Taylor expansion for 
$H_{j,i}(\theta_l)$ exactly around its true center $\theta_{c(l)}^*$, we have 
\begin{align*}
H_{j,i}(\theta_l) &= \delta_{jl}v_{j,i}^{\top}(\theta_{l} - \theta_{c(l)}^{*}) + \frac{1}{2} (\theta_{l} - \theta_{c(l)}^{*})^{\top}D^2H_{j,i}(\xi_l) (\theta_{l} - \theta_{c(l)}^{*}) \\
&= \delta_{jl}\|\theta_i-\theta^*_{c(i)}\| + \frac{1}{2} (\theta_{l} - \theta_{c(l)}^{*})^{\top}D^2H_{j,i}(\xi_l) (\theta_{l} - \theta_{c(l)}^{*}).
\end{align*}
Thus, we arrive at $$ \int H_{j,i}(\theta) d\nu(\theta) = \pi_{i}\|\theta_{i} - \theta_{c(i))}^{*}\| + \frac{1}{2} \sum_{l=1}^{k} \pi_{l} (\theta_{l} - \theta_{c(l)}^{*})^{\top} D^2H_{i,i}(\xi_l)(\theta_{l} - \theta_{c(l)}^{*}),$$
which indicates that
\begin{align}
\label{eqn:dung_thm_4_3_first_moment_estimation}
    \pi_{i}\|\theta_{i} - \theta_{c(i)}^{*}\| &\leq \underbrace{\left|\int H_{j,i}(\theta) d\nu(\theta)\right|}_{\text{Test Function Dual Bound}} + \frac{1}{2} \underbrace{\sum_{l=1}^{k} \pi_{l} \|D^2H_{j,i}(\xi_l)\|_{\mathrm{op}} \|\theta_{l} - \theta_{c(l)}^{*}\|_2^2}_{\text{Error}}
\end{align} 
For the test function dual bound, using Lemma \ref{lemma:mean_test_function} and Lemma \ref{lemma:Hellinger_to_Polynomial}, we have
\begin{align}
\nonumber
\left| \int H_{j,i}(\theta) d\nu(\theta) \right| &\le C_{\mathrm{poly}}\|H_{j,i}\|_{\infty}d_H(p_G,p_{G_*})\\
&\leq C_{\mathrm{poly}}C_{\text{norm},H} \Delta_{\mathrm{sep}}^{-(2k_*-2)} d_H(p_{G},p_{G_*}). \label{eqn:dung_thm_4_3_multivariate_global_mean_unique_point_test_function_dual_bound}
\end{align}
For the error, using Lemma \ref{lemma:mean_test_function} and the estimation in equation~\eqref{eqn:dung_thm__4_3_exact_no_group_global_second_moment_bound}, we obtain that
\begin{align}
\frac{1}{2}\sum_{l=1}^{k}  \pi_l \|D^2H_{j,i}(\xi_l)\|_{\mathrm{op}}\|\theta_l-\theta_{c(l)}^*\|^2 \leq \frac{1}{2} C_{2,E} C_{\text{var,4}} \Delta_{\mathrm{sep}}^{-(4k_*-4)} d_H(p_{G},p_{G_*}).\label{eqn:dung_thm_4_3_global_multivariate_sum_of_second_derivative_of_H}
\end{align}
Plugging the results from equations \eqref{eqn:dung_thm_4_3_multivariate_global_mean_unique_point_test_function_dual_bound} and \eqref{eqn:dung_thm_4_3_global_multivariate_sum_of_second_derivative_of_H} into equation~\eqref{eqn:dung_thm_4_3_first_moment_estimation}, we have 
\begin{align}
\label{eqn:dung_thm_4_3_multivariate_global_unique_near_first_moment_bound}
    \pi_i\|\theta_i-\theta_{c(i)}^*\|\leq C_{\mathrm{mean}}\Delta_{\mathrm{sep}}^{-(4k_*-4)} d_H(p_{G},p_{G_*})
\end{align}
where 
\begin{equation*}
C_{\mathrm{mean}}= (2R)^{2k_*-2}C_{\mathrm{poly}}C_{\mathrm{norm},H} +\frac{1}{2} C_{E,2} C_{\mathrm{var,1}}.
\end{equation*}

\emph{Step 5 - Bounding $W_2^2(G,G_*)$ and conclusion.} Plugging the results from equations~\eqref{eqn:dung_thm__4_3_exact_no_group_global_second_moment_bound}, \eqref{eqn:dung_thm_4_3_multivariate_global_delta_small_pi_étimation}, \eqref{eqn:dung_thm_4_3_multivariate_global_unique_near_first_moment_bound} into equation~\eqref{eqn:dung_thm_4_3_multivariate_global_W2_decomposition}, we have  
\begin{align*}
    W_2^2(G,G_*)&\leq S_3(G,G_*)\leq C^{-1}_{\mathrm{global},6} \Delta_{\mathrm{sep}}^{-(4k_*-3)} d_H(p_{G},p_{G_*}),
\end{align*}
where 
\begin{equation*}
    C^{-1}_{\mathrm{global},6} = 2\cdot(2R)^{2k_*-1}C_{\mathrm{var},1} +  8R^2k_*\cdot C_{\mathrm{mass}}+ 2R^2k_*\cdot C_{\mathrm{mean}}. 
\end{equation*}
As a result, we have 
\begin{equation*}
    d_H(p_{G},p_{G_*}) \geq C_{\mathrm{global},6}\Delta_{\mathrm{sep}}^{4k_*-3} S_3(G,G_*) \geq C_{\mathrm{global},6}\Delta_{\mathrm{sep}}^{4k_*-3} W_2^2(G,G_*). 
\end{equation*}


\section{Test Functions}
\label{section:test_functions}
Let $\mathcal{V}$ be the vector space spanned by all true centers $\{\theta_1^*,\ldots,\theta_{k_*}^{*}\}$ and $\{\theta_1,\ldots,\theta_k\}$. Through the proof, as we construct test functions in $\mathcal{V}$, by changing the coordinate, we can suppose that $\mathcal{V}$ is exactly the embedded space $\mathbb{R}^{d_*}$ in $\mathbb{R}^{d}$. In this case, $d_*\leq \min\{k+k_*,d\}$. 

\subsection{Variance extractor test function}
\label{sec:variance_test_function}
For a family of center $\{\theta^*_1,\ldots,\theta^*_{k_*}\}$ belong to the closed ball $\bar{B}(0,R)$ such that $\min_{1\leq j,l\leq {k_*}}\|\theta^*_i - \theta^*_l\|\geq \Delta_{\mathrm{sep}}$, we construct the variance extractor test function 
\begin{align*}
    P_{\mathrm{var}}(\theta) : = \prod_{l=1}^{k_*} \|\theta - \theta_l^*\|^2.
\end{align*}
\begin{lemma}
    \label{lemma:variance_test_function}
    Given the construction of the test function $P_{\mathrm{var}}$, we have
    \begin{align*}
        \|P_{\mathrm{var}}\|_{\infty}\leq C_{\mathrm{var}}:=(2R)^{2k_*}.
    \end{align*}
\end{lemma}
\begin{proof}[Proof of Lemma~\ref{lemma:variance_test_function}]
Note that for any $\theta, \theta_l^* \in B(0,R)$, we have $\|\theta - \theta^*_l\|_2\leq 2R$. As a result, we have 
\begin{equation*}
    \|P_{\mathrm{var}}\|_{\infty} \leq \sup_{\theta \in B(0,R)}|P_{\mathrm{var}}(\theta)| = (2R)^{2k_*} := C_{\mathrm{var}}. 
\end{equation*}
\end{proof}
\subsection{Mean extractor test function}
\label{sec:mean_test_function}

\subsubsection{Non multi-cluster setting} 
\label{sec:non_multi_cluster_mean_test_function}
For any $1 \leq i,j \leq k_*$ we construct the mean  extractor test function, that is, an interpolated Hermite polynomial $H_{j,i}(\theta)$ such that $H_{j,i}(\theta_l^*) = 0$ and $\nabla H_{j,i}(\theta_l^*) = 0$ for all $l\neq j$. The formula of this polynomial can be expressed as
\begin{align*}
    H_{j,i}(\theta) : = \ell_j^2(\theta) \left[\langle v_{j,i}, \theta - \theta_j^*\rangle\right],  \quad \text{where} \quad \ell_j(\theta) : = \prod_{q \neq j} \frac{\|\theta - \theta_q^*\|_2}{\|\theta_j^* - \theta_q^*\|_2}, 
\end{align*}
and $v_{j,i} = {(\theta_i-\theta_j^*)}/{\|\theta_i-\theta_j^*\|}$ if $\theta_i \neq \theta_j^*$ and $v_{j,i}=(1, 0,\ldots ,0)^\top$ otherwise. It can be checked that $H_{j,i}$ is of degree $2k_*-1$ and $H_{j,i}(\theta_l^*) = 0$ for all true center $\theta_l^*$. For the gradient of $H_{j,i}$, it is easy to check that $\nabla H_{j,i}(\theta^*_l) = 0$ for all $l \neq j$. To achieve  $\nabla H_{j,i}(\theta^*_j)$, we first calculate 
\begin{equation}
\label{eqn:lemma_ell_function_gradient}
    \nabla \ell_j(\theta^*_j) =\left(\sum_{q\neq j}  \frac{\nabla|_{\theta=\theta^*_j}\|\theta - \theta_q^*\|_2}{\|\theta_j^* - \theta_q^*\|_2}\right)\underbrace{\ell_j(\theta^*_j)}_{=1} = \sum_{q\neq j} \frac{\theta^*_j-\theta^*_q}{\|\theta_j^*-\theta_q^*\|_2^2}. 
\end{equation}
which leads to an estimation to be used later 
\begin{equation}
\label{eqn:lemma_ell_function_gradient_estimation}
    \|\nabla \ell_j(\theta^*_j)\|_2 \leq \sum_{q\neq j} \frac{\|\theta^*_j-\theta^*_q\|}{\|\theta_j^*-\theta_q^*\|_2^2} = \sum_{q\neq j} \frac{1}{\|\theta_j^*-\theta_q^*\|_2}\leq\frac{k_*-1}{\Delta_{\mathrm{sep}}}. 
\end{equation}
Thus, we have 
\begin{align*}
    \nabla H_{j,i}(\theta^*_j) = 2\ell_{j}(\theta_j^*)\nabla\ell_j(\theta^*_j)\left[\langle v_{j,i}, \theta_j^* - \theta_j^*\rangle\right] + \ell^2_j(\theta^*_j)v_{j,i} = v_{j,i}. 
\end{align*}
\begin{lemma}
\label{lemma:mean_test_function}
(a) (One-cluster assumption) Given the construction of the test function $H_{j,i}$, under one-cluster assumption, for any $1\leq l\leq k_*$, we have 
\begin{align*}
\|H_{j,i}\|_{\infty} \le C_{\mathrm{norm},H} \Delta_{\mathrm{sep}}^{-(2k_*-2)}\\
\max_{\theta \in \mathcal{V}_l: \|\theta - \theta_{l}^{*}\| \le \Delta_{\mathrm{sep}}/4}\|D^2{H}_{j,i}(\theta)\|_{\mathrm{op}}\leq C_{H,1}\Delta_{\mathrm{sep}}^{-1}.
\end{align*}

(b) (No-cluster assumption) Under no-cluster assumption, we have for any $1\leq l \leq k_*$
\begin{align*}
\|H_{j,i}\|_{\infty} \le C_{\mathrm{norm},H} \Delta_{\mathrm{sep}}^{-(2k_*-2)}\\
\max_{\theta\in \bar{B}(0,R)}\|D^2{H}_{j,i}(\theta)\|_{\mathrm{op}}\leq C_{H,2}\Delta_{\mathrm{sep}}^{-(2k_*-2)}.
\end{align*}
\end{lemma}
\begin{proof}[Proof of Lemma~\ref{lemma:mean_test_function}]
(a) In the case of one-cluster assumption, for any $\theta\in \bar{B}(0,R)$, we have $\|\theta-\theta_q^*\|\leq 2R$ for all $q$. Thus, 
\begin{equation}
\label{eqn:dung_lemma_7.2_estimation_of_ell}
    \ell^2_j(\theta) \leq (2R)^{2k_*-2}\Delta^{-(2k_*-2)}_{\mathrm{sep}}.
\end{equation}

In addition, we have 
\begin{equation}
    \label{eqn:dung_lemma_7.2_estimation_of_monomial}
    |\langle v_{j,i},\theta-\theta_j^*\rangle|\leq \|v_{j,i}\|_2\cdot \|\theta-\theta_j^*\|_2 \leq 2R. 
\end{equation}
Combining the estimation in equations~\eqref{eqn:dung_lemma_7.2_estimation_of_ell} and \eqref{eqn:dung_lemma_7.2_estimation_of_monomial}, we have
\begin{equation*}
    |H_{j,i}(\theta)| \leq (2R)^{2k_*-1}\Delta^{-(2k_*-2)}_{\mathrm{sep}}:= C_{\text{norm},H} \Delta_{\mathrm{sep}}^{-(2k_*-2)}.
\end{equation*}
We now proceed to bound $\|D^2H_{j,i}(\theta)\|_{\mathrm{op}}$, for any $\theta \in \mathcal{V}_{l}$ and $\|\theta - \theta_{l}^{*}\| \leq \Delta_{\mathrm{sep}}/4$. We perform the following change of coordinates $z = \frac{\theta - \theta_{l}^*}{\Delta_{\mathrm{sep}}}$. For $\theta \in \mathcal{V}_{l}$ and $\|\theta - \theta_{l}^{*}\| \leq \Delta_{\mathrm{sep}}/4$, we have $\|z\| \leq 1/4$. Let $c_{lq} = \frac{\theta_l^* - \theta_q^*}{\Delta_{\mathrm{sep}}}$ for any $1 \leq l \leq k_{*}$. From the assumption, $\|c_{lq}\| \le C_0$. Also $\|c_{jq}\| \ge 1$ for $q \neq j$. Therefore, we obtain that 
\begin{align}
    \label{eqn:dung_lemma_7_2_coordinate_change_polynomial}
    \ell_j(\theta) = \prod_{q\neq j} \frac{\|\theta_l^*-\theta_q^* + \Delta_{\mathrm{sep}}z\|}{\|\theta^*_{j}-\theta_q^*\|} = \prod_{q\neq j}\frac{\|c_{lq}+z\|}{\|c_{jq}\|} = \tilde{\ell}_j(z). 
\end{align}
It is clear that $\tilde{\ell}^2_j(z)$ is a polynomial in $z$ whose coefficients are determined strictly by bounds $C_0$ and $1$. Hence $\tilde{\ell}^2_j(z)$ and its derivatives are bounded by an absolute constant $C_{\ell}(R, k_*, C_0)$, for any $\|z\|\leq 1/4$.

We have $H_{j,i}(\theta) =   \tilde{\ell}_j^{2}(z) \Delta_{\mathrm{sep}}\langle v_{j,i},z + c_{lj}\rangle  = \Delta_{\mathrm{sep}} \tilde{H}_{j,i}(z)$. Note that the second-order derivative of $\tilde{H}_{j,i}(z)$ with respect to $z$ are bounded by $C_{\tilde{H},\mathrm{a}}(R, k_*, C_0)$.
By the chain rule, we have $\frac{d^k}{d\theta^k} {H}_{j,i}(\theta) = \Delta_{\mathrm{sep}}^{-|k|+1}\cdot \frac{d^k}{dz^k}\tilde{H}_{j,i}(z)$ for any multi-index $k$. Thus, we have for any $u,v\in [1,d]$,
\begin{align*}
\max_{\theta \in \mathcal{V}_l: \|\theta - \theta_{l}^{*}\| \leq \Delta_{\mathrm{sep}}/4} \left|\frac{\partial^2}{\partial\theta^u\partial\theta^v}{H}_{j,i}(\theta)\right|&\le \Delta_{\mathrm{sep}}^{-1} C_{\tilde{H},\mathrm{a}}. 
\end{align*}
As a result, noting that $\|M\|_{\mathrm{op}}\leq d_*\max_{1\leq i,j\leq d_{*}}\{|M_{ij}|\}$ for any symmetric matrix $M$, where $d_*$ here is the dimension of $M$, we have 
\begin{align*}
    \max_{\theta \in \mathcal{V}_l: \|\theta - \theta_{l}^{*}\| \leq \Delta_{\mathrm{sep}}/4}\|D^2{H}_{j,i}(\theta)\|_{\mathrm{op}}\leq d_*C_{\tilde{H},\mathrm{a}}\Delta_{\mathrm{sep}}^{-1} :=  C_{H,1}\Delta_{\mathrm{sep}}^{-1}. 
\end{align*}

(b) In the case of no-cluster assumption, for $\|H_{j,i}\|_{\infty}$, we use the same argument as in part (a) to achieve identical result. For $D^2H_{j,i}$, from the definition of $\Delta_{\mathrm{sep}}$, it is clear that the denominator of the polynomial $\ell_j$ is at least $\Delta_{\mathrm{sep}}^{k_*-1}$, thus, we can write $\ell^2_j(\theta)$ as  $\ell^2_j(\theta) = \Delta_{\mathrm{sep}}^{-2(k_*-1)} \sum_{m\succeq 0, |m|\leq 2k_*-2} d_m \theta^m$. As a result, we can express $\ell^2_{j}$ as 
\begin{align}
\label{eqn:dung_lemma_another_representation_for_ell_j}
    \ell_j^2(\theta) = \Delta_{\mathrm{sep}}^{-(2k_*-2)} \sum_{m \succeq 0,|m|\leq 2k_*-2} e_m \theta^m \text{ such that }|e_m| \le R^{2k_*-2-|m|}.\prod_{i=1}^{d}\binom{2k_*-2}{m_i}.
\end{align} 
Therefore, the polynomial $H_{j,i}(\theta) = \ell_j^2(\theta) \langle v_{j,i}, \theta - \theta_j^*\rangle  $ can be rewritten as $H_{j,i}(\theta) = \sum_{m \succeq 0,|m|\leq 2k_*-1} h_m \theta^m$ with 
\begin{align*}
    |h_m| \le \Delta_{\mathrm{sep}}^{-(2k_*-2)} \tilde{h}_m,
\end{align*}
for any $m \succeq 0,|m|\leq 2k_*-1$, where $$\tilde{h}_m = R^{2k_*-1-|m|}\prod_{i=1}^d\binom{2k_*-1}{m_i}.$$
As a result, the second order partial derivatives of $H_{j,i}$ can be bounded as  

\begin{align*}
    \max_{\theta \in \bar{B}(0,R)} \left|\frac{\partial^2}{\partial\theta^u\partial\theta^v}H_{j}(\theta)\right| &\le \sum_{0\leq |m| \leq 2k_*-1}(m_{u}m_v -\mathbf{1}_{\{u=v\}}m_u)|h_m|R^{|m|-2}\\ &=\left(\sum_{0\leq |m| \leq 2k_*-1}(m_um_v -\mathbf{1}_{\{u=v\}}m_u)R^{|m|-2}\tilde{h}_m\right)\Delta_{\mathrm{sep}}^{-(2k_*-2)}\\
    &:= C_{H,\mathrm{partial}} \Delta_{\mathrm{sep}}^{-(2k_*-2)}. 
\end{align*}
As a result, noting that $\|M\|_{\mathrm{op}}\leq d_*\max_{1\leq i,j\leq d_{*}}\{|M_{ij}|\}$ for any symmetric matrix $M$, where $d_*$ here is the dimension of $M$, we have 
\begin{align*}
    \max_{\theta \in \bar{B}(0,R)}\|D^2{H}_{j,i}(\theta)\|_{\mathrm{op}}\leq d_*C_{H,\mathrm{partial}}\Delta_{\mathrm{sep}}^{-(2k_*-2)} :=  C_{H,2}\Delta_{\mathrm{sep}}^{-(2k_*-2)}. 
\end{align*}
\end{proof} 

\subsubsection{Multi-cluster setting}
\label{sec:dung_mean_extractor_multicluster}
For multi-cluster setting, we construct the mean extractor test function, that is, an interpolated Hermite polynomial $\bar{H}_{j,i}(\theta)$ such that $\bar{H}_{j,i}(\theta_l^*) = 0$ and $\nabla H_{j,i}(\theta_l^*) = 0$ for all $l\neq j$. The formula of $\bar{H}_{j,i}$ is given by
$$\bar{H}_{j,i}(\theta) = \ell_{j,\text{micro}}^2(\theta) [\langle \bar{v}_{j,i},\theta - \theta_{j}^{*}\rangle]P_{\text{macro}}(\theta),$$ 
where we define $\ell_{j,\text{micro}}(\theta) = \prod_{q \in \mathcal{C}_{m} \neq j} \frac{\|\theta - \theta_q^*\|_2}{\|\theta_j^* - \theta_q^*\|_2}$, $P_{\text{macro}}(\theta) = \prod_{p \neq m} \prod_{q \in \mathcal{C}_p} \left( \frac{\|\theta - \theta_q^*\|_2}{2R} \right)^{2s_{\max}}$, and $$\bar
v_{j,i} = P^{-1}_{\text{macro}}(\theta_j^*)\frac{\theta_i-\theta_j^*}{\|\theta_i-\theta_j^*\|},$$
when $\theta_i\neq \theta_j^*$ and $\bar{v}_{j,i} = P^{-1}_{\text{macro}}(\theta_j^*)(1,0,\ldots,0)^{\top}$ otherwise. It can be checked that $\bar{H}_{j,i}(\theta_l^*) = 0$ for all true center $\theta_l^*$. For the gradient of $\bar{H}_{j,i}$, it is easy to check that $\nabla \bar{H}_{j,i}(\theta^*_l) = 0$ for all $l \neq j$. To achieve  $\nabla \bar{H}_{j,i}(\theta^*_j)$, we have 
\begin{align*}
    \nabla \bar{H}_{j,i}(\theta^*_j) &= \nabla(\ell_{j,\text{micro}}^2(\theta^*_j) P_{\text{macro}}(\theta^*_j))[\langle \bar{v}_{j,i},\theta_j^*-\theta_j^*\rangle] +\ell_{j,\text{micro}}^2(\theta^*_j) P_{\text{macro}}(\theta^*_j)\bar{v}_{j,i}\\
    &= P_{\text{macro}}(\theta^*_j)\bar{v}_{j,i} = \frac{\theta_i-\theta_j^*}{\|\theta_i-\theta_j^*\|}. 
\end{align*}

\begin{lemma}
\label{lemma:dung_multicluster_mean_test_function}
Given the construction of the test function $\bar{H}_{j,i}$,  we have: 
\begin{itemize}
    \item [(a)] (Supremum norm bound) 
\begin{align*}
\|\bar{H}_{j,i}\|_{\infty} \le C_{\mathrm{norm},\bar{H}} \Delta_{\mathrm{sep}}^{-(2s_{m}-2)}
\end{align*}

\item [(b)] (Local Hessian bound) For any $1\leq l\leq k_*$ such that $\theta_l^* \in \mathcal{C}_m$, we have 
\begin{equation*}
    \max_{\theta \in \mathcal{V}_l: \|\theta - \theta_{l}^{*}\| \le \Delta_{\mathrm{sep}}/4}\|D^2\bar{H}_{j,i}(\theta)\|_{\mathrm{op}}\leq C_{\bar{H}}\Delta_{\mathrm{sep}}^{-1}.
\end{equation*}
Otherwise, for any $1\leq l\leq k_*$ such that to cluster $\theta_l^*\notin \mathcal{C}_m$, we have for $\theta \in \mathcal{V}_l: \|\theta - \theta_{l}^{*}\| \le \Delta_{\mathrm{sep}}/4$, 
\begin{equation*}
   |\bar{H}_{j,i}(\theta)|\leq C_{\mathrm{cross},\bar{H}}\|\theta - \theta_{l}^{*}\|_2^{2}.
\end{equation*}
\end{itemize}
\end{lemma}

\begin{proof}[Proof of Lemma \ref{lemma:dung_multicluster_mean_test_function}]
The first step is to bound $\|\bar{v}_{j,i}\|$. We have 
\begin{align}
\|\bar{v}_{j,i}\|_2 \le \left( \frac{2R}{D_{0}} \right)^{2s_{\max}(k_* - s_m)} \leq (2R/D_0)^{2k^2_*}:= \bar{C}_{\bar{v}}.
\label{eqn:dung_estimation_for_norm_of_bar_v}
\end{align}

(a) For any $\theta\in \bar{B}(0,R)$, we have $\|\theta-\theta_q^*\|\leq 2R$ for all $q$. Thus, 
\begin{equation}
\label{eqn:dung_lemma_7.3_estimation_of_ell_micro}
    \ell^2_{j,\text{micro}}(\theta) \leq (2R)^{2s_m-2}\Delta^{-(2s_m-2)}_{\mathrm{sep}}.
\end{equation}
In addition, we have 
\begin{equation}
    \label{eqn:dung_lemma_7.3_estimation_of_monomial_factor}
    |\langle v_{j,i},\theta-\theta_j^*\rangle|\leq \|v_{j,i}\|_2\cdot \|\theta-\theta_j^*\|_2 \leq 2R\bar{C}_{\bar{v}}. 
\end{equation}
Combining the estimation in equations~\eqref{eqn:dung_lemma_7.3_estimation_of_ell_micro} and \eqref{eqn:dung_lemma_7.3_estimation_of_monomial_factor}, and note that $|P_{\text{macro}}(\theta)|\leq 1$, we have
\begin{equation*}
    |\bar{H}_{j,i}(\theta)| \leq \bar{C}_{\bar{v}}(2R)^{2s_m-1}\Delta^{-(2s_m-2)}_{\mathrm{sep}}\leq C_{\text{norm},\bar{H}} \Delta_{\mathrm{sep}}^{-(2s_m-2)},
\end{equation*}
where $C_{\text{norm},\bar{H}}:=\bar{C}_{\bar{v}}\max\{(2R)^{2k_*-1},1\}.$

(b) Now, we move to bound $\|D^2\bar{H}_{j,i}(\theta)\|_{\mathrm{op}}$, for any $\theta \in \mathcal{V}_{l}$ and $|\theta - \theta_{l}^{*}| \leq \Delta_{\mathrm{sep}}/4$. We perform the following change of coordinates $z = \frac{\theta - \theta_{l}^*}{\Delta_{\mathrm{sep}}}$. For $\theta \in \mathcal{V}_{l}$ and $\|\theta - \theta_{l}^{*}\| \leq \Delta_{\mathrm{sep}}/4$, we have $\|z\|\leq 1/4$. Let $c_{lq} = \frac{\theta_l^* - \theta_q^*}{\Delta_{\mathrm{sep}}}$ for any $l,q\in \mathcal{C}_m$ and $l\neq q$. From the assumption, $\|c_{lq}\| \le C_0$. Also $\|c_{jq}\| \ge 1$ for $q \in \mathcal{C}_m\neq j$. Therefore, we obtain that 
\begin{align}
\label{eqn:dung_lemma_7_3_coordinate_change_polynomial_ell_macro}
    \ell_{j,\text{micro}}(\theta) = \prod_{q\in \mathcal{C}_m\neq j} \frac{\|\theta_l^*-\theta_q^* + \Delta_{\mathrm{sep}}z\|}{\|\theta^*_{j}-\theta_q^*\|} = \prod_{q\in \mathcal{C}_m\neq j}\frac{\|c_{lq}+z\|}{\|c_{jq}\|} = \tilde{\ell}_{j,\text{micro}}(z). 
\end{align}
Regarding the polynomial $P_{\text{macro}}(\theta)$, we have
\begin{equation}
\label{eqn:dung_p_macro_change_of_variable}
    P_{\text{macro}}(\theta) = P_{\text{macro}}(\theta_l^* + \Delta_{\mathrm{sep}} z) = \prod_{p \neq m} \prod_{q \in \mathcal{C}_p}  \left(\frac{\|\theta_l^* - \theta_q^* + \Delta_{\mathrm{sep}} z\|}{2R}\right) ^{2s_{\max}}:=\tilde{P}_{\text{macro}}(z).
\end{equation}
Because $\|\theta\|,\|\theta_q^*\| \le R$ , we have $\frac{\|\theta - \theta_q^* \|}{2R} \leq 1$. Furthermore, by taking derivatives of $\tilde{P}_{\text{macro}}(z)$ with respect to $z$, we can extract chain-rule scaling factors of $\frac{\Delta_{\mathrm{sep}}}{2R} \le 1$. Therefore, $\tilde{P}_{\text{macro}}(z)$ and its derivatives are bounded by an absolute constant $C_{\text{macro}}(R, k_*, s_{\max})$.
We have 
\begin{equation*}
    \bar{H}_{j,i}(\theta) = \Delta_{\mathrm{sep}}\tilde{\ell}^2_{j,\text{micro}}(z)\tilde{P}_{\text{macro}}(z)\langle v_{j,i},z+c_{lj}\rangle:= \Delta_{\mathrm{sep}}{\hat{H}}_{j,i}(z).
\end{equation*}
From the argument above, the second-order partial derivatives of $\hat{H}_{j,i}(z)$ has coefficients bounded by $C_{\hat{H},\mathrm{a}}(R,k_*,C_0)$. Using the chain rule, we have $\frac{d^k}{d\theta^k} \bar{H}_{j,i}(\theta) = \Delta_{\mathrm{sep}}^{-|k|+1}\cdot \frac{d^k}{dz^k}\hat{H}_{j,i}(z)$ for any multi-index $k$. Thus, we have for any $u,v\in [1,d]$,
\begin{align*}
\max_{\theta \in \mathcal{V}_l: \|\theta - \theta_{l}^{*}\| \leq \Delta_{\mathrm{sep}}/4} \left|\frac{\partial^2}{\partial\theta^u\partial\theta^v}\bar{H}_{j,i}(\theta)\right|&\le \Delta_{\mathrm{sep}}^{-1} \bar{C}_{\bar{H},\mathrm{a}}. 
\end{align*}
As a result, noting that $\|M\|_{\mathrm{op}}\leq d_*\max_{1\leq i,j\leq d_{*}}\{|M_{ij}|\}$ for any symmetric matrix $M$, where $d$ denote the dimension of $M$, we have 
\begin{align*}
    \max_{\theta \in \mathcal{V}_l: \|\theta - \theta_{l}^{*}\| \leq \Delta_{\mathrm{sep}}/4}\|D^2\bar{H}_{j,i}(\theta)\|_{\mathrm{op}}\leq d_*\bar{C}_{\tilde{H},\mathrm{a}}\Delta_{\mathrm{sep}}^{-1} :=  C_{\bar{H}}\Delta_{\mathrm{sep}}^{-1}. 
\end{align*}

For the case when $l \notin \mathcal{C}_m$, we have 
\begin{align}
\label{eqn:dung_lemma_mean_extractor_global_bound_P_macro_outside}
    |P_{\text{macro}}(\theta)| &\leq (2R)^{-2s_{\max}} \|\theta - \theta_{l}^{*}\|_2^{2s_{\max}},\\
\label{eqn:dung_lemma_mean_extractor_global_bound_l_micro_outside}
    |\ell_{j,\text{micro}}^2(\theta)| &\leq (2R)^{2(s_{m} - 1)} \Delta_{\mathrm{sep}}^{-2(s_{m} - 1)},
\end{align}
for any $l \notin \mathcal{C}_{m}$. By combining these bounds with the estimation in equation \eqref{eqn:dung_lemma_7.3_estimation_of_monomial_factor}, we have 
\begin{align*}
\nonumber
    |\bar{H}_{j,i}(\theta)| & \leq |\ell_{j,\text{micro}}^2(\theta)| \cdot |P_{\text{macro}}(\theta)|\cdot |\langle \bar{v}_{j,i},\theta-\theta_j^*\rangle|  \\
    \nonumber
    & \leq (2R)^{2s_{m} - 2s_{\max}-1} \bar{C}_{\bar{v}}\Delta_{\mathrm{sep}}^{-2(s_{m} - 1)}\|\theta - \theta_{l}^{*}\|_2^{2s_{\max}}
\end{align*}
Since $\|\theta - \theta_{l}^{*}\| \leq \Delta_{\mathrm{sep}}/4$, $\|\theta - \theta_{l}^{*}\|_2^{2s_{\max}} \leq \left(\frac{\Delta_{\mathrm{sep}}}{4}\right)^{2(s_{\max} - 1)}\|\theta - \theta_{l}^{*}\|_2^{2}$. As a result, we further attain that
\begin{align*}
|\bar{H}_{j,i}(\theta_{l})|
&\leq \bar{C}_{\bar{v}} (2R)^{2s_{m} - 2s_{\max} - 1} \cdot 4^{2(1-s_{\max})}\Delta_{\mathrm{sep}}^{2(s_{\max} - s_m)} \|\theta - \theta_{l}^{*}\|_2^{2}\\ 
& \leq\bar{C}_{\bar{v}} (2R)^{2s_{m} - 2s_{\max} - 1} \cdot 4^{2(1-s_{\max})}(2R)^{2(s_{\max} - s_m)} \leq C_{\text{cross}}\|\theta - \theta_{l}^{*}\|_2^{2}.
\end{align*}
for any $l \notin \mathcal{C}_{m}$. Thus, $|\bar{H}_{j,i}(\theta_{l})|\leq C_{\text{cross},\bar{H}}\|\theta - \theta_{l}^{*}\|_2^{2}$, where
$$
C_{\mathrm{cross},\bar{H}} = \bar{C}_{\bar{v}}\cdot \max\{(2R)^{2(k^*-1)-1},1\}\cdot \max\{(2R)^{2k^*-2},1\}. 
$$
\end{proof}

\subsection{Point-wise mass test function}
\label{sec:prior_transition_test_function}

\subsubsection{Non multi-cluster setting}
\label{sec:point_wise_mass_extractor_non_multicluster}
For any $1 \leq j \leq k_*$, we construct the point-wise mass  extractor test function, that is, an interpolated Hermite polynomial $E_j$ such that $E_j(\theta_l^*) = \delta_{jl}$ and $\nabla E_j(\theta_l^*) = 0$ for all $1\leq l\leq k_*$. $E_j$ is a polynomial of degree $2k_*-1$ defined as
\begin{align*}
    E_j(\theta) : = \ell_j^2(\theta) \left[ \bar{A}_j + \langle\bar{B}_j,\theta - \theta_j^*\rangle\right], \quad \text{where} \quad \ell_j(\theta) = \prod_{q \neq j} \frac{\|\theta - \theta_q^*\|_2}{\|\theta_j^* - \theta_q^*\|_2},
\end{align*}
and $\bar{A}_j = 1$, $\bar{B}_j = - 2\nabla\ell_j(\theta_j^*)$. We can rewrite $E_j(\theta) : = \ell_j^2(\theta) \left[ u_0 + \langle u_1,\theta\rangle \right]$, where $u_0 = \bar{A}_j - \langle \bar{B}_j,\theta_j^*\rangle $, and $u_1 = \bar{B}_j$. 
From the estimation in equation~\eqref{eqn:lemma_ell_function_gradient_estimation}, we have \begin{align}
    \label{eq:bound_Bj}
    \|\bar{B}_j\|_2\leq \Delta^{-1}_{\mathrm{sep}}(2k_*-2):=\Delta^{-1}_{\mathrm{sep}}C_B.
\end{align} 
It is obvious to check that $E_j(\theta_l^*) = \delta_{jl}$ and $\nabla E_j'(\theta_l^*) = 0$ for all $1\leq l\leq k_*$ and $l\neq j$. To verify that $\nabla E_j(\theta_j^*) = 0$ too, noting that $\ell_j(\theta_j^*) = 1$, we write 
\begin{align*}
    \nabla E_j(\theta^*_j) = 2\ell_j(\theta_j^*)\nabla \ell_j(\theta_j^*)[\bar{A}_j + \langle \bar{B}_j,\theta_j^*-\theta_j^*\rangle] +\ell^2_j(\theta_j^*)\bar{B}_j = 2\nabla\ell_j(\theta_j) - 2\nabla\ell_j(\theta_j) = 0.
\end{align*}

\begin{lemma}
\label{lemma:mass_test_function}
(a) (One cluster assumption) Given the construction of the test function $E_{j}$, the following properties hold for any $1\leq l\leq k_*$ \begin{align*}
\|E_j\|_{\infty} \le  {C}_{\mathrm{norm},E} \Delta_{\mathrm{sep}}^{-(2k_*-1)}\\
\max_{\theta \in \mathcal{V}_l: |\theta - \theta_{l}^{*}| \leq \Delta_{\mathrm{sep}}/4} \|D^2 E_j(\theta)\|_{\mathrm{op}} \le C_{E,1} \Delta_{\mathrm{sep}}^{-2}.  
\end{align*}
(b) (No cluster assumption) Under no-cluster assumption, we have for any $1\leq l \leq k_*$
\begin{align*}
\|E_j\|_{\infty} \le  {C}_{\mathrm{norm},E} \Delta_{\mathrm{sep}}^{-(2k_*-1)}\\
\max_{\theta\in \bar{B}(0,R)} \|D^2 E_j(\theta)\|_{\mathrm{op}} \le C_{E,2} \Delta_{\mathrm{sep}}^{-(2k_*-1)}.  
\end{align*}

\end{lemma}
\begin{proof}[Proof of Lemma~\ref{lemma:mass_test_function}]
(a) From the estimation in equation~\eqref{eqn:dung_lemma_7.2_estimation_of_ell}, we have $\ell^2_j(\theta) \leq (2R)^{2k_*-2}\Delta^{-(2k_*-2)}_{\mathrm{sep}}$. 
In addition, we have as $\Delta_{\mathrm{sep}} \leq 2R$
\begin{equation}
\label{eqn:dung_lemma_7.3_estimation_of_remainder_term}
|\bar{A}_j + \langle \bar{B}_j, \theta-\theta_j^*\rangle| \leq 1 + \|\bar{B}_j\|_2\cdot \|\theta-\theta_j^*\|_2 \leq 1 + \frac{2R(k_*-1)}{\Delta_{\mathrm{sep}}}\leq \frac{2Rk_*}{\Delta_{\mathrm{sep}}}. 
\end{equation}
Thus, we have 
\begin{equation*}
    |E_j(\theta)|\leq (2R)^{2k_*-1}k_*\cdot \Delta^{-(2k_*-1)}_{\mathrm{sep}} := \bar{C}_{\text{norm},E}\Delta^{-(2k_*-1)}_{\mathrm{sep}}. 
\end{equation*}

We now proceed to bound $\|D^2E_j(\theta)\|_{\mathrm{op}}$, for any $\theta \in \mathcal{V}_l$ and $\|\theta - \theta_{l}^{*}\|_2 \leq \Delta_{\mathrm{sep}}/4$. Let $z = \frac{\theta - \theta_{l}^*}{\Delta_{\mathrm{sep}}}$ and $c_{lq} = \frac{\theta_l^* - \theta_q^*}{\Delta_{\mathrm{sep}}}$, for any $1 \leq l \leq k_{*}$, then $\|z\|\leq 1/4$, $\|c_{lq}\|\leq C_0$ from the assumption, and $\|c_{jq}\|\geq 1$ for $q\neq j$. Noting that $\nabla\ell_j(\theta^*_j)=\sum_{q\neq j}\frac{\theta^*_j-\theta^*_q}{\|\theta^*_j-\theta^*_q\|^2}$, we can rewrite the linear factor in the representation of $E_j(\theta)$ as $$\bar{A}_j + \langle\bar{B}_j,\theta - \theta_j^*\rangle =  1  - 2 \sum_{q \ne j} \frac{1}{\|c_{jq}\|^2}\langle c_{jq}, c_{lj} + z \rangle = \tilde{L}(z).$$
We also obtain that 
\begin{equation*}
    \ell_j(\theta) = \prod_{q\neq j} \frac{\|\theta_l^*-\theta_q^* + \Delta_{\mathrm{sep}}z\|}{\|\theta^*_{j}-\theta_q^*\|} = \prod_{q\neq j}\frac{\|c_{lq}+z\|}{\|c_{jq}\|} = \tilde{\ell}_j(z). 
\end{equation*}
The polynomials $\tilde{L}(z)$ and $\tilde{\ell}^2(z)$ and their partial derivatives with respect to $z$ are strictly bounded by $C_{\ell}(R, k_*, C_0)$. From the definitions of $\tilde{\ell}_{j}(\cdot)$ in equation~\eqref{eqn:dung_lemma_7_2_coordinate_change_polynomial} and $\tilde{L}(\cdot)$, we have $E_j(\theta) = \tilde{\ell}_j^{2}(z) \tilde{L}(z) = \tilde{E}_j(z)$. As a result, the function $\tilde{E}_j(z)$ and its first two partial derivatives with respect to $z$ are bounded by $C_{\tilde{E}}(R, k_*, C_0)$ for any $\|z\|_2\leq 1/4$ . 

By the chain rule, we have $\frac{d^k}{d\theta^k} E_j(\theta) = \Delta_{\mathrm{sep}}^{-|k|}\cdot \frac{d^k}{dz^k}\tilde{E}_j(z)$ for any multi-index $k$. Thus, we have for any $u,v\in [1,d]$,
\begin{align*}
\max_{\theta \in \mathcal{V}_l: \|\theta - \theta_{l}^{*}\| \leq \Delta_{\mathrm{sep}}/4} \left|\frac{\partial^2}{\partial\theta^u\partial\theta^v}E_j(\theta)\right|&\le \Delta_{\mathrm{sep}}^{-2} \bar{C}_{2,E}. 
\end{align*}
As a result, noting that $\|M\|_{\mathrm{op}}\leq d_*\max_{1\leq i,j\leq d_*}\{|M_{ij}|\}$ for any symmetric matrix $M$, where $d$ denote the dimension of $M$, we have 
\begin{align*}
    \max_{\theta \in \mathcal{V}_l: \|\theta - \theta_{l}^{*}\| \leq \Delta_{\mathrm{sep}}/4}\|D^2 E_j(\theta)\|_{\mathrm{op}}\leq d_*\bar{C}_{2,E}\Delta_{\mathrm{sep}}^{-2} :=  C_{E,2}\Delta_{\mathrm{sep}}^{-2}. 
\end{align*}

(b) From the definition of $\bar{A}_j$ and the estimation in equation \eqref{eq:bound_Bj} for $\bar{B}_j$, we have the following estimations for $u_0$ and $u_1$: 
\begin{align*}
    \|u_1\|_2 & \le \Delta_{\mathrm{sep}}^{-1} C_B, \\
    \|u_0\|_2 & \le \Delta_{\mathrm{sep}}^{-1}(2R  + R C_B) := \Delta_{\mathrm{sep}}^{-1} C_{A}.
\end{align*} 
From the definition of $E_j$ and equation \eqref{eqn:dung_lemma_another_representation_for_ell_j}, we can express $E_j(\theta)$ as  
$E_{j}(\theta) = \sum_{m \succeq 0,|m|\leq 2k_*-1} w_m \theta^m$ with 
\begin{align*}
    |w_m| \le \Delta_{\mathrm{sep}}^{-(2k_*-1)} \tilde{w}_m,
\end{align*}
for any $m \succeq 0, |m| \leq 2 k_{*} - 1$, where $$\tilde{w}_m = C_AR^{2k_*-2-|m|}\prod_{i=1}^d\binom{2k_*-2}{m_i}  + C_BR^{2k_*-3-|m|}\prod_{i=1}^d\binom{2k_*-2}{m_i}\sum_{i=1}^d\binom{2k_*-2}{m_i}^{-1}\binom{2k_*-2}{m_i+1}.$$
As a result, the second order partial derivatives of $E_{j}$ can be bounded as  
\begin{align*}
    \max_{\theta \in \bar{B}(0,R)} \left|\frac{\partial^2}{\partial\theta^u\partial\theta^v}E_{j}(\theta)\right| &\le \sum_{0\leq |m| \leq 2k_*-1}(m_{u}m_v -\mathbf{1}_{\{u=v\}}m_u)|w_m|R^{|m|-2}\\ &=\left(\sum_{0\leq |m| \leq 2k_*-1}(m_um_v -\mathbf{1}_{\{u=v\}}m_u)R^{|m|-2}\tilde{w}_m\right)\Delta_{\mathrm{sep}}^{-(2k_*-1)}\\
    &:= C_{\mathrm{partial},E} \Delta_{\mathrm{sep}}^{-(2k_*-1)}. 
\end{align*}

As a result, noting that $\|M\|_{\mathrm{op}}\leq d_*\max_{1\leq i,j\leq d_{*}}\{|M_{ij}|\}$ for any symmetric matrix $M$, where $d_*$ denote the dimension of $M$, we have 
\begin{align*}
    \max_{\theta \in \bar{B}(0,R)}\|D^2{E}_{j}(\theta)\|_{\mathrm{op}}\leq d_*C_{\mathrm{partial},E}\Delta_{\mathrm{sep}}^{-(2k_*-1)} :=  C_{E,2}\Delta_{\mathrm{sep}}^{-(2k_*-1)}. 
\end{align*}
\end{proof} 

\subsubsection{Multi-cluster setting} 
\label{sec:dung_mass_extractor_multicluster}
For multi-cluster setting, we construct the mass extractor test function, that is, a polynomial
$\bar{E}_{j}$ such that $\bar{E}_j(\theta_l^*) = \delta_{jl}$ and $\nabla \bar{E}_{j}(\theta^*_l) = 0$ for all $l$. The formula of $\bar{E}_{j}$ is given by 
$$\bar{E}_{j}(\theta) = \ell_{j,\text{micro}}^2(\theta) [\bar{A}_j + \langle \bar{B}_j,\theta - \theta_{j}^{*}\rangle]P_{\text{macro}}(\theta),$$ 
where we define $\ell_{j,\text{micro}}(\theta) = \prod_{q \in \mathcal{C}_{m} \neq j} \frac{\|\theta - \theta_q^*\|_2}{\|\theta_j^* - \theta_q^*\|_2}$, $P_{\text{macro}}(\theta) = \prod_{p \neq m} \prod_{q \in \mathcal{C}_p} \left( \frac{\|\theta - \theta_q^*\|_2}{2R} \right)^{2s_{\max}}$ as in Section \ref{sec:dung_mean_extractor_multicluster}, and $\bar{A}_j = 1/ P_{\text{macro}}(\theta_{j}^{*})$,
\begin{align*}
    \bar{B}_j & = -\bar{A}_j[2\nabla\ell_{j,\text{micro}}(\theta_j^*) + (P_{\text{macro}}(\theta^*_j))^{-1}\nabla P_{\text{macro}}(\theta^*_j)]\\
    &= -\bar{A}_j \left(2\sum_{q \in \mathcal{C}_m \setminus \{j\}} \frac{\theta_j^* - \theta_q^*}{\|\theta_j^* - \theta_q^*\|^2_2}+ 2s_{\max}\sum_{p \neq m} \sum_{q \in \mathcal{C}_p} \frac{\theta_j^* - \theta_q^*}{\|\theta_j^* - \theta_q^*\|^2_2} \right).
\end{align*}
It can be checked that $\bar{E}_j(\theta_l^*) = \delta_{jl}$ for all true center $\theta_l^*$. For the gradient of $\bar{E}_{j}$, it is easy to check that $\nabla \bar{E}_{j}(\theta^*_l) = 0$ for all $l \neq j$. To achieve  $\nabla \bar{E}_{j}(\theta^*_j)$, noting that $\ell_{j,\text{micro}}(\theta^*_j) = 1$, we have 
\begin{align*}
    \nabla \bar{E}_j(\theta^*_j) &= \nabla(\ell_{j,\text{micro}}^2(\theta^*_j) P_{\text{macro}}(\theta^*_j))[\bar{A}_j + \langle \bar{B}_j,\theta_{j}^{*} - \theta_{j}^{*}\rangle] + P_{\text{macro}}(\theta^*_j)\bar{B}_j\\
    &= 2\bar{A}_jP_{\text{macro}}(\theta^*_j)\nabla\ell_{j,\text{micro}}(\theta^*_j) + \bar{A}_j\nabla P_{\text{macro}}(\theta^*_j)+ P_{\text{macro}}(\theta^*_j)\bar{B}_j = 0. 
\end{align*}
\begin{lemma}
\label{lemma:dung_multicluster_mass_test_function}
Given the construction of the test function $\bar{E}_{j}$,  we have: 
\begin{itemize}
    \item [(a)] (Supremum norm bound) 
\begin{align*}
\|\bar{E}_{j}\|_{\infty} \le C_{\mathrm{norm},\bar{E}} \Delta_{\mathrm{sep}}^{-(2s_{m}-1)}
\end{align*}

\item [(b)] (Local Hessian bound) For any $1\leq l\leq k_*$ such that $\theta_l^* \in \mathcal{C}_m$, we have 
\begin{equation*}
    \max_{\theta \in \mathcal{V}_l: \|\theta - \theta_{l}^{*}\| \le \Delta_{\mathrm{sep}}/4}\|D^2\bar{E}_{j}(\theta)\|_{\mathrm{op}}\leq C_{\bar{E},2}\Delta_{\mathrm{sep}}^{-2}.
\end{equation*}
Otherwise, for any $1\leq l\leq k_*$ such that to cluster $\theta_l^*\notin \mathcal{C}_m$ such that $\theta \in \mathcal{V}_l: \|\theta - \theta_{l}^{*}\| \le \Delta_{\mathrm{sep}}/4$, we have 
\begin{equation*}
    |\bar{E}_{j}(\theta)|\leq C_{\mathrm{cross},\bar{E}}\Delta_{\mathrm{sep}}^{-1}\|\theta-\theta_l^*\|_2^2.
\end{equation*}
\end{itemize}
\end{lemma}
\begin{proof}[Proof of Lemma \ref{lemma:dung_multicluster_mass_test_function}] 
The first step is to bound $\bar{A}_j$ and $\|\bar{B}_j\|_2$. Since  $\|\theta_j^* - \theta_q^*\| \ge D_0$ for any $q \in \mathcal{C}_{p}$ and $p \neq m$, we have 
$$ |P_{\text{macro}}(\theta_j^*)| \ge \prod_{p \neq m} \prod_{q \in \mathcal{C}_p} \left( \frac{D_0}{2R} \right)^{2s_{\max}} = \left( \frac{D_0}{2R} \right)^{2s_{\max}(k_* - s_m)}.$$
Therefore, we have an upper bound of $\bar{A}_{j}$ as follows: 
$$|\bar{A}_j| \le  \left( \frac{2R}{D_{0}} \right)^{2s_{\max}(k_* - s_m)}\leq \left( \frac{2R}{D_{0}} \right)^{2k_*^2} := \bar{C}_{A}.$$
Now, moving to $\bar{B}_{j}$, an application of the triangle inequality yields that
\begin{align*}
\|\bar{B}_j\|_2 & \leq |\bar{A}_j| \left(  \sum_{q \in \mathcal{C}_m \setminus \{j\}} \frac{2}{\|\theta_j^* - \theta_q^*\|_2} + \sum_{p \neq m} \sum_{q \in \mathcal{C}_p} \frac{2s_{\max}}{\|\theta_j^* - \theta_q^*\|_2} \right) \\
& \leq \bar{C}_A \left(  \frac{2(s_m - 1)}{\Delta_{\mathrm{sep}}} + \frac{2s_{\max}(k_* - s_m)}{D_0}\right),
\end{align*}
where the second inequality is due to the fact that $\|\theta_{j}^{*} - \theta_{q}^{*}\| \geq \Delta_{\mathrm{sep}}$, for all $q \in \mathcal{C}_{m} \setminus \{j\}$, and $\|\theta_{j}^{*} - \theta_{q}^{*}\| \geq D_{0}$, for any $q \in \mathcal{C}_{p}$ and $p \neq m$. Since $\Delta_{\mathrm{sep}} \le 2R$, we can further bound $\bar{B}_j$ as
\begin{align*}
    \|\bar{B}_j\|_2 & \leq \Delta_{\mathrm{sep}}^{-1} \cdot \bar{C}_A \left(2(s_m - 1)+ \frac{2s_{\max}(k_* - s_m)R}{D_0}  \right) \\
    &\leq \Delta_{\mathrm{sep}}^{-1} \cdot \bar{C}_A \left(2(k_* - 1)+ \frac{4k^2_*R}{D_0}  \right):=\bar{C}_{B} \Delta_{\mathrm{sep}}^{-1}.
\end{align*}
(a)
We have
\begin{align}
\nonumber
    |\bar{A}_{j} + \langle \bar{B}_{j},\theta-\theta_j^*\rangle|&\leq |\bar{A}_j| +  \|\bar{B}_{j}\|_2\cdot \|\theta-\theta_j^*\|_2 \leq \bar{C}_A + R\bar{C}_B\Delta^{-1}_{\mathrm{sep}}\\
    &\leq R(2\bar{C}_A + \bar{C}_B)\Delta^{-1}_{\mathrm{sep}}.
\label{eqn:dung_lemma_7.5_mass_extractor_multicluster_linear_factor_estimation}
\end{align}
Combining the above estimation with estimation in equation~\eqref{eqn:dung_lemma_7.3_estimation_of_ell_micro}, and note that $|P_{\text{macro}}(\theta)|\leq 1$, we have
\begin{align*}
    |\bar{E}_j(\theta)| &\leq R(2R)^{2s_m-2}(2\bar{C}_A + \bar{C}_B)\Delta^{-(2s_m-1)}_{\mathrm{sep}}\\
    &\leq \max\{1,(2R)^{2k_*-2}\}\Delta^{-(2s_m-1)}_{\mathrm{sep}}:= C_{\text{norm,}\bar{E}}
    \Delta_{\mathrm{sep}}^{-(2s_m-1)}.
\end{align*}
(b) Now, we move to bound $\|D^2\bar{E}_{j}(\theta)\|_{\mathrm{op}}$, for any $\theta \in \mathcal{V}_{l}$ and $\|\theta - \theta_{l}^{*}\|_2 \leq \Delta_{\mathrm{sep}}/4$. As in the proof of Lemma \ref{lemma:dung_multicluster_mean_test_function}, we perform the following change of coordinates $z = \frac{\theta - \theta_{l}^*}{\Delta_{\mathrm{sep}}}$. For $\theta \in \mathcal{V}_{l}$ and $\|\theta - \theta_{l}^{*}\| \leq \Delta_{\mathrm{sep}}/4$, we have $\|z\|\leq 1/4$. Let $c_{lq} = \frac{\theta_l^* - \theta_q^*}{\Delta_{\mathrm{sep}}}$ for any $l,q\in \mathcal{C}_m$ and $l\neq q$. From the assumption, $|c_{lq}| \le C_0$. Also $|c_{jq}| \ge 1$ for $q \in \mathcal{C}_m\neq j$.
We represent the linear term $\bar{A}_j + \langle \bar{B}_j, \theta - \theta_j^*\rangle$ as 
\begin{align*}
    \bar{A}_j + \langle \bar{B}_j, \theta - \theta_j^*\rangle&= \bar{A}_j \left[ 1 -\left\langle 2\sum_{q \in \mathcal{C}_m \setminus \{j\}} \frac{c_{jq}}{\|c_{jq}\|^2} + 4s_{\max}\Delta_{\mathrm{sep}}\sum_{p \neq m} \sum_{q \in \mathcal{C}_p} \frac{\theta_j^* - \theta_q^*}{\|\theta_j^* - \theta_q^*\|^2_2} ,c_{lj} + z\right\rangle \right] \\
&:= \tilde{L}_{1}(z).
\end{align*}
Note that $\tilde{L}_{1}(z)$ is a polynomial in $z$ with bounded coefficients. Thus, $\tilde{L}_{1}(z)$
and its derivatives are strictly bounded by $C_L(R, k_*, D_{0},C_0)$. Combining this result with the change of variable performed in equations \eqref{eqn:dung_lemma_7_3_coordinate_change_polynomial_ell_macro} and \eqref{eqn:dung_p_macro_change_of_variable}, we have 
\begin{equation*}
    \bar{E}_{j}(\theta) = \tilde{\ell}^2_{j,\text{micro}}(z)\tilde{P}_{\text{macro}}(z)\tilde{L}_1(z):= \hat{E}_{j}(z).
\end{equation*}
From the argument above, the second-order partial derivatives of $\hat{E}_{j}(z)$ has coefficients bounded by $C_{\hat{E},\mathrm{a}}(R,k_*,C_0)$. Using the chain rule, we have $\frac{d^k}{d\theta^k} \bar{E}_j(\theta) = \Delta_{\mathrm{sep}}^{-|k|}\cdot \frac{d^k}{dz^k}\hat{E}_{j}(z)$ for any multi-index $k$. Thus, we have for any $u,v\in [1,d_*]$,
\begin{align*}
\max_{\theta \in \mathcal{V}_l: \|\theta - \theta_{l}^{*}\| \leq \Delta_{\mathrm{sep}}/4} \left|\frac{\partial^2}{\partial\theta^u\partial\theta^v}\bar{E}_{j}(\theta)\right|&\le \Delta_{\mathrm{sep}}^{-2} C_{\mathrm{partial},\bar{E}}. 
\end{align*}
As a result, noting that $\|M\|_{\mathrm{op}}\leq d_*\max_{1\leq i,j\leq d_{*}}\{|M_{ij}|\}$ for any symmetric matrix $M$, where $d$ denote the dimension of $M$, we have 
\begin{align*}
    \max_{\theta \in \mathcal{V}_l: \|\theta - \theta_{l}^{*}\| \leq \Delta_{\mathrm{sep}}/4}\|D^2\bar{E}_{j}(\theta)\|_{\mathrm{op}}\leq d_*C_{\mathrm{partial},\bar{E}}\Delta_{\mathrm{sep}}^{-2} :=  C_{\bar{E}}\Delta_{\mathrm{sep}}^{-2}. 
\end{align*}
(c) For the case when $l \notin \mathcal{C}_m$, combining the estimation in equations \eqref{eqn:dung_lemma_7.5_mass_extractor_multicluster_linear_factor_estimation}, \eqref{eqn:dung_lemma_mean_extractor_global_bound_P_macro_outside} and \eqref{eqn:dung_lemma_mean_extractor_global_bound_l_micro_outside}, we have 
\begin{align*}
\nonumber
    |\bar{E}_{j}(\theta)| & \leq |\ell_{j,\text{micro}}^2(\theta)| \cdot |P_{\text{macro}}(\theta)|\cdot |\bar{A}_{j} + \langle \bar{B}_{j},\theta-\theta_j^*\rangle|  \\
    \nonumber
    & \leq R(2\bar{C}_A + \bar{C}_B)(2R)^{2s_{m} - 2s_{\max}-2} \Delta_{\mathrm{sep}}^{-(2s_{m} - 1)}\|\theta - \theta_{l}^{*}\|_2^{2s_{\max}}. 
\end{align*}
Since $\|\theta - \theta_{l}^{*}\| \leq \Delta_{\mathrm{sep}}/4$, we have $\|\theta - \theta^*_l\|^{2(s_{\max}-1)}\leq (\Delta_{\mathrm{sep}}/4)^{2(s_{\max}-1)}$. In addition, $(\Delta_{\mathrm{sep}})^{2s_{\max}-2s_m} \leq (2R)^{2s_{\max}-2s_m}$. Thus, we further attain that
\begin{align*}
|\bar{E}_{j}(\theta)| & \leq R(2\bar{C}_A + \bar{C}_B) (2R)^{-2} \cdot 4^{2(1-s_{\max})}\Delta_{\mathrm{sep}}^{-1} \|\theta - \theta_{l}^{*}\|_2^{2}\\ 
& \leq R(2\bar{C}_A + \bar{C}_B)(2R)^{-2} \cdot\max\{(2R)^{2k_*},1\}\Delta_{\mathrm{sep}}^{-1}\|\theta - \theta_{l}^{*}\|_2^{2} \\
&\leq C_{\text{cross},\bar{E}} \Delta_{\mathrm{sep}}^{-1} \|\theta - \theta_{l}^{*}\|_2^{2},
\end{align*}
for any $l \notin \mathcal{C}_{m}$, where
$$
C_{\mathrm{cross},\bar{E}} = (4R)^{-1}(2\bar{C}_A + \bar{C}_B)\cdot \max\{(2R)^{2k^*},1\}. 
$$
\end{proof}

\subsection{Cluster-wise mass test function}
\label{sec:cluster_wise_test_function} We now construct test function which can isolate a cluster of data with the remainder. For a data including $k_*$ centers in $\mathbb{R}^d$, partitioned into $k_0$ disjoint topological clusters $\mathcal{C}_1,\ldots\mathcal{C}_{k_0}$, we recall the following condition imposed on these clusters.
\begin{itemize}
    \item \textit{Inter-cluster assumption}: For any $\theta_j^*\neq \theta_q^* \in \mathcal{C}_m$, we have $\|\theta_j^* - \theta_q^*\|_2\leq C_0\Delta_{\mathrm{sep}}$. 
    \item \textit{Intra-cluster assumption}: For any $\theta_j^* \in \mathcal{C}_m$ and $\theta_q^* \notin \mathcal{C}_m$, we have $\|\theta_j^* - \theta_q^*\|_2\geq D_0$. 
\end{itemize}
Recall that to make sure these clusters are isolated from each other, we impose the rigid topological ordering $\Delta_{\mathrm{sep}} \leq D_0/(4C_0)$. 

For any $1 \leq m \leq k_*$, we construct the cluster-wise mass  extractor test function, that is, an interpolated Hermite polynomial $P_m(\theta)$ such that $P_m(\theta^*_q) = \mathbf{1}_{\{\theta^*_q\in \mathcal{C}_m\}}$ and $\nabla P_m(\theta^*_q) = 0$ for all $q\in[k_*]$. Without loss of generality, we can suppose that $\theta_1^* \in \mathcal{C}_m$. We seek a radial function depending only on the squared distance from $\theta_1^*$, i.e.,
\[
P_m(\theta)=Q_m\left(\|\theta-\theta_1^*\|^2\right),
\]
where the choice of $Q_m$ will be determined later.

\textbf{Step 1: Smooth step function.} We build a function $\Phi \in C^{\infty}(\mathbb{R})$ such that $\Phi(r)\leq 1$ for all $r$, $\Phi(r) = 1$ for $0\leq r \leq D_0^2/2$, and $\Phi(r) = 0$ for $r \geq D_0^2$. In fact, consider the standard function 
\begin{equation*}
    g(t) := \begin{cases}
        \exp\left(-\dfrac{1}{t}\right), \text{ if } t > 0,\\
        0, \text{ if } t\leq 0. 
    \end{cases}
\end{equation*}
and let $S(t) = g(1-t)/(g(t)+g(1-t))$ for $t \in \mathbb{R}$. Then, it is obvious to check that $S=1$ in $(-\infty,0]$ and $S=0$ in $[1,\infty)$. Then, choose $\Phi(r) = S(2r/D^2_0-1)$, then this function satisfy 
\begin{itemize}
    \item $\Phi(r)\leq 1$ for all $r$.
    \item $\Phi(r) = 1$ for $0\leq r \leq D_0^2/2$.
    \item $\Phi(r) = 0$ for $r \geq D_0^2$. 
\end{itemize}
A direct  consequence of this construction is that $\Phi(r_i) = \mathbf{1}_{\{i\leq l\}}$ and $\Phi'(r_i) = 0$ for $1\leq i\leq k$. In addition, we define $\bar{M}_n = \sup_{t\leq 4R^2}|\Phi^{(n)}(t)|$, then it is obvious to check that $\bar{M}_n$ is bounded constant depending only on $D_0$. 

\textbf{Step 2: Interpolation polynomial $Q_m$ based on $\Phi$.} Consider $0\leq r_1\leq \ldots\leq r_k\leq 4R^2$ be increasing sequence of positive reals such that there exists $l$ satisfying $r_1,\ldots,r_l \leq C_0^2\Delta^2_{\mathrm{sep}}$ and $r_{l+1},\ldots,r_{k} \geq D_0^2$. Then, we construct a polynomial $Q_m(r)$ with degree less than $2k$ such that $$Q_m(r_i) = \mathbf{1}_{\{i\leq l\}} \text{ and } Q_m'(r_i) = 0, 1\leq i\leq k.$$  By taking out the coincident points in $r_1,\ldots,r_k$, without loss of generality, we can suppose that $0\leq r_1<\ldots < r_k$.
Let $\mathcal{Z} = (z_1, z_2, \dots, z_{2k})$ denote the ordered multi-set of these $2k$ interpolation nodes. Then, we define 
\begin{align*}
    z_{2i} = z_{2i - 1} = r_{i}, \quad \text{for any } 1 \leq i \leq k. 
\end{align*} 
The coefficients of the Newton polynomial are the divided differences of $\Phi_m(x)$ evaluated over the elements of $\mathcal{Z}$. Because $\mathcal{Z}$  contains adjacent identical elements, the standard recursive formula $\frac{\Phi_{m}(b) - \Phi_{m}(a)}{b - a}$ becomes $0/0$ when $a = b$ are adjacent in $\mathcal{Z}$. We  resolve this issue by invoking the confluent divided difference procedure. For any single element $z_{i}$, we evaluate
$$ [z_{i}] = \Phi(z_{i}).$$
Then, for any two consecutive elements $z_{2i - 1} = z_{2i} = r_i$ for some $l$, we define
$$[z_{2i - 1}, z_{2i}] = \lim_{x \to r_i} \frac{\Phi(x) - \Phi(r_i)}{x - r_i} = \Phi'(r_i). $$
On the other hand, for any two consecutive elements $z_{2i} \neq z_{2i+1}$, we perform the standard divided difference procedure
$$[z_{2i}, z_{2i+1}] = \frac{\Phi(z_{2i+1}) - \Phi(z_{2i})}{z_{2i+1} - z_{2i}}.$$
Since no true center is repeated more than twice in $\mathcal{Z}$, the first node $z_i$ and the last node $z_{i+l}$ of any subsequence $(z_i, \dots, z_{i+l})$ of length 3 or greater are guaranteed to be distinct, namely $z_{i+l} \neq z_i$ when $l \geq 3$. Therefore, for any $l \geq 3$ we can perform the standard divided differences procedure
$$ [z_i, \dots, z_{i+l}] = \frac{[z_{i+1}, \dots, z_{i+l}] - [z_i, \dots, z_{i+l-1}]}{z_{i+l} - z_i}.$$

Based on that confluent divided difference procedure, the explicit form of the polynomial $P_{m}$ via the Newton interpolation procedure is given by
\begin{align}
\label{eqn:dung_divided_difference}
    Q_m(r) : = \sum_{j=0}^{2k-1} [z_1, \dots, z_{j+1}] \prod_{i=1}^j (r - z_i),
\end{align}
where we impose the convention that $\prod_{i=1}^j (r - z_i) = 1$ when $j = 0$. We can verify that $Q_m(r_i) = \Phi(r_i)$ and $Q_m'(r_i) = \Phi'(r_i)$ for all $1\leq l\leq k_*$. 

\textbf{Step 3: Construction of $P_m$.} Consider the cluster $\mathcal{C}_m$ contains center $\theta_1^*$. Let $\theta_2^*,\ldots\theta^*_{k}$ be the points in our data in the order such that for $r_k:=\|\theta^*_k - \theta^*_{1}\|^2_2$, we have $$0=r_1<\ldots< r_l \leq C^2_0\Delta^2_{\mathrm{sep}}< D^2_0\leq r_{l+1}< \ldots < r_k\leq 4R^2.$$
From previous step, we have that $Q_m$ is of degree at most $2k$ such that $Q_m(r_i) = \mathbf{1}_{\{i\leq l\}}$ and $Q'_m(r_i)=0$. Let $P_m(\theta) = Q_m(\|\theta-\theta^*_1\|^2)$, then it is obvious that $P_m$ is a polynomial of $\theta$. It is obvious that $P_m(\theta^*_i) = \mathbf{1}_{\{i\leq l\}}$, and 
\begin{equation*}
    \nabla P_m(\theta^*_i) = 2Q_m'(r_i)\cdot (\theta_i^*-\theta_1^*) = 0. 
\end{equation*}

\begin{lemma}
\label{lemma:dung_bound_for_cluster_mass_extracting}
The polynomial $P_m$ defined above satisfies 
\begin{itemize}
    \item [(a)](Supremum norm bound) 
    \begin{equation*}
    \|P_m\|_{\infty}\leq C_{\mathrm{norm},P}. 
    \end{equation*}
    \item [(b)] (Local Hessian bound) 
\begin{equation*}
    \max_{\theta\in B(0,R)}\|D^2P_m(\theta)\|_{\mathrm{op}} \leq C_{P,2}. 
\end{equation*}
    \end{itemize}
\end{lemma}
\begin{proof}
First, we show that the coefficients of $Q_m$ is less than a constant $C_{\mathrm{cluster}}$ depending only on $D_0$. We start with establishing upper bounds for $[z_1, \dots, z_{j+1}]$, for any $0 \leq j \leq 2k_{*} - 1$. We will demonstrate by induction that $$[z_1,\ldots,z_{n}] \leq \bar{M}_{n}:=\frac{1}{(n-1)!}\sup_{0 \leq x \leq 4R^2}|\Phi^{(n)}(x)|$$ and for any $n \geq 1$. When $j = 0$, from the definition of $[z_{1}]$, we have that $|[z_{1}]| = |\Phi_{m}(z_{1})| \leq \bar{M}_{0}$. Assume that the inequality holds for $j \geq 0$, we show that it also holds for $j + 1$ using the technique from Cauchy \cite{cauchy1885oeuvres}.
Let $\varpi_{n}(r) = \sum_{j=0}^{n-1}[z_1,\ldots,z_{j+1}]\prod_{i=1}^{j}(r - z_i)$, where $1\leq n \leq 2k$, and consider
$\gamma:\mathbb{R}\to \mathbb{R}$, $\gamma(\theta) = \Phi(\theta) - \varpi_n(\theta)$, then it can be validated that
\begin{align*}
    &\gamma(z_i) = 0, \quad \text{if } i\in[1,n],\\
    &\gamma'(z_i) = 0, \quad \text{if } i \leq n/2.
\end{align*}
Let $q_1 <\ldots < q_r$ be different points of $z_1,\ldots,z_{n}$ in increasing order. Then, Rolle's theorem indicates that for $1\leq i \leq r-1$, there exists a point $\xi_i \in (q_{i},q_{i+1})$ such that $\gamma'(\xi_i) = 0$. In addition, for $1\leq i \leq r$, we also have that
\begin{align*}
    &\gamma'(q_i) = 0, \quad \text{when } n \text{ is even,}\\
    &\gamma'(q_i) = 0, \quad \text{when } n \text{ is odd, except at most one point.}
\end{align*}
In these two cases, we observe that $\gamma'$ has at least $n-1$ zero points. By recurrently applying Rolle's theorem to $\gamma'$, there exists a point $\zeta \in \left[\min_{1\leq i\leq n}z_i, \max_{1\leq i\leq n}z_i\right]$ such that the $n-1$-th derivative $\gamma^{(n-1)}(\zeta) = 0$. In other words, we have
\begin{equation*}
    |[z_1,\ldots,z_{n}]| = \dfrac{1}{(n-1)!}\left|\dfrac{d^{n-1}}{dr^{n-1}}\varpi_n(\zeta)\right| = \dfrac{1}{(n-1)!}\left|\Phi^{(n-1)}(\zeta)\right| \leq \dfrac{\bar{M}_{n-1}}{(n-1)!}.
\end{equation*}

(a) Suppose that $Q_{m}(r) = \sum_{i=0}^{2k-1}w_ir^{i}$ where $|w_i|\leq C_{\text{cluster}}$. Then, we have for $\theta \in B(0,R)$, 
    \begin{equation*}
        |P_m(\theta)| \leq \sum_{i=0}^{2k-1}|w_i|\|\theta-\theta^*_1\|^{2i}\leq C_{\text{cluster}}\sum_{i=0}^{2k-1}(2R)^{2i} \leq 2kC_{\text{cluster}}\max\{1,(2R)^{4k-2}\}. 
    \end{equation*}
    Thus, $\|P_m\|_{\infty} \leq C_{\mathrm{norm},P}$, where $C_{\mathrm{norm},P} := 2kC_{\text{cluster}}\max\{1,(2R)^{4k-2}\}$. 
    
    (b) For the second derivative of $P_m$, we have for any index $i\neq j$
    \begin{align*}
        \left|\dfrac{\partial^2}{\partial\theta^i\partial\theta^j}P_m(\theta)\right|&\leq \sum_{l=2}^{2k-1}4l(l-1)|w_l|\|\theta-\theta_l^*\|_2^{2l-4}|\theta^{i}\theta^{j}|\leq C_{\text{cluster}}\sum_{l=1}^{2k-1}4l(l-1)R^{2l-2} \\
        &\leq 12C_{\text{cluster}}k^3\max\{1,(2R)^{4k-4}\},
    \end{align*}
    and for $i$ 
     \begin{align*}
        \left|\dfrac{\partial^2}{\partial\theta^i\partial\theta^i}P_m(\theta)\right|&\leq \sum_{l=2}^{2k-1}(2l)|w_l|\|\theta-\theta_l^*\|_2^{2l-4}(|\theta^{i}|^2 +(2l-2)\|\theta-\theta_l\|_2^2)\leq C_{\text{cluster}}\sum_{l=1}^{2k-1}(2l)(2l-1)R^{2l-2} \\
        &\leq 12C_{\text{cluster}}k^3\max\{1,(2R)^{4k-4}\}. 
    \end{align*}
    Combining these estimations, noting that $\|M\|_{\mathrm{op}}\leq d_{*}\max_{1\leq i,j\leq d_{*}}\{|M_{ij}|\}$, by choosing $C_{P,2}:=12d_*\cdot C_{\text{cluster}}k^3\max\{1,(2R)^{4k-4}\}$ we have 
    \begin{equation*}
\|D^2P_m(\theta)\|_{\mathrm{op}} \leq C_{P,2}. 
    \end{equation*}
\end{proof}
\section{Density Estimation Rate}
\label{sec:density_estimation_rate}
In this section, we provide proof for Proposition \ref{prop:density_estimation_rate} about the density estimation rate in Gaussian mixture models. We follow the schema from \cite{Vandegeer}, which utilize the ingredient from point process estimation and covering theory. 

\subsection{Basic notations for covering theory}
Before proceeding with the detailed argument, we first recall several basic notions from covering theory. For a detailed and systematic reference, please refer to \cite{Vandegeer}. 

For a metric space $(\mathcal{P}, d)$, an \textit{$\varepsilon$-net} of $(\mathcal{P}, d)$ is a collection of balls of radius $\varepsilon$ whose union covers $\mathcal{P}$. The \textit{covering number} $N(\varepsilon, \mathcal{P}, d)$ is defined as the minimal cardinality of such an $\varepsilon$-net, while the corresponding \textit{entropy number} is given by
$H(\varepsilon, \mathcal{P}, d)=
\log N(\varepsilon, \mathcal{P}, d).$

When $\mathcal{P}$ is a family of densities in Euclidean space $\mathbb{R}^{d}$, the metric $d$ is typically induced by the $\mathcal{L}^2(m)$ norm, where $m$ denotes the Lebesgue measure. In this setting, whenever no ambiguity arises, we abbreviate
$N(\varepsilon,\mathcal{P},\|\cdot\|_{\mathcal{L}^{2}(m)})$ and $H(\varepsilon,\mathcal{P},\|\cdot\|_{\mathcal{L}^{2}(m)})$ by
$N(\varepsilon,\mathcal{P},m)$ and $H(\varepsilon,\mathcal{P},m),$
respectively.

The \textit{bracketing number} $N_B(\varepsilon, \mathcal{P}, d)$ is defined as the smallest integer $n$ such that there exists $n$ pairs of functions
$\{(\underline{g}_i,\overline{g}_i)\}_{i=1}^n$
satisfying
$d(\underline{g}_i,\overline{g}_i) < \varepsilon$, and for every $f \in \mathcal{P}$, there exists some $i \in \{1,\dots,n\}$ such that
$\underline{g}_i \le f \le \overline{g}_i$. The associated \textit{bracketing entropy} is then defined by
$H_B(\varepsilon,\mathcal{P},d)=
\log N_B(\varepsilon,\mathcal{P},d)$. Similarly, when the metric is induced by the $\mathcal{L}^2(m)$ norm, we write $N_B(\varepsilon,\mathcal{P},m)$ and $H_B(\varepsilon,\mathcal{P},m)$
for the corresponding bracketing number and bracketing entropy.

In our problem, we denote $\Gamma = \Pi \times \Theta$, where the simplex of dimension $k-1$ is denoted by $\Pi := \{c_1,\ldots,c_k:c_i \geq 0\ , c_1+\ldots +c_k = 1\}$, and the expert space $\Theta = \{\theta\in \mathbb{R}^d:\|\theta\| \leq R\}$. To analyze the covering numbers and the convergence rates of our estimator, we consider
$\mathcal{D}(\Gamma) =\{G:G = \sum_{i=1}^k\pi_i \delta_{\theta_i},(\pi_1\cdots\pi_k) \in \Pi, \theta_i \in \Theta\}$ and $\mathcal{P}(\Gamma) = \{p_{G}:G \in \mathcal{D}(\Gamma)\}$, where we recall that 
\begin{equation*}
    p_{G}(x) = \sum_{i=1}^k \pi_i f(x|\theta_i, \Sigma),
\end{equation*}
here $f(\cdot|\theta,\Sigma) $ is multivariate Gaussian density of mean $\theta$ and covariance $\Sigma$. 

In addition, for the true underlying measure $G_* \in \mathcal{P}(\Gamma)$ to be estimated, we introduce the following notation:
\begin{equation*}
    \overline{p}_{G} = (p_{G} + p_{G_*})/2, \quad \quad   \overline{\mathcal{P}}(\Gamma) =\{\overline{p}_{G}: p_G \in \mathcal{P}(\Theta)\}
\end{equation*}
\begin{equation*}
    \overline{\mathcal{P}}^{1/2}(\Gamma) =\{\overline{p}^{1/2}_{G}: p_G \in \mathcal{P}(\Gamma)\},\quad \quad 
    \overline{\mathcal{P}}^{1/2}(\Gamma,\varepsilon) =\{\overline{p}^{1/2}_{G}: p_G \in \mathcal{P}(\Gamma),h(\overline{p}_{G},p_{G_*})\leq \varepsilon\}.
\end{equation*}
To quantify the connection between local bracket entropy and the complexity of a class of distribution, we define the \textit{bracket entropy integral}
\begin{equation*}
    \mathcal{J}_B\left(\varepsilon,  \overline{\mathcal{P}}^{1/2}(\Gamma,\varepsilon),m\right) = \left(\int_{\varepsilon^2/2^{13}}^{\varepsilon} H_B^{1/2}(u, \overline{\mathcal{P}}^{1/2}(\Gamma,\varepsilon),m)du\right)\vee \varepsilon,
\end{equation*}
where $u\vee \varepsilon = \max\{u,\varepsilon\}$.

\subsection{Proof for density estimation}

The main idea of the proof is to derive an upper bound for the bracketing entropy integral appearing in Lemma \ref{lemma:dung_bracket_integral_entropy_bound}. This constitutes the key ingredient for controlling the supremum of the empirical processes $\mu_n$ associated with the empirical distribution in Lemma \ref{lemma:dung_concentration_rate_of_empirical_process}. We then reduce the problem of bounding the Hellinger distance to that of controlling the corresponding empirical processes, whose supremum is analyzed separately over dyadic shells. This approach is commonly referred to as the \textit{peeling technique}.

\begin{lemma}
\label{lemma:dung_bracket_integral_entropy_bound}
    For a universal constant $J$, there exists $N \geq 0$ such that for all $n\geq N$ and $\varepsilon \geq (d\log(n)/n)^{1/2}$, we have 
    \begin{equation}
\label{eqn:bracket_entropy_estimation}
    \mathcal{J}_B\left(\varepsilon,  \overline{\mathcal{P}}^{1/2}(\Gamma,\varepsilon),m\right)\leq J\sqrt{n}\varepsilon^2.  
    \end{equation}
\end{lemma}
\begin{proof} 
\emph{Step 1 - Bounding bracketing number and bracketing entropy.} In this step, we derive estimations for the covering number $N$ and the bracketing entropy $H_B$ associated with the class $\mathcal{P}(\Gamma)$ under consideration. 
\begin{enumerate}
        \item [(1)] $\log N(\varepsilon,\mathcal{P}(\Gamma),\|\cdot\|_{\infty}) \lesssim d\log(1/\varepsilon)$.
        \item [(2)] $H_B(\varepsilon,\mathcal{P}(\Gamma),h) \lesssim d\log(1/\varepsilon)$.
\end{enumerate}
For (1), consider $\mathcal{E}_{\varepsilon}(\Gamma)$ be $\varepsilon$-net of $\Gamma$, i.e. a set such that for each $x \in \Theta$, there exists $\overline{x} \in \mathcal{E}_{\varepsilon}(\Gamma)$ such that $\|x-\overline{x} \|_2 \leq d\varepsilon$. Using similar argument in proof of Lemma 6 in \cite{ho2022gaussian}, we observe that a vigilant choice of $\varepsilon$-net can lead to $\log |\mathcal{E}_{\varepsilon}(\Gamma)| \lesssim d\log(1/\varepsilon)$. 

For the density space $\mathcal{P}(\Gamma)$, we consider the set $$\mathcal{E}_{\varepsilon}(\mathcal{P}(\Gamma)):= \{p_G:G = \sum_{i=1}^k \lambda_i\delta_{\theta_i},\ (\lambda,\theta_1,\ldots,\theta_k) \in \mathcal{E}_{\varepsilon}(\Gamma)\}.$$ Consider $p_G \in \mathcal{P}(\Theta)$, choose $G' \in \mathcal{E}_{\varepsilon}(\Gamma)$ such that $\|G-G'\|_2 \leq \varepsilon$. Then, using the fact that Lipschitz property is preserved through the sum and (inner) product of bounded function, we have $\|p_G-p_{G'}\|_\infty \lesssim \|G-G'\|_2 \leq \varepsilon$ uniformly. Thus, by reasonably scaling $\varepsilon$, we have $\mathcal{E}_{\varepsilon}(\mathcal{P}(\Gamma))$ is an $\varepsilon$-net of $\mathcal{P}(\Gamma)$. As a result, we have $\log N(\varepsilon,\mathcal{P}(\Gamma),\|\cdot\|_{\infty}) \lesssim d\log(1/\varepsilon)$. 

For (2), let $\eta \lesssim \varepsilon$ which will be chosen later, from Part 1, we construct an $\eta$-net of size $N$, namely $p_1,\ldots,p_N$. 
    
Noting that for any parameter $\theta \in \Theta = \bar{B}(0,R)$, we achieve a general bound for $f(\cdot|\theta,\Sigma)$
\begin{equation*}
    f(x|\theta,\Sigma) = (2\pi)^{-d/2}(\det\Sigma)^{-1/2}\exp\left(-\frac{1}{2}(x-\theta)^\top\Sigma^{-1}(x-\theta)\right) \leq (2\pi)^{-d/2}\lambda^{-d/2}_{\min}. 
\end{equation*}
In addition, as $\|\theta\|\leq R$, for any $\|x\| \geq 2R$, we have 
\begin{equation*}
    \|x-\theta\| \geq \|x\|-\|\theta\| \geq \|x\|-R\geq \|x\|/2. 
\end{equation*}
In addition, since the maximum eigenvalue of $\Sigma$ is less than $\lambda_{\max}$, we have 
\begin{equation*}
    (x-\theta)^{\top}\Sigma^{-1}(x-\theta) \geq \frac{\|x-\theta\|^2}{\lambda_{\max}} \geq \dfrac{\|x\|^2}{4\lambda_{\max}}. 
\end{equation*}
Thus, the Gaussian density satisfies the tail bound 
\begin{equation*}
    f(x|\theta,\Sigma) \leq (2\pi)^{-d/2}\lambda^{-d/2}_{\min}\exp\left(-\frac{\|x\|^2}{8\lambda_{\max}}\right). 
\end{equation*}
Thus, as $p_G$ is a mixture of Gaussian whose means belong to $\bar{B}(0,R)$ and variance $\Sigma$, the function $H:\mathbb{R}^d\to \mathbb{R}_{+}$ defined as
    \begin{equation*}
        H(x) = 
        \begin{cases}
            A, \text{ when } \|x\| \leq 2R,\\
            Ae^{-c\|x\|^2}, \text{ when }\|x\|> 2R,
        \end{cases}
    \end{equation*}
where $A = (2\pi\lambda_{\min})^{-d/2}$, $c = (8\lambda_{\max})^{-1}$, dominates all distributions in $\{p_{G},G \in \mathcal{P}(\Gamma)\}$, i.e. $p_G(x) \leq H(x),\forall G \in \mathcal{P}(\Gamma)$. For $i \in [1,N]$, we construct the bracket $[p^{L}_i, p^{U}_i]$ as 
    \begin{equation*}
            p^{L}_i(x) = \max\{f_i(x)-\eta,0\},\quad
            p^{U}_i(x) = \min\{f_i(x)+\eta,H(x)\}.
    \end{equation*}
    Then, we can easily verify that family of brackets $[p^{L}_i, p^{U}_i]$ cover the family of distribution: for each distribution $p_G$, there exists $i$ such that $p^{L}_i(\cdot) \leq p_G(\cdot) \leq p^{U}_i(\cdot)$. In addition, $0 \leq p^{U}_i(x) - p^{L}_i(x) \leq \min\{2\eta,H(x)\}$, which leads to 
\begin{equation}
\label{eqn:difference_between_bracket}
    \int_{\mathbb{R}^d} (p^{U}_i(x) - p^{L}_i(x))dx \leq \int_{\mathbb{R}^d}\min\{2\eta,H(x)\} dx:=I_{\eta}.
    \end{equation}
Now it is necessary to evaluate $I_{\eta}$ in the RHS of equation \eqref{eqn:difference_between_bracket}. Define $r_\eta$ by
$Ae^{-cr_\eta^2}=2\eta,$ or equivalently,
$$r_\eta^2
=
\frac{1}{c}\log\frac{A}{2\eta}
=
8\lambda_{\max}\log\frac{A}{2\eta}.$$
For $\eta < \frac{A}{2}\exp\left(-\frac{R^2}{16\lambda^2_{\max}}\right)$ , we have $r_\eta>2R$. Hence,
$$I_\eta=\int_{\|x\|\le r_\eta}2\eta\,dx+\int_{\|x\|>r_\eta}Ae^{-c\|x\|^2}\,dx.$$
For the first term,
$$\int_{\|x\|\le r_\eta}2\eta\,dx=2\eta\,\mathrm{Vol}(B(0,r_\eta))=2\eta v_d r_\eta^d = 2\eta v_d
\left(
8\lambda_{\max}\log\frac{A}{2\eta}
\right)^{d/2},$$
where $v_d$ denotes the volume of the unit ball in $\mathbb R^d$. For the second term, using the standard Gaussian tail estimate,
$$\int_{\|x\|>r_\eta}Ae^{c\|x\|^2}\,dx\lesssim\eta r_\eta^{d-2}\lesssim\eta\left(\log(1/\eta)\right)^{(d-2)/2}.
$$
Consequently,
$$\int_{\mathbb R^d}\min\{2\eta,H(x)\}\,dx \lesssim \eta\left(\log(1/\eta)\right)^{d/2}
$$
as $\eta < \frac{A}{2}\exp\left(-\frac{R^2}{16\lambda^2_{\max}}\right)$. This result implies that there exists a constant $c'$ such that 
\begin{equation*}
H_B(c'\eta(\log(1/\eta))^{d/2}) \leq \log(N) \lesssim d\log(1/\eta).
\end{equation*}
Thus, for $\varepsilon = c'\eta(\log(1/\eta))^{d/2}$, which implies that $\log(1/\varepsilon) \sim \log(1/\eta)$, we have $$ H_B(\varepsilon,\mathcal{P}(\Gamma),\|\cdot\|_1) \lesssim d\log(1/\varepsilon).$$ Also remember that $h^2 \leq \|\cdot\|_1$, we have for $\varepsilon \leq c_\varepsilon < 1/2$ sufficiently small, $$H_B(\varepsilon,\mathcal{P}(\Gamma),h) \lesssim d\log(1/\varepsilon).$$ 
When $1/2 > \varepsilon > c_\varepsilon > 0$, as $H_B(\varepsilon,\mathcal{P}(\Gamma),\|\cdot\|_1)$ is a decreasing function with respect to $\varepsilon$, we have $H_B(\varepsilon,\mathcal{P}(\Gamma),\|\cdot\|_1) \leq H_B(c_\varepsilon,\mathcal{P}(\Gamma),h)$, and $\log(1/\varepsilon) > \log(2)$, thus we also have $H_B(\varepsilon,\mathcal{P}(\Gamma),h) \lesssim d\log(1/\varepsilon)$. 
\newline 

\emph{Step 2 - Bounding bracket entropy integral. } We have the following estimation for $H_B$ term:
    \begin{align*}
        H_B(u, \overline{\mathcal{P}}^{1/2}(\Gamma,\varepsilon),m)& \overset{(a)}{\leq} H_B(u, \overline{\mathcal{P}}^{1/2}(\Gamma),m)  \overset{(b)}{=} H_B(\dfrac{u}{\sqrt{2}}, \overline{\mathcal{P}}(\Gamma),h)\\
        &\overset{(c)}{\leq} H_B(u, {\mathcal{P}}(\Gamma),h) \overset{(d)}{\lesssim}d \log(1/u). 
    \end{align*}
    Here, $(a)$ is due to $\overline{\mathcal{P}}^{1/2}(\Gamma,\varepsilon) \subset \overline{\mathcal{P}}^{1/2}(\Gamma)$, $(b)$ is from the definition of Hellinger distance, $(c)$ is due to the inequality $h^2\left(\frac{f_*+f_1}{2},\frac{f_*+f_2}{2} \right) \leq \frac{1}{2}h^2(f_1,f_2)$ for $f_*,f_1,f_2$ are densities, (Lemma 4.2, \cite{Vandegeer}). As a result, 
    \begin{equation*}
        \mathcal{J}_B\left(\varepsilon,  \overline{\mathcal{P}}^{1/2}(\Gamma,\varepsilon),m\right) \lesssim \sqrt{d} \int_{\varepsilon^2/2^{13}}^{\varepsilon}\sqrt{\log(1/u)}du \lesssim \sqrt{d}\varepsilon \sqrt{\log\dfrac{2^{13}}{\varepsilon^2}}\lesssim \sqrt{n}\varepsilon^2,
    \end{equation*}
    given that $\varepsilon\geq (d\log(n)/n)^{1/2}$ and $n$ sufficiently large. This proves equation \eqref{eqn:bracket_entropy_estimation}. 
\end{proof}

Before giving the proof for the consistency of MLE, we recall the following lemma, whose proof is a direct consequence of Theorem 5.11 and estimation (7.7), \cite{Vandegeer}.

\begin{lemma}
\label{lemma:dung_concentration_rate_of_empirical_process}
Let $R > 0$ and $k \geq 1$. Assume $C_1 < \infty$. Then, for any sufficiently large constant $C$, for all $n \in \mathbb{N}$ and any $t > 0$ satisfying
\begin{align}
t &\leq (8\sqrt{n}R) \wedge \left( {C_1 \sqrt{n} R^2} \right), \label{eqn:cond1} \\
t &\geq C^2 (C_1 + 1) \left( R \,\vee\, \int_{t/(2^6 \sqrt{n})}^{R} 
H_B^{1/2}\left(\frac{u}{\sqrt{2}}, \, \overline{\mathcal{P}}^{1/2}(\Gamma, R), m \right) du \right), \label{eqn:cond2}
\end{align}
we have the concentration bound
\begin{align}
\label{eqn:concentration_of_empirical_process}
\mathbb{P}_{G_*} \left(
\sup_{ h(\bar{p}_{G},\, p_{G_*}) \leq R}
|\mu_n(G)| \geq t
\right)
\leq C \exp\left(
- \frac{t^2}{C^2 (C_1 + 1) R^2}
\right).
\end{align}
Here, for $P_n$ be the empirical distribution from the sample $X_1,\ldots,X_n$ and $P$ is measure induced from density $p_{G_*}$, the empirical process $\mu_n$ is defined as 
\begin{equation}
\label{eqn:empirical_process}
    \mu_n := \left\{\mu_{n}(G) = \sqrt{n}\int \dfrac{1}{2}\boldsymbol{1}_{p_{G_*}> 0}\log (\overline{p}_G/p_{G_*})\ d(P_n -P):G\in \mathcal{D}(\Gamma)\right\}
\end{equation}
\end{lemma}

\begin{proof}[Proof of Proposition \ref{prop:density_estimation_rate}]
We prove that for all $\delta \geq \delta_{n}:= (d\log(n)/n)^{1/2}$, \begin{equation}
\label{eqn:tail_estimation}\mathbb{P}_{G_*}\left(h(p_{\widehat{G}_n},p_{G_*}) > \delta\right)\leq c\exp\left(-\dfrac{n\delta^2}{cd}\right),\  \forall G_* \in \mathcal{D}(\Gamma). 
\end{equation}

To prove this claim, with $\mu_n$ defined in equation~\eqref{eqn:empirical_process}, we first observe that 
\begin{equation*}
    \dfrac{1}{16}h^2(p_{\widehat{G}_n},p_{G_*}) \overset{(i)}{\leq} h^2(\overline{p}_{\widehat{G}_n},p_{G_*}) \overset{(ii)}{\leq} \dfrac{1}{\sqrt{n}}\mu_n(\widehat{G}_n). 
\end{equation*}
where $(i)$ is due to Lemma 4.2, $(ii)$ is due to Lemma 4.1, \cite{Vandegeer}. As a result, we have 
\begin{align*}
\mathbb{P}_{G_*}\left(h(p_{\widehat{G}_n},p_{G_*}) > \delta\right) &\leq  \mathbb{P}_{G_*}\left(\mu_n(\widehat{G}_n) - \sqrt{n}h^2(\overline{p}_{\widehat{G}_n},p_{G_*})\geq 0,\ h(\overline{p}_{\widehat{G}_n},p_{G_*}) > \delta\right)\\
&\leq \mathbb{P}_{G_*}\left(\sup_{G:h(\overline{p}_G,p_{G_*}) \geq \delta/4} [\mu_n(G)-\sqrt{n}h^2(\overline{p}_G,p_{G_*})] \geq 0\right)
\end{align*}
Now it is our time to use peeling technique as in the proof of Theorem 7.4, \cite{Vandegeer}. For any $\delta >0$ and $S:= \lceil \log_2(1/\delta)\rceil + 1$ be the smallest number satisfying $2^{S+1}\delta/4 >1$, we have 
\begin{align*}
\mathbb{P}_{G_*}\left(h(p_{\widehat{G}_n},p_{G_*}) > \delta\right) &\leq \sum_{s = 0}^S\mathbb{P}_{G_*}\left(\sup_{G:h(\overline{p}_G,p_{G_*}) \geq \delta/4} [\mu_n(G)-\sqrt{n}h^2(\overline{p}_G,p_{G_*})] \geq 0\right)\\
&\leq \sum_{s = 0}^S\mathbb{P}_{G_*} \left(\sup_{G:2^s\delta/4 \leq h(\overline{p}_G,p_{G_*}) \leq 2^{s+1}\delta/4} |\mu_n(G)| \geq \sqrt{n}2^{2s}(\delta/4)^2 \right).
\end{align*}
Thus, we have 
\begin{align*}
\mathbb{P}_{G_*}\left(h(p_{\widehat{G}_n},p_{G_*}) > \delta\right)\leq  \sum_{s = 0}^S\mathbb{P}_{G_0} \left(\sup_{G: h(\overline{p}_G,p_{G_*}) \leq 2^{s+1}\delta/4} |\mu_n(G)| \geq \sqrt{n}2^{2s}(\delta/4)^2 \right).
\end{align*}
At this stage, in order to apply Lemma \ref{lemma:dung_concentration_rate_of_empirical_process}, we verify its assumptions. We set
$R = 2^{s+1}\delta$, $C_1 = 15$,
and $t = n^{1/2}2^{2s}(\delta/4)^2 = (n^{1/2}R^2)/2^6$. With these choices, the condition in equation~\eqref{eqn:cond1} is immediately satisfied for every $s = 0,1,\ldots,S$. 


It remains to verify the condition in equation \eqref{eqn:cond2}. First, by change of variable, we have 
\begin{align*}
    \int^R_{t/(2^6\sqrt{n})}H_B^{1/2}\left(u/\sqrt{2},\overline{\mathcal{P}}^{1/2}(\Gamma,R),m\right)du &= \sqrt{2}\int^{R/\sqrt{2}}_{R^2/2^{12+1/2}}H_B^{1/2}\left(u,\overline{\mathcal{P}}^{1/2}(\Gamma,R),m\right)du\\
    &\leq  2\int^{R}_{R^2/2^{13}}H_B^{1/2}\left(u,\overline{\mathcal{P}}^{1/2}(\Gamma,R),m\right)du. 
\end{align*}
Consequently, we have 
\begin{align*}
R&\vee \int^R_{t/(2^6\sqrt{n})}H_B^{1/2}\left(u/\sqrt{2},\overline{\mathcal{P}}^{1/2}(\Gamma,R),m\right)du\\
&\leq 2\mathcal{J}_B\left(R,\overline{\mathcal{P}}^{1/2}(\Gamma,R),m\right)
\leq 2dJ\sqrt{n} R^2 = 2^7dJt. 
\end{align*}

As all conditions are justified, we can apply Lemma \ref{lemma:dung_concentration_rate_of_empirical_process}, which implies that
\begin{equation*}    \mathbb{P}_{G_*}\left(h(p_{\widehat{G}_n},p_{G_*}) > \delta\right) \leq C \sum_{s=0}^\infty \exp\left(-\dfrac{2^{2s}n\delta^2}{2^7Jd}\right) \leq c\exp\left(-\dfrac{n\delta^2}{cd}\right) \end{equation*}
for some universal constant $c$ and for $n$ sufficiently large. 
\end{proof}

\begin{lemma}
\label{lemma:small_wasserstein_implies_center_in_voronoi_cell} 
Suppose that $G = \sum_{i=1}^{k_*}\pi_i\delta_{\theta_{i}}$ and $G_* = \sum_{i=1}^{k_*}\pi^*_i\delta_{\theta^*_{i}}$ be two discrete measures in $\mathbb{R}^d$ with exactly $k_*$ mass for each of them, such that 

\begin{itemize}
    \item [(i)] (Non-vanishing mass) $\pi^*_{\min} = \min_{1\leq i \leq k_*}\pi^*_{i} > 0$.
    \item [(ii)] (Separation condition) For each $i\neq j$, we have $\|\theta_i - \theta_j\|\geq \Delta_{\mathrm{sep}}$.  
\end{itemize}
Then, if $W_1(G,G_*) < \pi^*_{\min}\Delta_{\mathrm{sep}}/4$, then for each $j \in [k_*]$, there exist $c(j)\in [k_*]$ such that $\|\theta_j - \theta^*_{c(j)}\|\leq \Delta_{\mathrm{sep}}/4$. 
\end{lemma}

\begin{proof}[Proof of Lemma \ref{lemma:small_wasserstein_implies_center_in_voronoi_cell}]
For each $j \in [k_*]$, set $B_{j} = \{\theta:\|\theta-\theta^*_j\|_2\leq \Delta_{\mathrm{sep}}/4\}$. Then from the separation condition, the balls $\{B_j\}_{1\leq j\leq k_*}$ are pair-wise disjoint. Suppose that there exists a support point $\theta_i$ in $G$ such that $\theta_i \notin \cup_{j=1}^{k_*}B_j$, then there exist a ball $B_j$ not containing any support points $\theta_l$ of $G$. Then, each map transporting $G_*$ to $G$ takes at least a distance of $\Delta_{\mathrm{sep}}/4$ to move $\theta^*_{j}$ to any point of the support $G$. As a result, the Wasserstein distance between $G_*$ and $G$ can be estimated as 
\begin{equation*}
    W_1(G,G_*) = W_1(G_*,G) \geq \pi_j \Delta_{\mathrm{sep}}/4 \geq \pi^*_{\min}\Delta_{\mathrm{sep}}/4. 
\end{equation*}
This implies that for $W_1(G,G_*) \leq \pi^*_{\min}\Delta_{\mathrm{sep}}/4$, each support point of $G$ belongs to exactly one of the balls $B_j$'s. 
    
\end{proof}

\bibliographystyle{plain}
\bibliography{references}
\end{document}